\documentclass[10.5pt,amsfonts]{amsart}

\usepackage{amsmath,amssymb,amsthm,enumerate}
\usepackage[all]{xy}
\usepackage{graphicx}
\usepackage{xcolor}
\usepackage[colorlinks=true]{hyperref}
\usepackage{comment}
\usepackage[T1]{fontenc}
\usepackage{enumitem}
\usepackage{mathrsfs}
\usepackage{tikz}

\SelectTips{eu}{12}

\newtheorem{thm}{Theorem}[section]
\newtheorem{prop}[thm]{Proposition}
\newtheorem{lem}[thm]{Lemma}
\newtheorem{cor}[thm]{Corollary}

\newtheorem{setting}[thm]{Setting}
\newtheorem{mthm}{Theorem}

\theoremstyle{definition}
\newtheorem{defn}[thm]{Definition}
\newtheorem{exa}[thm]{Example}

\theoremstyle{remark}
\newtheorem{rem}[thm]{Remark}

\numberwithin{equation}{section}

\newcommand{\QQ}{\mathbb{Q}}
\newcommand{\RR}{\mathbb{R}}
\newcommand{\ZZ}{\mathbb{Z}}
\newcommand{\NN}{\mathbb{N}}

\DeclareMathOperator{\cl}{\mathrm{cl}}
\DeclareMathOperator{\scl}{\mathrm{scl}}
\DeclareMathOperator{\Ker}{\mathrm{Ker}}
\DeclareMathOperator{\Aut}{\mathrm{Aut}}
\DeclareMathOperator{\Inn}{\mathrm{Inn}}
\DeclareMathOperator{\Out}{\mathrm{Out}}

\DeclareMathOperator{\asdim}{asdim}
\DeclareMathOperator{\asydim}{asdim}

\newcommand{\tor}{\mathrm{tor}}
\newcommand{\Gg}{G}
\newcommand{\Ng}{N}
\newcommand{\Lg}{L}
\newcommand{\Ff}{F}
\newcommand{\Mf}{M}
\newcommand{\Kf}{K}
\newcommand{\xg}{x}
\newcommand{\yg}{y}
\newcommand{\zg}{z}
\newcommand{\vf}{v}
\newcommand{\wf}{w}
\newcommand{\gG}{g}
\newcommand{\ff}{f}
\newcommand{\bff}{\underline{f}}
\newcommand{\bffi}{\underline{f}_i}
\newcommand{\bffone}{\underline{f}_1}
\newcommand{\bffell}{\underline{f}_{\ell}}
\newcommand{\QG}{\Gamma}
\newcommand{\nug}{\nu}
\newcommand{\muf}{\mu}
\newcommand{\psg}{\psi}
\newcommand{\phf}{\phi}
\newcommand{\kg}{k}
\newcommand{\hf}{h}
\newcommand{\pp}{p}
\newcommand{\ppi}{\pi}
\newcommand{\Rel}{R}

\newcommand{\WW}{\mathrm{W}}
\newcommand{\Wcal}{\mathcal{W}}
\newcommand{\PPh}{\Phi}
\newcommand{\PPs}{\Psi}
\newcommand{\Psig}{\sigma}
\newcommand{\Ptau}{\tau}
\newcommand{\CEt}{\tilde{\Theta}}
\newcommand{\CE}{\Theta}
\newcommand{\Rdim}{\dim_{\mathbb{R}}}
\newcommand{\DD}{\mathscr{D}}
\newcommand{\DDN}{\mathscr{D}_{\Ng}}
\newcommand{\DDL}{\mathscr{D}_{\Lg}}
\newcommand{\ODDN}{\hat{\mathscr{D}}_{\Ng}}
\newcommand{\ODDL}{\hat{\mathscr{D}}_{\Lg}}
\newcommand{\Ecal}{\mathcal{E}}

\newcommand{\om}{\omega}
\newcommand{\Coarse}{\mathbf{Coarse}}
\newcommand{\genus}{g}
\newcommand{\ttil}{\tilde{t}}
\newcommand{\util}{a'}

\newcommand{\bA}{\mathbf{A}}
\newcommand{\bB}{\mathbf{B}}

\newcommand{\bfalpha}{{\boldsymbol \alpha}}

\newcommand{\vm}{\vec{m}}
\newcommand{\vn}{\vec{n}}
\newcommand{\vu}{\vec{u}}
\newcommand{\va}{\vec{a}}
\newcommand{\vb}{\vec{b}}

\renewcommand{\Im}{\mathrm{Im}}

\newcommand{\Homeo}{\mathrm{Homeo}}

\newcommand{\IAA}{\mathrm{IA}}
\newcommand{\QQQ}{\mathrm{Q}}

\newcommand{\HHH}{\mathrm{H}}

\newcommand{\GL}{\mathrm{GL}}
\newcommand{\SL}{\mathrm{SL}}
\newcommand{\Sp}{\mathrm{Sp}}
\newcommand{\Mod}{\mathrm{Mod}}
\newcommand{\rot}{\mathrm{rot}}
\newcommand{\diam}{\mathrm{diam}}

\newcommand{\Hom}{\mathrm{Hom}_{\mathbb{R}}}

\makeatletter
\newsavebox{\@brx}
\newcommand{\llangle}[1][]{\savebox{\@brx}{\(\m@th{#1\langle}\)}%
  \mathopen{\copy\@brx\kern-0.5\wd\@brx\usebox{\@brx}}}
\newcommand{\rrangle}[1][]{\savebox{\@brx}{\(\m@th{#1\rangle}\)}%
  \mathclose{\copy\@brx\kern-0.5\wd\@brx\usebox{\@brx}}}
\makeatother

\allowdisplaybreaks[2]

\makeatletter
\@namedef{subjclassname@2020}{%
  \textup{2020} Mathematics Subject Classification}
\makeatother


\begin{document}

\title{Coarse group theoretic study on stable mixed commutator length}

\author[M. Kawasaki]{Morimichi Kawasaki}
\address[Morimichi Kawasaki]{Department of Mathematics, Faculty of Science, Hokkaido University, North 10, West 8, Kita-ku, Sapporo, Hokkaido 060-0810, Japan}
\email{kawasaki@math.sci.hokudai.ac.jp}

\author[M.~Kimura]{Mitsuaki Kimura}
\address[Mitsuaki Kimura]{Department of Mathematics, Osaka Dental University, 8-1 Kuzuha Hanazono-cho, Hirakata, Osaka 573-1121, Japan}
\email{kimura-m@cc.osaka-dent.ac.jp}

\author[S. Maruyama]{Shuhei Maruyama}
\address[Shuhei Maruyama]{School of Mathematics and Physics, College of Science and Engineering, Kanazawa University, Kakuma-machi, Kanazawa, Ishikawa, 920-1192, Japan}
\email{smaruyama@se.kanazawa-u.ac.jp}

\author[T. Matsushita]{Takahiro Matsushita}
\address[Takahiro Matsushita]{Department of Mathematical Sciences, Faculty of Science, Shinshu University, Matsumoto, Nagano, 390-8621, Japan}
\email{matsushita@shinshu-u.ac.jp}

\author[M. Mimura]{Masato Mimura}
\address[Masato Mimura]{Mathematical Institute, Tohoku University, 6-3, Aramaki Aza-Aoba, Aoba-ku, Sendai 980-8578, Japan}
\email{m.masato.mimura.m@tohoku.ac.jp}

\begin{abstract}
Let  $G$ be a group and $N$ a normal subgroup of $G$. We study the large scale behavior, not the exact values themselves, of the stable mixed commutator length $\mathrm{scl}_{G,N}$ on the mixed commutator subgroup $[G,N]$; when $N=G$, $\mathrm{scl}_{G,N}$ equals the stable commutator length $\mathrm{scl}_G$ on the commutator subgroup $[G,G]$. For this purpose, we regard $\mathrm{scl}_{G,N}$ not only as a function from $[G,N]$ to $\mathbb{R}_{\geq 0}$, but as a bi-invariant metric function $d^+_{\mathrm{scl}_{G,N}}\colon [G,N] \times [G,N]\to \mathbb{R}_{\geq 0}$ defined as  $d^+_{\mathrm{scl}_{G,N}}(x,y)=\mathrm{scl}_{G,N}(x^{-1}y)+1/2$ if $x\ne y$ and $d^+_{\mathrm{scl}_{G,N}}(x,y)=0$ if $x=y$ for $x,y\in [G,N]$. Our main focus is coarse group theoretic structures of $([G,N],d^+_{\mathrm{scl}_{G,N}})$. Our preliminary result (the absolute version) connects,  via the Bavard duality, $([G,N],d^+_{\mathrm{scl}_{G,N}})$ and the quotient vector space of the space of $G$-invariant quasimorphisms on $N$ over one of such homomorphisms. In particular, we prove that the dimension of  this vector space equals the asymptotic dimension of $([G,N],d^+_{\mathrm{scl}_{G,N}})$.

Our main result is the comparative version: we connect the coarse kernel, formulated by Leitner and Vigolo, of the coarse homomorphism $\iota_{G,N}$ from $([G,N],d^+_{\mathrm{scl}_{G,N}})$ to $([G,N],d^+_{\mathrm{scl}_{G}})$ sending $y\in [G,N]$ to $y$, and  a certain quotient vector space $\mathrm{W}(G,N)$ of the space of invariant quasimorphisms. Assume that $N=[G,G]$ and that $\mathrm{W}(G,N)$ is finite dimensional with dimension $\ell$. Then we prove that the coarse kernel of $\iota_{G,N}$ is isomorphic to $\mathbb{Z}^{\ell}$ as a coarse group. In contrast to  the absolute version, the space $\mathrm{W}(G,N)$ is finite dimensional in many cases, including all $(G,N)$  with finitely generated $G$ and nilpotent $G/N$.
As an application of our result, given a group homomorphism $\varphi\colon G\to  H$ between finitely generated groups, we define an $\mathbb{R}$-linear map `inside' the groups, which is dual to the naturally defined $\mathbb{R}$-linear map from $\mathrm{W}(H,[H,H])$ to $\mathrm{W}(G,[G,G])$ induced by $\varphi$.
\end{abstract}

\maketitle

\tableofcontents

\section{Introduction}

\subsection{Digests of our main results: absolute version and comparative version}\label{subsec=digest}

In the present paper, we study coarse geometry and coarse group structures of the mixed commutator subgroup $[\Gg,\Ng]$ equipped with the $\scl$-almost-metric $d_{\scl_{\Gg,\Ng}}$  (see the discussion at the beginning of Subsection~\ref{subsec=absolute} for the difference between the words `coarse geometry' and `coarse group structure' in this paper). More precisely, we study this in the \emph{comparative} aspect. Before presenting our results in the \emph{absolute} version (Proposition~\ref{prop=presclkika}) and the comparative version (Theorem~\ref{thm=main1}), we start from basic settings and motivations. Let $\Gg$ be a group and $\Ng$ a normal subgroup of $\Gg$. Let $\mathcal{S}_{\Gg,\Ng}=\{[\gG,\xg]=\gG\xg\gG^{-1}\xg^{-1}\,|\,\gG\in \Gg,\xg\in \Ng\}$ be the set of simple $(\Gg,\Ng)$-commutators. The \emph{mixed commutator subgroup} $[\Gg,\Ng]$ is the group generated by $\mathcal{S}_{\Gg,\Ng}$; the \emph{mixed commutator length} $\cl_{\Gg,\Ng}\colon [\Gg,\Ng]\to \ZZ_{\geq 0}$ is the word length with respect to $\mathcal{S}_{\Gg,\Ng}$. The \emph{stable mixed commutator length} $\scl_{\Gg,\Ng}\colon [\Gg,\Ng]\to \RR_{\geq 0}$ is the stabilization of $\cl_{\Gg,\Ng}$, more precisely, for every $\yg\in [\Gg,\Ng]$, we define
\[
\scl_{\Gg,\Ng}(\yg)=\lim_{n\to \infty}\frac{\cl_{\Gg,\Ng}(\yg^n)}{n}.
\]
For the case of $\Ng=\Gg$, $\cl_{\Gg,\Gg}$ and $\scl_{\Gg,\Gg}$ coincide with the (ordinary) commutator length $\cl_{\Gg}$ and stable commutator length $\scl_{\Gg}$, respectively. We refer the reader to \cite{Calegari} for backgrounds of $\scl_{\Gg}$, and to \cite{KKMMMsurvey} for a survey on $\scl_{\Gg,\Ng}$.

The main theme of the present paper is \emph{large scale behavior} of the stable mixed commutator length $\scl_{\Gg,\Ng}$. To study this, we regard $\scl_{\Gg,\Ng}$ not only as a function $\scl_{\Gg,\Ng}\colon [\Gg,\Ng]\to \RR_{\geq 0}$, but as an \emph{almost metric} on $[\Gg,\Ng]$. More precisely, we define the following \emph{$\scl_{\Gg,\Ng}$-almost-metric}
\begin{equation}\label{eq=dscl}
d_{\scl_{\Gg,\Ng}}\colon [\Gg,\Ng]\times [\Gg,\Ng]\to \mathbb{R}_{\geq 0};\quad (\yg_1,\yg_2)\mapsto d_{\scl_{\Gg,\Ng}}(\yg_1,\yg_2)=\scl_{\Gg,\Ng}(\yg_1^{-1}\yg_2),
\end{equation}
and we study the almost metric space $([\Gg,\Ng],d_{\scl_{\Gg,\Ng}})$. Here the terminology of the `almost metric' means that it satisfies the triangle inequality \emph{only up to a uniform additive error}; we describe in  Subsection~\ref{subsec=absolute} (and Subsection~\ref{subsec=scl-distance}) more details and motivations. The almost metric $d_{\scl_{\Gg,\Ng}}$ is bi-$[\Gg,\Ng]$-invariant, meaning that for every $\lambda,\lambda',\yg_1,\yg_2\in [\Gg,\Ng]$,
\begin{equation}\label{eq=biinv}
d_{\scl_{\Gg,\Ng}}(\lambda\yg_1\lambda',\lambda\yg_2\lambda')=d_{\scl_{\Gg,\Ng}}(\yg_1,\yg_2).
\end{equation}
Thus, we equip $([\Gg,\Ng],d_{\scl_{\Gg,\Ng}})$ with the \emph{coarse group structure}. Here coarse groups are group objects of the category of coarse spaces, and the theory of coarse groups has been recently developed by Leitner and Vigolo \cite{LV}; we briefly state basic definitions in Subsection~\ref{subsec=CK}. In this manner, our goal is to study \emph{coarse group structures} of $([\Gg,\Ng],d_{\scl_{\Gg,\Ng}})$.

The celebrated \emph{Bavard duality theorem} \cite{Bavard} connects $\scl_{\Gg}$ with the quotient vector space $\QQQ(\Gg)/\HHH^1(\Gg)$. Here, the symbol $\HHH^1(\Gg)$ means $\HHH^1(\Gg;\RR)$ (the first group cohomology), namely, the space of genuine ($\RR$-valued) homomorphisms on $\Gg$. The symbol $\QQQ(\Gg)$ denotes the vector space of \emph{homogeneous quasimorphisms} on $\Gg$: we say a map $\psg\colon \Gg\to \RR$ is a \emph{quasimorphism} if the defect of $\psg$,
\begin{equation}\label{eq=defect}
\DD(\psg)=\sup\left\{\left|\psg(\gG_1\gG_2)-\psg(\gG_1)-\psg(\gG_2)\right| \,\middle|\, \gG_1,\gG_2\in \Gg\right\}
\end{equation}
is finite. A function from $\Gg$ to $\RR$ is \emph{homogeneous} if its restriction to every cyclic subgroup is a homomorphism. The \emph{Bavard duality theorem for $\scl_{\Gg,\Ng}$}, proved in \cite{KKMM1}, generalizes this duality to that between $\scl_{\Gg,\Ng}$ and the quotient space $\QQQ(\Ng)^{\Gg}/\HHH^1(\Ng)^{\Gg}$. Here, we equip the space of $\RR$-valued functions on $\Ng$ with the $\Gg$-action by conjugation: for $\psg\colon \Ng\to \RR$ and for $\gG\in \Gg$, $(\gG\cdot \psg)(\xg)=\psg(\gG^{-1}\xg\gG)$ $(\xg\in \Ng)$. The symbol $\cdot^{\Gg}$ means the invariant part of the $\Gg$-action above. In this setting, the Bavard duality theorem for $\scl_{\Gg,\Ng}$,  \cite{KKMM1} and \cite{KK}, states that for every $\yg\in [\Gg,\Ng]$,
\begin{equation}\label{eq=KKMMBavard}
\scl_{\Gg,\Ng}(\yg)=\sup_{[\nug]\in \QQQ(\Ng)^{\Gg}/\HHH^1(\Ng)^{\Gg}} \frac{|\nug(\yg)|}{2\DD(\nug)}.
\end{equation}
holds, where $[\nug]$ is the equivalence class in $\QQQ(\Ng)^{\Gg}/\HHH^1(\Ng)^{\Gg}$ of $\nug$. Here, we regard $\scl_{\Gg,\Ng}\equiv 0$ if $\QQQ(\Ng)^{\Gg}/\HHH^1(\Ng)^{\Gg}=0$. We rephrase this as Theorem~\ref{thm=Bavard} in the present paper. In this manner, we have the coarse group $([\Gg,\Ng],d_{\scl_{\Gg,\Ng}})$ (`\emph{scl-side}') and the vector space $\QQQ(\Ng)^{\Gg}/\HHH^1(\Ng)^{\Gg}$ (`\emph{quasimorphism-side}'), and the Bavard duality theorem for $\scl_{\Gg,\Ng}$ \eqref{eq=KKMMBavard} connects these two concepts. Our first result, the \emph{absolute} version, determines the coarse group structure of $([\Gg,\Ng],d_{\scl_{\Gg,\Ng}})$ as long as $\Rdim \left(\QQQ(\Ng)^{\Gg}/\HHH^1(\Ng)^{\Gg}\right)<\infty$. We emphasize that Proposition~\ref{prop=presclkika} is \emph{not} our main result, as we state it as a proposition. Here we set $\ZZ^{0}=0$ if $\ell=0$.

\begin{prop}[absolute version]\label{prop=presclkika}
Let $\Gg$ be a group and $\Ng$ a normal subgroup of $\Gg$. Assume that $\QQQ(\Ng)^{\Gg}/\HHH^1(\Ng)^{\Gg}$ is finite dimensional, and  set $\ell=\Rdim \left(\QQQ(\Ng)^{\Gg}/\HHH^1(\Ng)^{\Gg}\right)$.
Then, the group $([\Gg,\Ng],d_{\scl_{\Gg,\Ng}})$  is isomorphic to $(\ZZ^{\ell},\|\cdot\|_1)$ \textup{(}equipped with the $\ell^1$-norm\textup{)} as a coarse group, and this coarse group isomorphism may be given by a quasi-isometry.
\end{prop}

However, Proposition~\ref{prop=presclkika} is \emph{not} as powerful as its appearance. Indeed, some of the readers might notice that there are many examples of group pairs $(\Gg,\Ng)$ for which the space $\QQQ(\Ng)^{\Gg}/\HHH^1(\Ng)^{\Gg}$ is infinite dimensional. For instance, if $\Gg$ is a non-elementary Gromov-hyperbolic group (more generally, an acylindrically hyperbolic group) and if $\Ng$ is an infinite normal subgroup, then this space is always \emph{infinite dimensional} (Example~\ref{exa=AH}). Proposition~\ref{prop=prescldim} stated in Subsection~\ref{subsec=absolute} implies that the asymptotic dimension of $([\Gg,\Ng],d_{\scl_{\Gg,\Ng}})$ is infinite in this case. Hence, $([\Gg,\Ng],d_{\scl_{\Gg,\Ng}})$ is `\emph{too huge}' as a coarse group in many cases.

Our main theorem is the \emph{comparative version} of the study on coarse group structures of $([\Gg,\Ng],d_{\scl_{\Gg,\Ng}})$. More precisely, we replace $\QQQ(\Ng)^{\Gg}/\HHH^1(\Ng)^{\Gg}$ (quasimorphism-side) and $([\Gg,\Ng],d_{\scl_{\Gg,\Ng}})$ ($\scl$-side) simultaneously with the following `comparative' versions.

\noindent
\underline{\emph{The space $\WW(\Gg,\Ng)$}.} We define the quotient vector space $\WW(\Gg,\Ng)$ as
\begin{equation}\label{eq=W-space}
\WW(\Gg,\Ng)=\QQQ(\Ng)^{\Gg}/\left(\HHH^1(\Ng)^{\Gg}+i^{\ast}\QQQ(\Gg)\right).
\end{equation}
Here the inclusion map $i\colon \Ng\hookrightarrow \Gg$ induces an $\RR$-linear map $i^{\ast}\colon \QQQ(\Gg)=\QQQ(\Gg)^{\Gg}\to \QQQ(\Ng)^{\Gg}$ (\emph{cf.} Lemma~\ref{lem=Ginv}); namely, for $\psg\in \QQQ(\Gg)$, $i^{\ast}\psg=\psg\circ i=\psg|_{\Ng}$. Hence, elements in $i^{\ast}\QQQ(\Gg)$ may be seen as  ($\Gg$-invariant) quasimorphisms on $\Ng$ \emph{that are extendable to $\Gg$}.

\noindent
\underline{\emph{The coarse kernel} of the map $\iota_{\Gg,\Ng}$.} We consider the following map
\begin{equation}\label{eq=iotaGN}
\iota_{\Gg,\Ng}\colon ([\Gg,\Ng],d_{\scl_{\Gg,\Ng}})\to ([\Gg,\Ng],d_{\scl_{\Gg}});\quad \yg\mapsto \yg.
\end{equation}
Clearly, this map $\iota_{\Gg,\Ng}$ is  an isomorphism in the category of groups. However, $\iota_{\Gg,\Ng}$ is \emph{not} necessarily an isomorphism in the category of \emph{coarse groups}. Indeed, despite the fact that $\iota_{\Gg,\Ng}$ is a \emph{coarse homomorphism} in the sense of \cite{LV} (homomorphism in the category of coarse groups), the group-theoretical inverse map $\iota_{\Gg,\Ng}^{-1}\colon ([\Gg,\Ng],d_{\scl_{\Gg}})\to ([\Gg,\Ng],d_{\scl_{\Gg,\Ng}});$ $\yg\mapsto \yg$ is \emph{not} necessarily a coarse homomorphism. More precisely, for every $\yg_1,\yg_2\in [\Gg,\Ng]$ we have
\begin{equation}\label{eq=sclnohikaku}
d_{\scl_{\Gg}}(\yg_1,\yg_2)\leq d_{\scl_{\Gg,\Ng}}(\yg_1,\yg_2);
\end{equation}
however, there is \emph{no} estimate in the converse direction to \eqref{eq=sclnohikaku}. In \cite{LV}, Leitner and Vigolo introduced the concept of the \emph{coarse kernel} of a coarse homomorphism: the coarse kernel $A$ of $\iota_{\Gg,\Ng}$ is the maximal subset up to coarse containments among all subsets of $[\Gg,\Ng]$ that are bounded in $d_{\scl_{\Gg}}$; we will discuss it in more detail in Subsection~\ref{subsec=comparative}. When $\iota_{\Gg,\Ng}$ is not an isomorphism in the category of coarse groups, it should be explained by the non-triviality of the coarse kernel ($\iota_{\Gg,\Ng}$ is always an epimorphism in the category of coarse groups).

Theorem~\ref{thm=main1} mentioned below is a special case of our main result, which will show up as Theorem~\ref{mthm=main} in Section~\ref{sec=intro2}. To state Theorem~\ref{thm=main1} and other results in this paper, the following formulation of group presentations is convenient.

\begin{defn}[`group presentations' in this paper]\label{defn=pres}
Let $\Gg$ be a group. Then, we say that $(\Ff\,|\,\Rel)$ is a \emph{presentation} of $\Gg$ if $\Ff$ is a free group, $\Rel$ is a normal subgroup of $\Ff$, and $\Gg \cong  \Ff/\Rel$ holds.
\end{defn}
A system $(S\,|\,\mathcal{R})$ is called a presentation of a group $\Gg$ in the standard literature if and only if $(\Ff(S) \,|\, \llangle \mathcal{R}\rrangle)$ is a presentation of $\Gg$ in the sense of Definition~\ref{defn=pres}. Here, $\Ff(S)$ denotes the free group with free basis $S$, and $\llangle \mathcal{R}\rrangle$ means the normal closure of $\mathcal{R}$ in $F(S)$.

\begin{thm}[comparative version: special case of our main theorem (Theorem~\ref{mthm=main})]\label{thm=main1}
Let $\Gg$ be a group and $\Ng$ a normal subgroup of $\Gg$. Let $i\colon \Ng\hookrightarrow \Gg$ be the inclusion map. Assume that either of the following two conditions is fulfilled:
\begin{enumerate}[label=\textup{(\alph*)}]
 \item $\Ng=[\Gg,\Gg]$; or
 \item  $\Ng\geqslant [\Gg,\Gg]$ and $\Gg$ admits a presentation $(\Ff\,|\,\Rel)$ \textup{(}in the sense of Definition~$\ref{defn=pres}$\textup{)} such that $\Rel \leqslant [\Ff,\Ff]$.
\end{enumerate}
Assume that $\WW(\Gg,\Ng)$ \textup{(}defined in \eqref{eq=W-space}\textup{)} is finite dimensional, and set $\ell=\Rdim \WW(\Gg,\Ng)$. Then the coarse kernel of the map $\iota_{\Gg,\Ng}$ \textup{(}defined in \eqref{eq=iotaGN}\textup{)} is isomorphic to $(\ZZ^{\ell},\|\cdot\|_1)$ as a coarse group. Furthermore, this coarse isomorphism may be given by a quasi-isometry.
\end{thm}

We note that in general a coarse kernel of a coarse homomorphism may \emph{not} exist. Therefore, the existence of the coarse kernel of $\iota_{\Gg,\Ng}$ itself is a part of the assertion of Theorem~\ref{thm=main1}. The vector space $\WW(\Gg,\Ng)$ is the cokernel of the map $\QQQ(\Gg)/\HHH^1(\Gg)\to \QQQ(\Ng)^{\Gg}/\HHH^1(\Ng)^{\Gg}$ induced by $i\colon \Ng \hookrightarrow \Gg$. By the Bavard duality \eqref{eq=KKMMBavard}, it might be natural to expect that this space relates to the coarse kernel of $\iota_{\Gg,\Ng}$. In addition, we note that the space $\WW(\Gg,\Ng)$ played a key role to a certain study in symplectic geometry in \cite{KKMM2}.

We emphasize that $\Rdim\WW(\Gg,\Ng)$ can be computed for several pairs $(\Gg,\Ng)$. The background here is the study of  $\WW(\Gg,\Ng)$  pursued in \cite{KKMMM}; there, we showed that $\WW(\Gg,\Ng)$ is isomorphic to a certain subspace of the \emph{ordinary} cohomology $\HHH^2(\Gg)$ (\emph{not} bounded cohomology $\HHH^2_b(\Gg)$), provided that $\QG=\Gg/\Ng$ is boundedly $3$-acyclic (Definition~\ref{defn=bdd_acyc}). Amenable groups, including all nilpotent groups, are boundedly $3$-acyclic. See Subsection~\ref{subsec=cohom} (Theorem~\ref{thm=KKMMMmain2} and related results) for more details. In contrast to the space $\QQQ(\Ng)^{\Gg}/\HHH^1(\Ng)^{\Gg}$, the space $\WW(\Gg,\Ng)$ is \emph{finite dimensional} in many cases. For instance, if either (a) or (b) in Theorem~\ref{thm=main1} is satisfied, then $\WW(\Gg,\Ng)$ is always finite dimensional as long as $\Gg$ is finitely generated.

In several examples, this dimension estimate enables us to compute $\Rdim \WW(\Gg,\Ng)$ without providing concrete linearly independent elements in $\WW(\Gg,\Ng)$. In particular, for several pairs $(\Gg,\Ng)$, we can show that $\WW(\Gg,\Ng)\ne 0$ by checking $\Rdim \WW(\Gg,\Ng)>0$; we can take this approach \emph{without constructing} any concrete element $\nug$ in $\QQQ(\Ng)^{\Gg}$ representing a non-zero element in $\WW(\Gg,\Ng)$. For instance, if we set  $\Gg$ as  the surface group $\pi_1(\Sigma_{\genus})$ of genus $\genus$,  with $\genus\in \NN_{\geq 2}$, and if we set $\Ng=[\Gg,\Gg]$, then we have
\begin{equation}\label{eq=surfacedim}
\Rdim \WW(\Gg,\Ng)=1
\end{equation}
(Theorem~\ref{thm=WW}~(2)). For this example, we had obtained \eqref{eq=surfacedim} in \cite{KKMMM} without providing any $\nug\in \QQQ(\Ng)^{\Gg}$  that represents a generator of the one-dimensional space $\WW(\Gg,\Ng)$. Later, a concrete example of such $\nug\in \QQQ(\Ng)^{\Gg}$ was constructed in \cite{MMM}. By Theorem~\ref{thm=main1} (for $(\Gg,\Ng)$ satisfying condition (a)), in this example the coarse kernel of $\iota_{\Gg,\Ng}$ is isomorphic to $(\ZZ,|\cdot|)$ as a coarse group. In Example~\ref{exa=surfacecup}~(2), we will also exhibit an example of $(\Gg,\Ng)$ with $\Ng\ne [\Gg,\Gg]$ for which condition (b) in Theorem~\ref{thm=main1} is fulfilled with positive $\ell$.

\subsection{Motivation of studying coarse geometry/coarse group structures of $([\Gg,\Ng],d_{\scl_{\Gg,\Ng}})$}\label{subsec=absolute}

Here, we describe our motivation of studying coarse geometry and the coarse group structure of $([\Gg,\Ng],d_{\scl_{\Gg,\Ng}})$ with providing basic definitions.  In this paper, we use the word `coarse geometry' for the study of coarse structures. More precisely, even for a group with a bi-invariant metric, we regard it just as a coarse space, but \emph{not} as a coarse group when we refer to as coarse structures. In Example~\ref{exa=CG}, we will describe a notable difference between the study of coarse structures alone and that of coarse group structures.

Specially for the case of $\Ng=\Gg$, the stable commutator length $\scl_{\Gg}$ has been studied by various researchers. Some importance of $\scl_{\Gg}$ comes from its geometric interpretation in terms of admissible surfaces (see for instance, \cite[Proposition~2.10]{Calegari}). In \cite[Section~4]{KKMM1}, some of the authors provided a geometric interpretation of $\scl_{\Gg,\Ng}$. By definition, the map $\scl_{\Gg,\Ng}\colon [\Gg,\Ng]\to \RR_{\geq 0}$ is \emph{semi-homogeneous}: for every $\yg\in [\Gg,\Ng]$ and for every $n\in \ZZ$,
\begin{equation}\label{eq=semihom}
\scl_{\Gg,\Ng}(\yg^n)=|n|\cdot \scl_{\Gg,\Ng}(\yg).
\end{equation}
Also in this aspect, the values of $\scl_{\Gg,\Ng}$ might be seen as some spectral-like invariant (such as the Lyapnov exponents in dynamical systems). The $\scl_{\Gg}$ (and $\scl_{\Gg,\Ng}$) have been studied intensively, for instance, in the following two directions:
\begin{itemize}
 \item (elementwise) study the exact value or number theoretic properties of it;
 \item (local structure) study when a sequence $(\yg_m)_{m\in \NN}$ in $[\Gg,\Ng]$ satisfies that $\lim\limits_{m\to\infty}\scl_{\Gg,\Ng}(\yg_m)=0$.
\end{itemize}
Results in the first direction include a well-known result by Calegari \cite{Calegari09}, stating that for every free group $F$, $\scl_{F}$ takes values in $\QQ$ and providing an explicit algorithm to compute the value. In the second direction, \emph{gap theorems} have been proved for instance in \cite{CaFu10}, \cite{BBF16} and \cite{Chen20} for several class of groups. These theorems assert the existence of a uniform gap of $\scl_{\Gg}(\gG)$ from $0$ over all $\gG\in [\Gg,\Gg]$ with $\scl_{\Gg}(\gG)\ne 0$, as well as they characterize elements $\gG\in [\Gg,\Gg]$ satisfying $\scl_{\Gg}(\gG)=0$.

Contrastingly, our focus in the present paper is the \emph{large scale geometry} of $([\Gg,\Ng],d_{\scl_{\Gg,\Ng}})$. Interests in this direction already appeared in the work \cite{CZ08} of Calegari and Zhuang in 2008 for the $\cl$-metric. There, they proved that if $\Gg$ is finitely presented, then $([\Gg,\Gg],d_{\cl_{\Gg}})$ is large scale simply connected. They also proved that if $\Gg$ is a torsion-free Gromov-hyperbolic group, then $([\Gg,\Gg],d_{\cl_{\Gg}})$ is one-ended, and has asymptotic dimension at least $2$. See \cite[Theorem~A, Theorem~B and Corollary~5.3]{CZ08} for details. One of the importance of large scale perspective naturally appears if we regard $\scl_{\Gg,\Ng}\colon [\Gg,\Ng]\to \RR_{\geq 0}$ as a `metric-like' function $d_{\scl_{\Gg,\Ng}}\colon [\Gg,\Ng]\times [\Gg,\Ng]\to \RR_{\geq 0}$ by \eqref{eq=dscl}. Despite the fact that $d_{\scl_{\Gg,\Ng}}$ is \emph{not} a genuine metric in general, by replacing $d_{\scl_{\Gg,\Ng}}$ with the function $d^+_{\scl_{\Gg,\Ng}}=d^{+,1/2}_{\scl_{\Gg,\Ng}}$, defined for every $\yg_1,\yg_2\in [\Gg,\Ng]$ as
\begin{equation}\label{eq=scl+half}
d_{\scl_{\Gg,\Ng}}^{+}(\yg_1,\yg_2)=\begin{cases}
d_{\scl_{\Gg,\Ng}}(\yg_1,\yg_2)+\frac{1}{2}, & \textrm{if }\yg_1\ne \yg_2,\\
0, & \textrm{if }\yg_1= \yg_2,
\end{cases}
\end{equation}
we obtain a genuine metric. By this, we say that $d_{\scl_{\Gg,\Ng}}$ is an \emph{almost} metric; see Definition~\ref{defn=generalizedmet} for details. This replacement process destroys the local structure; for instance, the new genuine metric space $([\Gg,\Ng], d^+_{\scl_{\Gg,\Ng}})$ is \emph{uniformly discrete}, meaning that, $\inf\{d^+_{\scl_{\Gg,\Ng}}(\yg_1,\yg_2)\,|\yg_1\ne\yg_2 \in [\Gg,\Ng]\}>0$. Nevertheless, the coarse structure of $([\Gg,\Ng], d^+_{\scl_{\Gg,\Ng}})$ is \emph{exactly the same} as that of $([\Gg,\Ng], d_{\scl_{\Gg,\Ng}})$. Thus, we may regard as if $([\Gg,\Ng], d_{\scl_{\Gg,\Ng}})$ were a genuine metric space, as long as we only study large scale geometry. We emphasize that the perspective of large scale geometry naturally shows up in this manner if we study $\scl_{\Gg,\Ng}$ not elementwise, but as a metric-like object.

On  almost metric spaces, the concepts of \emph{controlled maps}, \emph{closeness} between maps, \emph{quasi-isometries}, \emph{quasi-isometric embeddings}, \emph{coarse embeddings} and \emph{asymptotic dimensions} are defined and they have been intensively studied in coarse geometry (we review basic concepts including them in Subsections~\ref{subsec=coarsegeom} and \ref{subsec=asdim}; see \cite{Roe} for a comprehensive treatise on coarse geometry).  Here we only recall the definitions of controlled maps, closeness and quasi-isometric embeddings: let $(X,d_X)$ and $(Y,d_Y)$ be two almost metric spaces. A map $\alpha\colon (X,d_X)\to (Y,d_Y)$ is called a \emph{controlled map} if for every $S\in \RR_{\geq 0}$ there exists $T\in \RR_{\geq 0}$ such that for every $x,x'\in X$, $d_X(x,x')\leq S$ implies $d_Y(\alpha(x),\alpha(x'))\leq T$. Two maps $\alpha,\beta\colon (X,d_X)\to (Y,d_Y)$ are said to be \emph{close}, $\alpha\approx \beta$, if $\sup\limits_{x\in X}d_Y(\alpha(x),\beta(x))<\infty$ holds. A map $\alpha \colon (X,d_X)\to (Y,d_Y)$ is called a \emph{quasi-isometric embedding} if there exist $C_1,C_2\in \RR_{>0}$ and $D_1,D_2\in \RR_{\geq 0}$ such that for every $x,x'\in X$,
\[
C_1\cdot d_X(x,x')-D_1\leq d_Y(\alpha(x),\alpha(x'))\leq C_2\cdot d_X(x,x')+D_2
\]
holds. In particular, quasi-isometric embeddings are controlled maps.

As we mentioned in Subsection~\ref{subsec=digest}, an almost metric $d_{\scl_{\Gg,\Ng}}$ is bi-$[\Gg,\Ng]$-invariant (see \eqref{eq=biinv}. Moreover, if we regard $d_{\scl_{\Gg,\Ng}}$ as an almost  generalized metric on $\Gg$, then it is bi-$\Gg$-invariant; see Subsection~\ref{subsec=scl-distance}). Therefore, we can furthermore study the \emph{coarse group structure} of $([\Gg,\Ng],d_{\scl_{\Gg,\Ng}})$. The theory of coarse groups has been recently developed by Leitner and Vigolo \cite{LV}; we recall fundamental notions that will be used in the present paper in Subsection~\ref{subsec=CK}.  For two groups $(\Gg_1,d_1)$ and $(\Gg_2,d_2)$ equipped with bi-invariant almost metrics, we say that a map $\alpha\colon (\Gg_1,d_1)\to (\Gg_2,d_2)$ is a \emph{pre-coarse homomorphism} if
\[
\sup_{\gG,\gG'\in \Gg_1}d_2(\alpha(\gG\gG'),\alpha(\gG)\alpha(\gG'))<\infty
\]
holds. Leitner and Vigolo \cite{LV} formulated the concept of coarse homomorphisms: $\alpha\colon (\Gg_1,d_1)\to (\Gg_2,d_2)$ is a \emph{coarse homomorphism} if it is a pre-coarse homomorphism and if it is a controlled map; see also Remark~\ref{rem=bf} below.
 Two groups $(\Gg_1,d_1)$ and $(\Gg_2,d_2)$ equipped with bi-invariant almost metrics are isomorphic \emph{as coarse groups} if there exist two pre-coarse homomorphisms $\alpha\colon (\Gg_1,d_1)\to (\Gg_2,d_2)$ and $\beta\colon (\Gg_2,d_2)\to (\Gg_1,d_1)$ such that $\alpha$ and $\beta$ are both controlled maps (hence, coarse homomorphisms), and they are \emph{coarse inverse} to each other, meaning that, $\beta\circ \alpha\approx \mathrm{id}_{\Gg_1}$ and $\alpha\circ \beta\approx \mathrm{id}_{\Gg_2}$.

\begin{rem}\label{rem=bf}
Strictly speaking, the terminology of a coarse homomorphism in \cite{LV} is defined for a \emph{coarse map}, \emph{i.e.,} an equivalence class of controlled maps with respect to closeness (Definition~\ref{defn=coarse_map}). However, by abuse of notation, we use this terminology also for a set map. In the present paper, we write coarse notions in bold symbol (such as a coarse map $\bfalpha$), and we write set theoretical ones in non-bold symbol (such as a set map $\alpha$).
\end{rem}

Strong attentions have been paid to coarse geometric study on a finitely generated group $H$ equipped with the word metric with respect to a finite generating set. In contrast, our main object $([\Gg,\Ng],d_{\scl_{\Gg,\Ng}})$ is  \emph{locally infinite}, namely, bounded subsets can consist of infinitely many elements. However if we study coarse group theoretic objects coming from a genuine group, then the (almost) metric is supposed to be bi-invariant and hence locally infinite spaces naturally show up in general. Our work supplies new examples of coarse subgroups of interest, including one in Remark~\ref{rem=retract}.

With these backgrounds, we explain the precise form of  Proposition~\ref{prop=presclkika} as follows: if $\QQQ(\Ng)^{\Gg}/\HHH^1(\Ng)^{\Gg}$ is a finite dimensional $\RR$-vector space with dimension $\ell$, then there exist two maps  $\Psig\colon [\Gg,\Ng]\to \ZZ^{\ell}$ and $\Ptau\colon \ZZ^{\ell}\to [\Gg,\Ng]$ satisfying the following properties:
\begin{enumerate}[label=(\arabic*)]
  \item $\Psig\colon [\Gg,\Ng]\to (\ZZ^{\ell},\|\cdot\|_1)$ and $\Ptau\colon \ZZ^{\ell}\to ([\Gg,\Ng],d_{\scl_{\Gg,\Ng}})$ are both \emph{pre-coarse homomorphisms}.
  \item $\Psig\colon ([\Gg,\Ng],d_{\scl_{\Gg,\Ng}})\to (\ZZ^{\ell},\|\cdot\|_1)$ and $\Ptau\colon (\ZZ^{\ell},\|\cdot\|_1)\to ([\Gg,\Ng],d_{\scl_{\Gg,\Ng}})$ are both \emph{quasi-isometric embeddings}.
  \item $\Psig\colon ([\Gg,\Ng],d_{\scl_{\Gg,\Ng}})\to (\ZZ^{\ell},\|\cdot\|_1)$ and $\Ptau\colon (\ZZ^{\ell},\|\cdot\|_1)\to ([\Gg,\Ng],d_{\scl_{\Gg,\Ng}})$ are \emph{coarse inverse} to each other.
\end{enumerate}

We also have the following result, which characterizes $\Rdim\left(\QQQ(\Ng)^{\Gg}/\HHH^1(\Ng)^{\Gg}\right)$ in terms of coarse geometric properties of $[\Gg,\Ng]$. This result works even when $\Rdim\left(\QQQ(\Ng)^{\Gg}/\HHH^1(\Ng)^{\Gg}\right)=\infty$.
\begin{prop}[coarse geometric characterization of $\Rdim \left(\QQQ(\Ng)^{\Gg}/\HHH^1(\Ng)^{\Gg}\right)$]\label{prop=prescldim}
Let $\Gg$ be a group and $\Ng$ a normal subgroup of $\Gg$. Then,
\[
\Rdim \left(\QQQ(\Ng)^{\Gg}/\HHH^1(\Ng)^{\Gg}\right) =\asdim ([\Gg,\Ng],d_{\scl_{\Gg,\Ng}})
\]
and
\[
\Rdim \left(\QQQ(\Ng)^{\Gg}/\HHH^1(\Ng)^{\Gg}\right) \leq \asdim ([\Gg,\Ng],d_{\cl_{\Gg,\Ng}}).
\]
Here, $\asdim$ denotes the asymptotic dimension.
\end{prop}

See Subsection~\ref{subsec=asdim} for details on the asymptotic dimension; here we only note that $\asdim(\ZZ^{n},\|\cdot\|_1)=n$ for every $n\in \ZZ_{\geq 0}$ (we recall this as Theorem~\ref{thm=asdimZ}).
As we mentioned in Subsection~\ref{subsec=digest}, if $\Gg$ is an acylindrically hyperbolic group and if $\Ng$ is an infinite normal subgroup (we allow $\Ng=\Gg$), then $\QQQ(\Ng)^{\Gg}/\HHH^1(\Ng)^{\Gg}$ is infinite dimensional. Hence, Proposition~\ref{prop=prescldim} implies that $\asdim ([\Gg,\Ng],d_{\cl_{\Gg,\Ng}})=\infty$ in this case. This result greatly generalizes the aforementioned result \cite[Corollary~5.3]{CZ08}, stating $\asdim ([\Gg,\Gg],d_{\cl_{\Gg}})\geq 2$ for a torsion-free non-elementary Gromov-hyperbolic group $\Gg$. This is an application of our absolute versions.

\subsection{Coarse kernels and more details on the comparative version}\label{subsec=comparative}
First, let us explain  precisely what the `coarse kernel' of $\iota_{\Gg,\Ng}$ means in Theorem~\ref{thm=main1}. We note that $\iota_{\Gg,\Ng}$ is  a coarse homomorphism. The point here is the group theoretical inverse map $([\Gg,\Ng],d_{\scl_{\Gg}})\to ([\Gg,\Ng],d_{\scl_{\Gg,\Ng}})$; $\yg\mapsto \yg$ is \emph{not} necessarily a controlled map; this is why $\iota_{\Gg,\Ng}$ may \emph{not} be a monomorphism as a coarse homomorphism. In \cite{LV}, Leitner and Vigolo introduced the concept of \emph{coarse kernels} in the category of coarse groups. We briefly describe the definition in our setting;  see Examples~\ref{exa=CKgroup} and \ref{exa=CKiota} for more details. A subset $A$ of a set $X$ equipped with an almost metric $d_1$ is \emph{$d_1$-bounded} if the diameter of $A$ with respect to $d_1$ is finite (Definition~\ref{defn=bounded}). For an almost metric space $(X,d_2)$ and for $A,B\subseteq X$, $B$ is \emph{coarsely contained} in $A$ (in the almost metric $d_2$), written as $B\preccurlyeq_{d_2} A$, if there exists $D\in \RR_{\geq 0}$ such that $B\subseteq \mathcal{N}_{D,d_2}(A)$. Here, $\mathcal{N}_{D,d_2}(A)=\left\{x\in X\,\middle|\, \inf\limits_{a\in A}d_2(a,x)\leq D\right\}$. These sets $A$ and $B$ are \emph{asymptotic}, written as $A\asymp_{d_2} B$, if $A\preccurlyeq_{d_2} B$ and $B\preccurlyeq_{d_2} A$. The equivalence class with respect to $\asymp_{d_2}$ is called a \emph{coarse subspace}. Then, a subset $A$ of $[\Gg,\Ng]$ is a \emph{coarse kernel} of $\iota_{\Gg,\Ng}$ if we have
\begin{enumerate}[label=(\arabic*)]
  \item the set $A$ is $d_{\scl_{\Gg}}$-bounded; and
  \item the set $A$ is \emph{maximal} with respect to $\preccurlyeq_{d_{\scl_{\Gg,\Ng}}}$ among all $d_{\scl_{\Gg}}$-bounded sets in $[\Gg,\Ng]$. That is, for every $d_{\scl_{\Gg}}$-bounded subset $B$ of $[\Gg,\Ng]$, $B\preccurlyeq_{d_{\scl_{\Gg,\Ng}}} A$  holds.
\end{enumerate}
If such $A$ exists, then the coarse subspace $\mathbf{A}$ (with respect to $\asymp_{d_{\scl_{\Gg,\Ng}}}$) represented by $A$ is uniquely determined. Strictly speaking, $\mathbf{A}$ is referred to as \emph{the} coarse kernel of $\iota_{\Gg,\Ng}$;  recall Remark~\ref{rem=bf}.

We present the following strong form of Theorem~\ref{thm=main1}.
\begin{thm}[strong form of Theorem~\ref{thm=main1}: coarse group isomorphism class of the coarse kernel]\label{thm=main1precise}
We stick to the setting of Theorem~$\ref{thm=main1}$.
Then, there exist two maps $\PPh\colon ([\Gg,\Ng],d_{\scl_{\Gg,\Ng}})\to (\ZZ^{\ell},\|\cdot\|_1)$ and $\PPs\colon (\ZZ^{\ell},\|\cdot\|_1)\to ([\Gg,\Ng],d_{\scl_{\Gg,\Ng}})$ such that the following hold true.
\begin{enumerate}[label=\textup{(}$\arabic*$\textup{)}]
  \item The maps $\PPh$ and $\PPs$ are pre-coarse homomorphisms.
  \item The image $\PPs(\ZZ^{\ell})$ is $d_{\scl_{\Gg}}$-bounded.
  \item The map $\PPs$ is a quasi-isometric embedding. For every $d_{\scl_{\Gg}}$-bounded set $A\subseteq [\Gg,\Ng]$, the map $\PPh|_A$ is a quasi-isometric embedding.
  \item If we regard $\PPs$ as a map from $\ZZ^{\ell}$ to $\PPs(\ZZ^{\ell})$, then $\PPh|_{\PPs(\ZZ^{\ell})}$ and $\PPs$ are coarse inverse to each other.
  \item For every $d_{\scl_{\Gg}}$-bounded set $A\subseteq [\Gg,\Ng]$, $\sup\limits_{\yg\in A}d_{\scl_{\Gg,\Ng}}(\yg,(\PPs\circ \PPh)(\yg))<\infty$ holds.
\end{enumerate}
In particular, the coarse kernel of the map $\iota_{\Gg,\Ng}$ is the coarse subspace represented by $\PPs(\ZZ^{\ell})$.
\end{thm}

\begin{rem}\label{rem=zerodim}
In \cite[Theorem~2.1~(1)]{KKMMM}, the authors already showed that in general, $\scl_{\Gg}$ and $\scl_{\Gg,\Ng}$ are \emph{bi-Lipschitzly equivalent} on $[\Gg,\Ng]$ if $\WW(\Gg,\Ng)=0$. That is, there exists $C\in \RR_{\geq 1}$ such that for every $\yg\in [\Gg,\Ng]$, $\scl_{\Gg,\Ng}(\yg)\leq C\cdot \scl_{\Gg}(\yg)$ holds (if furthermore $\Ng\geqslant [\Gg,\Gg]$, then by \cite[Theorem~2.1~(3)]{KKMMM} $\scl_{\Gg}$ and $\scl_{\Gg,\Ng}$ in fact coincide on $[\Gg,\Ng]$). Hence, if $\ell=0$ in Theorem~\ref{thm=main1precise}, then the coarse kernel of $\iota_{\Gg,\Ng}$ is trivial (namely, represented by the trivial subgroup). Therefore, Theorem~\ref{thm=main1precise}  is new for the case where $\ell$ lies in $\NN=\ZZ_{>0}$.
\end{rem}

In the previous work \cite{KK}, \cite{KKMMM} and \cite{MMM},  the authors studied the problem whether $\scl_{\Gg}$ and $\scl_{\Gg,\Ng}$ are bi-Lipschitzly equivalent on $[\Gg,\Ng]$. Theorem~\ref{thm=main1precise} can be seen as a far-reaching `generalization' of the answers given in these results in the following two aspects.
\begin{enumerate}[label=(\arabic*)]
  \item We obtain a general machinery for comparison theorems of mixed $\scl$'s.  \item We formulate a refined problem.
\end{enumerate}
First, we explain the details of (1): in \cite{KK}, the first counterexample to this problem was given from symplectic geometry as follows. Let $\genus \in \NN_{\geq 2}$, and let $\Sigma_{\genus}$ denote the closed connected orientable surface of genus $\genus$. Equip $\Sigma_{\genus}$ with a symplectic form $\omega$. Let $\Gg=\mathrm{Symp}_0(\Sigma_{\genus},\omega)$, \emph{i.e.,} the identity component of the group of  symplectomorphisms of $(\Sigma_{\genus},\omega)$, and $\Ng=[\Gg,\Gg]$ (this group equals $\mathrm{Ham}(\Sigma_{\genus},\omega)$, \emph{i.e.,} the group of Hamiltonian diffeomorphisms of $(\Sigma_{\genus},\omega)$), then $\scl_{\Gg}$ and $\scl_{\Gg,\Ng}$ are not bi-Lipschitzly equivalent on $[\Gg,\Ng]$. Later, a similar example was provided with a smaller $\Gg$ in \cite[Example~7.15]{KKMMM} (based on \cite{KKMM2}). In \cite{MMM}, some of the authors proved that $\scl_{\Gg}$ and $\scl_{\Gg,\Ng}$ are not bi-Lipschitzly equivalent on $[\Gg,\Ng]$ if $\Gg$ is the surface group $\pi_1(\Sigma_{\genus})$ with $\genus\in \NN_{\geq 2}$ and $\Ng=[\Gg,\Gg]$: this is the first family of such examples for finitely generated $\Gg$. In all of these examples, the proof of the non-(bi-Lipschitz-)equivalence used information of  a concrete quasimorphism $\nug\in \QQQ(\Ng)^{\Gg}$ whose equivalence class $[\nug]\in \WW(\Gg,\Ng)$ is non-zero. In \cite{KK} and \cite{KKMMM}, Py's Calabi quasimorphism \cite{Py06} was employed as such $\nug$; in \cite{MMM}, some of the authors constructed such $\nug$ from a surface group action on a circle (see also \cite{KKMMMsurvey} for the constructions of these invariant quasimorphisms). By Remark~\ref{rem=zerodim}, if $\scl_{\Gg}$ and $\scl_{\Gg,\Ng}$ are not bi-Lipschitzly equivalent on $[\Gg,\Ng]$, then $\WW(\Gg,\Ng)\ne 0$. A natural question may be to ask whether the converse holds in general. Recall from Subsection~\ref{subsec=digest} that showing $\WW(\Gg,\Ng)\ne 0$ by the computation of $\Rdim \WW(\Gg,\Ng)$ is considerably easier than constructing a concrete element $\nug$ in $\QQQ(\Ng)^{\Gg}$ such that the equivalence class $[\nug]\in\WW(\Gg,\Ng)$ is non-zero. Therefore, this question is a challenge.

Nevertheless, the conclusions of Theorem~\ref{thm=main1precise} in particular imply that $\scl_{\Gg}$ and $\scl_{\Gg,\Ng}$ are not bi-Lipschitzly equivalent on $[\Gg,\Ng]$; see the discussion below on (2). The point here is that our general machinery in the proof of Theorem~\ref{thm=main1precise} does \emph{not} require  any information of concrete quasimorphisms $\nug$ satisfying $[\nug]\ne 0$. We call the heart of this machinery the theory of \emph{core extractors}, and will discuss it in Sections~\ref{sec=CE} and \ref{sec=CEabel}; see Subsection~\ref{subsec=org} for the organization of the present paper. As a corollary to our main results, we state as Corollary~\ref{cor=notbilip} a partial answer to the natural question above in a more general framework (we switch from a pair $(\Gg,\Ng)$ to a triple $(\Gg,\Lg,\Ng)$; see Subsection~\ref{subsec=mainthm}). In particular, if $\Ng=[\Gg,\Gg]$, then $\WW(\Gg,\Ng)\ne 0$ implies the non-(bi-Lipschitz-)equivalence of $\scl_{\Gg}$ and $\scl_{\Gg,\Ng}$ on $[\Gg,\Ng]$.

Next, we explain (2): by \cite[Chapter~7]{LV}, once we have the coarse kernel of a coarse homomorphism, we can apply the isomorphism theorem in the category of coarse groups. In particular, in the setting of Theorem~\ref{thm=main1precise}, if we set the coarse kernel of $\iota_{\Gg,\Ng}$ as $\mathbf{A}$, then we have the coarse group isomorphism
\begin{equation}\label{eq=coarsequot}
([\Gg,\Ng],d_{\scl_{\Gg,\Ng}})/\mathbf{A} \cong ([\Gg,\Ng],d_{\scl_{\Gg}}).
\end{equation}
If $\iota_{\Gg,\Ng}$ has a non-trivial coarse kernel, then $\scl_{\Gg}$ and $\scl_{\Gg,\Ng}$ are \emph{not} bi-Lipschitzly equivalent on $[\Gg,\Ng]$. Furthermore, the coarse group isomorphism \eqref{eq=coarsequot} explains \emph{the} difference between $([\Gg,\Ng],d_{\scl_{\Gg,\Ng}})$ and $([\Gg,\Ng],d_{\scl_{\Gg}})$. Hence, coarse group theoretic language, specially the coarse kernel, supplies the right formulation for the comparison problem of mixed $\scl$'s in a large scale aspect.

The following theorem is a special case of Theorem~\ref{mthm=dim} in Section~\ref{sec=intro2}.

\begin{thm}[special case of Theorem~\ref{mthm=dim}: coarse characterization of $\Rdim \WW(\Gg,\Ng)$]\label{thm=main2}
Let $\Gg$ be a group and $\Ng$ a normal subgroup of $\Gg$. Assume that either of the conditions \textup{(a)} and \textup{(b)} in Theorem~$\ref{thm=main1}$ is fulfilled. Then, we have
\[
\Rdim \WW(\Gg,\Ng)=\sup\{\asdim (A,d_{\scl_{\Gg,\Ng}})\,|\, A\subseteq [\Gg,\Ng] \mathrm{\ is\ }d_{\scl_{\Gg}}\textrm{-}\mathrm{bounded}\}
\]
and
\[
\Rdim \WW(\Gg,\Ng)\leq \sup\{\asdim (A,d_{\cl_{\Gg,\Ng}})\,|\, A\subseteq [\Gg,\Ng] \mathrm{\ is\ }d_{\cl_{\Gg}}\textrm{-}\mathrm{bounded}\}.
\]
\end{thm}

\subsection{Application of determining the coarse kernel of $\iota_{\Gg,\Ng}$}\label{subsec=intro_main}
We have already explained  in \eqref{eq=coarsequot} an importance of coarse kernels. Here, we present a further theoretical importance of Theorem~\ref{thm=main1precise}. In the present paper, we use the following notation for the lower central series, where we set $\NN=\{1,2,\ldots\}$.

\begin{defn}[formulation of the lower central series]\label{defn=lowercentral}
Let $\Gg$ be a group. Then, the  \emph{lower central series} $(\gamma_q(\Gg))_{q\in \NN}$ of $\Gg$ is defined as $\gamma_1(\Gg)=\Gg$ and for every $q\in \NN$, $\gamma_{q+1}(\Gg)=[\Gg,\gamma_q(\Gg)]$.
\end{defn}

Let $\Gg$ and $H$ be two groups, and let $\varphi\colon \Gg\to H$ be a group homomorphism. It is then straightforward to see that $\varphi$ induces an $\RR$-linear map $T^{\varphi}\colon \WW(H,\gamma_2(H))\to \WW(\Gg,\gamma_2(\Gg))$.
Now we furthermore assume that $\WW(\Gg,\gamma_2(\Gg))$ and $\WW(H,\gamma_2(H))$ are both finite dimensional (this holds if $\Gg$ and $H$ are finitely generated; see Corollary~\ref{cor=nilp}). Set $\ell=\Rdim\WW(\Gg,\gamma_2(\Gg))$ and $\ell'=\Rdim\WW(H,\gamma_2(H))$. Then, by using coarse kernels, we can define an \emph{$\RR$-linear map} $S_{\varphi}\colon \RR^{\ell}\to \RR^{\ell'}$ that is \emph{dual} to $T^{\varphi}$ in an appropriate sense. The outline of the construction of this map $S_{\varphi}$ goes as follows. Let $\mathbf{A}_{\Gg}$ and $\mathbf{A}_{H}$ be the coarse kernels of $\iota_{\Gg,\gamma_2(\Gg)}$ and $\iota_{H,\gamma_2(H)}$, respectively. They are coarse subspaces of $(\gamma_3(\Gg),d_{\scl_{\Gg,\gamma_2(\Gg)}})$ and $(\gamma_3(H),d_{\scl_{H,\gamma_2(H)}})$, and by Theorem~\ref{thm=main1precise}, they are isomorphic to $(\ZZ^{\ell},\|\cdot\|_1)$ and $(\ZZ^{\ell'},\|\cdot\|_1)$, and hence to $(\RR^{\ell},\|\cdot\|_1)$ and $(\RR^{\ell'},\|\cdot\|_1)$ as coarse groups, respectively. By the universality of the coarse kernel $\mathbf{A}_{H}$, $\varphi$ induces a coarse homomorphism $\mathbf{S}_{\varphi}\colon \mathbf{A}_{\Gg}\to \mathbf{A}_{H}$  (recall Remark~\ref{rem=bf}). If we fix coarse isomorphisms $\mathbf{A}_{\Gg}\cong (\RR^{\ell},\|\cdot\|_1)$ and $\mathbf{A}_{H}\cong (\RR^{\ell'},\|\cdot\|_1)$, then we obtain a unique representative of $\mathbf{S}_{\varphi}$ by an $\RR$-linear map (we will argue in Lemma~\ref{lem=crushing}); this $\RR$-linear map is defined to be $S_{\varphi}\colon \RR^{\ell}\to \RR^{\ell'}$.

The $\RR$-linear map $T^{\varphi}$ is the map induced on function spaces, and there is no mystery to obtain an $\RR$-linear map. Contrastingly, $S_{\varphi}$ is a map `induced on subsets of groups,' and there is a priori \emph{no}  real vector space structure. Nevertheless, by Theorem~\ref{thm=main1precise} we are able to find such structures to define the \emph{$\RR$-linear map} $S_{\varphi}$. Furthermore, there is a certain duality between $S_{\varphi}$ and $T^{\varphi}$ (the \emph{coarse duality formula} in  Theorem~\ref{thm=duality}), which enables us to switch from function spaces to coarse kernels, and vice versa. We discuss this in Subsection~\ref{subsec=induced}.

As an application of this theory, we in particular obtain the following \emph{crushing theorem}.

\begin{thm}[special case of crushing theorem (Theorem~\ref{thm=crushing})] \label{thm=crushingq2}
Let $\Gg$ and $H$ be groups, and let $\varphi \colon G \to H$ be a group homomorphism. Assume that $\WW(\Gg, \gamma_2(\Gg))$ and  $\WW(H, \gamma_2(H))$ are both finite dimensional. Set $\ell=\Rdim\WW(\Gg, \gamma_2(\Gg))$ and $\ell'=\Rdim\WW(H, \gamma_2(H))$. Assume furthermore that $\ell>\ell'$.
Then, there exists $X \subseteq \gamma_{3}(\Gg)$ satisfying the following three conditions:
\begin{enumerate}[label=\textup{(\arabic*)}]
\item the set $X$ is \emph{not} $d_{\scl_{\Gg, \gamma_{2}(\Gg)}}$-bounded;
\item the set $X$ is $d_{\scl_\Gg}$-bounded; and
\item the image $\varphi(X)$ is $d_{\scl_{H},\gamma_{2}(H)}$-bounded.
\end{enumerate}
Moreover, there exists $X \subseteq \gamma_{3}(\Gg)$ satisfying \textup{(2)} and \textup{(3)} such that $\asdim(X,d_{\scl_{\Gg, \gamma_{2}(\Gg)}})\geq \ell-\ell'$.
\end{thm}
We will develop this theory in a much broader framework in a forthcoming work.

\section{Main results and outlined proofs}\label{sec=intro2}
Our main results, Theorem~\ref{mthm=main}, Theorem~\ref{mthm=dim} and Theorem~\ref{mthm=dimfin}, may be seen as comparative versions of study on coarse group structures of stable mixed commutator lengths for \emph{triples} $(\Gg,\Lg,\Ng)$. In Subsection~\ref{subsec=mainthm}, we present them. In Subsection~\ref{subsec=CKapplication}, we exhibit examples  of the coarse kernels. In Subsection~\ref{subsec=org}, we provide the outlines of the proofs of Proposition~\ref{prop=sclkika} (absolute version) and Theorem~\ref{mthm=main} (comparative version). Finally, we describe the organization of the present paper.

\subsection{Main theorems}\label{subsec=mainthm}

We can generalize the comparative version, Theorem~\ref{thm=main1precise}, to a triple $(\Gg,\Lg,\Ng)$, where $\Lg$ and $\Ng$ are normal subgroups of $\Gg$ with $\Lg \geqslant \Ng$. In this case, we  compare $([\Gg,\Ng],d_{\scl_{\Gg,\Lg}})$ and $([\Gg,\Ng],d_{\scl_{\Gg,\Ng}})$. For this purpose, we present the following generalizations of the vector space $\WW(\Gg,\Ng)$ and the map $\iota_{\Gg,\Ng}$ to the setting of group triples.

\begin{defn}[the vector space $\Wcal(\Gg,\Lg,\Ng)$]\label{defn=thespaceV}
Let $\Gg$ be a group, and let $\Lg$ and $\Ng$ be two normal subgroups of $\Gg$ with $\Lg\geqslant \Ng$. The inclusion map $i\colon \Ng \hookrightarrow \Lg$ induces an $\RR$-linear map $i^{\ast}\colon \QQQ(\Lg)^{\Gg}\to \QQQ(\Ng)^{\Gg}$; $\psg\mapsto \psg|_{\Ng}$. The vector space $\Wcal(\Gg,\Lg,\Ng)$ is defined as \[
\Wcal(\Gg,\Lg,\Ng)=\QQQ(\Ng)^{\Gg}/\left(\HHH^1(\Ng)^{\Gg}+i^{\ast}\QQQ(\Lg)^{\Gg}\right).
\]

For the case where $\Lg=\Gg$, we write $\WW(\Gg,\Ng)$ for $\Wcal(\Gg,\Gg,\Ng)$.
\end{defn}

\begin{defn}[the map $\iota_{(\Gg,\Lg,\Ng)}$]\label{defn=iota}
We stick to the setting of Definition~$\ref{defn=thespaceV}$. Define the map $\iota_{(\Gg,\Lg,\Ng)}$ as
\[
\iota_{(\Gg,\Lg,\Ng)}\colon ([\Gg,\Ng],d_{\scl_{\Gg,\Ng}})\to ([\Gg,\Ng],d_{\scl_{\Gg,\Lg}});\ \yg\mapsto \yg.
\]

For the case where $\Lg=\Gg$, then we write $\iota_{\Gg,\Ng}$ for $\iota_{(\Gg,\Gg,\Ng)}$ (as in \eqref{eq=iotaGN}).
\end{defn}

By Lemma~\ref{lem=Ginv}, the formulation of $\WW(\Gg,\Ng)$ in Definition~\ref{defn=thespaceV} coincides with \eqref{eq=W-space}. In Theorem~\ref{thm=main1precise}, we assume that the group pair $(\Gg,\Ng)$ satisfies either (a) or (b). We will generalize this to a group triple $(\Gg,\Lg,\Ng)$. In our main theorem (Theorem~\ref{mthm=main}), we are able to treat the case where $\Ng$ and $\Lg$ differ \emph{by one step} in the lower central series of $\Gg$. Recall our convention of the lower central series from Definition~\ref{defn=lowercentral}.

\begin{mthm}[main theorem 1: determining coarse group isomorphism class of the coarse kernel]\label{mthm=main}
Let $\Gg$ be a group, and let $\Lg$, $\Ng$ be two normal subgroups of $\Gg$ with $\Lg\geqslant \Ng$. Assume that there exists $q\in \NN_{\geq 2}$ such that either of the following two conditions is fulfilled:
\begin{enumerate}
  \item[\textup{(a$_q$)}]  $\Ng=\gamma_{q}(\Gg)$ and $\Lg=\gamma_{q-1}(\Gg)$; or
  \item[\textup{(b$_q$)}]   $\Ng\geqslant \gamma_{q}(\Gg)$, $\Lg=\gamma_{q-1}(\Gg)\Ng$, and $\Gg$ admits a group presentation  $(\Ff\,|\,\Rel)$  such that $\Rel\leqslant \gamma_{q}(\Ff)$.
\end{enumerate}
Assume that $\Wcal(\Gg,\Lg,\Ng)$ is finite dimensional, and set $\ell=\Rdim \Wcal(\Gg,\Lg,\Ng)$. Then there exist two maps $\PPh\colon ([\Gg,\Ng],d_{\scl_{\Gg,\Ng}})\to (\ZZ^{\ell},\|\cdot\|_1)$ and $\PPs\colon (\ZZ^{\ell},\|\cdot\|_1)\to ([\Gg,\Ng],d_{\scl_{\Gg,\Ng}})$ such that the following hold true.
\begin{enumerate}[label=\textup{(}$\arabic*$\textup{)}]
  \item The maps $\PPh$ and $\PPs$ are pre-coarse homomorphisms.
  \item The image $\PPs(\ZZ^{\ell})$ is $d_{\scl_{\Gg,\Lg}}$-bounded.
  \item The map $\PPs$ is a quasi-isometric embedding.
  \item For every $d_{\scl_{\Gg,\Lg}}$-bounded set $A\subseteq [\Gg,\Ng]$, $\PPh|_A$ is a quasi-isometric embedding.
  \item For every $d_{\scl_{\Gg,\Lg}}$-bounded set $A\subseteq [\Gg,\Ng]$, $\sup\limits_{\yg\in A}d_{\scl_{\Gg,\Ng}}(\yg,(\PPs\circ \PPh)(\yg))<\infty$.
  \item We have $\PPh\circ \PPs\approx \mathrm{id}_{(\ZZ^{\ell},\|\cdot\|_1)}$. That is, $\sup \left\{\left\|(\PPh\circ \PPs)(\vm)- \vm\right\|_1\,\middle|\,\vm\in \ZZ^{\ell}\right\}<\infty$.
  \item The coarse subspace represented by $\PPs(\ZZ^{\ell})$ is the coarse kernel of the map $\iota_{(\Gg,\Lg,\Ng)}$.
\end{enumerate}
In particular, if we regard $\PPs$ as a map from $\ZZ^{\ell}$ to $\PPs(\ZZ^{\ell})$, then $\PPh|_{\PPs(\ZZ^{\ell})}$ and $\PPs$ give coarse isomorphisms between  the coarse kernel of $\iota_{(\Gg,\Lg,\Ng)}$ and $(\ZZ^{\ell},\|\cdot\|_1)$.

Furthermore, we may take these maps $\PPh$ and $\PPs$ in such a way that $\PPh\circ \PPs=\mathrm{id}_{\ZZ^{\ell}}$ holds.
\end{mthm}
Similar to Remark~\ref{rem=zerodim}, the main case of Theorem~\ref{mthm=main} is that of $\ell\in \NN$; see also Theorem~\ref{thm=biLip}.

The next result is the second main theorem of this paper, which treats the coarse group theoretic characterization of the dimension of the space $\Wcal(\Gg,\Lg,\Ng)$. We say a map $\alpha\colon (X,d_X)\to (Y,d_Y)$ between almost metric spaces is \emph{coarsely proper} (Definition~\ref{defn=several_maps} and Example~\ref{exa=several_maps}) if for every $S \in \RR_{>0}$ there exists $T\in \RR_{>0}$ such that for every $x_1,x_2\in X$, $d_X(x_1,x_2) > T$ implies $d_Y(\alpha(x_1),\alpha(x_2)) > S$; we say $\alpha$ is \emph{$d_Y$-bounded} (Definition~\ref{defn=bounded}) if $\alpha(X)$ is $d_Y$-bounded.

\begin{mthm}[main theorem~2: coarse group theoretic characterization of $\Rdim \Wcal(\Gg,\Lg,\Ng)$]\label{mthm=dim}
Let $\Gg$ be a group, and let $\Lg$, $\Ng$ be two normal subgroups of $\Gg$ with $\Lg\geqslant \Ng$. Assume that there exists $q\in \NN_{\geq 2}$ such that either of the conditions \textup{(a$_q$)} and \textup{(b$_q$)} in Theorem~{\textup{\ref{mthm=main}}} is fulfilled. Then, we have
\begin{align*}
&\Rdim \Wcal(\Gg,\Lg,\Ng)\\
={}&\sup\left\{\ell\in \ZZ_{\geq 0}\,\middle|\,\begin{gathered}\exists \mathrm{\ coarsely\ proper\ }d_{\scl_{\Gg,\Lg}}\textrm{-}\mathrm{bounded\ coarse\ homomorphism}\\ (\ZZ^{\ell},\|\cdot\|_1)\to ([\Gg,\Ng],d_{\scl_{\Gg,\Ng}})
\end{gathered}\right\} \\
={}&\inf\left\{\ell\in \ZZ_{\geq 0}\,\middle|\,\begin{gathered}\forall A\subseteq [\Gg,\Ng]\ d_{\scl_{\Gg,\Lg}}\textrm{-}\mathrm{bounded;}\\
\exists \mathrm{\ coarsely\ proper\ coarse\ homomorphism}\ (A,d_{\scl_{\Gg,\Ng}})\to (\ZZ^{\ell},\|\cdot\|_1)
\end{gathered}\right\}
\end{align*}
and
\begin{align*}
&\Rdim \Wcal(\Gg,\Lg,\Ng)\\
\leq{} &\sup\left\{\ell\in \ZZ_{\geq 0}\,\middle|\,\begin{gathered} \exists \mathrm{\ coarsely\ proper\ }d_{\cl_{\Gg,\Lg}}\textrm{-}\mathrm{bounded\ coarse\ homomorphism}\\ (\ZZ^{\ell},\|\cdot\|_1)\to ([\Gg,\Ng],d_{\cl_{\Gg,\Ng}})
\end{gathered}\right\}
\end{align*}
In particular, we  have
\[
\Rdim \Wcal(\Gg,\Lg,\Ng)=\sup\{\asdim (A,d_{\scl_{\Gg,\Ng}})\,|\, A\subseteq [\Gg,\Ng] \mathrm{\ is\ }d_{\scl_{\Gg,\Lg}}\textrm{-}\mathrm{bounded}\}
\]
and
\[
\Rdim \Wcal(\Gg,\Lg,\Ng)\leq \sup\{\asdim (A,d_{\cl_{\Gg,\Ng}})\,|\, A\subseteq [\Gg,\Ng] \mathrm{\ is\ }d_{\cl_{\Gg,\Lg}}\textrm{-}\mathrm{bounded}\}.
\]
\end{mthm}

Theorem~\ref{thm=main1precise} and Theorem~\ref{thm=main2} immediately follow from Theorem~\ref{mthm=main} and Theorem~\ref{mthm=dim} for $q=2$, respectively.

Theorem~\ref{mthm=dim} assumes that either of the conditions (a${}_q$) and (b${}_q$) is satisfied for some $q\in \NN_{\geq 2}$. If we drop this assumption and treat the general case, then at present  we only have an inequality in one direction in Theorem~\ref{mthm=dim}.

\begin{mthm}[main theorem~3: dimension estimate for general cases]\label{mthm=dimfin}
Let $\Gg$ be a group, and let $\Lg$, $\Ng$ be two normal subgroups of $\Gg$ with $\Lg\geqslant \Ng$. Then, we have
\begin{align*}
\Rdim \Wcal(\Gg,\Lg,\Ng)
\geq &\inf\left\{\ell\in \ZZ_{\geq 0}\,\middle|\,\begin{gathered}\forall A\subseteq [\Gg,\Ng]\ d_{\scl_{\Gg,\Lg}}\textrm{-}\mathrm{bounded;}\\
\exists \mathrm{\ coarsely\ proper\ coarse\ homomorphism}\\
(A,d_{\scl_{\Gg,\Ng}})\to (\ZZ^{\ell},\|\cdot\|_1)
\end{gathered}\right\}.
\end{align*}
In particular, we have
\[
\Rdim \Wcal(\Gg,\Lg,\Ng)\geq \sup\{\asdim (A,d_{\scl_{\Gg,\Ng}})\,|\, A\subseteq [\Gg,\Ng] \mathrm{\ is\ }d_{\scl_{\Gg,\Lg}}\textrm{-}\mathrm{bounded}\}.
\]
\end{mthm}

Theorem~\ref{mthm=main} and Theorem~\ref{mthm=dim} have more general forms, Theorem~\ref{thm=maingeneral} and Theorem~\ref{thm=dimgeneral}, respectively. A natural question might be to ask whether we can treat the case for the triple of the form $(\Gg,\gamma_{p}(\Gg),\gamma_q(\Gg))$ for $p,q\in \NN$ with $p<q$. We will address this question in a forthcoming work.

\begin{rem}\label{rem=retract}
In the conclusions of Theorem~\ref{mthm=main}, take two maps $\PPh$ and $\PPs$ such that moreover $\PPh\circ \PPs=\mathrm{id}_{\ZZ^{\ell}}$. Set $A=(\PPs\circ \PPh)([\Gg,\Ng])(=\PPs(\ZZ^{\ell}))$, and view $\PPs$ as a map from $\ZZ^{\ell}$ to $A$.
Then, this set $A$ is a representative of the coarse kernel of $\iota_{(\Gg,\Lg,\Ng)}$. This $A\subseteq [\Gg,\Ng]$ might be seen as a `retract,' and the map $\PPs\circ \PPh\colon [\Gg,\Ng]\to A$ might be seen as a `retraction.' More precisely, $\PPs\circ \PPh$ is a coarse homomorphism and $(\PPs\circ \PPh)|_A=\mathrm{id}_A$. The point we stress here is that $(\PPs\circ \PPh)|_A$ is \emph{not merely close to $\mathrm{id}_A$, but identical to $\mathrm{id}_A$.} This might have a potential application in future to study the coarse kernel of $\iota_{(\Gg,\Lg,\Ng)}$ \emph{set theoretically}. For instance,  we can define the  coarse group multiplication `$\bullet$' and the coarse group inverse $c$ on $A$ set theoretically as follows: for every $\vm,\vn\in \ZZ^{\ell}$,
\[
\PPs(\vm) \bullet \PPs(\vn)=\PPs(\vm+\vn)\quad \textrm{and}\quad c(\PPs(\vm))=\PPs(-\vm).
\]
Then, $((A,d_{\scl_{\Gg,\Ng}}),\bullet,e_{\Gg},c)$ is isomorphic to $(\ZZ^{\ell},\|\cdot\|_1)$ as a coarse group. This $\bullet$ does not coincide with the genuine group multiplication. In fact, since $A$ is $d_{\scl_{\Gg,\Lg}}$-bounded, even starting from a genuine group, we can construct a (set theoretical) coarse `subgroup' that is far from any genuine subgroup in general.
\end{rem}

\begin{exa}\label{exa=CG}
Define $S_1^{\sharp}\colon (\RR_{\geq 0},|\cdot|)\to (\mathbb{C},|\cdot|)$ by $S_1^{\sharp}(u)=ue^{\sqrt{-1}\log u}$ for every $u\in \RR_{>0}$ and $S_1^{\sharp}(0)=0$. Then, since $\big|(S_1^{\sharp})'(u)\big|\leq \sqrt{2}$ for $u\in \RR_{>0}$, we have
\[
|u_1-u_2|\leq |S_1^{\sharp}(u_1)-S_1^{\sharp}(u_2)|\leq \sqrt{2}|u_1-u_2|
\]
for every $u_1,u_2\in \RR_{\geq 0}$; in particular, $S_1^{\sharp}$ is a  bi-Lipschitz map. Since the natural map $\rho \colon (\mathbb{C},|\cdot|)\to (\RR^2,\|\cdot\|_1)$ is a bi-Lipschitz map, so is $S_2^{\sharp}=\rho \circ S_1^{\sharp}$. Now, define a map $S^{\sharp}\colon (\RR,|\cdot|)\to (\RR^4,\|\cdot\|_1)$ by
\[
S^{\sharp}(u)=\begin{cases}
(S_2^{\sharp}(|u|), 0,0), & \textrm{if }u\in \RR_{\geq 0},\\
(0,0,S_2^{\sharp}(|u|)) & \textrm{if }u\in \RR_{<0}
\end{cases}
\]
for every $u\in \RR$. Then, we can show that this map $S^{\sharp}$ is a bi-Lipschitz map. In particular, if we regard $(\RR,|\cdot|)$ and $(\RR^4,\|\cdot\|_1)$ as coarse spaces, \emph{not} equipped with the coarse group structures, then this map $S^{\sharp}$ is a controlled map that is coarsely proper; in other words, $S^{\sharp}$ is a coarse embedding. The existence of such a map $S^{\sharp}$ makes the study of coarse embeddings from $(\RR,|\cdot|)$ to $(\RR^4,\|\cdot\|_1)$ (or, even that of quasi-isometric embeddings from $(\RR,|\cdot|)$ to $(\RR^4,\|\cdot\|_1)$) up to closeness quite difficult.

However, if we regard $(\RR,|\cdot|)$ and $(\RR^4,\|\cdot\|_1)$ as coarse groups, then the map $S^{\sharp}$ is \emph{not} a coarse homomorphism. In fact, Lemma~\ref{lem=crushing} implies that every coarse homomorphism $\mathbf{S}\colon(\RR,|\cdot|)\to(\RR^4,\|\cdot\|_1)$ (as a coarse map; recall Remark~\ref{rem=bf}) admits a unique $\RR$-linear representative. Therefore, the classification of coarsely proper coarse homomorphisms  from $(\RR,|\cdot|)$ to $(\RR^4,\|\cdot\|_1)$ is straightforward: such a coarse map corresponds bijectively to a non-zero vector $\xi$ in $\RR^4$ by the $\RR$-linear representative $\RR\ni u\mapsto u\xi\in \RR^4$.

This example suggests that for a subset $A$ of a coarse group, the coarse group structure of $A$ (if it exists) be much finer than the coarse structure of $A$.
\end{exa}

\subsection{Concrete examples and applications of coarse kernels}\label{subsec=CKapplication}

We will see several examples of Theorem~\ref{mthm=main} in Section~\ref{sec=examplemain}; in Section~\ref{sec=remarks}, we also see applications of the coarse kernels in Theorem~\ref{mthm=main}. Here, we present some concrete examples from these two sections. For simplicity, we consider the setting of pairs $(\Gg,\Ng)$ (Theorem~\ref{thm=main1precise}). In some cases, we are able to obtain an explicit representative of the coarse kernel; see discussions above Theorem~\ref{thm=WW} for terminology appearing in Proposition~\ref{prop=explicitexample}. Here, we recall the \emph{Andreadakis--Johnson filtration} \cite{Andreadakis} of the automorphism group of a free group.  Let $n\in \NN_{\geq 2}$ and let $F_n$ denote a free group of rank $n$. For every $q\in \NN$, the action of $\Aut(F_n)$ on the $(q+1)$-st nilpotent group quotient $F_n/\gamma_{q+1}(F_n)$ induces the group homomorphism $\Aut(F_n)\to \Aut (F_n/\gamma_{q+1}(F_n))$. Its kernel is denoted by $\mathcal{A}_{n}(q)$; thus, we have the descending filtration
\begin{equation}\label{eq=Andreadakis}
\Aut(F_n) \geqslant \mathcal{A}_{n}(1)\geqslant \mathcal{A}_{n}(2)\geqslant \cdots\geqslant \mathcal{A}_{n}(q)\geqslant  \mathcal{A}_{n}(q+1)\geqslant\cdots.
\end{equation}
We note that $\mathcal{A}_{n}(1)$ equals $\IAA_n$, the group of IA-automorphisms of $F_n$.

\begin{prop}[examples with explicit coarse kernels]\label{prop=explicitexample}
For a group $\Gg$ and a normal subgroup $\Ng$ of $\Gg$, let $\iota_{\Gg,\Ng}$ be the map defined as \eqref{eq=iotaGN}.
\begin{enumerate}[label=\textup{(\arabic*)}]
  \item \textup{(}surface group\textup{)} Let $\genus\in \NN_{\geq 2}$. Let $\Gg=\pi_1(\Sigma_{\genus})$ be the surface group of genus $\genus$:
\[
\Gg= \langle a_1,  \cdots,a_{\genus}, b_1, \cdots, b_{\genus} \, | \, [a_1, b_1] \cdots [a_{\genus}, b_{\genus}]=e_G \rangle.
\]
Let $\Ng=[\Gg,\Gg]$. Set $A=\left\{[a_1, b_1^m] \cdots [a_{\genus}, b_{\genus}^m]\,\middle|\, m\in \ZZ\right\}$ and
\[
\PPs\colon \ZZ\to A;\ m\mapsto [a_1, b_1^m] \cdots [a_{\genus}, b_{\genus}^m].
\]
Then, $A\subseteq [\Gg,\Ng]$ represents the coarse kernel of $\iota_{\Gg,\Ng}$, and it is isomorphic to $(\ZZ,|\cdot|)$  by the coarse group isomorphism $\PPs$.
  \item \textup{(}free-by-cyclic group\textup{)} Let $n\in \NN_{\geq 2}$ and let $F_n$ denote a free group of rank $n$ with free basis $\{a_1,\ldots,a_n\}$. Let $\chi\in \Aut(F_n)$ be an atoroidal automorphism. Assume that $\chi$ belongs to the group $\mathcal{A}_n(2)$ appearing in the Andreadakis--Johnson filtration \eqref{eq=Andreadakis}. Let $\Gg$ be the semi-direct product $F_n\rtimes_{\chi}\ZZ$ associated with the action $\ZZ\curvearrowright F_n$ by application of powers of $\chi$. Let $\Ng=[\Gg,\Gg]$.

Set $A^{\flat}=\left\{\chi(a_1)^{m_1}a_1^{-m_1} \cdots \chi(a_n)^{m_n}a_n^{-m_n}\,\middle|\, (m_1,\ldots ,m_n)\in \ZZ^n\right\}$ and
\[
\PPs^{\flat}\colon \ZZ^n\to A^{\flat};\ (m_1,\ldots,m_n)\mapsto \chi(a_1)^{m_1}a_1^{-m_1} \cdots \chi(a_n)^{m_n}a_n^{-m_n}.
\]
Then, $A^{\flat}\subseteq [\Gg,\Ng]$ represents the coarse kernel of $\iota_{\Gg,\Ng}$, and it is isomorphic to $(\ZZ^n,\|\cdot\|_1)$  by the coarse group isomorphism $\PPs^{\flat}$.
\end{enumerate}
\end{prop}
In the setting of Proposition~\ref{prop=explicitexample}~(2), we will also have another representative $A'{}^{\flat}$ for the coarse kernel of $\iota_{\Gg,\Ng}$ in Remark~\ref{rem=appli}. Under a weaker assumption that the atoroidal automorphism $\chi$ belongs to the IA-automorphism group $\IAA_n$, we may have an explicit expression of a representative of the coarse kernel, but the expression will be more complicated. We will see this in Example~\ref{exa=explicitker}~(3). We will also exhibit a representative of the coarse kernel in the setting of a certain mapping torus in Example~\ref{exa=explicitker}~(4). 

Proposition~\ref{prop=explicitexample} has the following applications.  Recall from \eqref{eq=W-space} (the definition of $\WW(\Gg,\Ng)$) that  $\nug\in \QQQ(\Ng)^{\Gg}$ represents the zero element in $\WW(\Gg,\Ng)$ if and only if there exists $\kg\in \HHH^1(\Ng)^{\Gg}$ such that the invariant homogeneous quasimorphism $\nug-\kg$ is extendable to $\Gg$.

\begin{prop}[characterization of extendability up to invariant homomorphisms for surface groups and free-by-cyclic groups]\label{prop=extexample}
The following hold true.
\begin{enumerate}[label=\textup{(\arabic*)}]
  \item Let $(\Gg,\Ng)$ and $A$ be as in Proposition~$\ref{prop=explicitexample}$~\textup{(1)}. Let $\nug\in \QQQ(\Ng)^{\Gg}$. Then, $\nug$ represents the zero element in $\WW(\Gg,\Ng)$ if and only if $\nug$ is bounded on $A$.
  \item Let $(\Gg,\Ng)$, $a_1,\ldots,a_n$, $\chi$ be as in Proposition~$\ref{prop=explicitexample}$~\textup{(2)}. Let $\nug\in \QQQ(\Ng)^{\Gg}$. Then, $\nug$ represents the zero element in $\WW(\Gg,\Ng)$ if and only if for every $i\in \{1,\ldots,n\}$, $\nug$ is bounded on the set $\left\{\chi(a_i)^ma_i^{-m} \,\middle| \,m\in \ZZ\right\}$.
\end{enumerate}
\end{prop}

\subsection{Outline of the proofs}\label{subsec=org}

In this subsection, we will present an outline of the proof of Theorem~\ref{mthm=main}. Theorem~\ref{mthm=main} is the comparative version, and its counterpart in the absolute version is Proposition~\ref{prop=presclkika}; in Section~\ref{sec=absolute}, we state the precise version of Proposition~\ref{prop=presclkika} as Proposition~\ref{prop=sclkika}. The proof of Proposition~\ref{prop=presclkika} may be seen as a `prototype' of that of Theorem~\ref{mthm=main}. Hence, before proceeding to the outlined proof of Theorem~\ref{mthm=main}, we  first present that of Proposition~\ref{prop=presclkika}. We summarize the correspondence between absolute and comparative versions in Table~\ref{table=abcom}.

\begin{table}[h]
\caption{correspondence between absolute and comparative versions}
\label{table=abcom}
\begin{tabular}{c|c|c}
 \hline  & absolute version & comparative version\\
 \hline \hline
\begin{tabular}{c}
 theorem \\
 (coarse isomorphisms) 
\end{tabular} &
\begin{tabular}{c}
 Proposition~\ref{prop=sclkika} \\
 (Proposition~\ref{prop=presclkika}) 
\end{tabular} &
\begin{tabular}{c}
 Theorem~\ref{mthm=main},  \\
 Theorem~\ref{thm=maingeneral} (general form),\\
 Theorems~\ref{thm=main1precise} and  \ref{thm=main1} ($\Lg=\Gg$)
\end{tabular} \\
 \hline
 \begin{tabular}{c}
  theorem \\
 (dimension)
\end{tabular} &
\begin{tabular}{c}
  Proposition~\ref{prop=metricdim}  \\
 (Proposition~\ref{prop=prescldim}) 
\end{tabular} &
\begin{tabular}{c}
 Theorem~\ref{mthm=dim}, \\
 Theorem~\ref{thm=dimgeneral} (general form), \\
  Theorem~\ref{thm=main2} ($\Lg=\Gg$), \\
  (Theorem~\ref{mthm=dimfin}: one direction)
\end{tabular} \\
 \hline \hline
$\scl$ side & $([\Gg,\Ng],d_{\scl_{\Gg,\Ng}})$ & 
\begin{tabular}{c}
 the coarse kernel of $\iota_{(\Gg,\Lg,\Ng)}$\\
 $([\Gg,\Ng],d_{\scl_{\Gg,\Ng}})\to ([\Gg,\Ng],d_{\scl_{\Gg,\Lg}})$ \\
\end{tabular} \\
\hline
quasimorphism side & $\QQQ(\Ng)^{\Gg}/\HHH^1(\Ng)^{\Gg}$ & 
\begin{tabular}{c}
$\Wcal(\Gg,\Lg,\Ng)$:\\
 the cokernel of the map \\
 $\QQQ(\Lg)^{\Gg}/\HHH^1(\Lg)^{\Gg}\to \QQQ(\Ng)^{\Gg}/\HHH^1(\Ng)^{\Gg}$
\end{tabular}\\
\hline \hline
\begin{tabular}{c}
 evaluation map to $\RR^{\ell}$ \\
 (Step~$1'$/Step~1)
\end{tabular} & 
Proposition~\ref{prop=evmap} ($\Psig^{\RR}$) & 
\begin{tabular}{c}
Theorem~\ref{thm=evmap} ($\PPh^{\RR}$)\\
 for $(\Gg,\Lg,\Ng)$ with $\Rdim \Wcal(\Gg,\Lg,\Ng)=\ell$\\
 (using Theorem~\ref{thm=comparisonDD})
\end{tabular}  \\
  \hline
\begin{tabular}{c}
 map from $\ZZ^{\ell}$ \\
 (Step~$2'$/Step~2)
\end{tabular} & 
Proposition~\ref{prop=dimQ} ($\Ptau$) & 
\begin{tabular}{c}
Theorem~\ref{thm=PPs} ($\PPs$)\\
 for the `abelian case' (\eqref{eq=abelian})\\
 (using Corollary~\ref{cor=InjCE}) \medskip
 \end{tabular}  \\ 
 form of the map & 
 \begin{tabular}{c}
 $(m_1,\ldots,m_{\ell})\mapsto$ \\
 $\yg_1^{m_1}\cdots \yg_{\ell}^{m_{\ell}}$
\end{tabular} &
\begin{tabular}{c} 
$(m_1,\ldots,m_{\ell})\mapsto$\\
$ [\gG_1^{(1)},(\gG'_1{}^{(1)})^{m_1}]\cdots [\gG_{t_{\ell}}^{(\ell)},(\gG'_{t_{\ell}}{}^{(\ell)})^{m_\ell}]$ (\eqref{eq=defnPPs2})
\end{tabular} \\
\hline
\begin{tabular}{c}
 compositions \\
 of two maps \\
 (Step~$3'$/Step~3)
\end{tabular} &
\begin{tabular}{c}
 Proposition~\ref{prop=itteQ} \\
 ($\Ptau\circ\Psig$ and $\Psig\circ\Ptau$)
\end{tabular} &
\begin{tabular}{c}
Theorem~\ref{thm=itte} \\
 ($\PPs\circ\PPh|_A$ for $d_{\scl_{\Gg,\Lg}}$-bounded $A$,  \\
 and $\PPh\circ\PPs$)\end{tabular} \\
\hline
\end{tabular}
\end{table}

To prove Proposition~\ref{prop=presclkika}, it suffices to show  the existence of the maps $\Psig$ and $\Ptau$ with conditions (1)--(3) above the presentation of Proposition~\ref{prop=prescldim} (recall Subsection~\ref{subsec=absolute}). We take the following three steps.

\begin{enumerate}
  \item [\underline{Step~$1'$}.] construct $\Psig^{\RR}\colon [\Gg,\Ng]\to \RR^{\ell}$;
  \item [\underline{Step~$2'$}.] construct $\Ptau\colon \ZZ^{\ell}\to [\Gg,\Ng]$;
  \item [\underline{Step~$3'$}.] take an appropriate $\Psig\colon [\Gg,\Ng]\to \ZZ^{\ell}$ out of $\Psig^{\RR}$, and study the compositions $\Ptau\circ\Psig$ and $\Psig\circ \Ptau$.
\end{enumerate}

In Step~$1'$, $\Psig^{\RR}$ can be taken as the \emph{evaluation map}:
\[
\Psig^{\RR}=\Psig^{\RR}_{(\nug_1,\ldots,\nug_{\ell})}\colon [\Gg,\Ng]\ni \yg\mapsto (\nug_1(\yg),\ldots,\nug_{\ell}(\yg))\in \RR^{\ell}.
\]
Here we can take an arbitrary tuple $(\nug_1,\ldots,\nug_{\ell})$ such that $\{[\nug_1],\ldots,[\nug_{\ell}]\}$ forms a basis of $\QQQ(\Ng)^{\Gg}/\HHH^1(\Ng)^{\Gg}$, where $[\cdot]$ means the equivalence class in  $\QQQ(\Ng)^{\Gg}/\HHH^1(\Ng)^{\Gg}$. To show that this $\Psig^{\RR}\colon ([\Gg,\Ng],d_{\scl_{\Gg,\Ng}})\to (\RR^{\ell},\|\cdot\|_1)$ is a quasi-isometric embedding, we employ the Bavard duality theorem for $\scl_{\Gg,\Ng}$ (Theorem~\ref{thm=Bavard}) and the finite dimensionality of $\QQQ(\Ng)^{\Gg}/\HHH^1(\Ng)^{\Gg}$. This step is done in Subsection~\ref{subsec=Psig}. In Step~$2'$, we construct the map $\Ptau$ of the form
\[
\Ptau\colon \ZZ^{\ell}\ni (m_1,\ldots,m_{\ell})\mapsto \yg_1^{m_1}\cdots \yg_{\ell}^{m_{\ell}}\in [\Gg,\Ng].
\]
To construct such $\Ptau$, we need to choose $\yg_1,\ldots,\yg_{\ell}\in [\Gg,\Ng]$ in an appropriate manner; such a choice of the tuple $(\yg_1,\ldots,\yg_{\ell})$ is provided by a lemma on function spaces, Lemma~\ref{lem=dual}. This lemma at the same time supplies the tuple $(\nug_1,\ldots,\nug_{\ell})$, where $\{[\nug_1],\ldots,[\nug_{\ell}]\}$ are linearly independent in  $\QQQ(\Ng)^{\Gg}/\HHH^1(\Ng)^{\Gg}$. More precisely, $(\yg_i)_{i}$ and $(\nug_j)_j$ satisfy that
\[
\textrm{for every }i,j\in \{1,\ldots,\ell\},\quad \nug_j(\yg_i)=\delta_{i,j}.
\]
Here $\delta_{\cdot,\cdot}$ means the Kronecker delta.
This step is treated in Subsection~\ref{subsec=Ptau}. In Step~$3'$, we first construct the map $\Psig$ as follows: take $\Psig^{\RR}=\Psig^{\RR}_{(\nug_1,\ldots,\nug_{\ell})}$ as in Step~$1'$ associated with the tuple $(\nug_1,\ldots,\nug_{\ell})$ supplied in Step~$2'$. Since a coarse inverse $\rho\colon \RR^{\ell}\to \ZZ^{\ell}$ to the inclusion $\ZZ^{\ell}\to \RR^{\ell}$ exists, we can take the composition $\Psig=\rho\circ \Psig^{\RR}$.
Finally, we prove that $\Ptau\circ \Psig\approx \mathrm{id}_{([\Gg,\Ng],d_{\scl_{\Gg,\Ng}})}$ and $\Psig\circ \Ptau \approx \mathrm{id}_{(\ZZ^{\ell},\|\cdot\|_1)}$. This step is pursued in Subsection~\ref{subsec=sclkika}.

We note that Step~$2'$ works under the weaker  assumption $\ell\leq \Rdim \left(\QQQ(\Ng)^{\Gg}/\HHH^1(\Ng)^{\Gg}\right)$ than the original assumption $\ell= \Rdim \left(\QQQ(\Ng)^{\Gg}/\HHH^1(\Ng)^{\Gg}\right)$ (contrastingly, Step~$1'$ works only in the original setting). This is the reason why we can remove the assumption of the finite dimensionality of $\QQQ(\Ng)^{\Gg}/\HHH^1(\Ng)^{\Gg}$ in Proposition~\ref{prop=metricdim}, which is a generalization of Proposition~\ref{prop=prescldim}.

Now we proceed to the outlined proof of Theorem~\ref{mthm=main}: it consists of the following three steps.

\begin{enumerate}
  \item [\underline{Step~1}.]  construct $\PPh^{\RR}\colon [\Gg,\Ng]\to \RR^{\ell}$;
  \item  [\underline{Step~2}.] construct $\PPs\colon \ZZ^{\ell}\to [\Gg,\Ng]$;
  \item  [\underline{Step~3}.] take an appropriate $\PPh\colon [\Gg,\Ng]\to \ZZ^{\ell}$ out of $\PPh^{\RR}$, and study $\PPs\circ\PPh$ and $\PPh\circ \PPs$.
\end{enumerate}

These three steps look similar to those for the proof of Proposition~\ref{prop=presclkika}. However, we need several modifications from the constructions of $\Psig^{\RR}$ and $\Ptau$ in the proof of Proposition~\ref{prop=presclkika}. First, in Step~1, we can only use the finite dimensionality of $\Wcal(\Gg,\Lg,\Ng)$; however, as \eqref{eq=KKMMBavard} indicates, $d_{\scl_{\Gg,\Ng}}$ itself is closely related to the space $\QQQ(\Ng)^{\Gg}/\HHH^1(\Ng)^{\Gg}$, which is a much larger space than $\Wcal(\Gg,\Lg,\Ng)$. (Recall that our emphasis in the introduction is that there exists several examples of $(\Gg,\Ng)$ with $\Rdim(\QQQ(\Ng)^{\Gg}/\HHH^1(\Ng)^{\Gg})=\infty$ but $\Rdim \WW(\Gg,\Ng)<\infty$.) In particular, the latter implies that what we need here is the property that $\PPh^{\RR}|_A$ is a quasi-isometric embedding \emph{for every $d_{\scl_{\Gg,\Lg}}$-bounded set $A\subseteq [\Gg,\Ng]$}, not necessarily on the whole $[\Gg,\Ng]$. For this purpose, we can define the evaluation map
\[
\PPh^{\RR}=\PPh^{\RR}_{(\nug_1,\ldots,\nug_{\ell})}\colon [\Gg,\Ng]\ni \yg\mapsto (\nug_1(\yg),\ldots,\nug_{\ell}(\yg))\in \RR^{\ell}
\]
associated with an arbitrary tuple $(\nug_1,\ldots,\nug_{\ell})$ such that $\{[\nug_1],\ldots,[\nug_{\ell}]\}$ forms a basis of $\Wcal(\Gg,\Lg,\Ng)$. Then, we can show that this map $\PPh^{\RR}$ works. However, now the proof  is not a direct application of Theorem~\ref{thm=Bavard}. For this part, we prove the \emph{comparison theorem of defects} (Theorem~\ref{thm=comparisonDD}) prior to Step~1, which may be of independent interest. We note that Step~1 works under the general assumption $\Rdim \Wcal(\Gg,\Lg,\Ng)<\infty$; condition (a${}_q$) or (b${}_q$) for any $q\in \NN_{\geq 2}$ in Theorem~\ref{mthm=main} is not needed here. Hence, Step~1 yields Theorem~\ref{mthm=dimfin} as well.

Step~2 is the \emph{key step} to the proof. In this step, we need a \emph{complete modification} of the construction from the map $\Ptau$ in the proof of Proposition~\ref{prop=presclkika}; see the discussion at the beginning of Section~\ref{sec=CE} for details. In the setting of Theorem~\ref{mthm=main} (`abelian case': $\Ng\geqslant [\Gg,\Lg]$), we can take $\PPs$ of the form
\begin{align*}
\PPs\colon \ZZ^{\ell}\ni &(m_1,\ldots,m_{\ell})\\
\mapsto & [\gG_{1}^{(1)},(\gG'_{1}{}^{(1)})^{m_1}]\cdots [\gG_{t_1}^{(1)},(\gG'_{t_1}{}^{(1)})^{m_1}]\cdots [\gG_{1}^{(\ell)},(\gG'_{1}{}^{(\ell)})^{m_{\ell}}]\cdots [\gG_{t_{\ell}}^{(\ell)},(\gG'_{t_{\ell}}{}^{(\ell)})^{m_{\ell}}] \in [\Gg,\Ng].
\end{align*}
Here, $t_i\in \NN$ for every $i\in \{1,\ldots,\ell\}$, and $\gG_s^{(i)}\in \Gg$ and $\gG'_s{}^{(i)}\in \Lg$ for every $s\in \{1,\ldots,t_i\}$. To find appropriate $(t_i)_{i}$, $(\gG_s^{(i)})_{i,s}$, $(\gG'_s{}^{(i)})_{i,s}$, we initiate the theory of \emph{core extractors}. This theory morally provides such tuples by consideration of  a lift $(\Ff,\Kf,\Mf)$ of the given triple $(\Gg,\Lg,\Ng)$ that behaves `better' in terms of the $\Wcal$-space: the upshot of the theory of core extractors is that we can define an \emph{injective homomorphism} $\CE$, which we call the \emph{core extractor}, from $\Wcal(\Gg,\Lg,\Ng)$ to a certain space of \emph{invariant homomorphisms} defined on the lift, provided certain three conditions, written as  (i), (ii) and (D) in the present paper, are fulfilled (Definition~\ref{defn=CoreE} and Theorem~\ref{thm=InjCE}). This part is one of the most novel points in the present work; as we mentioned in Subsection~\ref{subsec=comparative}, this theory supplies a general machinery for comparison problems of mixed $\scl$'s. This, together with the map $\alpha_{\bff}$ defined in Definition~\ref{defn=Pspower} (which works in the abelian case), enables us to obtain a map $\PPs$ of the form above in an appropriate manner. The construction of $\PPs$ will be done in Theorem~\ref{thm=PPs}; Corollary~\ref{cor=InjCE}, which is the upshot of the theory of core extractors in the abelian case (introduced and developed in Sections~\ref{sec=CE} and \ref{sec=CEabel}), is the key to that construction. Points here are that the image $\CE_{[\nug]}$ of an element $[\nug]$ in $\Wcal(\Gg,\Lg,\Ng)$ under the core extractor $\CE$ is a \emph{genuine} homomorphism, and that we can compute the values of $\CE_{[\nug]}$ at several elements. The assumption in Theorem~\ref{mthm=main} of the existence of $q\in \NN_{\geq 2}$ fulfilling either  (a${}_q$) or (b${}_q$) is employed in this step.

Once Step~2 is accomplished, the construction of $\PPh$ out of $\PPh^{\RR}$ in Step~3 is similar to that of $\Psig$ out of $\Psig^{\RR}$ in the proof of Proposition~\ref{prop=presclkika}. In the current setting, in order to close up Step~3, it suffices to verify that $\PPh\circ\PPs\approx \mathrm{id}_{(\ZZ^{\ell},\|\cdot\|_1)}$ and that for every $d_{\scl_{\Gg,\Lg}}$-bounded set $A\subseteq [\Gg,\Lg]$,
\[
\sup_{\yg\in A}d_{\scl_{\Gg,\Ng}}(\yg,(\PPs\circ \PPh)(\yg))<\infty .
\]
For the proof of this inequality, we employ Theorem~\ref{thm=comparisonDD}  (again from Step~1). Finally, to verify the `furthermore' part, we slightly modify these constructions; we discuss this in the setting of Proposition~\ref{prop=presclkika} (Lemma~\ref{lem=sigmatau}) in Subsection~\ref{subsec=remark}. This ends the outline of our proof of Theorem~\ref{mthm=main}.

\

\noindent
\textbf{Organization of the present paper:} In Section~\ref{sec=pre}, we collect preliminary facts, including basics of the theory of coarse groups and coarse kernels in \cite{LV}. In Section~\ref{sec=absolute}, we present complete proofs of Propositions~\ref{prop=sclkika}, \ref{prop=presclkika} and \ref{prop=prescldim}. In Section~\ref{sec=defect}, we prove the comparison theorem of defects (Theorem~\ref{thm=comparisonDD}), which will be employed in Sections~\ref{sec=Phi} and \ref{sec=proofmain}. Section~\ref{sec=Phi} is devoted to the construction of the map $\PPh^{\RR}$ (Theorem~\ref{thm=evmap}), which corresponds to Step~1 in the outlined proof above. There, we also prove Theorem~\ref{mthm=dimfin}. In Sections~\ref{sec=CE} (general theory) and \ref{sec=CEabel} (abelian case), we introduce the theory of core extractors; Corollary~\ref{cor=InjCE} plays a key role in the proofs of Theorem~\ref{mthm=main} and Theorem~\ref{mthm=dim}. Section~\ref{sec=Psi} is devoted to the construction of the map $\PPs$ (Theorem~\ref{thm=PPs}); as we described in the outlined proof above, this is the key step to the proof of Theorem~\ref{mthm=main}. In Section~\ref{sec=proofmain}, we prove Theorem~\ref{thm=itte}, which corresponds to Step~3 in the outlined proof of Theorem~\ref{mthm=main}. We establish Theorem~\ref{thm=maingeneral} and  Theorem~\ref{thm=dimgeneral}, which are general forms of Theorem~\ref{mthm=main} and Theorem~\ref{mthm=dim}, respectively. In Section~\ref{sec=examplemain}, we exhibit several examples to which Theorem~\ref{mthm=main} applies: some of them provide explicit coarse kernels, including Proposition~\ref{prop=explicitexample}. Finally, in Section~\ref{sec=remarks}, we present applications of the coarse kernels obtained in Theorem~\ref{mthm=main}; Theorem~\ref{thm=crushingq2} and Proposition~\ref{prop=extexample} are verified there.

\

\noindent
\textbf{Notation and conventions:}
Throughout the present paper,  as mentioned in Remark~\ref{rem=bf}, we write coarse notions in bold symbol (such as a coarse map $\bfalpha$ and a coarse subspace $\mathbf{A}$) to distinguish them from set theoretical notions in non-bold symbol (such as a set map $\alpha$ and a subset $A$). For a group $H$, the symbol  $e_H$ denotes the group unit of $H$. The symbol $\NN$ means the set $\{1,2,3,\ldots\}$ of positive integers; in particular, $0\not\in\NN$ in this paper. The symbol $\delta_{\cdot,\cdot}$ denotes the Kronecker delta function. For $\genus\in \NN$, let $\Sigma_{\genus}$ be the closed connected orientable surface of genus $\genus$. If a real vector space $\mathcal{K}$ is infinite dimensional, then we set the real dimension $\Rdim \mathcal{K}$ of $\mathcal{K}$ to be $\infty$.  In the setting of Definition~\ref{defn=thespaceV}, for $\nug\in \QQQ(\Ng)^{\Gg}$, let $[\nug]$ mean the equivalence class represented by $\nug$ in $\Wcal(\Gg,\Lg,\Ng)$.

\section{Preliminaries}\label{sec=pre}
Here we collect several preliminary facts needed in the present paper. The reader who is familiar with the topic of a subsection can skip that subsection. As we mentioned in Section~\ref{sec=intro2}, our main theorems (Theorem~\ref{mthm=main}, Theorem~\ref{mthm=dim} and Theorem~\ref{mthm=dimfin}) treat a group triple $(\Gg,\Lg,\Ng)$. However, to define several preliminary concepts, we decompose  this triple as two pairs $(\Gg,\Lg)$ and $(\Gg,\Ng)$. For this reason, we frequently use the following setting in this section.

\begin{setting}\label{setting=itsumono}
Let $\Gg$ be a group and $\Ng$ a normal subgroup of $\Gg$. Let $i=i_{\Ng,\Gg}\colon \Ng\hookrightarrow \Gg$ be the inclusion map.
\end{setting}
We abbreviate $i_{\Ng,\Gg}$ as $i$ in the present paper unless we specially hope to clarify $\Ng$ or $\Gg$.

\subsection{Fundamental properties of invariant quasimorphisms}\label{subsec=qm}
Here we collect basic properties of invariant quasimorphisms and stable mixed commutator lengths;  \cite{KKMMMsurvey} is  a survey on these topics, where we describe more backgrounds and history for invariant quasimorphisms. We also refer the reader to \cite{Calegari} for a treatise on (ordinary) quasimorphisms and stable commutator lengths.

\begin{defn}\label{defn=qm}
Assume Setting~$\ref{setting=itsumono}$.
\begin{enumerate}[label=\textup{(}\arabic*\textup{)}]
  \item A function $\psg\colon \Gg\to \RR$ is said to be \emph{homogeneous} if for every $g\in \Gg$ and for every $n\in \ZZ$, $\psg(g^n)=n\psg(g)$ holds.
A function $\nug\colon \Ng\to \RR$ is said to be \emph{$\Gg$-invariant} if for every $x\in \Ng$ and for every $g\in \Gg$, $\nug(gxg^{-1})=\nug(x)$
holds.
  \item A function $\psg\colon \Gg\to \RR$ is called a  \emph{quasimorphism} on $\Gg$ if the \emph{defect} $\DD(\psg)$ of $\psg$, defined as \eqref{eq=defect}, is finite.  The $\RR$-vector space $\QQQ(\Gg)$ is defined as the space of homogeneous quasimorphisms on $\Gg$.
  \item The $\RR$-vector space $\QQQ(\Ng)^{\Gg}$ is defined as the space of homogeneous quasimorphisms on $\Ng$ that are $\Gg$-invariant.
  \item The $\RR$-vector space $\HHH^1(\Ng)^{\Gg}$ is defined as the space of \textup{(}genuine\textup{)} homomorphisms $\Ng\to \RR$ that are $\Gg$-invariant.
\end{enumerate}
\end{defn}

\begin{lem}\label{lem=Ginv}
Let $\Gg$ be a group. Then, $\QQQ(\Gg)=\QQQ(\Gg)^{\Gg}$.
\end{lem}

\begin{proof}
Let $\psg\in \QQQ(\Gg)$. Let $g,\lambda\in \Gg$ and $n\in \ZZ$. Then we have $
(\lambda g \lambda^{-1})^n=\lambda g^n \lambda^{-1}$ and hence $|n||\psg(\lambda g \lambda^{-1})-\psg(g)|\leq 2\DD(\psg)$. This shows that $\psg\in \QQQ(\Gg)^{\Gg}$.
\end{proof}

\begin{prop}[{\cite[Lemma~3.6]{Bavard} (see also \cite[Lemma~2.24]{Calegari})}]\label{prop=Bavard}
Let $\Gg$ be a group and $\psg\in \QQQ(\Gg)$. Then, $\DD(\psg)= \sup\{|\psg([\gG_1,\gG_2])|\,|\, \gG_1,\gG_2\in \Gg\}$.
\end{prop}

Proposition~\ref{prop=Bavard} has the following two corollaries.

\begin{cor}\label{cor=uptohom}
Assume Setting~$\ref{setting=itsumono}$. Let $\nug_1,\nug_2\in \QQQ(\Ng)^{\Gg}$. If $\nug_1$ and $\nug_2$ coincide on $[\Ng,\Ng]$, then
\[
\nug_1-\nug_2\in \HHH^1(\Ng)^{\Gg}.
\]
\end{cor}

\begin{proof}
Let $\nug\in \QQQ(\Ng)^{\Gg}$. Note that in particular $\nug\in \QQQ(\Ng)$. By applying Proposition~\ref{prop=Bavard}, we obtain that\begin{equation}\label{eq=defectNN}
\DD(\nug)= \sup\{|\nu([\xg_1,\xg_2])|\,|\, \xg_1,\xg_2\in \Ng\}.
\end{equation}
By setting $\nug=\nug_1-\nug_2\in \QQQ(\Ng)^{\Gg}$,  we conclude that $\DD(\nug)=0$. Hence, $\nug\in \HHH^1(\Ng)^{\Gg}$.
\end{proof}

\begin{cor}\label{cor=BavardD}
Assume Setting~$\ref{setting=itsumono}$. Let $\nug\in \QQQ(\Ng)^{\Gg}$. Then $\DD(\nug)= \sup\{|\nu([\gG,\xg])|\,|\, \gG\in \Gg,\,\xg\in \Ng\}$.
\end{cor}

\begin{proof}
It is straightforward to show that $\DD(\nug)\geq  \sup\{|\nu([\gG,\xg])|\,|\, \gG\in \Gg,\,\xg\in \Ng\}$. Conversely, \eqref{eq=defectNN} implies that
\[
\sup\{|\nu([\gG,\xg])|\,|\, \gG\in \Gg,\,\xg\in \Ng\}
\geq  \sup\{|\nu([\xg_1,\xg_2])|\,|\, \xg_1,\xg_2\in \Ng\}
=\DD(\nug).\qedhere
\]
\end{proof}

\begin{defn}[$\cl_{\Gg,\Ng}$ and $\scl_{\Gg,\Ng}$]\label{defn=sclGN}
Assume Setting~\ref{setting=itsumono}.
\begin{enumerate}[label=(\arabic*)]
  \item A \emph{simple $(\Gg,\Ng)$-commutator} means an element in $\Gg$ of the form $[\gG,\xg]$, where $\gG\in \Gg$ and $\xg\in \Ng$. The \emph{mixed commutator subgroup} $[\Gg,\Ng]$ is a subgroup of $\Gg$ generated by the set of simple $(\Gg,\Ng)$-commutators.
  \item The \emph{mixed commutator length} $\cl_{\Gg,\Ng}$ is defined as the word length on $[\Gg,\Ng]$ with respect to the set of all simple $(\Gg,\Ng)$-commutators. Namely, for $\yg\in [\Gg,\Ng]$, $\cl_{\Gg,\Ng}(\yg)$ is defined to be the minimal number of $n\in\ZZ_{\geq 0}$ such that $\yg$ can be written as the product of $n$ simple $(\Gg,\Ng)$-commutators. In particular, we set $\cl_{\Gg,\Ng}(e_{\Gg})=0$.
    \item The \emph{stable mixed commutator length} $\scl_{\Gg,\Ng}$ of $\yg\in [\Gg,\Ng]$ is defined as $\scl_{\Gg,\Ng}(\yg)=\lim\limits_{n\to \infty} \frac{\cl_{\Gg,\Ng}(\yg^n)}{n}$.
\end{enumerate}
\end{defn}

For future purposes in relation to the generalized mixed Bavard duality theorem \cite{KKMMMgmB}, it is natural to extend the domain of $\scl_{\Gg,\Ng}$ in the following manner. This extension will not show up after this preliminary section, as the main subject  of the present paper is the coarse group structure of $([\Gg,\Ng],\scl_{\Gg,\Ng})$, not that of $(\Gg,\scl_{\Gg,\Ng})$; we employ this extension only to formulate Corollary~\ref{cor=+1/2} in Subsection~\ref{subsec=scl-distance} in full generality.

\begin{defn}\label{defn=torsionscl}
Assume Setting~\ref{setting=itsumono}. Let $(\Ng/[\Gg,\Ng])_{\tor}$ be the subgroup of torsion elements of $\Ng/[\Gg,\Ng]$, and let $\mathrm{proj}_{(\Gg,\Ng)}\colon \Ng\twoheadrightarrow \Ng/[\Gg,\Ng]$ be the natural group quotient map. Then $\scl_{\Gg,\Ng}$ is defined as a map from $\mathrm{proj}_{(\Gg,\Ng)}^{-1}\Big((\Ng/[\Gg,\Ng])_{\tor}\Big)$ to $\RR_{\geq 0}$ in the following manner. Let $\xg\in \mathrm{proj}_{(\Gg,\Ng)}^{-1}\Big((\Ng/[\Gg,\Ng])_{\tor}\Big)$. Then, there exists $n\in \NN$ such that $\xg^n\in [\Gg,\Ng]$. By using this $n\in \NN$, we define
\begin{equation}\label{eq=torsionscl}
\scl_{\Gg,\Ng}(\xg)=\frac{\scl_{\Gg,\Ng}(\xg^n)}{n}.
\end{equation}
\end{defn}
By semi-homogeneity of $\scl_{\Gg,\Ng}\colon [\Gg,\Ng]\to \RR_{\geq 0}$ (recall \eqref{eq=semihom}), the definition in \eqref{eq=torsionscl} does not depend on the choice of $n$. We also note that 
\[[\Gg,\Ng]\subseteq \mathrm{proj}_{(\Gg,\Ng)}^{-1}\Big((\Ng/[\Gg,\Ng])_{\tor}\Big)\subseteq \Ng.\]
If $\Ng=\Gg$, then $\cl_{\Gg,\Gg}$ and $\scl_{\Gg,\Gg}$ equal the commutator length $\cl_{\Gg}$ and the stable commutator length $\scl_{\Gg}$, respectively. In this case, the extension of the domain of $\scl_{\Gg}$ explained in Definition~\ref{defn=torsionscl} is standard.

We will frequently use the following lemma without mentioning. This immediately follows from Corollary~\ref{cor=BavardD}.

\begin{lem}\label{lem=2D-1}
Assume Setting~$\ref{setting=itsumono}$. Let $\nug\in \QQQ(\Ng)^{\Gg}$. Then, for every $\yg\in [\Gg,\Ng]\setminus \{e_{\Gg}\}$, we have
\[
|\nug(\yg)|\leq (2\cl_{\Gg,\Ng}(\yg)-1)\DD(\nug).
\]
\end{lem}

As mentioned in Subsection~\ref{subsec=digest}, the following \emph{Bavard duality theorem for mixed commutator length} has been shown in \cite[Theorem~1.2]{KKMM1}.

\begin{thm}[Bavard duality theorem for mixed $\scl$, \cite{KKMM1}]\label{thm=Bavard}
Assume Setting~$\ref{setting=itsumono}$. Then for every $\yg\in [\Gg,\Ng]$,
\[
\scl_{\Gg,\Ng}(\yg)=\sup_{\nug\in \QQQ(\Ng)^{\Gg}\setminus\HHH^1(\Ng)^{\Gg}}\frac{|\nug(\yg)|}{2\DD(\nug)}.
\]
Here, if $\QQQ(\Ng)^{\Gg}=\HHH^1(\Ng)^{\Gg}$, then $\scl_{\Gg,\Ng}\equiv 0$ on $[\Gg,\Ng]$.
\end{thm}

By homogeneity of elements in $\QQQ(\Ng)^{\Gg}$, we have the following immediate extension of Theorem~\ref{thm=Bavard} in the situation of Definition~\ref{defn=torsionscl}: for every $\xg\in \mathrm{proj}_{(\Gg,\Ng)}^{-1}\Big((\Ng/[\Gg,\Ng])_{\tor}\Big)$, we have
\begin{equation}\label{eq=torsionBavard}
\scl_{\Gg,\Ng}(\xg)=\sup_{\nug\in \QQQ(\Ng)^{\Gg}\setminus\HHH^1(\Ng)^{\Gg}}\frac{|\nug(\xg)|}{2\DD(\nug)}.
\end{equation}
Here, we note that for every $\kg\in \HHH^1(\Ng)^{\Gg}$, $\kg$ vanishes on $\mathrm{proj}_{(\Gg,\Ng)}^{-1}\Big((\Ng/[\Gg,\Ng])_{\tor}\Big)$.

\subsection{The (almost-)metrics $d_{\cl}$ and $d_{\scl}$}\label{subsec=scl-distance}

We will define a generalized metric $d_{\cl_{\Gg,\Ng}}$ \emph{on} $\Gg$ and an almost generalized metric $d_{\scl_{\Gg,\Ng}}$ \emph{on} $\Gg$. Before presenting their definitions, we first clarify what generalized metrics and  almost (generalized) metrics mean in the present paper.

\begin{defn}\label{defn=generalizedmet}
Let $X$ be a set and $d\colon X\times X\to\RR_{\geq 0}\cup \{\infty\}$ a map.
\begin{enumerate}[label=(\arabic*)]
\item (generalized metric) The pair $(X,d)$ is called a \emph{generalized metric space} if the following four conditions are fulfilled.
\begin{enumerate}
  \item[($1_{1}$)] For every $x\in X$, $d(x,x)=0$.
  \item[($1_{2}$)] For every $x_1,x_2\in X$, if $d(x_1,x_2)=0$, then $x_1=x_2$.
  \item[($1_{3}$)] For every $x_1,x_2\in X$, $d(x_1,x_2)=d(x_2,x_1)$.
  \item[($1_{4}$)] For every $x_1,x_2,x_3\in X$, $d(x_1,x_3)\leq d(x_1,x_2)+d(x_2,x_3)$.
\end{enumerate}
This $d$ is called a \emph{generalized metric} on $X$.
\item (almost generalized metric/almost metric) The pair $(X,d)$ is called an \emph{almost generalized metric space} if  conditions ($1_{1}$) and ($1_3$)  hold and if there exists $C\in \mathbb{R}_{\geq 0}$ such that the following condition ($1_4^{+,C}$) holds true.
\begin{enumerate}
  \item[($1_{4}^{+,C}$)] For every $x_1,x_2,x_3\in X$, $d(x_1,x_3)\leq d(x_1,x_2)+d(x_2,x_3)+C$.
\end{enumerate}
This $d$ is called an \emph{almost generalized metric} on $X$.

The pair $(X,d)$ is called an \emph{almost metric space} if it is an almost generalized metric space and if $d(X\times X)\subseteq \mathbb{R}_{\geq 0}$. This $d$ is called an \emph{almost metric} on $X$.
\end{enumerate}
\end{defn}

The notion of almost metric spaces with constant $C$ coincides with that of $(1,C)$-metric spaces in \cite{ThompsonHemmila}. The following lemma is obvious.

\begin{lem}\label{lem=+C}
Let $(X,d)$ be an almost generalized metric space. Let $C\in \mathbb{R}_{> 0}$ be a constant that satisfies \textup{(}$1_4^{+,C}$\textup{)} in Definition~$\ref{defn=generalizedmet}$. Define $d^+=d^{+,C}\colon X\times X\to \mathbb{R}_{\geq 0}\cup \{\infty\}$ by
\[
d^+(x_1,x_2)=
\begin{cases}
d(x_1,x_2)+C, &\textrm{if $x_1\ne x_2$}, \\
0, &\textrm{if $x_1=x_2$}
\end{cases}
\]
for every $x_1,x_2\in X$. Then $(X,d^+)$ is a generalized metric space.
\end{lem}

 In item (2) of the following definition, recall Definition~\ref{defn=torsionscl}.
\begin{defn}\label{defn=scldistance}
Assume Setting~\ref{setting=itsumono}.
\begin{enumerate}[label=\textup{(}\arabic*\textup{)}]
  \item \textup{(}$\cl$-(generalized-)metric\textup{)} For $g_1,g_2\in \Gg$, set
\[
d_{\cl_{\Gg,\Ng}}(g_1,g_2)=\left\{
\begin{array}{cl}
\cl_{\Gg,\Ng}(g_1^{-1}g_2), &\textrm{if $g_1^{-1}g_2\in [\Gg,\Ng]$}, \\
\infty, &\textrm{otherwise}.
\end{array}\right.
\]
  \item \textup{(}$\scl$-almost-(generalized-)metric\textup{)} For $g_1,g_2\in \Gg$, set
\[
d_{\scl_{\Gg,\Ng}}(g_1,g_2)=\left\{
\begin{array}{cl}
\scl_{\Gg,\Ng}(g_1^{-1}g_2), &\textrm{if $g_1^{-1}g_2\in  \mathrm{proj}_{(\Gg,\Ng)}^{-1}\Big((\Ng/[\Gg,\Ng])_{\tor}\Big)$}, \\
\infty, &\textrm{otherwise}.
\end{array}\right.
\]
\end{enumerate}
We also set $d_{\cl_{\Gg}}=d_{\cl_{\Gg,\Gg}}$ and $d_{\scl_{\Gg}}=d_{\scl_{\Gg,\Gg}}$.
\end{defn}

For every $g_1,g_2\in \Gg$,
\begin{equation}\label{eq=hikaku}
d_{\scl_{\Gg,\Ng}}(g_1,g_2)\leq d_{\cl_{\Gg,\Ng}}(g_1,g_2)
\end{equation}
holds.
The following example shows that $d_{\scl_{\Gg,\Ng}}$ is \emph{not} a generalized metric in general.

\begin{exa}\label{exa=nottriangle}
Let $A$, $B$ be  non-trivial finite groups. Set $\Gg=\Ng=A\star B$, the free product of $A$ and $B$. Let $a\in  A\setminus \{e_A\}$ and $b\in  B\setminus \{e_B\}$, and set $m_a$ and $m_b$ as the orders of $a$ and $b$, respectively. Since $m_a<\infty$ and $m_b<\infty$, we have $d_{\scl_{\Gg}}(a,e_G)=0$ and $d_{\scl_{\Gg}}(e_G,b)=0$ (recall Definition~\ref{defn=torsionscl}). However, \cite[Theorem~2.93]{Calegari} implies that
\[
d_{\scl_{\Gg}}(a,b)=\frac{1}{2}\left(1-\frac{1}{m_a}-\frac{1}{m_b}\right),
\]
which is non-zero if $(m_a,m_b)\ne (2,2)$. Hence, ($1_4$) in Definition~\ref{defn=generalizedmet} fails for $([\Gg,\Ng],d_{\scl_{\Gg,\Ng}})$ in general.
\end{exa}

The following proposition may be proved by direct commutator calculus; see \cite[Proposition~1]{Kotschick}. For the reader's convenience, we include another  proof using Theorem~\ref{thm=Bavard}  (strictly speaking, using \eqref{eq=torsionBavard}).
\begin{prop}\label{prop=scladditive}
Assume Setting~$\ref{setting=itsumono}$. Then for every  $\xg_1,\xg_2\in \mathrm{proj}_{(\Gg,\Ng)}^{-1}\Big((\Ng/[\Gg,\Ng])_{\tor}\Big)$,
\begin{equation}\label{eq=scl1/2}
\scl_{\Gg,\Ng}(\xg_1\xg_2)\leq \scl_{\Gg,\Ng}(\xg_1)+\scl_{\Gg,\Ng}(\xg_2)+\frac{1}{2}
\end{equation}
holds.
For every $g_1,g_2,g_3\in \Gg$, we have
\begin{equation}\label{eq=dscl1/2}
d_{\scl_{\Gg,\Ng}}(g_1,g_3)\leq d_{\scl_{\Gg,\Ng}}(g_1,g_2)+d_{\scl_{\Gg,\Ng}}(g_2,g_3)+\frac{1}{2}.
\end{equation}
\end{prop}

\begin{proof}
For \eqref{eq=scl1/2},  first observe that $\xg_1,\xg_2\in \mathrm{proj}_{(\Gg,\Ng)}^{-1}\Big((\Ng/[\Gg,\Ng])_{\tor}\Big)$ implies $\xg_1\xg_2\in \mathrm{proj}_{(\Gg,\Ng)}^{-1}\Big((\Ng/[\Gg,\Ng])_{\tor}\Big)$. Indeed, since $\Ng/[\Gg,\Ng]$ is an abelian group, the set $(\Ng/[\Gg,\Ng])_{\tor}$ forms a subgroup of $\Ng/[\Gg,\Ng]$. Now, take an arbitrary $\varepsilon \in \RR_{>0}$. Then by  \eqref{eq=torsionBavard}, there exists $\nug_{\varepsilon}\in \QQQ(\Ng)^{\Gg}$ such that
\[
|\nug_{\varepsilon}(\xg_1\xg_2)|\geq 2(1-\varepsilon)\scl_{\Gg,\Ng}(\xg_1\xg_2)\cdot \DD(\nug_{\varepsilon}).
\]
For this $\nug_{\varepsilon}$, we also have $|\nug_{\varepsilon}(\xg_1\xg_2)|\leq |\nug_{\varepsilon}(\xg_1)|+|\nug_{\varepsilon}(\xg_2)|+\DD(\nug_{\varepsilon})$. Again by  \eqref{eq=torsionBavard},
\[
(1-\varepsilon)\scl_{\Gg,\Ng}(\xg_1\xg_2)\leq \scl_{\Gg,\Ng}(\xg_1)+\scl_{\Gg,\Ng}(\xg_2)+\frac{1}{2}
\]
holds. By letting $\varepsilon \searrow 0$, we obtain \eqref{eq=scl1/2}. For \eqref{eq=dscl1/2},  it trivially holds if either $g_1^{-1}g_2$ or $g_2^{-1}g_3$ is outside of $\mathrm{proj}_{(\Gg,\Ng)}^{-1}\Big((\Ng/[\Gg,\Ng])_{\tor}\Big)$. If $g_1^{-1}g_2$ and $g_2^{-1}g_3$ both belong to $\mathrm{proj}_{(\Gg,\Ng)}^{-1}\Big((\Ng/[\Gg,\Ng])_{\tor}\Big)$, then \eqref{eq=scl1/2} implies \eqref{eq=dscl1/2}.
\end{proof}

By \eqref{eq=dscl1/2}, Lemma~\ref{lem=+C} yields the following corollary.
\begin{cor}\label{cor=+1/2}
Assume Setting~$\ref{setting=itsumono}$. We define $d^+_{\scl_{\Gg,\Ng}}\colon G\times G\to \mathbb{R}_{\geq 0}\cup \{\infty\}$  by
\begin{equation}\label{eq=+1/2}
d^+_{\scl_{\Gg,\Ng}}(\gG_1,\gG_2)=
\begin{cases}
\scl_{\Gg,\Ng}(\gG_1^{-1}\gG_2)+\frac{1}{2}, &\textrm{if $\gG_1\ne \gG_2$ and $\gG_1^{-1}\gG_2\in  \mathrm{proj}_{(\Gg,\Ng)}^{-1}\Big((\Ng/[\Gg,\Ng])_{\tor}\Big)$}, \\
0, &\textrm{if $\gG_1=\gG_2$},\\
\infty, &\textrm{if $\gG_1^{-1}\gG_2\in \Gg\setminus  \mathrm{proj}_{(\Gg,\Ng)}^{-1}\Big((\Ng/[\Gg,\Ng])_{\tor}\Big)$}
\end{cases}
\end{equation}
for every $\gG_1,\gG_2\in \Gg$. Then, $d^+_{\scl_{\Gg,\Ng}}$ is a generalized metric on $\Gg$.
\end{cor}

\begin{lem}\label{lem=Ginvmetric}
Assume Setting~$\ref{setting=itsumono}$.
\begin{enumerate}[label=\textup{(}$\arabic*$\textup{)}]
 \item Let $y$ be a simple $(\Gg,\Ng)$-commutator. Then for every $\lambda\in \Gg$, $\lambda y\lambda^{-1}$ is a simple $(\Gg,\Ng)$-commutator. For every $\xg\in \Ng$ and $\gG\in \Gg$, $[\xg,\gG]$ is a simple $(\Gg,\Ng)$-commutator.
 \item The generalized metric $d_{\cl_{\Gg,\Ng}}$ is bi-$G$-invariant.
 \item The almost generalized metric $d_{\scl_{\Gg,\Ng}}$ and the generalized metric $d^+_{\scl_{\Gg,\Ng}}$ defined by \eqref{eq=+1/2} are both bi-$\Gg$-invariant.
\end{enumerate}
\end{lem}

\begin{proof}
For $(1)$, write $y=[g,x]$ with $g\in \Gg$ and $x\in \Ng$. Then, $\lambda y \lambda^{-1}=[\lambda g \lambda^{-1}, \lambda x \lambda^{-1}]$ is a simple $(\Gg,\Ng)$-commutator. Also, observe that for every $g\in \Gg$ and every $\xg\in \Ng$, $[\xg,g]=$$g[g^{-1},\xg]g^{-1}$$=[g,g\xg g^{-1}]$.
For (2), the non-trivial part, the right-$\Gg$-invariance, follows from (1). Finally, (3) is implied by (2).
\end{proof}

\begin{rem}\label{rem=notGtorsion}
In Definition~\ref{defn=torsionscl}, we define $\scl_{\Gg,\Ng}(\xg)$ for $\xg\in \Ng$ satisfying the condition that  there exists $n\in \NN$ such that $\xg^n\in [\Gg,\Ng]$. Some reader might wonder why we do not define $\scl_{\Gg,\Ng}(\gG)$ for $\gG\in  \Gg\setminus \Ng$ satisfying the same condition. The reason is for validity of Proposition~\ref{prop=scladditive} and Corollary~\ref{cor=+1/2}. More precisely, the point here is that, unlike $(\Ng/[\Gg,\Ng])_{\tor}$, the set of torsion elements in $\Gg/[\Gg,\Ng]$ may \emph{not} form a subgroup of $\Gg/[\Gg,\Ng]$. 

For instance, let $\Gg$ be a group that admits two torsion elements $a,b\in G$ such that $a^{-1}b$ is of infinite order, such as the infinite dihedral group $D_{\infty}$. Let $\Ng=\{e_{\Gg}\}$. If we extend the formulation of Definition~\ref{defn=torsionscl} even to the case of $\gG\in \Gg\setminus \Ng$ whose positive power can lie in $[\Gg,\Ng]$, then under this formulation we will have $d_{\scl_{\Gg,\Ng}}(a,e_{\Gg})=0$ and $d_{\scl_{\Gg,\Ng}}(e_{\Gg},b)=0$, but $d_{\scl_{\Gg,\Ng}}(a,b)=\infty$; hence, \eqref{eq=dscl1/2} will \emph{fail}. 

For this reason, we set $d_{\scl_{\Gg,\Ng}}(g_1,g_2)=\infty$ if $g_1^{-1}g_2\in \Gg\setminus \Ng$, even when some positive power of $g_1^{-1}g_2$ belongs to $[\Gg,\Ng]$. We note that this subtlety only shows up when we study $\scl_{\Gg,\Ng}$ (rather than $\scl_{\Gg}$ alone), because $\Gg\setminus \Ng=\emptyset$ if $\Ng=\Gg$.
\end{rem}


\begin{defn}\label{defn=choron}
Assume Setting~$\ref{setting=itsumono}$. Let $g_1,g_2\in \Gg$ and $C\in \RR_{\geq 0}$. We write $g_1\mathrel{\overset{(\Gg,\Ng)}{\underset{C}{\eqsim}}}g_2$ to mean that $d_{\cl_{\Gg,\Ng}}(g_1,g_2)\leq C$.
If $(\Gg,\Ng)$ is clear in the context, then we abbreviate as ${\underset{C}{\eqsim}}$.
\end{defn}

\begin{lem}\label{lem=choron}
Assume Setting~$\ref{setting=itsumono}$. We abbreviate $\overset{(\Gg,\Ng)}{\underset{C}{\eqsim}}$ as ${\underset{C}{\eqsim}}$.
\begin{enumerate}[label=\textup{(}$\arabic*$\textup{)}]
  \item For every $g_1,g_2,g_3\in \Gg$ and every $C_1,C_2\in \RR_{\geq 0}$, if $g_1\mathrel{\underset{C_1}{\eqsim}}g_2$ and $g_2\mathrel{\underset{C_2}{\eqsim}}g_3$, then $g_1\mathrel{\underset{C_1+C_2}{\eqsim}}g_3$.
 \item For every $g_1,g_2,\lambda_1,\lambda_2\in \Gg$ and every $C\in \RR_{\geq 0}$, if $g_1\mathrel{\underset{C}{\eqsim}}g_2$, then $\lambda_1g_1\lambda_2\mathrel{\underset{C}{\eqsim}}\lambda_1g_2\lambda_2$.
  \item For every $g\in\Gg$ and $x\in \Ng$, $gx\mathrel{\underset{1}{\eqsim}}xg$.
\end{enumerate}
\end{lem}

\begin{proof}
Items (1) and (3) hold by definition. Item (2) follows from Lemma~\ref{lem=Ginvmetric}.
\end{proof}

We state the following commutator calculus, which will be employed in Section~\ref{sec=CEabel}. This follows from a direct computation, so we omit the proof.

\begin{lem}\label{lem=commutatoreq}
Let $\Gg$ be a group. Then for every $g_1,g_2,g_3\in \Gg$, the following hold.
\begin{enumerate}[label=\textup{(\arabic*)}]
 \item $[g_1,g_2g_3]=[g_1,g_2][g_1,g_3][[g_1,g_3]^{-1},g_2]$.
 \item $[g_1,g_2g_3]=[g_1,g_2]g_2[g_1,g_3]g_2^{-1}$.
\end{enumerate}
\end{lem}

\subsection{Cohomological results on $\WW(\Gg,\Ng)$ and $\Wcal(\Gg,\Lg,\Ng)$}\label{subsec=cohom}

For a group $\Gg$, it is well known that the space $\QQQ(\Gg)/\HHH^1(\Gg)$ is isomorphic to a certain subspace of the second \emph{bounded cohomology} $\HHH^2_b(\Gg)$ with trivial real coefficient (Lemma~\ref{lem=comparison}). In Setting~\ref{setting=itsumono}, in \cite{KKMMM} the authors relate $\WW(\Gg,\Ng)$ with \emph{ordinary group cohomology} under the assumption that the quotient group is \emph{boundedly $3$-acyclic}. In this subsection, we recall related definitions and results. We refer the reader to the books \cite{Monodbook} and \cite{Frigerio} for more details on bounded cohomology.

\begin{setting}\label{setting=QG}
Under Setting~$\ref{setting=itsumono}$, \emph{i.e.,} for a group $\Gg$ and a normal subgroup $\Ng$  of $\Gg$, set $\QG=\Gg/\Ng$. Let $\pp\colon \Gg \twoheadrightarrow \QG$ be the natural group quotient map.
\end{setting}

In Settings~\ref{setting=itsumono} and \ref{setting=QG}, we consider the short exact sequence
\begin{equation}\label{eq=shortex}
1 \longrightarrow \Ng \stackrel{i}{\longrightarrow} \Gg \stackrel{\pp}{\longrightarrow} \QG \longrightarrow 1
\end{equation}
of groups.

We first briefly recall the definition of ordinary and bounded cohomology of groups. Let $\Gg$ be a group. Let $n\in \mathbb{Z}$. Let $C^n(\Gg) = C^n(\Gg; \RR)$ be the space of real-valued functions on
$\Gg^{n}$
if $n\geq 0$, and set $C^n(\Gg) = 0$ if $n < 0$. Here, we set $\Gg^0$ as the one-point set.
Define $\delta \colon C^n(\Gg) \to C^{n+1}(\Gg)$ by
\[ \delta c(g_0, \cdots, g_n) = c(g_1, \cdots, g_n) + \sum_{i=1}^n (-1)^i c(g_0, \cdots, g_{i-1} g_i, \cdots, g_n) + (-1)^{n+1} c(g_0, \cdots, g_{n-1}).\]
The \emph{$n$-th group cohomology} of $\Gg$ (with trivial real coefficient) is the $n$-th cohomology group of the cochain complex $(C^{\ast}(\Gg), \delta)$. Let $C^n_b(\Gg)$ be the space of bounded real-valued functions on
$\Gg^{n}$
if $n\geq 0$; $C^n_b(\Gg)=0$ if $n<0$. Then $C^{\ast}_b(\Gg)$ is a subcomplex of $C^{\ast}(G)$. Define $\HHH^{\ast}_b(G)$ as the cohomology of $C^{\ast}_b(G)$: this is \emph{bounded cohomology} of $\Gg$ (with trivial real coefficient). The map $c_{\Gg}^n \colon \HHH^{n}_b(\Gg) \to \HHH^{n}(\Gg)$ induced by the inclusion $C^{n}_b(\Gg) \hookrightarrow C^{n}(\Gg)$ is called the \emph{$n$-th comparison map}. We abbreviate $c_{\Gg}^n$ as $c_{\Gg}$ if $n$ is clear.

\begin{lem}[see for instance \cite{Calegari}]\label{lem=comparison}
Let $\Gg$ be a group. Then the quotient space $\QQQ(\Gg)/\HHH^1(\Gg)$ is isomorphic to the kernel of $c_{\Gg}^2\colon \HHH^{2}_b(\Gg) \to \HHH^{2}(\Gg)$.
\end{lem}

Let $C^{\ast}_{/b}(\Gg)$ denote the quotient complex $C^{\ast}(\Gg)/ C^{\ast}_b(\Gg)$. Then we define $\HHH^*_{/b}(\Gg)$ to be the cohomology of the cochain complex $C^{\ast}_{/b}(\Gg)$. In relation to Lemma~\ref{lem=comparison}, it is straightforward to see that the natural map $\QQQ(\Gg) \to \HHH^1_{/b}(\Gg)$ provides an isomorphism
\begin{equation}\label{eq=/bisom}
\QQQ(\Gg)\cong \HHH^1_{/b}(\Gg).
\end{equation}
See \cite[proof of Theorem~2.50]{Calegari} for more details.

We next recall the definition of  \emph{bounded $n$-acyclicity} of groups from \cite{Ivanov12} and \cite{MR21}. In particular, bounded $3$-acyclicity is an important assumption of the main results of \cite{KKMMM}.

\begin{defn}[bounded $n$-acyclicity]\label{defn=bdd_acyc}
Let $n\in \NN$. A group $\Gg$ is said to be \emph{boundedly $n$-acyclic} if $\HHH^i_b(\Gg) = 0$ holds for every $i\in \NN$ with $i \leq n$. We say that $\Gg$ is \emph{boundedly acyclic} if for every $n\in \NN$, $\Gg$ is boundedly $n$-acyclic.
\end{defn}

We collect known facts on the study of bounded acyclicity by various researchers; these results except (1) and (2) in Theorem~\ref{thm=bdd_acyc} are not used in the present paper. The class of amenable groups contains that of solvable groups; in particular, every nilpotent group is amenable.  We refer the reader to \cite{MR}, \cite{FLM1} and \cite{FLM2} for more details on the study of bounded acyclicity.

\begin{thm}[known results for boundedly $n$-acyclic groups]\label{thm=bdd_acyc}
The following hold.
\begin{enumerate}[label=\textup{(}$\arabic*$\textup{)}]
\item \textup{(}\cite{Gr}\textup{)} Every amenable group is boundedly acyclic.
\item \textup{(}see \cite{MR21}\textup{)} Let $n\in \NN$. Let $1\to N \to G \to \Gamma \to 1$ be a short exact sequence of groups. Assume that $N$ is boundedly $n$-acyclic. Then $G$ is boundedly $n$-acyclic if and only if $\Gamma$ is.
\item \textup{(}\cite{MatsumotoMorita}\textup{)}  Let $m\in \NN$. Then, the group $\Homeo_{c}(\RR^m)$ of homeomorphisms on $\RR^m$ with compact support is boundedly acyclic.
\item \textup{(}combination of \cite{Mon04} and \cite{MS04}\textup{)}
For $m \in \NN_{\geq 3}$, every lattice in $\SL(m,\RR)$ is boundedly $3$-acyclic.
\item \textup{(}\cite{BucherMonod}\textup{)} Burger--Mozes groups \textup{(}\cite{BurgerMozes}\textup{)} are boundedly $3$-acyclic.
\item \textup{(}\cite{Monod2021}\textup{)} Richard Thompson's group $F$ is boundedly acyclic.
\item \textup{(}\cite{Monod2021}\textup{)} Let $L$ be an arbitrary group. Let $\Gamma$ be an infinite amenable group. Then the wreath product $L\wr \Gamma =\left( \bigoplus_{\Gamma}L\right)\rtimes \Gamma$ is boundedly acyclic.
\item \textup{(}\cite{MN21}\textup{)} For $m\in \NN_{\geq 2}$, the identity component $\Homeo_0(S^m)$ of the group of orientation-preserving homeomorphisms of $S^n$ is boundedly $3$-acyclic. The group $\Homeo_0(S^3)$ is boundedly $4$-acyclic.
\end{enumerate}
\end{thm}

In Settings~\ref{setting=itsumono} and \ref{setting=QG}, consider short exact sequence \eqref{eq=shortex}. Then, the following exact sequence
\[
0 \to \HHH^1(\QG) \to \HHH^1(\Gg) \to \HHH^1(\Ng)^{\Gg} \to \HHH^2(\QG) \to \HHH^2(\Gg),
\]
called the \emph{five-term exact sequence} of group cohomology, is well known. The main result in \cite{KKMMM} is the counterpart of this five-term exact sequence for $\HHH_{/b}$ (recall \eqref{eq=/bisom}).

\begin{thm}[{\cite[Theorem~1.5]{KKMMM}}]\label{thm=KKMMMmain1}
Assume Settings~$\ref{setting=itsumono}$ and $\ref{setting=QG}$. Then there exists an $\RR$-linear map $\tau_{/b} \colon \QQQ(\Ng)^\Gg \to \HHH^2_{/b}(\QG)$ which satisfies the following.
\begin{enumerate}[label=\textup{(}$\arabic*$\textup{)}]
\item The following sequence is exact:
\begin{equation}
0 \to \QQQ(\QG) \to \QQQ(\Gg) \to \QQQ(\Ng)^{\Gg} \xrightarrow{\tau_{/b}} \HHH^2_{/b}(\QG) \to \HHH^2_{/b}(\Gg).
\end{equation}
\item The following diagram is commutative:
\begin{equation}\label{eq=diagram}
\xymatrix{
0 \ar[r] & \HHH^1(\QG) \ar[r] \ar[d] & \HHH^1(\Gg) \ar[r] \ar[d] & \HHH^1(\Ng)^{\Gg} \ar[r] \ar[d] & \HHH^2(\QG) \ar[r] \ar[d] & \HHH^2(\Gg) \ar[d] \\
0 \ar[r] & \QQQ(\QG) \ar[r] & \QQQ(\Gg) \ar[r] & \QQQ(\Ng)^{\Gg} \ar[r]^{\tau_{/b}} & \HHH^2_{/b}(\QG) \ar[r] & \HHH^2_{/b}(\Gg).
}
\end{equation}
\end{enumerate}
\end{thm}

From Theorem~\ref{thm=KKMMMmain1}, diagram chasing yields the following result; see \cite[Subsection~4.1]{KKMMM} for details.

\begin{thm}[{\cite[Theorem~1.9]{KKMMM}}]\label{thm=KKMMMmain2}
Assume Settings~$\ref{setting=itsumono}$ and $\ref{setting=QG}$. Assume that $\QG$ is boundedly $3$-acyclic. Then there exists an isomorphism
\[
\WW(\Gg,\Ng) \cong \Im (c_{\Gg}^2 \colon \HHH^2_b(\Gg) \to \HHH^2(\Gg)) \cap \Im (\pp^{\ast} \colon \HHH^2(\QG) \to \HHH^2(\Gg)).
\]
\end{thm}

In the setting of Theorem~\ref{thm=KKMMMmain2}, determining $\Rdim \WW(\Gg,\Ng)$ reduces to the study of $c_{\Gg}^2$ and $\pp^{\ast}$. The map $\pp^{\ast}$ is between ordinary group cohomology and relatively tractable. Despite the fact that studying the image of  $c_{\Gg}^2$ is difficult in general, the following theorem helps our study for Gromov-hyperbolic groups.

\begin{thm}[\cite{Gr}, \cite{Mineyev}]\label{thm=Mineyev}
Let $\Gg$ be a non-elementary Gromov-hyperbolic group. Then, for every $n\in \NN_{\geq 2}$, the $n$-th comparison map $c_{\Gg}^n\colon \HHH^n_b(\Gg)\to \HHH^n(\Gg)$ is surjective.
\end{thm}

With the aid of Theorems~\ref{thm=KKMMMmain1} and \ref{thm=KKMMMmain2}, we have a useful sufficient condition on pairs $(\Gg,\Ng)$ for having finite dimensional $\WW(\Gg,\Ng)$.

\begin{cor}[{\cite[Theorems~1.9 and 1.10]{KKMMM}}]\label{cor=fdW}
Assume Settings~$\ref{setting=itsumono}$ and $\ref{setting=QG}$. Assume that $\QG$ is boundedly $3$-acyclic. If either $\Gg$ or $\QG$ is finitely presented, then $\Rdim \WW(\Gg,\Ng)<\infty$.
\end{cor}

The following corollary is important in applications of Theorem~\ref{mthm=main}, including Subsection~\ref{subsec=induced}.
\begin{cor}\label{cor=nilp}
Let $\Gg$ be a finitely generated group and $q\in \NN_{\geq 2}$. Let $N$ be a normal subgroup of $\Gg$ satisfying $N\geqslant \gamma_q(\Gg)$. Then the spaces $\WW(\Gg,\Ng)$ is finite dimensional.
\end{cor}

\begin{proof}
Since $\QG=\Gg/\Ng$ is a finitely generated nilpotent group, it is finitely presented (see for instance \cite[2.3 and 2.4]{Mannbook}). Now  Corollary~\ref{cor=fdW} applies to this case.
\end{proof}

The following example shows a huge difference between $\WW(\Gg,\Ng)=\QQQ(\Ng)^{\Gg}/(\HHH^1(\Ng)^{\Gg}+i^{\ast}\QQQ(\Gg))$ and $\QQQ(\Ng)^{\Gg}/\HHH^1(\Ng)^{\Gg}$.

\begin{exa}[examples with $\Rdim \left(\QQQ(\Ng)^{\Gg} / \HHH^1(\Ng)^{\Gg}\right)=\infty$]\label{exa=AH}
Let $\Gg$ be an acylindrically hyperbolic group and $\Ng$ an infinite normal subgroup of $G$; see \cite{OsinAH} for this concept. For instance, the following groups are acylindrically hyperbolic: non-elementary Gromov-hyperbolic groups, the mapping class group $\Mod(\Sigma_{\genus})$ of $\Sigma_{\genus}$ with $\genus\geq 2$, groups defined by $s$ generators and $r$ relations with $s-r\geq 2$ \cite{Osin15}, and $\Aut(H)$ with $H$ non-elementary Gromov-hyperbolic \cite{GenevoisHorbez}.
In this setting, the space $i^{\ast}\QQQ(\Gg)$ (and thus $\QQQ(N)^G$) is known to be infinite-dimensional (\cite{BF}, \cite[Corollary 4.3]{FW}).
Moreover, we can show that $\Rdim\left(\QQQ(\Ng)^{\Gg}/\HHH^1(\Ng)^{\Gg}\right)=\infty$ as follows: by \cite[Corollary~1.5]{OsinAH}, $\Ng$ is also acylindrically hyperbolic. Hence, so is $[\Ng,\Ng]$. Then the argument of \cite[Corollary 4.3]{FW} shows that the image of $i_{[\Ng,\Ng],\Gg}^{\ast}\colon \QQQ(\Gg)\to \QQQ([\Ng,\Ng])^{\Gg}$ has an infinite dimensional image. Finally, note that $i_{[\Ng,\Ng],\Ng}^{\ast}\HHH^1(N)^G=0$.

Contrastingly, Corollary~\ref{cor=fdW} states that if $\QG$ is boundedly $3$-acyclic and if either $\Gg$ or $\QG$ is finitely presented, then $\Rdim \WW(\Gg,\Ng)<\infty$; see also Corollary~\ref{cor=nilp}.
\end{exa}

Together with Theorem~\ref{thm=Mineyev}, we have several examples of pairs $(\Gg,\Ng)$  for which $\Rdim \WW(\Gg,\Ng)$ is computed; these examples include ones where $\WW(\Gg,\Ng)$ is non-zero finite dimensional (for instance as in Theorem~\ref{thm=WW}).

For a triple $(\Gg,\Lg,\Ng)$ of a group $\Gg$ and two normal subgroups $\Lg,\Ng$ with $\Lg\geqslant \Ng$, the computation of $\Rdim \Wcal(\Gg,\Lg,\Ng)$ is rather difficult. Nevertheless, we might reduce the computation of $\Rdim\Wcal(\Gg,\Lg,\Ng)$ to that of $\Rdim \WW(\Gg,\Ng)$ and $\Rdim \WW(\Gg,\Lg)$ in a certain manner; see Theorem~\ref{thm=dimWcal} in Subsection~\ref{subsec=Wcal}.


We have the following corollary to Theorem~\ref{thm=KKMMMmain1}.

\begin{cor}\label{cor=intersection}
Assume Settings~$\ref{setting=itsumono}$ and $\ref{setting=QG}$. Then, there exists an isomorphism
\[
(\HHH^1(\Ng)^{\Gg} \cap i^*\QQQ(\Gg)) / i^* \HHH^1(\Gg) \cong \Im(\HHH^1(\Ng)^{\Gg} \to \HHH^2(\QG)) \cap \Im (\HHH^2_b(\QG) \to \HHH^2(\QG)).
\]
In particular, if $\QG$ is boundedly $2$-acyclic, then we have
\[
\HHH^1(\Ng)^{\Gg} \cap i^* \QQQ(\Gg) = i^*\HHH^1(\Gg).
\]
\end{cor}

\begin{proof}
The image of the map $\HHH^1(\Ng)^{\Gg} \to \HHH^2(\QG)$ restricted to $\HHH^1(\Ng)^{\Gg} \cap i^*\QQQ(\Gg)$ is equal to
$\Im(\HHH^1(\Ng)^{\Gg} \to \HHH^2(\QG)) \cap \Im (\HHH^2_b(\QG) \to \HHH^2(\QG))$
by Theorem~\ref{thm=KKMMMmain1} and the exactness of $\HHH_b^2(\QG) \to \HHH^2(\QG) \to \HHH_{/b}^2(\QG)$.
Moreover, the kernel of the restricted map coincides with $i^* \HHH^1(\Gg)$ by the exactness of $\HHH^1(\Gg) \to \HHH^1(\Ng)^{\Gg} \to \HHH^2(\QG)$.
\end{proof}

In Subsection~\ref{subsec=criterion}, we will prove a variant (Proposition~\ref{prop=intersectiontriple}) of Corollary~\ref{cor=intersection} for a triple $(\Gg,\Lg,\Ng)$; that result will be employed in Section~\ref{sec=CE}.

The following result is proved in \cite[Proposition~1.6]{KKMM1}. Despite the fact that the proof does not use Theorem~\ref{thm=KKMMMmain1}, this result suggests that studying $\WW(\Gg,\Ng)$ is related to studying short exact sequence \eqref{eq=shortex} itself. We will employ Proposition~\ref{prop=vsplit} in the proof of Theorem~\ref{thm=orerei}.

\begin{prop}[\cite{KKMM1}]\label{prop=vsplit}
Assume Settings~$\ref{setting=itsumono}$ and $\ref{setting=QG}$. Assume that \eqref{eq=shortex} \emph{virtually splits}, meaning that there exist $\QG_1\leqslant \QG$ with $[\QG:\QG_1]<\infty$ and a group homomorphism $s\colon \QG_1\to \Gg$ such that $\pp \circ s=\mathrm{id}_{\QG_1}$. Then $\QQQ(\Ng)^{\Gg}=i^{\ast}\QQQ(\Gg)$ and in particular, $\WW(\Gg,\Ng)=0$.
\end{prop}

In the rest of this subsection, we exhibit examples for which $\Rdim \WW(\Gg,[\Gg,\Gg])$ is computed from \cite{KKMMM}. We briefly recall some terminology appearing in Theorem~\ref{thm=WW}. For $\genus \in \NN_{\geq 2}$, the abelianization map $\pi_1(\Sigma_{\genus})\twoheadrightarrow \ZZ^{2\genus}$ induces a group homomorphism  from the automorphism group $\Aut(\pi_1(\Sigma_{\genus}))$ to $\GL({2\genus},\ZZ)$. The inverse image of $\SL({2\genus},\ZZ)$ under this map is written as $\Aut_+(\pi_1(\Sigma_{\genus}))$. The Dehn--Nielsen--Baer theorem states that $\Out_+(\pi_1(\Sigma_{\genus}))=\Aut_+(\pi_1(\Sigma_{\genus}))/\Inn(\pi_1(\Sigma_{\genus}))$ is isomorphic to the mapping class group $\Mod(\Sigma_{\genus})$ of $\Sigma_{\genus}$. The map on $\Aut_+(\pi_1(\Sigma_{\genus}))$ induces
\[
\mathrm{sr}\colon \Aut_+(\pi_1(\Sigma_{\genus}))\twoheadrightarrow \Out_+(\pi_1(\Sigma_{\genus}))\cong \Mod(\Sigma_{\genus})\stackrel{\overline{\mathrm{sr}}}{\twoheadrightarrow} \Sp({2\genus},\ZZ);
\]
the map $\overline{\mathrm{sr}}$ is called the \emph{symplectic representation} of $\Mod(\Sigma_{\genus})$. The \emph{Torelli group} $\mathcal{I}(\Sigma_{\genus})\leqslant \Mod(\Sigma_{\genus})$ is the kernel of $\overline{\mathrm{sr}}$. An element of $\Mod(\Sigma_{\genus})$ is \emph{pseudo-Anosov} if and only if for every (equivalently, some) lift $\chi\in \Aut_+(\pi_1(\Sigma_{\genus}))$, the semi-direct product $\pi_1(\Sigma_{\genus})\rtimes_{\chi}\ZZ$ is Gromov-hyperbolic (\cite{Thu1986}, see also \cite{MR1402300}).

For $n\in \NN_{\geq 2}$ and the free group $F_n$ of rank $n$, we can define the group homomorphism from $\Aut(F_n)$ to $\GL(n,\ZZ)$ via the abelianization of $F_n$. The \emph{IA-automorphism group} $\mathrm{IA}_n\leqslant \Aut(F_n)$ is the kernel of this map. We say that $\chi\in \Aut(F_n)$ is \emph{atoroidal} if no non-zero power of $\chi$ fixes the conjugacy class of any element of $F_n$; $\chi$ is atoroidal if and only if $F_n\rtimes_{\chi}\ZZ$ is Gromov-hyperbolic (\cite{BF1992}).

\begin{thm}[{See \cite[Theorems~4.5, 1.1, 1.2 and 4.11]{KKMMM}}]\label{thm=WW}
The following hold true.
\begin{enumerate}[label=\textup{(\arabic*)}]
  \item Let $F$ be a free group. Then, $\WW(F,\gamma_2(F))=0$.
  \item Let $\genus \in \NN_{\geq 2}$. Then,
\[
\Rdim \WW(\pi_1(\Sigma_{\genus}),\gamma_2(\pi_1(\Sigma_{\genus})))=1.
\]
  \item Let $\genus \in \NN_{\geq 2}$. Assume that $\chi\in \Aut_+(\pi_1(\Sigma_{\genus}))$ represents a pseudo-Anosov element in the Torelli group $\mathcal{I}(\Sigma_{\genus})$ of $\Mod(\Sigma_{\genus})(\cong \Out_+(\Sigma_{\genus}))$. Then for the semi-direct product $\pi_1(\Sigma_{\genus})\rtimes_{\chi}\ZZ$, we have
\[
\Rdim \WW(\pi_1(\Sigma_{\genus})\rtimes_{\chi}\ZZ,\gamma_2(\pi_1(\Sigma_{\genus})\rtimes_{\chi}\ZZ))=2\genus +1.
\]
  \item Let $F_n$ be a free group of rank $n\in\NN_{\geq 2}$. Assume that $\chi\in \Aut(F_n)$ lies in the IA-automorphism group $\mathrm{IA}_n$ and that $\chi$ is atoroidal. Then for the semi-direct product $F_n\rtimes_{\chi}\ZZ$, we have
\[
\Rdim \WW(F_n\rtimes_{\chi}\ZZ,\gamma_2(F_n\rtimes_{\chi}\ZZ))=n.
\]
\end{enumerate}
\end{thm}

The group $\pi_1(\Sigma_{\genus})\rtimes_{\chi}\ZZ$ in (3) above is isomorphic to the fundamental group of the mapping torus of a homeomorphism on $\Sigma_{\genus}$ that induces the automorphism  $\chi\in \Aut_+(\pi_1(\Sigma_{\genus}))$.

\subsection{Some functional analytic results}\label{subsec=FA}
In this subsection we collect some functional analytic results, which will be used in Section~\ref{sec=defect}. The first result is a corollary to the inverse mapping theorem for Banach spaces. For two real Banach spaces $(X,\|\cdot\|_X)$ and $(Y,\|\cdot\|_Y)$ and for a linear operator $T\colon (X,\|\cdot\|_X)\to(Y,\|\cdot\|_Y)$, the \emph{operator norm} $\|T\|_{\mathrm{op}}$ is defined by $\|T\|_{\mathrm{op}}=\sup\left\{\frac{\|T\xi\|_Y}{\|\xi\|_X}\,\middle| \,\xi\in X\setminus \{0\}\right\}$.

\begin{prop}\label{prop=inverse}
Let $(X,\|\cdot\|_X)$ and $(Y,\|\cdot\|_Y)$ be two real Banach spaces. Let $T\colon X\to Y$ be an injective linear operator with $\|T\|_{\mathrm{op}}<\infty$. Assume that $\ell=\Rdim (Y/T(X))$ is finite. Take an arbitrary basis of $Y/T(X)$. Take an arbitrary set $\{\eta_1,\ldots,\eta_{\ell}\}\subseteq Y$  of representatives of this basis. Then, there exist $C_1, C_2\in \mathbb{R}_{>0}$ such that for every $\xi\in X$ and for every $(a_1,\ldots a_{\ell})\in \RR^{\ell}$,
\[
\left\|T\xi+\sum_{j\in \{1,\ldots,\ell\}} a_j\eta_j \right\|_Y\geq C_1^{-1}\cdot \left(\|\xi\|_X+C_2^{-1}\cdot\sum_{j\in \{1,\ldots,\ell\}} |a_j|\right)
\]
holds.
\end{prop}
Here, if $\ell=0$, then we regard $\{\eta_1,\ldots,\eta_{\ell}\}=\emptyset$ and $\RR^{\ell}=0$ so that we can take $C_2=1$.

\begin{proof}[Proof of Proposition~$\ref{prop=inverse}$]
Let $Z_0$ be the $\RR$-span of $\eta_1,\ldots,\eta_{\ell}$. Then the linear map
\[
S\colon \RR^{\ell}\stackrel{\cong}{\longrightarrow} Z_0; \ \RR^{\ell}\ni (a_1,\ldots ,a_{\ell})\mapsto \sum_{j\in \{1,\ldots,\ell\}} a_j\eta_j \in Z_0
\]
is an isomorphism. Equip $Z_0$ and $\RR^{\ell}$ with the induced norm from $\|\cdot\|_Y$ and the $\ell^1$-norm $\|\cdot\|_1$, respectively. Define  $C_2$ as the operator norm of $S^{-1}\colon (Z_0,\|\cdot\|_Y)\to (\RR^{\ell},\|\cdot\|_1)$. Thus we have for every $(a_1,\ldots,a_{\ell})\in \RR^{\ell}$,
\begin{equation}\label{eq=normell1}
\left\|\sum_{j\in \{1,\ldots,\ell\}} a_j\eta_j \right\|_Y\geq C_2^{-1}\cdot\sum_{j\in \{1,\ldots,\ell\}} |a_j|.
\end{equation}

Define $Z=X\oplus Z_0$ with the norm $\|\cdot\|_Z$ defined for every $\xi\in X$ and $\zeta\in Z_0$ by $\|(\xi,\zeta)\|_{Z}=\|\xi\|_X+\|\zeta\|_Y$. Then $(Z,\|\cdot\|_Z)$ is a real Banach space. Define a map $\tilde{T}\colon (Z,\|\cdot\|_Z)\to (Y,\|\cdot\|_Y)$ for every $\xi\in X$ and $\zeta\in Z_0$ by $\tilde{T}(\xi,\zeta)=T\xi+\zeta$. Then, it is straightforward to show that $\tilde{T}$ is a bijective linear operator with $\|\tilde{T}\|_{\mathrm{op}}\leq \|T\|_{\mathrm{op}}+1$. Therefore, the inverse mapping theorem applies to $\tilde{T}$ and we conclude that $\|\tilde{T}^{-1}\|_{\mathrm{op}}<\infty$.
Set $C_1=\|\tilde{T}^{-1}\|_{\mathrm{op}}$. Then we have for every $\xi\in X$ and $\zeta\in Z_0$,
\begin{equation}\label{eq=normsum}
\|T\xi+\zeta\|_Y\geq C_1^{-1}(\|\xi\|_X+\|\zeta\|_Y).
\end{equation}
By combining \eqref{eq=normell1} and \eqref{eq=normsum}, we complete the proof.
\end{proof}

\begin{rem}\label{rem=const}
The constant $C_1$ is given by the inverse mapping theorem (coming from the Baire category theorem), and it is ineffective. We also remark that $C_1$ (as well as $C_2$) in the proof does depend on the choice of $\eta_1,\ldots,\eta_{\ell}$. Indeed, the linear operator $\tilde{T}$ depends on the choice.
\end{rem}

The next result is proved in \cite[Theorem~7.4]{KKMMM}. In Setting~\ref{setting=itsumono}, we can see the defect as a map $\DD\colon\QQQ(\Ng)^{\Gg} \to \RR_{\geq 0}$.
In fact, $\DD$ descends to the map $\hat{\DD}\colon \QQQ(\Ng)^{\Gg}/\HHH^1(\Ng)^{\Gg}\to \mathbb{R}_{\geq 0}$. It is straightforward to see that $(\QQQ(\Ng)^{\Gg}/\HHH^1(\Ng)^{\Gg},\hat{\DD})$ is a real normed space: this $\hat{\DD}$ is called the \emph{defect norm} on $\QQQ(\Ng)^{\Gg}/\HHH^1(\Ng)^{\Gg}$.

\begin{prop}[\cite{KKMMM}]\label{prop=defectBanach}
Assume Setting~$\ref{setting=itsumono}$. Then $(\QQQ(\Ng)^{\Gg}/\HHH^1(\Ng)^{\Gg},\hat{\DD})$ is a Banach space.
\end{prop}

Combination of  Proposition~\ref{prop=defectBanach}, Proposition~\ref{prop=inverse} for the case of $\ell=0$ (this corresponds to the inverse mapping theorem for Banach spaces) and Theorem~\ref{thm=Bavard} yields the following theorem, which was obtained in \cite{KKMMM} for the case of $\Lg=\Gg$.

\begin{thm}\label{thm=biLip}
Let $\Gg$ be a group, and let $\Lg$ and $\Ng$ be two normal subgroups of $\Gg$ satisfying $\Lg\geqslant \Ng$. Assume that $\Wcal(\Gg,\Lg,\Ng)=0$. Then, $\scl_{\Gg,\Lg}$ and $\scl_{\Gg,\Ng}$ are bi-Lipschitzly equivalent on $[\Gg,\Ng]$.
\end{thm}

\begin{proof}
If $\Lg=\Gg$, then the assertion was proved in \cite[Theorem~2.1~(1)]{KKMMM} (as we mentioned in Remark~\ref{rem=zerodim}). The proof there can be extended to the general case; see also Remark~\ref{rem=l=0}.
\end{proof}

\subsection{A lemma on function spaces}
The following lemma will be employed in the constructions of maps $\Ptau$ (in the proof of Theorem~\ref{thm=main1precise}) and $\PPs$ (in the proof of Theorem~\ref{mthm=main}).

\begin{lem}\label{lem=dual}
Let $W$ be a set. Set $\RR^W$ as the real vector space of all real-valued functions on $W$. Let $\ell\in \NN$. Let $\Xi$ be a real subspace of $\RR^W$ with $\Rdim \Xi=\ell$. Then, there exist $\wf_1,\ldots ,\wf_{\ell}\in W$ and $\xi_1,\ldots ,\xi_{\ell}\in \Xi$ such that for every $i,j\in \{1,\ldots ,\ell\}$,
\begin{equation}\label{eq=delta}
\xi_j(\wf_i)=\delta_{i,j}.
\end{equation}
\end{lem}

Lemma~\ref{lem=dual} has a direct proof by induction on $\ell$. Here we present another conceptual proof.

\begin{proof}[Proof of Lemma~$\ref{lem=dual}$]
Set $A$ be the free $\RR$-module with basis $W$. Then $\RR^W$ is naturally isomorphic to $\Hom (A,\RR)$. Let $e\colon \Hom (A,\RR)\stackrel{\cong}{\to}\RR^W$ be the isomorphism. Take $\eta_1,\ldots,\eta_{\ell}$ a basis of $\Xi$. Define $\eta\colon A\to \RR^{\ell}$ by
\[
\eta\colon A\to \RR^{\ell};\ A\ni a\mapsto (\eta_1(a),\ldots,\eta_{\ell}(a))\in\RR^{\ell}.
\]
In what follows, we prove that $\eta$ is surjective. Set $B=\Ker(\eta)$ and $A_1=A/B$. Note that $\Rdim A_1\leq \Rdim \RR^{\ell}=\ell$. Define $\pp\colon A\twoheadrightarrow A_1$ as the natural projection; $\pp$ induces a homomorphism $\pp^{\ast}\colon \Hom(A_1,\RR)\to \Hom(A,\RR)$. Hence we have a homomorphism
\[
e\circ \pp^{\ast}\colon \Hom(A_1,\RR)\to \RR^W.
\]
We claim that $(e\circ \pp^{\ast})(\Hom(A_1,\RR))=\Xi$. Indeed, every $\xi \in \Xi$ is a linear combination of $\eta_1,\ldots,\eta_{\ell}$ and hence $(e^{-1}(\xi))(B)=0$ holds. Thus $\xi$ induces an element in $\Hom(A_1,\RR)$: this means that $\Xi\subseteq (e\circ \pp^{\ast})(\Hom(A_1,\RR))$. Now we have
\[
\ell=\Rdim \Xi\leq \Rdim \left((e\circ \pp^{\ast})(\Hom(A_1,\RR))\right)\leq \Rdim \Hom(A_1,\RR)= \Rdim A_1 \leq \ell.
\]
Therefore, we conclude that $(e\circ \pp^{\ast})(\Hom(A_1,\RR))=\Xi$. The argument above moreover shows that $\Rdim A_1 =\ell$. It then follows that $\eta$ is surjective. Also, the image of the map $e\circ \eta^{\ast}\colon \Hom(\RR^{\ell},\RR) \to \RR^W$ equals $\Xi$.

Since $A$ is generated by $W$ and $\eta$ is surjective, there exist $\wf_1,\ldots,\wf_{\ell}\in W$ such that $\eta(\wf_1),\ldots,\eta(\wf_{\ell})$ form a basis of $\RR^{\ell}$. Let $\wf_1',\ldots,\wf_{\ell}'$ be the dual basis to $\wf_1,\ldots,\wf_{\ell}$. Set $\xi_1,\ldots ,\xi_{\ell}$ by $\xi_j=(e\circ \eta^{\ast})(\wf_j')$ for every $j\in \{1,\ldots,\ell\}$. Then, by construction, we have \eqref{eq=delta}.
\end{proof}

\subsection{Basic concepts in coarse geometry}\label{subsec=coarsegeom}
In this subsection, we briefly recall basic concepts in coarse geometry. We  refer the reader to \cite{Roe} for a comprehensive treatise on this field. In this subsection and the next subsection, recall from our notation and conventions at the end of Section~\ref{sec=intro2} (and Remark~\ref{rem=bf}) that we use bold symbols for coarse notions, such as a coarse map $\bfalpha$ and a coarse subspace $\bA$; for their representatives we use non-bold symbols, such as $\alpha$ and $A$.

\begin{defn}[\cite{Roe}]\label{defn=coarsestr}
Let $X$ be a set. A \emph{coarse structure} on $X$ is a subset $\Ecal$ of the power set $\mathcal{P}(X\times X)$ of $X\times X$ that satisfies the following conditions.
\begin{enumerate}[label=\textup{(}\arabic*\textup{)}]
  \item For every $A,B\in \mathcal{P}(X\times X)$ with $A\subseteq B$, if $B\in \Ecal$, then $A\in \Ecal$.
  \item If $A,B\in \Ecal$, then $A\cup B\in \Ecal$.
  \item $\Delta_X=\{(x,x)\,|\,x\in X\}\in \Ecal$.
  \item If $E\in \Ecal$, then $E^{\mathrm{op}}=\{(y,x)\in X\times X\,|\, (x,y)\in E\}$ is an element of $\Ecal$.
  \item If $E_1,E_2\in \Ecal$, then
\[
E_1\circ E_2=\{(x,z)\in X\times X\,|\,\textrm{there exists $y\in X$ such that $(x,y)\in E_1$ and $(y,z)\in E_2$}\}
\]
is an element of $\Ecal$.
\end{enumerate}
A \emph{coarse space} $(X,\Ecal)$ means  a set $X$ equipped with a coarse structure $\Ecal$.
\end{defn}

Recall from Definition~\ref{defn=generalizedmet} the definition of almost generalized metric spaces.

\begin{defn}\label{defn=coarsestrd}
For an almost generalized metric space $(X,d)$, we define the \emph{coarse structure} $\Ecal_d\subseteq \mathcal{P}(X\times X)$ of $(X,d)$ as
\[
\Ecal_d=\left\{E\subseteq X\times X\,\middle|\, \sup_{(x_1,x_2)\in E}d(x_1,x_2)<\infty\right\}.
\]
\end{defn}


The following lemma is straightforward.
\begin{lem}\label{lem=coarsestr+C}
Let $(X,d)$ be an almost generalized metric space.
\begin{enumerate}[label=\textup{(}$\arabic*$\textup{)}]
  \item Let $C\in \mathbb{R}_{> 0}$ be a constant that satisfies \textup{(}$1_4^{+,C}$\textup{)} in Definition~$\ref{defn=generalizedmet}$. Let $d^{+}=d^{+,C}$ be the generalized metric defined in Lemma~$\ref{lem=+C}$. Then we have $\Ecal_{d^+}=\Ecal_d$.
  \item The family $\Ecal_d$ is a coarse structure.
\end{enumerate}
\end{lem}


Lemma~\ref{lem=coarsestr+C} justifies studying coarse geometry of the almost generalized metric space $([\Gg,\Ng],d_{\scl_{\Gg,\Ng}})$ in Setting~\ref{setting=itsumono}, as well as identifying $([\Gg,\Ng],d_{\scl_{\Gg,\Ng}})$ with $([\Gg,\Ng],d^{+}_{\scl_{\Gg,\Ng}})$ defined in \eqref{eq=+1/2} as coarse spaces.

\begin{defn}\label{defn=bounded}
Let $(X, \Ecal)$ be a coarse space.
\begin{enumerate}[label=\textup{(}$\arabic*$\textup{)}]
\item  A subset $A$ of $X$ is said to be \emph{bounded} if $A \times A \in \Ecal$. For an almost generalized metric space $(X,d)$, a subset $A \subseteq X$ is said to be \emph{$d$-bounded} if the \emph{diameter} of $A$ with respect to $d$, defined as
\[
\diam_d(A)=\sup\{d(x_1,x_2)\,|\, x_1,x_2\in A\}
\]
is finite.
\item Let $Y$ be a set. A map $\alpha\colon Y\to X$ is said to be \emph{bounded} if  $\alpha(Y)$ is bounded. For an almost generalized metric space $(X,d)$, a map $\alpha\colon Y\to X$ is said to be \emph{$d$-bounded} if $\alpha(Y)$ is $d$-bounded.
\end{enumerate}
\end{defn}

For an almost generalized metric space $(X,d)$ and $A\subseteq X$, $A$ is $d$-bounded if and only if $A$ is bounded with respect to the coarse structure $\Ecal_d$.




Let $X$ be a set and $E$ a subset of $X \times X$. For $x \in X$, set
\[ E_x = \{ y \in X \, | \, (y, x) \in E \}, \quad _x E = \{ y \in X \, | \, (x,y) \in E \}.\]
For $A \subseteq X$, set
\[ E[A] = \{ x \in X \, | \, \textrm{there exists $a \in A$ such that $(x,a) \in E$}\} = \bigcup_{a \in A} E_a. \]

\begin{defn}\label{defn=close}
Let $X$ be a set and $(Y, \Ecal')$ a coarse space. Two maps $\alpha,\beta \colon X \to Y$ are said to be \emph{close} if $(\alpha \times \beta)(\Delta_X) \in \Ecal'$. We write $\alpha \approx \beta$ to mean that $\alpha$ and $\beta$ are close. For a map $\alpha \colon X \to Y$, we write $\bfalpha=[\alpha]$ to indicate the equivalence class with respect to $\approx$ to which $\alpha$ belongs.
\end{defn}


\begin{defn}\label{defn=several_maps}
Let $(X, \Ecal)$ and  $(Y, \Ecal')$ be coarse spaces. Let $\alpha \colon X \to Y$ be a map.
\begin{enumerate}[label=\textup{(}$\arabic*$\textup{)}]
\item The map $\alpha$ is said to be \emph{controlled} if $E \in \Ecal$ implies $(\alpha \times \alpha)(E) \in \Ecal'$.

\item The map $\alpha$ is said to be \emph{coarsely proper} if $E' \in \Ecal'$ implies $(\alpha \times \alpha)^{-1}(E') \in \Ecal$. If $\alpha$ is moreover controlled, $\alpha$ is called a \emph{coarse embedding}.



\item Suppose that $\alpha$ is a controlled map. A \emph{coarse inverse of $\alpha$} is a controlled map $\beta \colon Y \to X$ such that $\beta \circ \alpha \approx {\rm id}_X$ and $\alpha \circ \beta \approx {\rm id}_Y$.

\item The map $\alpha$ is called a \emph{coarse equivalence} if $\alpha$ is controlled and has a coarse inverse.
\end{enumerate}
\end{defn}


\begin{defn}\label{defn=coarse_map}
Let $(X, \Ecal)$ and $(Y, \Ecal')$ be coarse spaces. A \emph{coarse map from $(X, \Ecal)$ to $(Y, \Ecal')$} is an equivalence class of controlled maps from $(X, \Ecal)$ to $(Y, \Ecal')$ with respect to the closeness $\approx$.
\end{defn}

Let $\Coarse$ denote the category whose objects are coarse spaces and whose morphisms are coarse maps.

We present what these concepts mean in the setting of almost metric spaces.

\begin{exa}\label{exa=several_maps}
Let $(X,d_X)$ and $(Y ,d_Y)$ be almost metric spaces. Let $\alpha \colon X \to Y$ be a map.
\begin{enumerate}[label=\textup{(}$\arabic*$\textup{)}]
\item Let $\beta\colon X\to Y$ be a map. Then $\alpha\approx \beta$ if and only if $\sup\limits_{x\in X}d_Y(\alpha(x),\beta(x))<\infty$ holds.
\item The map $\alpha$ is controlled if and only if for every $S \in \RR_{>0}$ there exists $T\in \RR_{>0}$ such that for every $x_1,x_2\in X$, $d_X(x_1,x_2) \le S$ implies $d_Y(\alpha(x_1), \alpha(x_2)) \le T$.


\item The map $\alpha$ is coarsely proper if and only if for every $S \in \RR_{>0}$ there exists $T\in \RR_{>0}$ such that for every $x_1,x_2\in X$, $d_X(x_1,x_2) > T$ implies $d_Y(\alpha(x_1),\alpha(x_2)) > S$.

\end{enumerate}
\end{exa}

The concept of quasi-isometries is defined for a map between almost metric spaces.

\begin{defn}[quasi-isometric embedding, quasi-isometry]
Let $(X,d_X)$ and $(Y,d_Y)$ be almost metric spaces. A map $\alpha$ is said to be \emph{large scale Lipschitz} if there exist $C \in \RR_{>0}$ and $D \in \RR_{\geq 0}$ such that for every $x_1,x_2\in X$, $d_Y(\alpha(x_1), \alpha(x_2)) \le C \cdot d_X(x_1,x_2) + D$ holds. The map $\alpha \colon (X,d_X) \to (Y,d_Y)$ is called a \emph{quasi-isometric embedding} (\emph{QI-embedding}) if there exist $C_1,C_2 \in \RR_{>0}$ and $D_1,D_2 \in \RR_{\geq 0}$ such that for every $x_1,x_2\in X$,
\[
C_1\cdot d_X(x_1,x_2) - D_1 \le d_Y(\alpha(x_1), \alpha(x_2)) \le C_2 \cdot d_X(x_1, y_1) + D_2
\]
holds. If $\alpha$ furthermore admits a coarse inverse that is a quasi-isometric embedding, then $\alpha$ is called a \emph{quasi-isometry}. The spaces $X$ and $Y$ are said to be \emph{quasi-isometric} if there exists a quasi-isometry from $X$ to $Y$.
\end{defn}


\begin{rem}\label{rem=QG}
A metric space $(X,d)$ is called a \emph{quasi-geodesic space} if there exist $a,b \in \mathbb{R}_{>0}$ such that for every $x,x' \in X$ there exist $x_0, \cdots, x_m$, where $m = \left\lceil \frac{d(x,x')}{b} \right\rceil$, satisfying
\[ x = x_0, \quad x' = x_m, \quad \textrm{and} \quad d(x_i, x_{i-1}) \leq  a \textrm{ for every }i\in \{1,\ldots, m\}.\]
Here $\lceil \cdot \rceil$ is the ceiling function.

If $(X,d_X)$ is a quasi-geodesic space and $(Y,d_Y)$ is an almost metric space, then every controlled map $\alpha\colon X\to Y$ is large scale Lipschitz. Indeed, for every $x,x'\in X$, take $x_0,\ldots,x_m$ as above. Then,
\[
d_Y(\alpha(x),\alpha(x'))\leq \sum_{i\in \{1,\ldots,m\}}d_Y(\alpha(x_{i-1}),\alpha(x_{i})).
\]
Since $\alpha$ is controlled, $\alpha$ must be large scale Lipschitz. In particular, for quasi-geodesic spaces $X$ and $Y$ if they are coarsely equivalent, then $X$ and $Y$ are in fact quasi-isometric.

Let $\Gg$ be a group and $\Ng$ a normal subgroup of $\Gg$. Then by definition, $([\Gg,\Ng],d_{\cl_{\Gg,\Ng}})$ is a quasi-geodesic space. In contrast, there is no reason to expect that so is $([\Gg,\Ng],d^+_{\scl_{\Gg,\Ng}})$.
\end{rem}


\subsection{A brief introduction to coarse groups and coarse kernels}\label{subsec=CK}
In this subsection, we present definitions in the theory of coarse groups, developed by Leitner and Vigolo \cite{LV}. As we described in the introduction and Section~\ref{sec=intro2}, the concept of the \emph{coarse kernel} (Definition~\ref{defn=coarsekernel}) of a coarse homomorphism is the main object in the comparative version, \emph{i.e.,} Theorem~\ref{mthm=main}. In the present paper, we only employ the notions in the setting of groups equipped with bi-invariant almost (generalized) metrics. Hence, the reader who is not familiar with this topic may consult Examples~\ref{exa=coarse group}, \ref{exa=precoarse}, \ref{exa=CKgroup} and \ref{exa=CKiota} only.

Let $(X, \Ecal)$ and $(Y, \Ecal')$ be coarse spaces. Set
\[ \Ecal_{X \times Y} = \{ E \subseteq (X \times Y)^2 \, | \, \textrm{$\pi_{13}(E) \in \Ecal$ and $\pi_{24}(E) \in \Ecal'$}\}.\]
Here, $\pi_{13}$ and $\pi_{24}$ are the maps defined by
\[ \pi_{13} \colon (X \times Y) \times (X \times Y) \to X \times X, ((x,y),(x', y')) \mapsto (x, x'),\]
\[ \pi_{24} \colon (X \times Y) \times (X \times Y) \to Y \times Y, ((x,y),(x', y')) \mapsto (y, y').\]
Then, $(X \times Y, \Ecal_{X \times Y})$ is a coarse space and is a product object of $(X, \Ecal)$ and $(Y, \Ecal')$ in $\Coarse$.

\begin{defn}[coarse group]\label{defn=coarse_group}
A \emph{coarse group} is a group object in $\Coarse$. Namely, a coarse group is a $4$-tuple $(G, \om, e, c)$ consisting of the following data:
\begin{enumerate}[label=\textup{(}$\arabic*$\textup{)}]
\item $G = (G, \Ecal_G)$ is a coarse space.

\item $\om \colon G \times G \to G$ is a coarse map.

\item $e \colon T \to G$ is a coarse map. Here $T$ is the coarse space consisting of one point.

\item $c \colon G \to G$ is a coarse map.
\end{enumerate}
These data commute the following diagrams:
\begin{align*}
\begin{gathered}
\xymatrix{
G \times G \times G \ar[r]^-{{\rm id}_G \times \om} \ar[d]_{\om \times {\rm id}_G} & G \times G \ar[d]^{\om} \\
G \times G \ar[r]^-{\om} & G,} \\
\text{(associativity)}
\end{gathered}
\quad
\begin{gathered}
\xymatrix{
G \ar[r]^-{({\rm id}_G, e)} \ar[rd]^-{{\rm id}_G} \ar[d]_{(e, {\rm id}_G)} & G \times G \ar[d]^\om \\
G \times G \ar[r]^-{\om} & G,} \\
\text{(identity)}
\end{gathered}
\quad
\begin{gathered}
\xymatrix{
G \ar[r]^-{({\rm id}_G, c)} \ar[rd]^-{e} \ar[d]_{(c, {\rm id}_G)} & G \times G \ar[d]^\om \\
G \times G \ar[r]^-{\om} & G.} \\
\text{(inverse)}
\end{gathered}
\end{align*}
Here $e \colon G \to G$ denotes the composition of $G \to T \xrightarrow{e} G$.
\end{defn}

\begin{exa} \label{exa=coarse group}
Let $G$ be a group and $d$ a bi-invariant almost generalized metric of $d$. Let $\Ecal_d$ be the coarse structure associated by $d$. Then,
$(G,\Ecal_d)$ is a coarse group.
\end{exa}

For the convenience of descriptions in the present paper, we introduce the notion of \emph{pre-coarse homomorphisms} in the following manner.

\begin{defn}[pre-coarse homomorphism]\label{defn=precoarse}
Let $G = (G, [\om_G], e, c)$ and $H = (H, [\om_H], e', c')$ be coarse groups.
 A map $\alpha \colon G \to H$ is called a \emph{pre-coarse homomorphism} if the following diagram commutes up to closeness.
\[ \xymatrix{
G \times G \ar[r]^{\alpha \times \alpha} \ar[d]_{\om_G} & H \times H \ar[d]^{\om_H} \\
G \ar[r]^{\alpha} & H.
}\]

\end{defn}

\begin{defn}[coarse homomorphism]\label{defn=coarsehom}
Let $G = (G, [\om_G], e, c)$ and $H = (H, [\om_H], e', c')$ be coarse groups.
A \emph{coarse homomorphism} is a coarse map from $G$ to $H$ that is represented by a controlled pre-coarse homomorphism.
\end{defn}
As we mentioned in Remark~\ref{rem=bf}, we also call a  representative of a coarse map (controlled pre-coarse homomorphism) itself a coarse homomorphism.

\begin{exa} \label{exa=precoarse}
Let $G$ and $H$ be groups, and let $d_G$ and $d_H$ be bi-invariant almost generalized metrics of $G$ and $H$, respectively. Then a map $\alpha \colon (G,d_G) \to (H,d_H)$ is a pre-coarse homomorphism if and only if $\sup\limits_{g_1,g_2\in \Gg}d_H(\alpha(g_1g_2), \alpha(g_1) \alpha(g_2)) <\infty$ holds. The map $\alpha$ is a coarse homomorphism if and only if it is a controlled pre-coarse homomorphism.

In this setting, the condition of $\alpha$ being a pre-coarse homomorphism is independent of the choice of the bi-invariant metric $d_G$ on $G$. Hence, by abuse of notation, we say that $\alpha \colon G \to (H,d_H)$ is a pre-coarse homomorphism without mentioning $d_G$.
\end{exa}


\begin{defn}\label{defn=asymp}
Let $(X, \Ecal)$ be a coarse space, and let $A$ and $B$ be subsets of $X$.
\begin{enumerate}[label=\textup{(}$\arabic*$\textup{)}]
\item We say that $A$ is \emph{coarsely contained in $B$} if there exists $E \in \Ecal$ such that $A \subseteq E[B]$. In this case, we write $A \preccurlyeq B$.

\item We say that $A$ and $B$ are \emph{asymptotic} if $A \preccurlyeq B$ and $B \preccurlyeq A$. In this case, we write $A \asymp B$. Then $\asymp$ is an equivalence relation of the power set of $X$. An equivalence class of $\asymp$ is called a \emph{coarse subspace of $X$}.

\item Let $\bA$ and $\bB$ be coarse subspaces of $(X, \Ecal)$. Then $\bA \subseteq \bB$ if for some  $A \in \bA$ and for some  $B \in \bB$ (equivalently, for every $A \in \bA$ and for every $B \in \bB$), $A\preccurlyeq B$ holds.
\end{enumerate}
\end{defn}
We note that $\bA = \bB$ if and only if $\bA \subseteq \bB$ and $\bB \subseteq \bA$. The following lemma is straightforward.

\begin{lem}\label{lem=image}
Let $(X, \Ecal)$ and $(Y, \Ecal')$ be coarse spaces. The following hold.
\begin{enumerate}[label=\textup{(}$\arabic*$\textup{)}]
\item Let $A$ be a subset of $X$, and let $\alpha, \beta \colon X \to Y$ be controlled maps such that $\alpha$ and $\beta$ are close. Then $\alpha(A) \asymp \beta(A)$ holds.

\item Let $A$ and $B$ be subsets of $X$, and let $\alpha \colon X \to Y$ be a controlled map. Then $A \preccurlyeq B$ implies $\alpha(A) \preccurlyeq \alpha(B)$. In particular, $A \asymp B$ implies $\alpha(A) \asymp \alpha(B)$.
\end{enumerate}
\end{lem}

\begin{defn}\label{defn=coarse_image}
Let $\bfalpha \colon X \to Y$ be a coarse map and $\mathbf{A}$ a coarse subspace of $(X, \Ecal)$. Let $\alpha$ be a controlled map representing $\bfalpha$ and $A$ a subset of $X$ representing $\mathbf{A}$. Then the \emph{coarse image $\bfalpha(\mathbf{A})$} is defined to be the coarse subspace represented by $\alpha(A)$.
\end{defn}

Lemma~\ref{lem=image} guarantees the well-definedness of $\bfalpha(\bA)$.

\begin{defn}\label{defn=coarse_pre-image}
Let $(X, \Ecal)$ and $(Y, \Ecal')$ be coarse spaces, and let $\bfalpha$ be a coarse map from $(X,\Ecal)$ to $(Y, \Ecal')$, and $\bB$ a coarse subspace of $(Y, \Ecal')$. A coarse subspace $\bA$ of $(X, \Ecal)$ is called the \emph{coarse preimage} if the following two conditions hold:
\begin{enumerate}[label=\textup{(}$\arabic*$\textup{)}]
\item $\bfalpha(\bA) \subseteq \bB$; and

\item for every coarse subspace $\bA'$ of $(X, \Ecal)$ satisfying $\bfalpha(\bA') \subseteq \bB$, $\bA' \subseteq \bA$ holds.
\end{enumerate}
\end{defn}
The coarse pre-image is unique if exists, but does not exist in general. If it exists, then we write $\bfalpha^{-1}(\bB)$ to indicate the coarse pre-image. Now we are ready to state the definition of coarse kernel. As is the case of coarse pre-image, the coarse kernel does not exist in general.

\begin{defn}[coarse kernel]\label{defn=coarsekernel}
Let $G=(G,[\om_G],e,c)$ and $H=(H,[\om_H],e',c')$ be coarse groups, and let $\bfalpha \colon G \to H$ be a coarse homomorphism. Then the \emph{coarse kernel of $\bfalpha$} is defined to be the coarse pre-image $\bfalpha^{-1}(e')$, where we identify $e'$ with $e'(T)$.
\end{defn}

\begin{exa}\label{exa=CKgroup}
Let $\Gg$ and $H$ be groups, and let $d_{\Gg}$ and $d_H$ be bi-invariant almost metrics on $\Gg$ and $H$, respectively. Let $\alpha\colon (\Gg,d_{\Gg})\to (H,d_H)$ be  a coarse homomorphism. Then, for $A\subseteq \Gg$, the coarse subspace $\mathbf{A}$ represented by $A$ is the coarse kernel of $\alpha$ if and only if the following two conditions are satisfied:
\begin{enumerate}[label=\textup{(\arabic*)}]
  \item the set $\alpha(A)$ is $d_{H}$-bounded; and
  \item for every $B\subseteq \Gg$ such that $\alpha(B)$ is $d_H$-bounded, $B\preccurlyeq A$ holds.
\end{enumerate}
\end{exa}

Let $\bfalpha \colon G \to H$ be a coarse homomorphism, and assume that $\bfalpha$ has the coarse kernel $\mathbf{K}$. Then $\mathbf{K}$ is coarsely normal, which implies that $g \mathbf{K} g^{-1} = \mathbf{K}$ for every $g \in G$. Then, we can define the quotient coarse group $G/\mathbf{K}$. Then, this quotient coarse group $G/\mathbf{K}$ is isomorphic to $\bfalpha(G)$; for details of this theory, see \cite[Chapter~7]{LV}.

\begin{exa}\label{exa=CKiota}
In this example, we recall our motivation from Subsections~\ref{subsec=digest} and \ref{subsec=comparative}. Let $\Gg$ be a group, and $\Lg$ and $\Ng$ normal subgroups of $\Gg$ with $\Lg\geqslant \Ng$. The map
\[
\iota_{(\Gg,\Lg,\Ng)} \colon ([\Gg, \Ng], d_{\scl_{\Gg, \Ng}}) \to ([\Gg, \Ng], d_{\scl_{\Gg, \Lg}});\ \yg\mapsto \yg
\]
(appearing in Definition~\ref{defn=iota}) is \emph{not} necessarily a monomorphism in the category of coarse groups: $\iota_{(\Gg,\Lg,\Ng)}$ may not be a coarse embedding. By Example~\ref{exa=CKgroup}, for $A\subseteq [\Gg,\Ng]$, the coarse subspace $\mathbf{A}$ represented by $A$ is the coarse kernel of $\iota_{(\Gg,\Lg,\Ng)}$ if and only if the following two conditions are satisfied:
\begin{enumerate}[label=\textup{(\arabic*)}]
  \item the set $A$ is $d_{\scl_{\Gg,\Lg}}$-bounded; and
  \item for every $d_{\scl_{\Gg,\Lg}}$-bounded set $B\subseteq [\Gg,\Ng]$, $B$ is coarsely contained in $A$ in $d_{\scl_{\Gg,\Ng}}$.
\end{enumerate}
If the coarse kernel $\mathbf{A}$ exists for $\iota_{(\Gg,\Lg,\Ng)}$, then we have an isomorphism
\[ ([\Gg, \Ng], d_{\scl_{\Gg, \Lg}}) \cong ([\Gg, \Ng], d_{\scl_{\Gg, \Ng}}) / \mathbf{A}\]
as coarse groups. This provides a motivation on Theorem~\ref{mthm=main} (and Theorem~\ref{thm=main1precise}).
\end{exa}

The following lemma is a special case of \cite[Proposition 12.2.1]{LV}. Since Lemma~\ref{lem=crushing} will be employed in Subsection~\ref{subsec=induced}, we include the proof for the reader's convenience. Recall from our notation that boldface such as $\mathbf{S}$ is used for a coarse notion in the present paper.

\begin{lem} \label{lem=crushing}
Let $\ell,\ell'\in \ZZ_{\geq 0}$, and let $\mathbf{S} \colon (\RR^{\ell}, \| \cdot \|_1) \to (\RR^{\ell'}, \| \cdot \|_1)$ be a coarse homomorphism. Then there exists a unique $\RR$-linear map $S \colon \RR^{\ell} \to \RR^{\ell'}$ that represents the coarse map  $\mathbf{S}$.
\end{lem}
\begin{proof}
The assertion trivially holds if $\ell=0$ or $\ell'=0$. Hence, in the rest of this proof, we assume that $\ell,\ell'\in \NN$.
Take an arbitrary representative $S_0\colon (\RR^{\ell}, \| \cdot \|_1) \to (\RR^{\ell'}, \| \cdot \|_1)$ of $\mathbf{S}$. Let $i_{\ell} \colon (\ZZ^{\ell}, \| \cdot \|_1) \hookrightarrow (\RR^{\ell}, \| \cdot \|_1)$ be the inclusion. For $i \in \{ 1, \cdots,\ell' \}$, let $p_i \colon \RR^{\ell'} \to \RR$ be the $i$-th projection. Then, $p_i \circ S_0 \circ i_{\ell}$ is a coarse homomorphism from $(\ZZ^{\ell}, \| \cdot\|_1)$ to $(\RR, | \cdot |)$, and hence is a quasimorphism on $\ZZ^{\ell}$. Let $\bar{S}_{0,i} \colon \ZZ^{\ell} \to (\RR, |\cdot|)$ be the homogenization of $p_i \circ S_0 \circ i_{\ell}$. Since $\ZZ^{\ell}$ is abelian, $\bar{S}_{0,i}$ is a homomorphism, and we have $\| \bar{S}_{0,i} - p_i \circ S_0 \|_{\infty} \le \DD(p_i \circ S_0)$ (see  \cite[Lemma~2.21]{Calegari} for details). Here $\|\cdot\|_{\infty}$ denotes the $\ell^{\infty}$-norm. Define the homomorphism $\bar{S}_0 \colon \ZZ^{\ell} \to \RR^{\ell'}$ by
\[
\bar{S}_0 = (\bar{S}_{0,1}, \cdots, \bar{S}_{0,\ell'}) \colon (\ZZ^{\ell}, \| \cdot\|_1) \to (\RR^{\ell'}, \| \cdot \|_1), \quad \vm \mapsto (\bar{S}_{0,1}(\vm), \cdots, \bar{S}_{0,\ell'} (\vm)).
\]
Then $\bar{S}_0$ and $S_0 \circ i_{\ell}$ are close since $\| \bar{S}_0 - S_0 \circ i_{\ell}\|_{\infty} \le\sum\limits_{i\in \{1,\ldots,\ell'\}} \DD(p_i \circ i_{\ell})$. Let $S \colon \RR^{\ell} \to \RR^{\ell'}$ be the $\RR$-linear map extending $\bar{S}_0 \colon \ZZ^{\ell} \to \RR^{\ell'}$. Since $\bar{S}_0 = S \circ i_{\ell}\approx S_0 \circ i_{\ell}$, we conclude $S\approx S_0$, as desired. The uniqueness of such $S$ follows from the fact that distinct two $\RR$-linear maps are not close.
\end{proof}

\subsection{Asymptotic dimension}\label{subsec=asdim}
Here we review the definition and the basic facts of asymptotic dimension, following \cite{BD} and \cite{Roe}.

Let $X$ be a set and $\mathcal{U}$ a family of subsets of $X$. Then the \emph{multiplicity of $\mathcal{U}$} is defined as
\[ \sup \left\{ \# S \, \middle| \, \textrm{$S \subseteq \mathcal{U}$ and $\bigcap_{A \in S} A \ne \emptyset$}\right\} \in \ZZ_{\geq 0} \cup \{ \infty \}.\]

Let $X$ be a metric space and $\mathcal{U}$ a cover of $X$. Then $\mathcal{U}$ is \emph{uniformly bounded} if there exists $S\in\RR_{\ge 0}$ such that $U \in \mathcal{U}$ implies $\diam(U) \le S$. We call a cover $\mathcal{V}$ of $X$ a \emph{refinement of a cover $\mathcal{U}$ of $X$} if for every $V \in \mathcal{V}$ there exists $U \in \mathcal{U}$ such that $V \subseteq U$.

\begin{defn}[asymptotic dimension of metric spaces]
Let $X$ be a metric space and $n\in \ZZ_{\geq 0}$. We write $\asydim X \le n$ if for every uniformly bounded open cover $\mathcal{V}$ of $X$, there exists a uniformly bounded open cover $\mathcal{U}$ of $X$ such that $\mathcal{V}$ is a refinement of $\mathcal{U}$ and the multiplicity of $\mathcal{U}$ is at most $n+ 1$. The \emph{asymptotic dimension} is defined as $\asdim X=\inf \{ n \, | \, \asydim X \le n\} \in \ZZ_{\geq 0} \cup \{ \infty\}$.
\end{defn}


Let $r\in \RR_{>0}$ and $\mathcal{U}$ a family of subsets in a metric space $X$. Then $\mathcal{U}$ is said to be \emph{$r$-disjoint} if $U, U' \in \mathcal{U}$ and $U \ne U'$ imply $d(U, U') > r$. Here we set $d(U, U') = \inf \{ d(x,y) \, | \, \textrm{$x \in U$ and $y \in U'$}\}$.

\begin{thm}[{\cite[Theorem~19]{BD}}]
Let $X$ be a metric space. The following conditions are equivalent.
\begin{enumerate}[label=\textup{(}$\arabic*$\textup{)}]
\item $\asydim X \le n$;
\item For every $r < \infty$ there exist $r$-disjoint families $\mathcal{U}^0, \cdots, \mathcal{U}^n$ of uniformly bounded subsets of $X$ such that $\bigcup\limits_{i\in \{0,\ldots,n\}} \mathcal{U}^i$ is a cover of $X$.
\end{enumerate}
\end{thm}

Using this theorem, the asymptotic dimension of the coarse space is defined as follows.

\begin{defn}
Let $(X, \Ecal)$ be a coarse space, and let $E \in \Ecal$. A family $\mathcal{U}$ of subsets of $X$ is said to be \emph{$E$-disjoint} if $U, U' \in \mathcal{U}$ and $U \ne U'$ imply $ (U \times U') \cap E = \emptyset$.
\end{defn}

Let $(X, \Ecal)$ be a coarse space. A family $\mathcal{U}$ of subsets of $X$ is  said to be \emph{uniformly bounded} if $\bigcup\limits_{U \in \mathcal{U}} (U \times U)$ belongs to $\Ecal$.

\begin{defn}[asymptotic dimension of coarse spaces]
Let $(X, \Ecal)$ be a coarse space, and $n\in \ZZ_{\geq 0}$. We write $\asydim (X, \Ecal) \le n$ if for every $E \in \Ecal$ there exist $E$-disjoint uniformly bounded families $\mathcal{U}^0, \cdots, \mathcal{U}^n$ of $X$ such that $\mathcal{U}^0 \cup \cdots \cup \mathcal{U}^n$ is a cover of $X$. The \emph{asymptotic dimension of $(X,\Ecal)$} is defined as $\asdim(X,\Ecal)=\inf \{ n \, | \, \asydim (X, \Ecal) \le n\} \in \ZZ_{\geq 0} \cup \{ \infty\}$.
\end{defn}

It immediately follows from the definition that the asymptotic dimension is invariant under coarse equivalences (see \cite[Section~9.1]{Roe} for details). This means that the asymptotic dimension of an  almost (generalized) metric space is well-defined. Finally, we state the following two results.

\begin{prop}[see {\cite[Proposition~9.10]{Roe}}]\label{prop=asdimCE}
Let $(X, \Ecal)$ and $(Y,\Ecal')$ be coarse spaces. Assume that there exists a coarse embedding $\alpha \colon (X, \Ecal) \to (Y, \Ecal')$. Then, $\asydim(X, \Ecal) \le \asydim(Y, \Ecal')$.
\end{prop}




\begin{thm}[see {\cite[Section~9.2]{Roe}}]\label{thm=asdimZ}
Let $\ell\in \ZZ_{\geq 0}$. Then, $\asdim (\ZZ^{\ell},\|\cdot\|_1)=\ell$.
\end{thm}

\section{Proofs of Propositions~$\ref{prop=presclkika}$ and $\ref{prop=prescldim}$}\label{sec=absolute}
In this section, we provide the proofs of Propositions~\ref{prop=presclkika} and \ref{prop=prescldim}. In fact, we prove the refined statements, Propositions~\ref{prop=sclkika} and \ref{prop=metricdim}. Here we exhibit Proposition~\ref{prop=sclkika}.

\begin{prop}[{precise version of Proposition~$\ref{prop=presclkika}$}]\label{prop=sclkika}
Assume Setting~$\ref{setting=itsumono}$. Assume that $\QQQ(\Ng)^{\Gg}/\HHH^1(\Ng)^{\Gg}$ is finite dimensional. Set $\ell=\Rdim \left(\QQQ(\Ng)^{\Gg}/\HHH^1(\Ng)^{\Gg}\right)$. Then there exist two maps $\Psig\colon [\Gg,\Ng]\to \ZZ^{\ell}$ and $\Ptau\colon \ZZ^{\ell}\to [\Gg,\Ng]$ such that the following hold true.
\begin{enumerate}[label=\textup{(}$\arabic*$\textup{)}]
  \item The map $\Psig\colon [\Gg,\Ng]\to (\ZZ^{\ell},\|\cdot\|_1)$ and $\Ptau\colon \ZZ^{\ell}\to ([\Gg,\Ng],d_{\cl_{\Gg,\Ng}})$ are pre-coarse homomorphisms.
  \item The map $\Psig\colon ([\Gg,\Ng],d_{\scl_{\Gg,\Ng}})\to (\ZZ^{\ell},\|\cdot\|_1)$ is a quasi-isometric embedding.
  \item There exist $C,C'\in \RR_{>0}$ and $D,D'\in \RR_{\geq 0}$ such that
for every $\vm,\vn\in \ZZ^{\ell}$, the inequalities
\begin{align*}
C\cdot \|\vm-\vn\|_1-D&\leq d_{\scl_{\Gg,\Ng}}(\Ptau(\vm),\Ptau(\vn))\\
&\leq d_{\cl_{\Gg,\Ng}}(\Ptau(\vm),\Ptau(\vn))\leq  C'\cdot \|\vm-\vn\|_1+D'
\end{align*}
hold.
  \item The maps $\Psig\colon ([\Gg,\Ng],d_{\scl_{\Gg,\Ng}})\to (\ZZ^{\ell},\|\cdot\|_1)$ and $\Ptau\colon (\ZZ^{\ell},\|\cdot\|_1)\to ([\Gg,\Ng],d_{\scl_{\Gg,\Ng}})$ are coarse inverse to each other.
\end{enumerate}

In particular, the pair of quasi-isometries $(\Psig,\Ptau)$ provides the isomorphism $([\Gg,\Ng],d_{\scl_{\Gg,\Ng}})\cong (\ZZ^{\ell},\|\cdot\|_1)$ as coarse groups.
\end{prop}

In Proposition~\ref{prop=sclkika}, we may take $\Psig$ and $\Ptau$ such that they furthermore satisfy $\Psig\circ \Ptau=\mathrm{id}_{\ZZ^{\ell}}$; see Subsection~\ref{subsec=remark}.

Recall the three steps in the outline of the proof of Proposition~\ref{prop=presclkika} from Subsection~\ref{subsec=org}:
\begin{enumerate}
  \item [\underline{Step~$1'$}.] construct $\Psig^{\RR}\colon [\Gg,\Ng]\to \RR^{\ell}$;
  \item [\underline{Step~$2'$}.] construct $\Ptau\colon \ZZ^{\ell}\to [\Gg,\Ng]$;
  \item [\underline{Step~$3'$}.] take an appropriate $\Psig\colon [\Gg,\Ng]\to \ZZ^{\ell}$ out of $\Psig^{\RR}$, and study the compositions $\Ptau\circ\Psig$ and $\Psig\circ \Ptau$.
\end{enumerate}
In this section, we will take these three steps in the proof of Proposition~\ref{prop=sclkika}.

We introduce the following formulations for the convenience of  stating results in quasi-isometric geometry.

\begin{defn}[QI-type estimates from below/above]\label{defn=QI}
Let $X$ and $Y$ be sets, and let  $\alpha\colon X\to Y$ be a map. Let $A\subseteq X$. Let $d_X$ and $d_Y$ be almost metrics on $X$ and $Y$, respectively.
\begin{enumerate}[label=(\arabic*)]
  \item A \emph{QI-type estimate from below on $A$} means an inequality of  the following form for every $x_1,x_2\in A$:
\[
C_1\cdot d_X(x_1,x_2)-D_1\leq d_Y(\alpha(x_1),\alpha(x_2)),
\]
where $C_1\in \RR_{>0}$ and $D_1\in \RR_{\geq 0}$ are constants not depending on $x_1,x_2\in A$.
 \item   A \emph{QI-type estimate from above on $A$} means an inequality of  the following form for every $x_1,x_2\in A$:
\[
d_Y(\alpha(x_1),\alpha(x_2))\leq C_2\cdot d_X(x_1,x_2)+D_2,
\]
where $C_2\in \RR_{>0}$ and $D_2\in \RR_{\geq 0}$ are constants not depending on $x_1,x_2\in A$.
\end{enumerate}
\end{defn}

\subsection{The construction of $\Psig^{\RR}$}\label{subsec=Psig}
In this subsection, we take Step~$1'$ in the outlined proof.

\begin{prop}\label{prop=evmap}
Assume Setting~$\ref{setting=itsumono}$. Assume that $\QQQ(\Ng)^{\Gg}/\HHH^1(\Ng)^{\Gg}$ is non-zero finite dimensional, and set $\ell=\Rdim \left(\QQQ(\Ng)^{\Gg}/\HHH^1(\Ng)^{\Gg}\right)$.
Take an arbitrary basis of $\QQQ(\Ng)^{\Gg}/\HHH^1(\Ng)^{\Gg}$. Take an arbitrary  set $\{\nu_1,\ldots ,\nu_{\ell}\}\subseteq  \QQQ(\Ng)^{\Gg}$ of representatives of this basis. Define $\Psig^{\RR}=\Psig^{\RR}_{(\nug_1,\ldots,\nug_{\ell})}\colon ([\Gg,\Ng],d_{{\scl}_{\Gg,\Ng}})\to (\mathbb{R}^{\ell},\|\cdot\|_1)$  by
\[
\Psig^{\RR}(\yg)=(\nug_1(\yg),\nug_2(\yg),\ldots ,\nug_{\ell}(\yg))
\]
for every $\yg\in [\Gg,\Ng]$. Then the following hold true.
\begin{enumerate}[label=\textup{(}$\arabic*$\textup{)}]
  \item The map $\Psig^{\RR}$ is a pre-coarse homomorphism.
  \item We have the following QI-type  estimate  from above on $[\Gg,\Ng]$:
\begin{equation}\label{eq=psipsi}
\|\Psig^{\RR}(\yg_1)-\Psig^{\RR}(\yg_2)\|_1\leq 2\left(\sum_{j\in \{1,\ldots,\ell\}}\DD(\nug_j)\right)\cdot d_{\scl_{\Gg,\Ng}}(\yg_1,\yg_2) + \sum_{j\in \{1,\ldots,\ell\}}\DD(\nug_j).
\end{equation}
  \item There exists $C\in \RR_{>0}$ such that we have the following QI-type estimate from below on $[\Gg,\Ng]$:
\begin{equation}\label{eq=psipsiC}
\|\Psig^{\RR}(\yg_1)-\Psig^{\RR}(\yg_2)\|_1\geq \frac{2}{C}\cdot d_{\scl_{\Gg,\Ng}}(\yg_1,\yg_2)-\frac{1}{C}.
\end{equation}
\end{enumerate}
In particular, $\Psig^{\RR}\colon ([\Gg,\Ng],d_{{\scl}_{\Gg,\Ng}})\to (\mathbb{R}^{\ell},\|\cdot\|_1)$ is a coarse homomorphism and a quasi-isometric embedding.
\end{prop}
The constant $C$ in (3) can be determined by \eqref{eq=CCC}.

\begin{proof}[Proof of Proposition~$\ref{prop=evmap}$]
Let $\yg_1,\yg_2\in [\Gg,\Ng]$. For (1), we have
\begin{equation}\label{eq=coarsehomPsig}
\|\Psig^{\RR}(\yg_1\yg_2)-\Psig^{\RR}(\yg_1)-\Psig^{\RR}(\yg_2)\|_1\leq \sum_{j\in \{1,\ldots,\ell\}}\DD(\nug_j).
\end{equation}
For (2), by \eqref{eq=coarsehomPsig} we have
\[
\|\Psig^{\RR}(\yg_1)-\Psig^{\RR}(\yg_2)\|_1\leq \|\Psig^{\RR}(\yg_1^{-1}\yg_2)\|_1+\sum_{j\in \{1,\ldots,\ell\}}\DD(\nug_j)=\sum_{j\in \{1,\ldots,\ell\}}|\nu_j(\yg_1^{-1}\yg_2)|+\sum_{j\in \{1,\ldots,\ell\}}\DD(\nug_j).
\]
By Theorem~\ref{thm=Bavard}, we have for every $j\in \{1,\ldots,\ell\}$,
\[
|\nu_j(\yg_1^{-1}\yg_2)|\leq 2\DD(\nug_j)\cdot d_{\scl_{\Gg,\Ng}}(\yg_1,\yg_2).
\]
Therefore, we obtain \eqref{eq=psipsi}.

In what follows, we prove (3). By Proposition~\ref{prop=defectBanach}, the following linear map
\[
(\RR^{\ell},\|\cdot\|_1)\to (\QQQ(\Ng)^{\Gg}/\HHH^1(\Ng)^{\Gg},\DD);\quad (\util_1,\ldots,\util_{\ell})\mapsto \util_1[\nug_1]+\util_2[\nug_2]+\cdots+\util_{\ell}[\nug_{\ell}]
\]
is an isomorphism of Banach spaces. Here, $[\cdot]$ means the equivalence class in $\QQQ(\Ng)^{\Gg}/\HHH^1(\Ng)^{\Gg}$. In particular, there exists $C\in \RR_{>0}$ such that for every $(\util_1,\ldots,\util_{\ell})\in \RR^{\ell}$,
\begin{equation}\label{eq=CCC}
\DD\left(\sum_{j\in \{1,\ldots,\ell\}}\util_j\nug_j\right)\geq C^{-1}\cdot \sum_{j\in \{1,\ldots,\ell\}}|\util_j|
\end{equation}
holds.

Let $\nug\in \QQQ(\Ng)^{\Gg}$. Then, there exist $\kg\in \HHH^1(\Ng)^{\Gg}$ and $(a_1,\ldots ,a_{\ell})\in \RR^{\ell}$ such that $\nug=\kg+\sum\limits_{j\in \{1,\ldots,\ell\}}a_j\nug_j$. Note that $\kg(\yg_1)=\kg(\yg_2)=0$ since $\yg_1,\yg_2\in [\Gg,\Ng]$. Hence, by \eqref{eq=CCC} we have
\[
\DD(\nu)=\DD\left(\kg+\sum_{j\in \{1,\ldots,\ell\}}a_j\nug_j\right)=\DD\left(\sum_{j\in \{1,\ldots,\ell\}}a_j\nug_j\right)\geq C^{-1}\cdot \sum_{j\in \{1,\ldots,\ell\}}|a_j|.
\]
This in particular implies that $\max\limits_{j\in \{1,\ldots,\ell\}}|a_j|\leq C\cdot\DD(\nug)$.
Therefore, we obtain that
\begin{align*}
|\nug(\yg_1^{-1}\yg_2)|&\leq |\nug(\yg_1)-\nug(\yg_2)|+\DD(\nug) =\left|\sum_{j\in \{1,\ldots,\ell\}}a_j(\nug_j(\yg_1)-\nug_j(\yg_2))\right|+\DD(\nug)\\
&\leq \left(\max_{j\in \{1,\ldots,\ell\}}|a_j|\right) \cdot \sum_{j\in \{1,\ldots,\ell\}}|\nug_j(\yg_1)-\nug_j(\yg_2)|+\DD(\nug)\\
&\leq C \DD(\nug)\cdot \|\Psig^{\RR}(\yg_1)-\Psig^{\RR}(\yg_2)\|_1+\DD(\nug).
\end{align*}
Then it follows from the Bavard duality theorem for $\scl_{\Gg,\Ng}$ (Theorem~\ref{thm=Bavard}) that
\[
d_{\scl_{\Gg,\Ng}}(\yg_1,\yg_2)\leq \frac{C}{2}\cdot  \|\Psig^{\RR}(\yg_1)-\Psig^{\RR}(\yg_2)\|_1+\frac{1}{2};
\]
equivalently, we obtain \eqref{eq=psipsiC}. This completes our proof.
\end{proof}

\subsection{The construction of $\Ptau$}\label{subsec=Ptau}
In this subsection, we take Step~$2'$ in the outlined proof. We note that in Proposition~\ref{prop=dimQ}, we allow the case of $\Rdim \left(\QQQ(\Ng)^{\Gg}/\HHH^1(\Ng)^{\Gg}\right)=\infty.$

\begin{prop}\label{prop=dimQ}
Assume Setting~$\ref{setting=itsumono}$. Assume that $\QQQ(\Ng)^{\Gg}/\HHH^1(\Ng)^{\Gg}\ne 0$. Let $\ell$ be an element in $\NN$ such that $\ell\leq \Rdim \left(\QQQ(\Ng)^{\Gg}/\HHH^1(\Ng)^{\Gg}\right)$. Assume that $\nug_1,\ldots,\nug_{\ell}\in \QQQ(\Ng)^{\Gg}$ and $\yg_1,\ldots ,\yg_{\ell}\in [\Gg,\Ng]$ satisfy for every $i,j\in \{1,\ldots ,\ell\}$,
\begin{equation}\label{eq=Kdelta}
\nu_j(y_i)=\delta_{i,j}.
\end{equation}
Define a map $\Ptau\colon \ZZ^{\ell}\to [\Gg,\Ng]$  by
\[
\Ptau((m_1,\ldots ,m_{\ell}))=\yg_1^{m_1}\cdots \yg_{\ell}^{m_{\ell}}
\]
for every $(m_1,\ldots ,m_{\ell})\in \ZZ^{\ell}$. Then this map $\Ptau$ satisfies the following two conditions.
\begin{enumerate}[label=\textup{(\arabic*)}]
  \item The map $\Ptau\colon \ZZ^{\ell}\to ([\Gg,\Ng],d_{\cl_{\Gg,\Ng}})$ is a pre-coarse homomorphism; in particular, $\Ptau\colon \ZZ^{\ell}\to ([\Gg,\Ng],d_{\scl_{\Gg,\Ng}})$ is a pre-coarse homomorphism as well.
  \item We have the following QI-type estimate from above on $\ZZ^{\ell}$:
\begin{equation}\label{eq=aboveQcl}
d_{\cl_{\Gg,\Ng}}(\Ptau(\vm),\Ptau(\vn))\leq \left(\max_{i\in \{1,\ldots ,\ell\}}\cl_{\Gg,\Ng}(\yg_i)\right)\cdot \|\vm-\vn\|_1+\ell-1.
\end{equation}
  \item We have the following QI-type estimate from below on $\ZZ^{\ell}$:
\begin{equation}\label{eq=belowQ}
d_{\scl_{\Gg,\Ng}}(\Ptau(\vm),\Ptau(\vn))\geq \frac{1}{2\ell \left(\max\limits_{j\in \{1,\ldots ,\ell\}}\DD(\nug_j)\right)}\cdot \|\vm-\vn\|_1- \frac{2\ell-1}{2}.
\end{equation}
\end{enumerate}
In particular, $\Ptau$ is a coarse homomorphism and a quasi-isometric embedding if we regard $\Ptau$ as a map $(\ZZ^{\ell},\|\cdot\|_1)\to ([\Gg,\Ng],d_{\scl_{\Gg,\Ng}})$. The same conclusion holds if we regard $\Ptau$ as a map $(\ZZ^{\ell},\|\cdot\|_1)\to ([\Gg,\Ng],d_{\cl_{\Gg,\Ng}})$.
\end{prop}

The idea of the construction above of the map $\Ptau$  already appears in symplectic geometry (for examples, see \cite{BK13} and \cite{KO19}).
In their settings, they were able to take pairwise commuting elements $\yg_1,\ldots,\yg_{\ell}$ so that they obtained a genuine group homomorphism $\Ptau$.

We note that by \eqref{eq=hikaku}, the inequality \eqref{eq=aboveQcl} in particular implies that
\[
d_{\scl_{\Gg,\Ng}}(\Ptau(\vm),\Ptau(\vn))\leq \left(\max_{i\in \{1,\ldots ,\ell\}}\cl_{\Gg,\Ng}(\yg_i)\right)\cdot \|\vm-\vn\|_1+\ell-1.
\]
The existences of such $\nug_1,\ldots,\nug_{\ell}\in \QQQ(\Ng)^{\Gg}$ and such $\yg_1,\ldots ,\yg_{\ell}$ in Proposition~\ref{prop=dimQ} are ensured by the following lemma.

\begin{lem}\label{lem=lemPtau}
Assume Setting~$\ref{setting=itsumono}$. Assume that $\QQQ(\Ng)^{\Gg}/\HHH^1(\Ng)^{\Gg}\ne 0$. Let $\ell\in \NN$ such that $\ell\leq \Rdim \left(\QQQ(\Ng)^{\Gg}/\HHH^1(\Ng)^{\Gg}\right)$. Then, there exist $\nug_1,\ldots,\nug_{\ell}\in \QQQ(\Ng)^{\Gg}$ and $\yg_1,\ldots ,\yg_{\ell}$ satisfying \eqref{eq=Kdelta} for every $i,j\in \{1,\ldots ,\ell\}$.
\end{lem}

\begin{proof}
Take an arbitrary $\ell$-dimensional subspace $X$ of $\QQQ(\Ng)^{\Gg}/\HHH^1(\Ng)^{\Gg}$. Take an arbitrary basis of $X$;  take an arbitrary  set of representatives in $\QQQ(\Ng)^{\Gg}$ of this basis. Let $Y$ be the $\ell$-dimensional subspace of $\QQQ(\Ng)^{\Gg}$ spanned by this set of representatives. Set $Z$ be the subspace of $\RR^{[\Gg,\Ng]}$ obtained by taking the restriction of elements in $Y$ to $[\Gg,\Ng]$; by Corollary~\ref{cor=uptohom}, $\Rdim Z=\ell$. Finally, apply Lemma~$\ref{lem=dual}$ to the case where $(W,\Xi)=([\Gg,\Ng],Z)$.
\end{proof}

The following proof of Proposition~\ref{prop=dimQ} uses the symbol `${\underset{C}{\eqsim}}$', more precisely `$\overset{(\Gg,\Ng)}{\underset{C}{\eqsim}}$' in Definition~\ref{defn=choron}.

\begin{proof}[Proof of Proposition~$\ref{prop=dimQ}$]
Let $\vm=(m_1,\ldots ,m_{\ell})$ and $\vn=(n_1,\ldots ,n_{\ell})$ be elements in $\ZZ^{\ell}$. First, we prove (1).  Then by Lemma~\ref{lem=choron}, we have
\begin{align*}
\Ptau(\vm)^{-1}\Ptau(\vn)&=\yg_{\ell}^{-m_{\ell}}\cdots \yg_1^{-m_1}\yg_1^{n_1}\cdots \yg_{\ell}^{n_{\ell}}=\yg_{\ell}^{-m_{\ell}}\cdots \yg_2^{-m_2}\yg_1^{n_1-m_1}\yg_2^{n_2}\cdots \yg_{\ell}^{n_{\ell}}\\
&\mathrel{\underset{1}{\eqsim}}\yg_1^{n_1-m_1}\yg_{\ell}^{-m_{\ell}}\cdots \yg_2^{-m_2}\yg_2^{n_2}\cdots \yg_{\ell}^{n_{\ell}}=\yg_1^{n_1-m_1}\yg_{\ell}^{-m_{\ell}}\cdots \yg_2^{n_2-m_2}\cdots \yg_{\ell}^{n_{\ell}}.
\end{align*}
Here, note that $\yg_1\in [\Gg,\Ng]\leqslant \Ng$. By continuing this process, we obtain that
\begin{equation}\label{eq=prehomQ}
\Ptau(\vm)^{-1}\Ptau(\vn)\mathrel{\underset{\ell-1}{\eqsim}}\Ptau(\vn-\vm).
\end{equation}
By replacing $\vn$ with $\vm+\vn$, we have for every $\vm,\vn\in \ZZ^{\ell}$, $\Ptau(\vm+\vn)\mathrel{\underset{\ell-1}{\eqsim}}\Ptau(\vm)\Ptau(\vn)$. This confirms (1).

Next we verify (2). By construction, we have
\begin{align*}
\cl_{\Gg,\Ng}(\Ptau(\vm))&\leq \sum_{i\in \{1,\ldots,\ell\}}\cl_{\Gg,\Ng}(\yg_i^{m_i})\leq \left(\max_{i\in \{1,\ldots ,\ell\}}\cl_{\Gg,\Ng}(\yg_i)\right) \sum_{i\in \{1,\ldots,\ell\}}|m_i|
=\left(\max_{i\in \{1,\ldots ,\ell\}}\cl_{\Gg,\Ng}(\yg_i)\right)\cdot \|\vm\|_1.
\end{align*}
By combining the inequalities above and \eqref{eq=prehomQ}, we have \eqref{eq=aboveQcl}; hence, we obtain (2).

Finally, we prove (3). Given $\vm$ and $\vn$, we can take $j_{\vm,\vn}\in \{1,\ldots ,\ell\}$ such that
\[
|m_{j_{\vm,\vn}}-n_{j_{\vm,\vn}}|\geq \frac{1}{\ell}\cdot\|\vm-\vn\|_1
\]
holds. Then, we have
\begin{align*}
|\nug_{j_{\vm,\vn}}(\Ptau(\vm)^{-1}\Ptau(\vn))|&\geq |\nug_{j_{\vm,\vn}}(\Ptau(\vn))-\nug_{j_{\vm,\vn}}(\Ptau(\vm))|-\DD(\nug_{j_{\vm,\vn}})\\
&\geq \left|\sum_{i\in \{1,\ldots,\ell\}}(n_i-m_i)\nug_{j_{\vm,\vn}}(\yg_i)\right|-(2\ell-1)\DD(\nug_{j_{\vm,\vn}})\\
&=|m_{j_{\vm,\vn}}-n_{j_{\vm,\vn}}|-(2\ell-1)\DD(\nug_{j_{\vm,\vn}})\\
&\geq \frac{1}{\ell}\cdot \|\vm-\vn\|_1-(2\ell-1)\DD(\nug_{j_{\vm,\vn}}).\end{align*}
By Theorem~\ref{thm=Bavard}, we obtain \eqref{eq=belowQ}, as desired.
\end{proof}

\subsection{Proofs of Propositions~\ref{prop=sclkika} and \ref{prop=presclkika}}\label{subsec=sclkika}
In this subsection, we take Step~$3'$ in the outlined proof. Recall the definition of the closeness ($\approx$) from Definition~\ref{defn=close} (and Example~\ref{exa=several_maps}).

\begin{prop}\label{prop=itteQ}
Assume Setting~$\ref{setting=itsumono}$. Assume that $\QQQ(\Ng)^{\Gg}/\HHH^1(\Ng)^{\Gg}$ is non-zero and finite dimensional, and  set $\ell=\Rdim \left(\QQQ(\Ng)^{\Gg}/\HHH^1(\Ng)^{\Gg}\right)$. Take $(\nug_{j})_{j\in \{1,\ldots,\ell\}}$ and $(\yg_i)_{i\in \{1,\ldots,\ell\}}$ as in Proposition~$\ref{prop=dimQ}$ \textup{(}by Lemma~$\ref{lem=lemPtau}$\textup{)}. Set $\Psig^{\RR}=\Psig^{\RR}_{(\nug_1,\ldots,\nug_{\ell})}\colon ([\Gg,\Ng],d_{\scl_{\Gg,\Ng}})\to (\RR^{\ell},\|\cdot\|_1)$ associated with the tuple $(\nug_1,\ldots,\nug_{\ell})$ above. Take an arbitrary map $\rho\colon (\RR^{\ell},\|\cdot\|_1)\to (\ZZ^{\ell},\|\cdot\|_1)$ satisfying that $\sup\limits_{\vu\in \RR^{\ell}}\|\vu-\rho(\vu)\|_1< \infty$. Set $\Psig=\rho\circ \Psig^{\RR}$. Then, the following hold true.
\begin{enumerate}[label=\textup{(\arabic*)}]
 \item $\Ptau\circ \Psig \approx \mathrm{id}_{([\Gg,\Ng],d_{\scl_{\Gg,\Ng}})}$.
  \item $\Psig\circ \Ptau \approx \mathrm{id}_{(\ZZ^{\ell},\|\cdot\|_1)}$.
\end{enumerate}
\end{prop}

We note that the map $\rho\colon (\RR^{\ell},\|\cdot\|_1)\to (\ZZ^{\ell},\|\cdot\|_1)$ above is automatically a coarse homomorphism. One such example of $\rho$ is the coordinatewise floor map: $\vu=(u_1,\ldots,u_{\ell})\mapsto (\lfloor u_1\rfloor,\ldots,\lfloor u_{\ell}\rfloor)$.

\begin{proof}[Proof of Proposition~$\ref{prop=itteQ}$]
Set $\kappa=\sup\limits_{\vu\in \RR^{\ell}}\|\vu-\rho(\vu)\|_1$. First, we prove (2). Let $\vm=(m_1,\ldots,m_{\ell})\in \ZZ^{\ell}$. Then,
\[
(\Psig\circ \Ptau)(\vm)=\rho \big((\Psig^{\RR}\circ \Ptau)(\vm)\big)=\rho \big((\nug_1(\Ptau(\vm)),\ldots,\nug_{\ell}(\Ptau(\vm)))\big).
\]
For every $j\in \{1,\ldots,\ell\}$,
\[
(\ell-1)\DD(\nug_j)\geq \left|\nug_j(\yg_1^{m_1}\cdots \yg_{\ell}^{m_{\ell}})-\sum_{i\in \{1,\ldots,\ell\}} \nug_j(\yg_i^{m_i})\right|=\left|\nug_j(\yg_1^{m_1}\cdots \yg_{\ell}^{m_{\ell}})-\sum_{i\in \{1,\ldots,\ell\}} m_i\nug_j(\yg_i)\right|
\]
and by \eqref{eq=Kdelta},
\begin{equation}\label{eq=KdeltaD}
\left|\nug_j(\yg_1^{m_1}\cdots \yg_{\ell}^{m_{\ell}})-m_j\right|\leq (\ell-1)\DD(\nug_j).
\end{equation}
By \eqref{eq=KdeltaD}, we have $\|\vm-(\Psig^{\RR}\circ \Ptau)(\vm)\|_1\leq (\ell-1)\sum\limits_{j\in \{1,\ldots,\ell\}}\DD(\nug_j)$.
By the definition of $\kappa$, we obtain
\begin{equation}\label{eq=sigmatau}
\left\|\vm-(\Psig\circ \Ptau)(\vm)\right\|_1\leq \kappa+(\ell-1)\sum_{j\in \{1,\ldots,\ell\}}\DD(\nug_j).
\end{equation}

Secondly, we verify (1). Let $\yg\in [\Gg,\Ng]$. Set $\Psig(\yg)=(m_1,\ldots,m_{\ell})\in \ZZ^{\ell}$. Then, $|m_j-\nug_j(\yg)|\leq \kappa$ for every $j\in \{1,\ldots,\ell\}$. Let $\nug\in \QQQ(\Ng)^{\Gg}$. Write $\nug=\kg+\sum\limits_{j\in \{1,\ldots,\ell\}}a_j\nug_j$ with $\kg\in \HHH^1(\Ng)^{\Gg}$ and $(a_1,\ldots,a_{\ell})\in \RR^{\ell}$.
Let $C\in \RR_{>0}$ be a constant such that for every $(\util_1,\ldots,\util_{\ell})\in \RR^{\ell}$, \eqref{eq=CCC} holds. Then by \eqref{eq=CCC} and \eqref{eq=KdeltaD}, we have
\begin{align*}
\left|\nug\left(\yg^{-1}(\Ptau\circ \Psig)(\yg)\right)\right|&\leq \DD(\nug)+ |\nug(\yg)-\nug(\yg_1^{m_1}\cdots \yg_{\ell}^{m_{\ell}})| \\
&\leq \DD(\nug)+ \sum_{\in \{1,\ldots,\ell\}}|a_j|\left|\nug_j(\yg)-\nug_j(\yg_1^{m_1}\cdots \yg_{\ell}^{m_{\ell}})\right|\\
&\leq \DD(\nug)+  \sum_{j\in \{1,\ldots,\ell\}} |a_j|\left(|\nug_j(\yg)-m_j|+(\ell-1)\DD(\nug_j)\right)\\
&\leq \DD(\nug)+\sum_{j\in \{1,\ldots,\ell\}} |a_j| \left(\kappa+(\ell-1)\DD(\nug_j)\right)\\
&\leq \DD(\nug)+\left\{\kappa+(\ell-1)\max_{j\in \{1,\ldots,\ell\}}\DD(\nug_j)\right\} \cdot \sum_{j\in \{1,\ldots,\ell\}}|a_j|\\
&\leq \DD(\nug)+C\left\{\kappa+(\ell-1)\max_{j\in \{1,\ldots,\ell\}}\DD(\nug_j)\right\}\cdot \DD(\nug).
\end{align*}
Therefore, Theorem~\ref{thm=Bavard} yields
\begin{equation}\label{eq=tausigma}
d_{\scl_{\Gg,\Ng}}(\yg, (\Ptau\circ \Psig)(\yg))\leq \frac{C}{2} \cdot \left\{\kappa+(\ell-1)\max_{j\in \{1,\ldots,\ell\}}\DD(\nug_j)\right\}+\frac{1}{2}.
\end{equation}
Now \eqref{eq=sigmatau} and \eqref{eq=tausigma}  end our proof.
\end{proof}

We are now ready to deduce Propositions~\ref{prop=sclkika} and \ref{prop=presclkika} from Propositions~\ref{prop=evmap}, \ref{prop=dimQ} and \ref{prop=itteQ}.

\begin{proof}[Proofs of Propositions~$\ref{prop=sclkika}$ and $\ref{prop=presclkika}$]
If $\ell=0$, then $\ZZ^0=0$ and set $\Psig$ as the zero map and $\Ptau$ as the trivial map. Then since $\scl_{\Gg,\Ng}\equiv 0$ in this case (by Theorem~\ref{thm=Bavard}), we have the conclusions. Hence, we may assume that $\ell\in \NN$. Let $(\Psig,\Ptau)$ be the pair of maps constructed in Proposition~\ref{prop=itteQ}. In what follows, we prove that this pair fulfills all assertions of Proposition~\ref{prop=sclkika}. First, assertion~(1) is from  Proposition~\ref{prop=evmap}~(1) and Proposition~\ref{prop=dimQ}~(1); note that $\rho$ is a quasi-isometric coarse homomorphism. Here note that a composition of pre-coarse homomorphisms is a pre-coarse homomorphism. Assertion~(2) follows from Proposition~\ref{prop=evmap}~(2) and (3). Assertion~(3) is deduced from Proposition~\ref{prop=dimQ}~(2) and (3). Assertion~(4) follows from Proposition~\ref{prop=itteQ}. We now complete the proof of Proposition~\ref{prop=sclkika} and hence of its weaker version Proposition~\ref{prop=presclkika}.
\end{proof}

\subsection{Coarse group theoretic characterizations of $\Rdim \left(\QQQ(\Ng)^{\Gg}/\HHH^1(\Ng)^{\Gg}\right)$}

We state the following result, which characterizes $\Rdim \left(\QQQ(\Ng)^{\Gg}/\HHH^1(\Ng)^{\Gg}\right)$ in the language of coarse groups. This proposition may be seen as a refinement of Proposition~\ref{prop=prescldim}.

\begin{prop}\label{prop=metricdim}
Assume Setting~$\ref{setting=itsumono}$. Then, the following hold.
\begin{enumerate}[label=\textup{(}$\arabic*$\textup{)}]
\item \textup{(}coarse group theoretic characterization of $\Rdim \left(\QQQ(\Ng)^{\Gg}/\HHH^1(\Ng)^{\Gg}\right)$\textup{)}
\begin{align*}
&\Rdim \left(\QQQ(\Ng)^{\Gg}/\HHH^1(\Ng)^{\Gg}\right)\\
={} &\sup\left\{\ell\in \ZZ_{\geq 0}\,\middle|\,\exists \mathrm{\ coarsely\ proper\ coarse\ homomorphism}\ (\ZZ^{\ell},\|\cdot\|_1)\to ([\Gg,\Ng],d_{\scl_{\Gg,\Ng}})\right\} \\
={} &\inf\left\{\ell\in \ZZ_{\geq 0}\,\middle|\,\exists \mathrm{\ coarsely\ proper\ coarse\ homomorphism}\ ([\Gg,\Ng],d_{\scl_{\Gg,\Ng}})\to (\ZZ^{\ell},\|\cdot\|_1)\right\}
\end{align*}
and
\begin{align*}
&\Rdim \left(\QQQ(\Ng)^{\Gg}/\HHH^1(\Ng)^{\Gg}\right)\\
\leq{} &\sup\left\{\ell\in \ZZ_{\geq 0}\,\middle|\, \exists \mathrm{\ coarsely\ proper\ coarse\ homomorphism}\ (\ZZ^{\ell},\|\cdot\|_1)\to ([\Gg,\Ng],d_{\cl_{\Gg,\Ng}})\right\}.
\end{align*}
\item \textup{(}asymptotic dimensional characterization of $\Rdim \left(\QQQ(\Ng)^{\Gg}/\HHH^1(\Ng)^{\Gg}\right)$\textup{)}
\[
\Rdim \left(\QQQ(\Ng)^{\Gg}/\HHH^1(\Ng)^{\Gg}\right)=\asdim ([\Gg,\Ng],d_{\scl_{\Gg,\Ng}})
\]
and
\[
\Rdim \left(\QQQ(\Ng)^{\Gg}/\HHH^1(\Ng)^{\Gg}\right)\leq \asdim ([\Gg,\Ng],d_{\cl_{\Gg,\Ng}}).
\]
\end{enumerate}
\end{prop}

In particular, we have $\asdim ([\Gg,\Ng],d_{\cl_{\Gg,\Ng}})=\infty$ if $\QQQ(\Ng)^{\Gg}/\HHH^1(\Ng)^{\Gg}$ is infinite dimensional.

\begin{proof}[Proofs of Propositions~$\ref{prop=metricdim}$ and $\ref{prop=prescldim}$]
Recall from Theorem~\ref{thm=asdimZ} that for every $n\in \ZZ_{\geq 0}$, $\asdim(\ZZ^{n},\|\cdot\|_1)=n$. Recall also Proposition~\ref{prop=asdimCE}. First, we prove Proposition~\ref{prop=metricdim}.  If $\QQQ(\Ng)^{\Gg}/\HHH^1(\Ng)^{\Gg}$ is infinite dimensional, then (1) follows from Propositions~\ref{prop=dimQ}. If $\QQQ(\Ng)^{\Gg}/\HHH^1(\Ng)^{\Gg}$ is finite dimensional, then Propositions~\ref{prop=sclkika} implies (1). This ends the proof of (1). Now, (2) immediately follows from (1). This completes the proof of Proposition~\ref{prop=metricdim}.

Finally, the assertion of Proposition~\ref{prop=prescldim} is exactly the same as that of  Proposition~\ref{prop=metricdim}~(2). Hence, Proposition~\ref{prop=prescldim} has been proved as well.
\end{proof}

\subsection{A remark on $\Psig\circ \Ptau$ in Proposition~\ref{prop=sclkika}}\label{subsec=remark}
The following lemma might be of independent interest; recall Remark~\ref{rem=retract} in the comparative version.

\begin{lem}\label{lem=sigmatau}
In the statement of Proposition~$\ref{prop=sclkika}$, we can  take $\Psig\colon  [\Gg,\Ng]\to \ZZ^{\ell}$ and $\Ptau\colon \ZZ^{\ell}\to [\Gg,\Ng]$ such that moreover $\Psig\circ \Ptau=\mathrm{id}_{\ZZ^{\ell}}$ holds.
\end{lem}

\begin{proof}
We only treat the non-trivial case: $\ell\in\NN$. Take $\Ptau$ as in Proposition~\ref{prop=itteQ} associated with $(\nug_j)_{j\in \{1,\ldots,\ell\}}$ and $(\yg_i)_{i\in \{1,\ldots,\ell\}}$ satisfying \eqref{eq=Kdelta}. Set
\[
\tilde{D}=\left\lfloor 2(\ell-1)\max_{j\in \{1,\ldots, \ell\}}\DD(\nu_j)\right\rfloor +1\]
and $\rho'_{\tilde{D}}\colon \RR\to \tilde{D}\ZZ$ be `the' nearest point projection. Here, each element in $\tilde{D}\ZZ+\frac{\tilde{D}}{2}$ has exactly two nearest points on $\tilde{D}\ZZ$, and we may choose either one to define $\rho'_{\tilde{D}}$. Define $\rho_{\tilde{D}}\colon \mathbb{R}^{\ell}\to (\tilde{D}\ZZ)^{\ell}$ by $\rho_{\tilde{D}}=\rho'_{\tilde{D}}\times \cdots \times \rho'_{\tilde{D}}$. Set the group isomorphism $\lambda_{\tilde{D}}\colon \ZZ^{\ell}\stackrel{\cong}{\to}(\tilde{D}\ZZ)^{\ell}$ given by  coordinatewise multiplication of $\tilde{D}$. Finally, set
\[
\tilde{\Psig}=\lambda_{\tilde{D}}^{-1}\circ \rho_{\tilde{D}}\circ \Psig^{\RR}_{(\nug_1,\ldots ,\nug_{\ell})}\quad \textrm{and} \quad \tilde{\Ptau}=\Ptau \circ \lambda_{\tilde{D}}.
\]
Then these $\tilde{\Psig}$ and $\tilde{\Ptau}$, respectively, serve in the role of $\Psig$ and $\Ptau$  in Proposition~\ref{prop=sclkika}. Furthermore, by \eqref{eq=KdeltaD}, for every $(m_1,\ldots ,m_{\ell})\in \mathbb{Z}^{\ell}$ and for every $j\in \{1,\ldots,\ell\}$ we have
\[
(\rho'_{\tilde{D}}\circ \nug_j)\left((m_1\tilde{D},\ldots ,m_{\ell}\tilde{D})\right)=m_j\tilde{D};
\]
this implies that $\tilde{\Psig}\circ \tilde{\Ptau}=\mathrm{id}_{\ZZ^{\ell}}$.
Finally replace $(\Psig,\Ptau)$ with $(\tilde{\Psig},\tilde{\Ptau})$.
\end{proof}

\section{Comparison theorem of defects}\label{sec=defect}
In this section, we use the following setting.
\begin{setting}\label{setting=GLN}
Let $\Gg$ be a group, and  let $\Lg$ and $\Ng$ be two normal subgroups of $\Gg$ with $\Lg\geqslant \Ng$. Set $i\colon \Ng\hookrightarrow \Lg$ as the inclusion map.
\end{setting}

We prove the \emph{comparison theorem of defects}, Theorem~\ref{thm=comparisonDD}. As we describe in Subsection~\ref{subsec=org}, we will employ this theorem for the proofs of Theorem~\ref{mthm=main} and Theorem~\ref{mthm=dim}. In this section, we treat the defect $\DD(\nug)$ of a map $\nug\in \QQQ(\Ng)^{\Gg}$ as well as the defect $\DD(\psg)$ of a map $\psg\in \QQQ(\Lg)^{\Gg}$. To clarify the difference of the domains of these two maps, we use the symbol $\DDN(\nug)$ for the former and the symbol $\DDL(\psg)$ for the latter. We use these two symbols only in the current section.

\subsection{The statement of the comparison theorem of defects}\label{subsec=comparison}
In Setting~\ref{setting=GLN}, we study defects of elements $\nug\in \QQQ(\Ng)^{\Gg}$. We assume that $\Wcal(\Gg,\Lg,\Ng)$ is finite dimensional, and set $\ell=\Rdim \Wcal(\Gg,\Lg,\Ng)$. Take an arbitrary basis of $\Wcal(\Gg,\Lg,\Ng)$. Take  an arbitrary set $\{\nu_1,\ldots ,\nu_{\ell}\}\subseteq  \QQQ(\Ng)^{\Gg}$ of representatives of this basis. Then, every $\nug\in \QQQ(\Ng)^{\Gg}$ can be expressed as $\nug=\kg+i^{\ast}\psg+\sum\limits_{j\in \{1,\ldots,\ell\}}a_j\nug_j$,
where $\kg\in \HHH^1(\Ng)^{\Gg}$, $\psg\in \QQQ(\Lg)^{\Gg}$ and $(a_1,\ldots,a_{\ell})\in \RR^{\ell}$ (we note that $(a_1,\ldots,a_{\ell})$ is unique, but the pair $(\kg,\psg)$ is not necessarily unique). Then, the triangle inequality shows that
\begin{equation}\label{eq=defectabove}
\DDN(\nug)\leq \DDN(i^{\ast}\psg)+\sum_{j\in \{1,\ldots,\ell\}}|a_j|\DDN(\nug_j)\leq \DDL(\psg)+\sum_{j\in \{1,\ldots,\ell\}}|a_j|\DDN(\nug_j).
\end{equation}
The following \emph{comparison theorem of defects} provides an estimate in the converse direction to \eqref{eq=defectabove}.

\begin{thm}[Comparison theorem of defects]\label{thm=comparisonDD}
Assume Setting~$\ref{setting=GLN}$. Assume that $\Wcal(\Gg,\Lg,\Ng)$ is finite dimensional, and set $\ell=\Rdim \Wcal(\Gg,\Lg,\Ng)$.
Take an arbitrary basis of $\Wcal(\Gg,\Lg,\Ng)$. Take an arbitrary  set $\{\nu_1,\ldots ,\nu_{\ell}\}\subseteq  \QQQ(\Ng)^{\Gg}$ of representatives of this basis.
Then, there exist $\mathscr{C}_{1,\mathrm{ctd}},\mathscr{C}_{2,\mathrm{ctd}}\in \RR_{>0}$, both depending on the choice of $\nu_1,\ldots ,\nu_{\ell}$, such that the following statement holds true: for every $\nu\in \QQQ(\Ng)^{\Gg}$, there exist $\kg\in \HHH^1(\Ng)^{\Gg}$ and $\psg\in \QQQ(\Lg)^{\Gg}$ such that $\nug=\kg+i^{\ast}\psg+\sum\limits_{j\in \{1,\ldots,\ell\}}a_j\nug_j$ and
\begin{equation}\label{eq=defectbelow}
\DDN(\nug)\geq \mathscr{C}_{1,\mathrm{ctd}}{}^{-1}\left(\DDL(\psg)+\mathscr{C}_{2,\mathrm{ctd}}{}^{-1}\cdot\sum_{j\in \{1,\ldots,\ell\}}|a_j|\right).
\end{equation}
Here $(a_1,\ldots,a_{\ell})\in \RR^{\ell}$ is the uniquely determined tuple such that
\begin{equation}\label{eq=decmpWcal}
[\nug]=\sum_{j\in \{1,\ldots,\ell\}}a_j[\nug_j] \quad \in \Wcal(\Gg,\Lg,\Ng).
\end{equation}
\end{thm}

\begin{rem}\label{rem=l=0}
If $\ell=0$, then Theorem~\ref{thm=comparisonDD} asserts that there exists a constant $C\in \RR_{>0}$ such that the following holds: for every $\nu\in \QQQ(\Ng)^{\Gg}$, there exist $\kg\in \HHH^1(\Ng)^{\Gg}$ and $\psg\in \QQQ(\Lg)^{\Gg}$ such that $\nug=\kg+i^{\ast}\psg$ and $\DDL(\psg)\leq C\cdot \DDN(\nug)$. This assertion, together with Theorem~\ref{thm=Bavard}, proves Theorem~\ref{thm=biLip}. This is how we proved the Theorem~\ref{thm=biLip} when $\Lg=\Gg$ in \cite[Theorem~2.1~(1)]{KKMMM}.
\end{rem}

\begin{rem}\label{rem=a_j}
In the setting of Theorem~$\ref{thm=comparisonDD}$, inequality~\eqref{eq=defectbelow} implies that
\[
\sum_{j\in \{1,\ldots,\ell\}}|a_j|\leq \mathscr{C}_{1,\mathrm{ctd}}\mathscr{C}_{2,\mathrm{ctd}}\cdot \DDN(\nug).
\]
Indeed, we have $\DDN(\nu)-\mathscr{C}_{1,\mathrm{ctd}}{}^{-1}\mathscr{C}_{2,\mathrm{ctd}}{}^{-1}\cdot\sum\limits_{j\in \{1,\ldots,\ell\}}|a_j|\geq \mathscr{C}_{1,\mathrm{ctd}}{}^{-1} \DDL(\psg)\geq 0$. We will employ this observation in Sections~\ref{sec=Phi} and \ref{sec=proofmain}.
\end{rem}




\begin{exa}\label{exa=rot}
In Theorem~\ref{thm=comparisonDD}, the factor $\kg\in \HHH^1(\Ng)^{\Gg}$ is \emph{unavoidable} in general. In what follows, we examine one such example. Consider
\[
1\longrightarrow \ZZ \stackrel{i}{\longrightarrow}\widetilde{\Homeo}_+(S^1)\longrightarrow \Homeo_+(S^1)\longrightarrow 1.
\]
Here $\ZZ$ is identified with the subgroup of $\ZZ$-shift in $\Homeo_+(\RR)$; $\widetilde{\Homeo}_+(S^1)$ equals the subgroup of $\Homeo_+(\RR)$ of elements that commute with the $\ZZ$-shifts. Set $\Gg=\Lg=\widetilde{\Homeo}_+(S^1)$ and $\Ng=\ZZ$. Consider $\nug\colon \Ng\to \RR$ sending $n\in \Ng$ to $n\in \RR$. Then $\nug\in \HHH^1(\Ng)^{\Gg}\subseteq \QQQ(\Ng)^{\Gg}$. Consider the \emph{translation number} (the lift of the \emph{Poincar\'{e} rotation number})
\[
\widetilde{\rot}\colon \Gg\to \RR;\ \Gg\ni g\mapsto \lim_{n\to\infty}\frac{g^n(0)}{n}.
\]
Then it is well known that $\widetilde{\rot}\in \QQQ(\Gg)$, $i^{\ast}\widetilde{\rot}=\nug$ and $\WW(\Gg,\Ng)=0$. However, $\DDN(\nug)=0$ but $\DD_{\Gg}(\widetilde{\rot})=1$. Hence, \eqref{eq=defectbelow} \emph{fails} for the pair $(\kg,\psg)=(0,\widetilde{\rot})$ (see Remark~\ref{rem=l=0}). In this example, it is \emph{impossible} to find $\psg\in \QQQ(\Gg)$ with $i^{\ast}\psg=\nug$ for which \eqref{eq=defectbelow} holds. Indeed, since $\DDN(\nug)=0$, \eqref{eq=defectbelow} forces $\psg$ to be an element of $\HHH^1(\Gg)$. However, then $\psg$ must be the zero map because $\Gg=\widetilde{\Homeo}_+(S^1)$ is perfect, a contradiction. Nevertheless, if we take a `good' pair $(\kg,\psg)=(\nug,0)$, then \eqref{eq=defectbelow} holds.
\end{exa}




\subsection{Proof of Theorem~\ref{thm=comparisonDD}}\label{subsec=proof_comparison}
The key tools to the proof are Propositions~\ref{prop=inverse} and \ref{prop=defectBanach}. Recall from  Subsection~\ref{subsec=FA} that $\QQQ(\Lg)^{\Gg}/\HHH^1(\Lg)^{\Gg}$ and $\QQQ(\Ng)^{\Gg}/\HHH^1(\Ng)^{\Gg}$ are equipped with the defect norms. Following the convention of this section, we write these defect norms as $\ODDL$ and $\ODDN$, respectively, in the current section. More precisely, for every $\psg\in \QQQ(\Lg)^{\Gg}$ and for every $\nug\in \QQQ(\Ng)^{\Gg}$, we define $\ODDL([\psg]_{\Lg})=\DDL(\psg)$ and $\ODDN([\nug]_{\Ng})=\DDN(\nug)$. Here, $[\cdot]_{\Lg}$ and $[\cdot]_{\Ng}$ mean the equivalence classes modulo $\HHH^1(\Lg)^{\Gg}$ and $\HHH^1(\Ng)^{\Gg}$, respectively.

\begin{proof}[Proof of Theorem~$\ref{thm=comparisonDD}$]
Set $X=\QQQ(\Lg)^{\Gg}/\HHH^1(\Lg)^{\Gg}$ and $Y=\QQQ(\Ng)^{\Gg}/\HHH^1(\Ng)^{\Gg}$.
By Proposition~\ref{prop=defectBanach}, the spaces $(X,\ODDL)$ and $(Y,\ODDN)$ are Banach spaces. The inclusion map $i\colon \Ng \hookrightarrow \Lg$ induces a linear operator
\[
T\colon (X,\ODDL)\to (Y,\ODDN);\quad X\ni [\psg]_{\Lg} \mapsto [i^{\ast}\psg]_{\Ng}\in Y.
\]
Since  $\DDN(i^{\ast}\psg)\leq \DDL(\psg)$ for every $\psg\in \QQQ(\Lg)^{\Gg}$, we have $\|T\|_{\mathrm{op}}\leq 1$. By assumption, $\Rdim (Y/T(X))=\ell$ and the set $\{[\nug_1]_{\Ng},\ldots,[\nug_{\ell}]_{\Ng}\}$ forms a set of representatives of a basis of $Y/T(X)$.

Note that $T\colon X\to Y$ need not injective. To deal with this issue, set $X_0=\Ker (T)$ and consider the quotient Banach space $X'=X/X_0$. Here, recall that the quotient norm $\|\cdot\|_{X'}$ is defined for every $\xi\in X$ as
\begin{equation}\label{eq=quotnorm}
\|[\xi]_{X_0}\|_{X'}=\inf_{\xi_0\in X_0}\ODDL(\xi+\xi_0).
\end{equation}
Here $[\cdot]_{X_0}$ means the equivalence class modulo $X_0$. Then $T$ induces an injective linear operator $\overline{T}\colon (X',\|\cdot\|_{X'})\to (Y,\ODDN)$ with $\|\overline{T}\|_{\mathrm{op}}\leq 1$.

Now we apply Proposition~\ref{prop=inverse} to $\overline{T}$ and $[\nug_1]_{\Ng},\ldots,[\nug_{\ell}]_{\Ng}$. Then we obtain two constants $\mathscr{C}'_{1,\mathrm{ctd}},\mathscr{C}_{2,\mathrm{ctd}}\in \RR_{>0}$, both depending on $[\nug_1]_{\Ng},\ldots,[\nug_{\ell}]_{\Ng}$ (see Remark~\ref{rem=const}), such that the following holds: for every $\psi\in \QQQ(\Lg)^{\Gg}$ and every $(a_1,\ldots,a_{\ell})\in \RR^l$, we have
\begin{equation}\label{eq=normcomparison}
\DDN\Bigl(i^{\ast}\psi+\sum_{j\in \{1,\ldots,\ell\}}a_j\nug_j\Bigr) \geq \mathscr{C}'_{1,\mathrm{ctd}}{}^{-1}\left(\|[[\psi]_{\Lg}]_{X_0}\|_{X'}+\mathscr{C}_{2,\mathrm{ctd}}{}^{-1}\cdot \sum_{j\in \{1,\ldots,\ell\}}|a_j|\right).
\end{equation}

Take an arbitrary constant $\mathscr{C}_{1,\mathrm{ctd}}$ with $\mathscr{C}_{1,\mathrm{ctd}}>\mathscr{C}'_{1,\mathrm{ctd}}$ and fix it. In what follows, we prove that these $\mathscr{C}_{1,\mathrm{ctd}}$ and $\mathscr{C}_{2,\mathrm{ctd}}$ work. Let $\nug\in \QQQ(\Ng)^{\Gg}$. Take the unique tuple $(a_1,\ldots ,a_{\ell})\in \RR^{\ell}$ such that \eqref{eq=decmpWcal} holds. Then, there exists $\psi'\in i^{\ast}\QQQ(\Lg)$ such that
\begin{equation}\label{eq=decomp2}
\nug=i^{\ast}\psi'+\sum_{j\in \{1,\ldots,\ell\}}a_j\nug_j.
\end{equation}
By applying \eqref{eq=normcomparison}, we obtain that
\begin{equation}\label{eq=normcomparison2}
\DDN(\nug)\geq \mathscr{C}'_{1,\mathrm{ctd}}{}^{-1}\left(\|[[\psi']_{\Lg}]_{X_0}\|_{X'}+\mathscr{C}_{2,\mathrm{ctd}}{}^{-1}\cdot \sum_{j\in \{1,\ldots,\ell\}}|a_j|\right).
\end{equation}
Take $\varepsilon \in \RR_{>0}$ with $(1-\varepsilon)\mathscr{C}_{1,\mathrm{ctd}}\geq \mathscr{C}'_{1,\mathrm{ctd}}$. Then by \eqref{eq=quotnorm}, there exists $\psi''\in \QQQ(\Lg)^{\Gg}$ such that $[\psi'']_{\Lg}\in X_0$ and
\begin{equation}\label{eq=defectquot}
(1-\varepsilon)\DDL(\psi'-\psi'')\leq \|[[\psi']_{\Lg}]_{X_0}\|_{X'}.
\end{equation}
Set $\kg=i^{\ast}\psi''$. Since $[\psi'']_{\Lg}\in X_0=\Ker (T)$, we have $\kg\in \HHH^1(\Ng)^{\Gg}$. Finally, set $\psi=\psi'-\psi''$. Then by \eqref{eq=decomp2}, we have
\[
\nug=i^{\ast}\psi+i^{\ast}\psi''+\sum_{j\in \{1,\ldots,\ell\}}a_j\nug_j\\=\kg+i^{\ast}\psi+\sum_{j\in \{1,\ldots,\ell\}}a_j\nug_j.
\]
By \eqref{eq=normcomparison2} and \eqref{eq=defectquot}, we have
\begin{align*}
\DDL(\nug)&\geq \mathscr{C}'_{1,\mathrm{ctd}}{}^{-1}\left((1-\varepsilon)\DDL(\psi)+\mathscr{C}_{2,\mathrm{ctd}}{}^{-1}\cdot \sum_{j\in \{1,\ldots,\ell\}}|a_j|\right)\\
&\geq \mathscr{C}_{1,\mathrm{ctd}}{}^{-1}\left(\DDL(\psi)+\mathscr{C}_{2,\mathrm{ctd}}{}^{-1}\cdot \sum_{j\in \{1,\ldots,\ell\}}|a_j|\right).
\end{align*}
Therefore, \eqref{eq=defectbelow} holds. This completes the proof.
\end{proof}

\subsection{Criterion for finite dimensionality of $\Wcal(\Gg,\Lg,\Ng)$}\label{subsec=criterion}

In this subsection, we exhibit a sufficient condition on the triple $(\Gg,\Lg,\Ng)$ in Setting~\ref{setting=GLN} such that $\Wcal(\Gg,\Lg,\Ng)$ is finite dimensional.

\begin{prop}\label{prop=findimWW}
Assume Settings~$\ref{setting=itsumono}$ and $\ref{setting=QG}$. Let $d_2$ be the map from $\HHH^1(\Ng)^{\Gg}$ to $\HHH^2(\QG)$ in the diagram \eqref{eq=diagram} of Theorem~$\ref{thm=KKMMMmain1}$. Assume that $\Ker(c_{\QG}^3\colon \HHH^3_b(\QG)\to \HHH^3(\QG))$ and $\HHH^2(\QG)/d_2(\HHH^1(\Ng)^{\Gg})$ are both finite dimensional. Then, for every normal subgroup $\Lg$ of $\Gg$ with $\Lg\geqslant \Ng$,
\[
\Rdim \Wcal(\Gg,\Lg,\Ng)<\infty.
\]
In particular, $\HHH^3_b(\QG)$ and $\HHH^2(\QG)$ are both finite dimensional, then for every normal subgroup $\Lg$ of $\Gg$ with $\Lg\geqslant \Ng$, we have $\Rdim \Wcal(\Gg,\Lg,\Ng)<\infty$.
\end{prop}

\begin{proof}
First, we treat the case where $\Ker(c_{\QG}^3)$ and $\HHH^2(\QG)/d_2(\HHH^1(\Ng)^{\Gg})$ are both finite dimensional. Let $\xi_4$ be the map from $\HHH^2(\QG)$ to $\HHH_{/b}^2(\QG)$ in \eqref{eq=diagram}.
Since the sequence
\[
\HHH^2(\QG)/d_2(\HHH^1(\Ng)^{\Gg}) \xrightarrow{\xi_4} \HHH_{/b}^2(\QG)/\xi_4d_2(\HHH^1(\Ng)^{\Gg}) \to \HHH_b^3(\QG) \to \HHH^3(\QG)
\]
is exact, the assumptions imply that $\HHH_{/b}^2(\QG)/\xi_4d_2(\HHH^1(\Ng)^{\Gg})$ is finite dimensional.
Since the space $\WW(\Gg,\Ng)$ injects to $\HHH_{/b}^2(\QG)/\xi_4d_2(\HHH^1(\Ng)^{\Gg})$ by Theorem \ref{thm=KKMMMmain1}, we conclude that $\WW(\Gg,\Ng)$ is finite dimensional. For every normal subgroup $\Lg$ of $\Gg$ with $\Lg\geqslant \Ng$, the space $\WW(\Gg,\Ng)$ surjects onto $\Wcal(\Gg,\Lg,\Ng)$. Hence, $\Wcal(\Gg,\Lg,\Ng)$ is also finite dimensional. Therefore, we obtain the conclusions in this case. Now, the final assertion immediately follows from the first case.
\end{proof}

We also have the following dimension estimate.

\begin{thm}\label{thm=dimWcal}
Assume Settings~$\ref{setting=GLN}$.
Denote by $i_{\Ng}^{\ast}$ the map $i_{\Ng}^{\ast}\colon \WW(\Gg,\Lg)\to \WW(\Gg,\Ng)$ induced by $i_{\Ng}=i\colon \Ng\hookrightarrow \Lg$. Assume that $\Rdim\WW(\Gg,\Ng)<\infty$. Then we have
\[
\Rdim \Wcal(\Gg,\Lg,\Ng)=\Rdim \WW(\Gg,\Ng)-\Rdim i_{\Ng}^{\ast}\WW(\Gg,\Lg).
\]
\end{thm}

\begin{proof}
Let $\mathcal{K}$ be the kernel of the surjection $\WW(\Gg, \Ng) \to \Wcal(\Gg,\Lg,\Ng)$.
Then it suffices to show that the map $i_{\Ng}^{\ast}$ surjects onto $\mathcal{K}$.
It is clear that the image of $i_{\Ng}^{\ast}$ is contained in $\mathcal{K}$.
Let $\nu \in \QQQ(\Ng)^{\Gg}$ satisfying $[\nu] \in \mathcal{K}$, that is, there exist $k \in \HHH^1(\Ng)^{\Gg}$ and $\psi \in \QQQ(\Lg)^{\Gg}$ such that $\nu = k + i_{\Ng}^{\ast} \psi$.
Then, $[\psi] \in \WW(\Gg, \Lg)$, and we have $i_{\Ng}^{\ast}[\psi] = [\nu - k] = [\nu] \in \mathcal{K}$. This ends our proof.
\end{proof}

In addition, we state the following variant of Corollary~\ref{cor=intersection} for a group triple $(\Gg,\Lg,\Ng)$.

\begin{prop}\label{prop=intersectiontriple}
Assume Setting~$\ref{setting=GLN}$. Assume that $\HHH^1(\Lg/\Ng) = \QQQ(\Lg/\Ng)$ and that $\HHH_b^2(\Lg/\Ng)^{\Gg} = 0$. Then we have $\HHH^1(\Ng)^{\Gg}\cap i^{\ast}\QQQ(\Lg)^{\Gg}=i^{\ast}\HHH^1(\Lg)^{\Gg}$.

In particular, under Setting~$\ref{setting=GLN}$, we have $\HHH^1(\Ng)^{\Gg}\cap i^{\ast}\QQQ(\Lg)^{\Gg}=i^{\ast}\HHH^1(\Lg)^{\Gg}$ if $\Lg/\Ng$ is boundedly $2$-acyclic.
\end{prop}

Recall from  Lemma~\ref{lem=comparison} that if $\HHH^1(\Lg/\Ng) = \QQQ(\Lg/\Ng)$ if and only if the comparison map $c_{\Lg/\Ng}^2$ is injective.

\begin{proof}[Proof of Proposition~$\ref{prop=intersectiontriple}$]
Set $\Lambda=\Lg/\Ng$. First, we  verify the former assertion. It is clear that $\HHH^1(\Ng)^{\Gg}\cap i^{\ast}\QQQ(\Lg)^{\Gg} \supseteq i^{\ast}\HHH^1(\Lg)^{\Gg}$ holds.
We show that $\HHH^1(\Ng)^{\Gg}\cap i^{\ast}\QQQ(\Lg)^{\Gg} \subseteq i^{\ast}\HHH^1(\Lg)^{\Gg}$ also holds.
Let us consider the following diagram, which is the diagram \eqref{eq=diagram} applied to the exact sequence $1 \to \Ng \to \Lg \xrightarrow{q} \Lambda \to 1$:
\begin{align*}
  \xymatrix{
  0 \ar[r] & \HHH^1(\Lambda) \ar[r]^-{q^{\ast}} \ar[d] & \HHH^1(\Lg) \ar[r]^-{i^{\ast}} \ar[d] & \HHH^1(\Ng)^{\Lg} \ar[r]^-{d_2} \ar[d] & \HHH^2(\Lambda) \ar[r] \ar[d] & \HHH^2(\Lg) \ar[d] \\
  0 \ar[r] & \QQQ(\Lambda) \ar[r]^-{q^{\ast}} & \QQQ(\Lg) \ar[r] & \QQQ(\Ng)^{\Lg} \ar[r] & \HHH^2_{/b}(\Lambda) \ar[r] & \HHH^2_{/b}(\Lg).
  }
\end{align*}
Let $\kg$ be an element in $\HHH^1(\Ng)^{\Gg}\cap i^{\ast}\QQQ(\Lg)^{\Gg}$.
Take an element $\psg$ in $\QQQ(\Lg)^{\Gg}$ satisfying $\kg = i^{\ast}\psg = \psg|_{\Ng}$.
It suffices to show that $\psg$ is contained in $\HHH^1(\Lg)$.
By the equality $\kg = \psg|_{\Ng}$, there uniquely exists a bounded cocycle $c_b \in C_b^2(\Lambda)$ such that $q^{\ast} c_b = \delta \psg$, and the equality $d_2 (\kg) = c_{\Lambda}^2 (c_b)$ holds (\cite[Lemma 4.6 and Lemma 5.3]{KM}).
Since $\psg$ is $\Gg$-invariant, so is $c_b$.
By the assumption that $\HHH_b^2(\Lambda)^{\Gg} = 0$, the bounded cohomology class $[c_b]\in \HHH_b^2(\Lambda)$ equals $0$.
In particular, the class $d_2(\kg) = c_{\Lambda}^2([c_b])$ is equal to zero.
From the exactness of the first low in the diagram, we may take an element $\kg_1$ in $\HHH^1(\Lg)$ such that $i^{\ast}\kg_1 = \kg$.
By the commutativity of the diagram and the exactness of the second low in the diagram, there exists an element $\kg_2$ in $\QQQ(\Lambda)=\HHH^1(\Lambda)$ such that $q^{\ast} \kg_2 = \psg - \kg_1 \in \QQQ(L)$.
Then we have $\psg = \kg_1 + q^{\ast} \kg_2 \in \HHH^1(\Lg)$, as desired.

To show the latter assertion, bounded $2$-acyclicity of $\Lambda$ in particular implies $\QQQ(\Lambda)/\HHH^1(\Lambda)=0$ (by Lemma~\ref{lem=comparison}) and $\HHH_b^2(\Lambda)^{\Gg} = 0$. Therefore, the former assertion applies to the current case.
\end{proof}

Proposition~\ref{prop=intersectiontriple} yields the following corollary to Theorem~\ref{thm=comparisonDD}; recall also Example~\ref{exa=rot}.

\begin{cor}\label{cor=DD2}
Assume Settings~$\ref{setting=GLN}$. Assume that $\HHH^1(\Lg/\Ng) = \QQQ(\Lg/\Ng)$ and that $\HHH_b^2(\Lg/\Ng)^{\Gg} = 0$. Assume that $\Wcal(\Gg,\Lg,\Ng)$ is finite dimensional. Then there exists $C\in \RR_{>0}$ such that the following holds true: for every $\nug\in i^{\ast}\QQQ(\Lg)^{\Gg}$, there exists $\psg\in \QQQ(\Lg)^{\Gg}$ such that $\nug =i^{\ast}\psg$ and $\DDL(\psg)\leq C\cdot \DDN(\nug)$ holds.
\end{cor}

\begin{proof}
Let $\ell=\Rdim \Wcal(\Gg,\Lg,\Ng)$. Fix a set of representative $\{\nug_1,\ldots,\nug_{\ell}\}$ of a basis of $\Wcal(\Gg,\Lg,\Ng)$. Take the constant $\mathscr{C}_{1,\mathrm{ctd}}\in \RR_{>0}$ as in Theorem~\ref{thm=comparisonDD}, and  set $C=\mathscr{C}_{1,\mathrm{ctd}}$. Then Theorem~\ref{thm=comparisonDD} applies to this $\nug$ with $(a_1,\ldots,a_{\ell})=(0,\ldots,0)$. Hence, we obtain $\kg\in \HHH^1(\Ng)^{\Gg}\cap i^{\ast}\QQQ(\Lg)^{\Gg}$ and $\psg'\in \QQQ(\Lg)^{\Gg}$ such that $\nug=\kg+i^{\ast}\psg'$ and $\DDL(\psg')\leq C\cdot \DDN(\nug)$. Note that $\kg=\nug-i^{\ast}\psg'$ belongs to $\HHH^1(\Ng)^{\Gg}\cap i^{\ast}\QQQ(\Lg)^{\Gg}$. By Proposition~\ref{prop=intersectiontriple}, $\kg\in i^{\ast}\HHH^1(\Lg)^{\Gg}$ holds. Write $\kg=i^{\ast}\tilde{\kg}$ with $\tilde{\kg}\in \HHH^1(\Lg)^{\Gg}$, and set $\psg=\psg'+\tilde{\kg}$. Then, $\nug=i^{\ast}\psg$ and $ \DDL(\psg)=\DDL(\psg')\leq C\cdot \DDN(\nug)$.
\end{proof}

\section{Construction of $\PPh^{\RR}$}\label{sec=Phi}
Recall the three steps in the outlined proof of Theorem~\ref{mthm=main} from Subsection~\ref{subsec=org}:

\begin{enumerate}
  \item [\underline{Step~1}.]  construct $\PPh^{\RR}\colon [\Gg,\Ng]\to \RR^{\ell}$;
  \item  [\underline{Step~2}.] construct $\PPs\colon \ZZ^{\ell}\to [\Gg,\Ng]$;
  \item  [\underline{Step~3}.] take an appropriate $\PPh\colon [\Gg,\Ng]\to \ZZ^{\ell}$ out of $\PPh^{\RR}$, and study $\PPs\circ\PPh$ and $\PPh\circ \PPs$.
\end{enumerate}

In this section, we take Step~1. In fact, we may construct the map $\PPh^{\RR}$ in a much broader situation than that of Theorem~\ref{mthm=main}. More precisely, the construction (Theorem~\ref{thm=evmap}) works as long as $\Wcal(\Gg,\Lg,\Ng)$ is finite dimensional.  Similar to Section~\ref{sec=absolute}, we focus on the case where $\Rdim\Wcal(\Gg,\Lg,\Ng)\in \NN$ (otherwise, we can take the zero map).
\subsection{The construction of $\PPh^{\RR}$}\label{subsec=Phi}
The following theorem is the precise form of Step~1. Recall our definition of QI-type estimates from below/above from Definition~\ref{defn=QI}.

\begin{thm}[construction of $\PPh^{\RR}$]\label{thm=evmap}
Assume Setting~$\ref{setting=GLN}$. Assume that $\Wcal(\Gg,\Lg,\Ng)$ is non-zero finite dimensional, and set $\ell=\Rdim \Wcal(\Gg,\Lg,\Ng)<\infty$.
Take an arbitrary basis of $\Wcal(\Gg,\Lg,\Ng)$. Take an arbitrary   set $\{\nu_1,\ldots ,\nu_{\ell}\}\subseteq  \QQQ(\Ng)^{\Gg}$ of representatives of this basis. Define $\PPh^{\RR}=\PPh^{\RR}_{(\nug_1,\ldots,\nug_{\ell})}\colon ([\Gg,\Ng],d_{{\scl}_{\Gg,\Ng}})\to (\mathbb{R}^{\ell},\|\cdot\|_1)$ by
\[
\PPh^{\RR}(\yg)=(\nug_1(\yg),\nug_2(\yg),\ldots ,\nug_{\ell}(\yg))
\]
for every $\yg\in [\Gg,\Ng]$.
Then the following hold true.
\begin{enumerate}[label=\textup{(}$\arabic*$\textup{)}]
  \item The map $\PPh^{\RR}$ is a pre-coarse homomorphism.
  \item We have the following QI-type estimate from above on $[\Gg,\Ng]$:
\[
\|\PPh^{\RR}(\yg_1)-\PPh^{\RR}(\yg_2)\|_1\leq 2\left(\sum_{j\in \{1,\ldots,\ell\}}\DD(\nug_j)\right)\cdot d_{\scl_{\Gg,\Ng}}(\yg_1,\yg_2) + \sum_{j\in \{1,\ldots,\ell\}}\DD(\nug_j).
\]
  \item Let $A\subseteq [\Gg,\Ng]$ be a $d_{\scl_{\Gg,\Lg}}$-bounded set, and set $D_A=\diam_{d_{\scl_{\Gg,\Lg}}}(A)$. Then we have the following QI-type estimate from below on  $A$:
\begin{equation}\label{eq=PPhPPh}
\|\PPh^{\RR}(\yg_1)-\PPh^{\RR}(\yg_2)\|_1\geq \frac{2}{\mathscr{C}_{1,\mathrm{ctd}}\mathscr{C}_{2,\mathrm{ctd}}}\cdot d_{\scl_{\Gg,\Ng}}(\yg_1,\yg_2)-\frac{(2D_A+1)\mathscr{C}_{1,\mathrm{ctd}}+1}{\mathscr{C}_{1,\mathrm{ctd}}\mathscr{C}_{2,\mathrm{ctd}}}.
\end{equation}
Here $\mathscr{C}_{1,\mathrm{ctd}}$ and $\mathscr{C}_{2,\mathrm{ctd}}$ are the constants associated with $(\nug_1,\ldots,\nug_{\ell})$ appearing in Theorem~$\ref{thm=comparisonDD}$.
\end{enumerate}
In particular, $\PPh^{\RR}\colon ([\Gg,\Ng],d_{{\scl}_{\Gg,\Ng}})\to (\mathbb{R}^{\ell},\|\cdot\|_1)$ is a coarse homomorphism and its restriction to every $d_{\scl_{\Gg,\Lg}}$-bounded set is a quasi-isometric embedding.
\end{thm}

\begin{proof}
The proofs of (1) and (2) are parallel to those of Proposition~\ref{prop=evmap}~(1) and (2). In what follows, we prove (3). Let $\yg_1,\yg_2\in A$. By Theorem~\ref{thm=Bavard}, for every $\psg\in \QQQ(\Lg)^{\Gg}$, we have
\begin{equation}\label{eq=gosapsg}
|\psg(\yg_1)-\psg(\yg_2)|\leq (2D_A+1)\DD(\psg).
\end{equation}
Let $\nug\in \QQQ(\Ng)^{\Gg}$. Then, by Theorem~\ref{thm=comparisonDD}, there exist $\kg\in \HHH^1(\Ng)^{\Gg}$, $\psg\in \QQQ(\Lg)^{\Gg}$ and $(a_1,\ldots ,a_{\ell})\in \RR^{\ell}$ such that $\nug=\kg+i^{\ast}\psg+\sum\limits_{j\in \{1,\ldots,\ell\}}a_j\nug_j$
and
\begin{equation}\label{eq=defectDDnug}
\DD(\nug)\geq \mathscr{C}_{1,\mathrm{ctd}}{}^{-1}\left(\DD(\psg)+\mathscr{C}_{2,\mathrm{ctd}}{}^{-1}\cdot \sum_{j\in \{1,\ldots,\ell\}}|a_j|\right).
\end{equation}
By Remark~\ref{rem=a_j}, the inequality \eqref{eq=defectDDnug} in particular implies that
\begin{equation}\label{eq=DDbelowa_j}
\max_{j\in \{1,\ldots,\ell\}}|a_j|\left(\leq \sum_{j\in \{1,\ldots,\ell\}}|a_j|\right)\leq \mathscr{C}_{1,\mathrm{ctd}}\mathscr{C}_{2,\mathrm{ctd}}\cdot\DD(\nug).
\end{equation}
Note that $\kg(\yg_1)=\kg(\yg_2)=0$ since $\yg_1,\yg_2\in [\Gg,\Ng]$. Together with \eqref{eq=gosapsg}, we have
\[
\left|\nug(\yg_1)-\nug(\yg_2)- \sum_{j\in \{1,\ldots,\ell\}}a_j(\nug_j(\yg_1)-\nug_j(\yg_2))\right|\leq |\psg(\yg_1)-\psg(\yg_2)| \leq (2D_A+1)\DD(\psg).
\]
By \eqref{eq=defectDDnug}, we obtain that
\begin{equation}\label{eq=mainnug}
\left|\nug(\yg_1)-\nug(\yg_2)- \sum_{j\in \{1,\ldots,\ell\}}a_j(\nug_j(\yg_1)-\nug_j(\yg_2))\right|\leq (2D_A+1)\mathscr{C}_{1,\mathrm{ctd}}\DD(\nug).
\end{equation}
By \eqref{eq=DDbelowa_j}, we have
\begin{align*}
\left|\sum_{j\in \{1,\ldots,\ell\}}a_j(\nug_j(\yg_1)-\nug_j(\yg_2))\right|&\leq \left(\max_{j\in \{1,\ldots,\ell\}}|a_j|\right) \cdot \sum_{j\in \{1,\ldots,\ell\}}|\nug_j(\yg_1)-\nug_j(\yg_2)|\\
&\leq \mathscr{C}_{1,\mathrm{ctd}}\mathscr{C}_{2,\mathrm{ctd}} \DD(\nug)\cdot \|\PPh^{\RR}(\yg_1)-\PPh^{\RR}(\yg_2)\|_1.
\end{align*}
Therefore, by \eqref{eq=mainnug} we obtain that
\[
|\nug(\yg_1^{-1}\yg_2)|\leq \left(\mathscr{C}_{1,\mathrm{ctd}}\mathscr{C}_{2,\mathrm{ctd}}\cdot  \|\PPh^{\RR}(\yg_1)-\PPh^{\RR}(\yg_2)\|_1+(2D_A+1)\mathscr{C}_{1,\mathrm{ctd}}+1\right)\cdot\DD(\nug).
\]
Then it follows from Theorem~\ref{thm=Bavard} that
\[
d_{\scl_{\Gg,\Ng}}(\yg_1,\yg_2)\leq \frac{\mathscr{C}_{1,\mathrm{ctd}}\mathscr{C}_{2,\mathrm{ctd}}}{2}\cdot  \|\PPh^{\RR}(\yg_1)-\PPh^{\RR}(\yg_2)\|_1+\frac{(2D_A+1)\mathscr{C}_{1,\mathrm{ctd}}+1}{2};
\]
equivalently, we obtain \eqref{eq=PPhPPh}.
\end{proof}

\subsection{Proof of Theorem~\ref{mthm=dimfin}}\label{subsec=dimfin}
Here, we deduce Theorem~\ref{mthm=dimfin} from Theorem~\ref{thm=evmap}.

\begin{proof}[Proof of Theorem~\textup{\ref{mthm=dimfin}}]
If $\Wcal(\Gg,\Lg,\Ng)=0$, then Theorem~\ref{thm=biLip} shows that every $d_{\scl_{\Gg,\Lg}}$-bounded set $A\subseteq [\Gg,\Ng]$ is $d_{\scl_{\Gg,\Ng}}$-bounded. Hence the assertions of Theorem~\ref{mthm=dimfin} hold. If $\Wcal(\Gg,\Lg,\Ng)$ is infinite dimensional, then there is nothing to prove. Finally, if $\Wcal(\Gg,\Lg,\Ng)$ is non-zero finite dimensional, then Theorem~\ref{thm=evmap} implies the conclusions. This completes our proof.
\end{proof}

\section{Core extractor}\label{sec=CE}
In Section~\ref{sec=Psi}, we will take Step~2 (constructing the map $\PPs$) in the outlined proof of Theorem~\ref{mthm=main}.  Sections~\ref{sec=CE} and \ref{sec=CEabel} are devoted to the introduction to the theory behind the construction of $\PPs$: it is  the theory of \emph{core extractors}.

The map $\Ptau$ in Proposition~\ref{prop=dimQ} is of the form
\[
\ZZ^{\ell}\ni (m_1,\ldots,m_{\ell})\mapsto \yg_1^{m_1}\cdots \yg_{\ell}^{m_{\ell}},
\]
for appropriate elements $\yg_1,\ldots,\yg_{\ell}\in [\Gg,\Ng]$. However, in this construction, $\Ptau(\ZZ^{\ell})$ is $d_{\scl_{\Gg,\Lg}}$-bounded if and only if the equalities
\[
\scl_{\Gg,\Lg}(\yg_1)=\cdots =\scl_{\Gg,\Lg}(\yg_{\ell})=0
\]
hold; recall the semi-homogeneity \eqref{eq=semihom}. These equalities are too strong conditions for general cases. For instance, if $\Gg$ is a non-elementary Gromov-hyperbolic group, then \cite{CaFu10} in particular implies the following (see also \cite[{\S}6.C2]{Gromov1993} and \cite{EpsteinFujiwara}): if $\gG \in [\Gg,\Gg]$ satisfies $\scl_{\Gg}(\gG)=0$, then $\gG$ admits a non-zero power that is conjugate to its inverse (in $\Gg$). We note that if $\yg\in \Gg$ lies in the group of the form $[\Gg,\Lg]$ and if $\scl_{\Gg,\Lg}(\yg)=0$, then $\scl_{\Gg}(\yg)=0$ (because $\scl_{\Gg}(\yg)\leq \scl_{\Gg,\Lg}(\yg)$). Hence, if $\Gg$ is a non-elementary Gromov-hyperbolic group, then for such $\yg$ there exist $n\in \NN$ and $\gamma\in \Gg$ such that $\gamma \yg^n\gamma^{-1}=\yg^{-n}$; if such $\yg$ lies in the group of the form $[\Gg,\Ng]$, then $\scl_{\Gg,\Ng}(\yg)=0$. Indeed, for $n$ and $\gamma$ above, we have for every $m\in \NN$,
\[
2mn\cdot \scl_{\Gg,\Ng}(\yg)=\scl_{\Gg,\Ng}(\yg^{mn} \yg^{mn})=\scl_{\Gg,\Ng}(\gamma \yg^{-mn}\gamma^{-1} \yg^{mn}) =\scl_{\Gg,\Ng}([\gamma, \yg^{-mn}])\leq 1;
\]
hence $\scl_{\Gg,\Ng}(\yg)=0$.
Therefore, in order to construct a map $\PPs\colon [\Gg,\Ng]\to \ZZ^{\ell}$ as in Theorem~\ref{mthm=main}, we need to modify the whole construction of the map $\Ptau$ in Proposition~\ref{prop=sclkika} in general.

A model case was found in \cite{MMM}:  as we mentioned in Subsection~\ref{subsec=comparative}, some of the authors showed there that for $\Gg=\pi_1(\Sigma_{\genus})$ with $\genus\in \NN_{\geq 2}$ and $\Ng=[\Gg,\Gg]$, $\scl_{\Gg}$ and $\scl_{\Gg,\Ng}$ are \emph{not} bi-Lipschitzly equivalent on $[\Gg,\Ng]$. In fact, they constructed the following sequence
\begin{equation}\label{eq=surfaceelem}
(\yg_m)_{m\in \NN}=(c_1[b_1^{-1},a_1^{m}]c_1^{-1} c_2[b_2^{-1},a_2^{m}]c_2^{-1} \cdots c_{\genus}[b_{\genus}^{-1},a_{\genus}^{m}]c_{\genus}^{-1})_{m\in \NN},
\end{equation}
where $a_1,\ldots,a_{\genus},b_1,\ldots,b_{\genus}$ are the standard generators of $G$ and $c_1,\ldots,c_{\genus}$ are certain elements in $\Gg$ (see \cite[Section~6]{MMM} for more details). Then, they showed that
\[
(\sup_{m\in \NN}\scl_{\Gg}(\yg_m)\leq \genus <\infty, \quad \textrm{but}) \quad \sup_{m\in \NN}\scl_{\Gg,\Ng}(\yg_m)=\infty,
\]
thus proving the non-(bi-Lipschitz-)equivalence above.

This example suggests that we should try to construct a map $\PPs\colon \ZZ^{\ell}\to [\Gg,\Ng]$ of the form
\[
(m_1,\ldots,m_{\ell})\mapsto \beta_1(m_1)\cdots \beta_{\ell}(m_{\ell}),
\]
where the maps $\beta_1,\ldots,\beta_{\ell}\colon \ZZ\to [\Gg,\Ng]$ are of the form $m\mapsto [\gG_1,\gG_1'{}^{m}]\cdots [\gG_{t},\gG_t'{}^{m}]$ for some $t\in \NN$, some $\gG_1,\ldots,\gG_{t}\in \Gg$ and some $\gG'_1,\ldots,\gG'_{t}\in \Lg$. In this manner, we can have a $d_{\cl_{\Gg,\Lg}}$-bounded map $\PPs\colon \ZZ^{\ell}\to [\Gg,\Ng]$.
However, it is a \emph{challenge} to construct such maps $\beta_1,\ldots,\beta_{\ell}$ in such a way that the resulting map $\PPs$ works. In the case of $(\Gg,\Lg,\Ng)=(\pi_1(\Sigma_{\genus}),\pi_1(\Sigma_{\genus}),\gamma_2(\pi_1(\Sigma_{\genus})))$, we have $\Rdim \WW(\Gg,\Ng)=1$ (Theorem~\ref{thm=WW}~(2)). As we briefly mentioned in Subsections~\ref{subsec=digest} and \ref{subsec=comparative}, in \cite{MMM} some of the authors constructed a representative $\nug_1\in \QQQ(\Ng)^{\Gg}$ of a generator of $\WW(\Gg,\Ng)$  ($q'{}^{\ast}\mu_{\rho_{\ell}}$ in \cite[Section~6]{MMM}). The construction \eqref{eq=surfaceelem} is based on the fact that we can estimate the values $(\nug_1(\yg_m))_{m\in\NN}$. Contrastingly, to show Theorem~\ref{mthm=main}, we need to treat \emph{all} triples of the form $(\Gg,\Lg,\Ng)=(\Gg,\gamma_{q-1}(\Gg),\gamma_q(\Gg))$, $q\in \NN_{\geq 2}$ such that $\Wcal(\Gg,\Lg,\Ng)$ is non-zero finite dimensional. In this generality, we can \emph{hardly} hope for fully concrete constructions of a set of representatives (in $\QQQ(\Ng)^{\Gg}$) of a basis of $\Wcal(\Gg,\Lg,\Ng)$. Therefore, the challenge above is to construct  maps $\beta_1,\ldots,\beta_{\ell}\colon \ZZ\to [\Gg,\Ng]$ of the form above appropriately \emph{without knowing concrete forms of non-extendable invariant quasimorphisms}. Surprisingly, this can be done; we introduce the theory of core extractors for this purpose. Despite the fact that we only use the theory in the `abelian case' (see Section~\ref{sec=CEabel}) in the present paper, this theory itself works under  weaker assumptions. Hence, we introduce the theory in a general setting in this section, and state the specialized theory to the `abelian case' in Section~\ref{sec=CEabel}.

\subsection{Conditions on the tuple $(\Ff,\Kf,\Mf,\Rel,\ppi)$}
One of the main ideas in the theory of core extractors is to \emph{lift} the triple $(\Gg,\Lg,\Ng)$ to $(\Ff,\Kf,\Mf)$ that behaves good in terms of the extandability of invariant quasimorphisms. First, we discuss our conditions on the new triple $(\Ff,\Kf,\Mf)$ to develop the theory of core extractors. More precisely, we introduce the following notion.

\begin{defn}[tuple associated with $(\Gg,\Lg,\Ng)$]\label{defn=FKM}
Let $(\Gg,\Lg,\Ng)$ be a triple in Setting~$\ref{setting=GLN}$. We say a  tuple $(\Ff,\Kf, \Mf,\Rel,\ppi)$ is a \emph{tuple associated with $(\Gg,\Lg,\Ng)$} if it satisfies the following:
\begin{enumerate}[label=\textup{(}$\arabic*$\textup{)}]
  \item $\Ff$ is a group, and $\Kf$ and $\Mf$ are two normal subgroups of $\Ff$ with $\Kf\geqslant \Mf$;
  \item$\ppi\colon \Ff\twoheadrightarrow \Gg$ is a surjective group homomorphism, and $\Rel=\Ker(\ppi)$; and
  \item $\ppi(\Kf)=\Lg$ and $\ppi(\Mf)=\Ng$.
\end{enumerate}
\end{defn}

For a tuple $(\Ff,\Kf, \Mf,\Rel,\ppi)$ associated with $(\Gg,\Lg,\Ng)$, we define the following two conditions. Recall our convention that $i_{J,H}\colon J\hookrightarrow H$ denotes the inclusion map for $J\leqslant H$ (we have abbreviated $i_{J,H}$ as $i$, but in this section we use the symbol $i_{J,H}$ to make the pair $(J,H)$ explicit).

\begin{defn}[two conditions for the theory of core extractors]\label{defn=FKMcondition}
Assume Setting~\ref{setting=GLN}. Let $(\Ff,\Kf, \Mf,\Rel,\ppi)$ be a tuple associated with $(\Gg,\Lg,\Ng)$. We define the following conditions (i) and (ii)  on the tuple $(\Ff,\Kf, \Mf,\Rel,\ppi)$.
\begin{enumerate}[label=(\roman*)]
 \item $\Wcal(\Ff,\Kf,\Mf)=0$ and $\HHH^1(\Mf)^{\Ff}\cap i^{\ast}\QQQ(\Kf)^{\Ff}=i^{\ast}\HHH^1(\Kf)^{\Ff}$. Here $i$ is the inclusion map $i_{\Mf,\Kf}\colon \Mf\hookrightarrow \Kf$.
 \item $(M\cap R)/(M\cap R\cap [\Ff,\Kf])$ is a torsion group.
\end{enumerate}
\end{defn}

In Setting~\ref{setting=GLN}, given a tuple $(\Ff,\Kf,\Mf,\Rel,\ppi)$ associated with $(\Gg,\Lg,\Ng)$, the natural group quotient map $\pi\colon \Ff\to \Gg$ induces a homomorphism $\pi^{\ast}\colon \QQQ(\Ng)^{\Gg}\to \QQQ(\Mf)^{\Ff}$ (we should write $(\pi|_{\Mf})^{\ast}$; but we abuse notation). For $\nug\in \QQQ(\Ng)^{\Gg}$, we can consider $\muf=\ppi^{\ast}\nug(=\nug\circ\ppi)\in \QQQ(\Mf)^{\Ff}$. Under condition (i) in Definition~\ref{defn=FKMcondition}, this $\muf$ can be decomposed as $\muf=\hf+i^{\ast}_{\Mf,\Kf}\phf$, where $\hf\in \HHH^1(\Mf)^{\Ff}$ and $\phf\in \QQQ(\Kf)^{\Ff}$; this is a key to  our theory. This argument is illustrated by the following diagrams.

\begin{equation}\label{diagram=core}
 \xymatrix@C=36pt{
\Ff \ar@{->>}[d]_{\ppi} \ar@{}[r]|*{\geqslant}  & \Kf \ar@{}[r]|*{\geqslant} \ar@{->>}[d]_{\ppi}  & \Mf \ar@{->>}[d]_{\ppi}  \\
\Gg\cong \Ff/\Rel \ar@{}[r]|*{\geqslant}  & \Lg \ar@{}[r]|*{\geqslant}  & \Ng,
} \qquad \qquad
\xymatrix{
\QQQ(\Mf)^{\Ff}\ar@{}[r]|*{\ni} &  \muf=\ppi^{\ast}\nug \ar@{}[r]|*{=}& \hf+i^{\ast}_{\Mf,\Kf}\phf \\
\QQQ(\Ng)^{\Gg}\ar[u]_{\ppi^{\ast}}\ar@{}[r]|*{\ni}& \nug. \ar@{|->}[u]. &}
\end{equation}

\begin{rem}\label{rem=torsion}
Condition~(ii) in Definition~\ref{defn=FKMcondition} holds if $(M\cap R)/(M\cap R\cap [\Ff,\Kf])$ is trivial, in other words,
\begin{enumerate}
\item[(ii$'$)] $M\cap R\leqslant [\Ff,\Kf]$
\end{enumerate}
holds. Examples discussed in the present paper in fact satisfy (ii$'$). Nevertheless, we formulate with a weaker assumption (ii) than (ii$'$) for further applications of our theory in a forthcoming work.
\end{rem}

We furthermore formulate the following `\emph{descending condition}' (D).

\begin{defn}[descending condition (D)]\label{defn=descending}
Assume Setting~\ref{setting=GLN}. Let $(\Ff,\Kf, \Mf,\Rel,\ppi)$ be a tuple associated with $(\Gg,\Lg,\Ng)$. We say that $(\Ff,\Kf, \Mf,\Rel,\ppi)$ satisfies \emph{condition \textup{(D)}}, or the \emph{descending condition}, if the following holds: for every $\nu\in \QQQ(\Ng)^{\Gg}$, if $\pi^{\ast}\nu\in i_{\Mf,\Kf}^{\ast}\QQQ(\Kf)^{\Ff}$, then $\nu\in i_{\Ng,\Lg}^{\ast}\QQQ(\Lg)^{\Gg}$.
\end{defn}

Condition (D) morally states that for every $\nu\in \QQQ(\Ng)^{\Gg}$, if the `lift' $\pi^{\ast}\nu$ of $\nu$ is extendable to $\Kf$, then there exists an extension of $\pi^{\ast}\nu$ on $\Kf$ that \emph{descends to} a quasimorphism on the group quotient $\Lg$. In the theory of core extractors, these three conditions (i), (ii) in Definition~\ref{defn=FKMcondition} and (D) in Definition~\ref{defn=descending} are important. In Example~\ref{exa=surface}, we will exhibit examples of $(\Gg,\Ng)$ (which corresponds to the triple $(\Gg,\Gg,\Ng)$) for which there exists a tuple $(\Ff,\Kf,\Mf,\Rel,\ppi)$ satisfying conditions (i), (ii) and (D). In Lemma~\ref{lem=maintogeneral} (together with Theorem~\ref{thm=abelD}), we will show that if a  group triple $(\Gg,\Lg,\Ng)$ satisfies the assumptions of Theorem~\ref{mthm=main}, then there exists a tuple $(\Ff,\Kf,\Mf,\Rel,\ppi)$ satisfying conditions (i), (ii) and (D). Under these three conditions, we can define the core extractor as an \emph{injective} $\RR$-linear map; we will argue in Definition~\ref{defn=CoreE} and Theorem~\ref{thm=InjCE}. Theorem~\ref{thm=InjCE} is the upshot of the theory of core extractors.

In what follows, we provide the following criteria (Propositions~\ref{prop=onetwo} and \ref{prop=D}) of a tuple $(\Ff,\Kf,\Mf,\Rel,\ppi)$ for condition (i) in Definition~\ref{defn=FKMcondition} and the descending condition (D).

\begin{prop}[criterion for condition~\textup{(i)}]\label{prop=onetwo}
Let $\Ff$ be a group, and let $\Kf$ and $\Mf$ be two  normal subgroups of $\Ff$ satisfying $\Kf\geqslant \Mf$. Assume that the following three conditions are fulfilled:
\begin{enumerate}[label=\textup{(}$\arabic*$\textup{)}]
  \item $\HHH^2(\Ff)=0$;
  \item $\Ff/\Mf$ is boundedly $3$-acyclic; and
  \item $\Kf/\Mf$ is boundedly $2$-acyclic.
\end{enumerate}
Then,  we have $\Wcal(\Ff,\Kf,\Mf)=0$ and $\HHH^1(\Mf)^{\Ff}\cap i^{\ast}_{\Mf,\Kf}\QQQ(\Kf)^{\Ff}=i_{\Mf,\Kf}^{\ast}\HHH^1(\Kf)^{\Ff}$.
\end{prop}

\begin{proof}
By Theorem~\ref{thm=KKMMMmain2}, assumptions (1) and (2) imply that $\WW(\Ff,\Mf)=0$. Then, we obtain that $\Wcal(\Ff,\Kf,\Mf)=0$. By Proposition~\ref{prop=intersectiontriple}, assumption (3) implies that $\HHH^1(\Mf)^{\Ff}\cap i^{\ast}_{\Mf,\Kf}\QQQ(\Kf)^{\Ff}=i_{\Mf,\Kf}^{\ast}\HHH^1(\Kf)^{\Ff}$.
\end{proof}

The following special case of Proposition~\ref{prop=onetwo} is important in our theory; this indicates a \emph{general way} to take a lift $(\Ff,\Kf,\Mf)$ of $(\Gg,\Lg,\Ng)$ satisfying condition~(i) in Definition~\ref{defn=FKMcondition}, provided that $\QG=\Gg/\Ng$ is boundedly $3$-acyclic and that $\Lg/\Ng$ is boundedly $2$-acyclic.




\begin{cor}[lifting to a free group]\label{cor=free}
Assume Setting~$\ref{setting=GLN}$. Assume that $\QG=\Gg/\Ng$ is boundedly $3$-acyclic and that $\Lg/\Ng$ is boundedly $2$-acyclic. Take an arbitrary \emph{free group} $\Ff$ such that there exists a surjective group homomorphism $\ppi\colon \Ff\twoheadrightarrow \Gg$, and set $\Rel=\Ker(\ppi)$. Set $\Kf=\ppi^{-1}(\Lg)$ and  $\Mf=\ppi^{-1}(\Ng)$. Then, the tuple $(\Ff,\Kf,\Mf,\Rel,\ppi)$ satisfies condition \textup{(i)} in Definition~$\ref{defn=FKMcondition}$.
\end{cor}

\begin{proof}
In the construction of the tuple $(\Ff,\Kf,\Mf,\Rel,\ppi)$, we have isomorphisms $\Ff/\Mf\cong \QG$ and $\Kf/\Mf \cong \Lg/\Ng$ of groups. Since $\Ff$ is a free group, $\HHH^2(\Ff)=0$. Hence, Proposition~\ref{prop=onetwo} applies to this setting.
\end{proof}

Corollary~\ref{cor=free} is the background of our use of the symbols $F$ and $R$ in the tuple $(\Ff,\Kf,\Mf,\Rel,\ppi)$: $F$ comes from a \emph{free} group, and $R$ comes from the \emph{relations}.

\begin{prop}[criterion for condition~\textup{(D)}]\label{prop=D}
Assume Setting~$\ref{setting=GLN}$. Let $(\Ff,\Kf, \Mf,\Rel,\ppi)$ be a tuple associated with $(\Gg,\Lg,\Ng)$. Assume the following condition:
\begin{enumerate}
 \item[\textup{(iii)}] $(K\cap R)/(M\cap R)$ is a torsion group.
\end{enumerate}
Then, the tuple $(\Ff,\Kf,\Mf,\Rel,\ppi)$ satisfies condition \textup{(D)}.
\end{prop}

For the proof of Proposition~\ref{prop=D}, we employ the following lemma; we include the proof of it for the reader's convenience.

\begin{lem}\label{lem=descending}
Let $\Ff$ be a group and $\Kf$  a normal subgroup of $\Ff$. Let $\Gg$ be a group quotient of $\Ff$, and $\ppi\colon \Ff\twoheadrightarrow \Gg$ be the natural group quotient map. Let $\Rel=\Ker(\ppi)$ and $\Lg=\ppi(\Kf)$. Let $\phf\in \QQQ(\Kf)^{\Ff}$. Then $\phf$ descends to a $\Gg$-invariant quasimorphism on $\Lg$ if and only if $\phf$ vanishes on $\Kf\cap \Rel$.
\end{lem}

\begin{proof}
The `only if' part is clear. In what follows, we prove the `if' part. We claim that for every $\vf_1,\vf_2\in \Kf$ with $\vf_1^{-1}\vf_2\in \Kf\cap \Rel$, $\phf(\vf_1)=\phf(\vf_2)$ holds. Indeed, for every $m\in \NN$, we have
\[
|\phf(\vf_1)-\phf(\vf_2)|=\frac{|\phf(\vf_1^m)-\phf(\vf_2^m)|}{m}\leq \frac{\DD(\phf)}{m}
\]
since $\phf|_{\Kf\cap \Rel}\equiv 0$. Therefore, the map $\phf\colon \Kf\to \RR$ (set-theoretically) descends to a map $\psg\colon \Lg=\Kf/(\Kf\cap \Rel)\to \RR$, \emph{i.e.}, $\ppi^{\ast}\psg=\phf$ holds. It is clear that the resulting map $\psg$ lies in $\QQQ(\Lg)^{\Gg}$.
\end{proof}

\begin{proof}[Proof of Proposition~$\ref{prop=D}$]
Let $\nug\in \QQQ(\Ng)^{\Gg}$ with $\ppi^{\ast}\nug\in i_{\Mf,\Kf}^{\ast}\QQQ(\Kf)^{\Ff}$. Take $\phf\in \QQQ(\Kf)^{\Ff}$ such that $\pi^{\ast}\nug=i_{\Mf,\Kf}^{\ast}\phf$. We claim that $\phf$ vanishes on $\Kf\cap \Rel$. To see this, first we observe that $\phf$ vanishes on $\Mf\cap \Rel$: note that $\ppi(\Mf\cap \Rel)=\{e_{\Gg}\}$ and that $\phf|_M\equiv \pi^{\ast}\nug$. By condition (iii) and homogeneity of $\phf$, this observation yields that $\phf|_{\Kf\cap \Rel}\equiv 0$. By Lemma~\ref{lem=descending}, there exists $\psg\in \QQQ(\Lg)^{\Gg}$ such that $\ppi^{\ast}\psg=\phf$. Hence, we have
\[
\ppi^{\ast}\nug=(i_{\Mf,\Kf}^{\ast}\circ \ppi^{\ast})\psg
=(\ppi^{\ast}\circ i_{\Ng,\Lg}^{\ast})\psg=\ppi^{\ast}(i_{\Ng,\Lg}^{\ast}\psg).\]
This implies that $\nug=i_{\Ng,\Lg}^{\ast}\psg \in i_{\Ng,\Lg}^{\ast}\QQQ(\Lg)^{\Gg}$. Therefore, the tuple $(\Ff,\Kf,\Mf,\Rel,\ppi)$ satisfies (D).
\end{proof}


\subsection{Core and core-surviving elements}

In the setting of diagram~\eqref{diagram=core}, we plan to focus on `the' invariant homomorphism $\hf$ in $\HHH^1(\Mf)^{\Ff}$ and to call it as `the' \emph{core} of $\nug\in \QQQ(\Ng)^{\Gg}$. However, `the' decomposition $\ppi^{\ast}\nug=\hf+i^{\ast}_{\Mf,\Kf}\phf$ is \emph{not} unique in general. To make this plan meaningful, we start from the rigorous definition of \emph{a} core of $\nug\in \QQQ(\Ng)^{\Gg}$ in a general setting, even without assuming condition (i) in Definition~\ref{defn=FKMcondition}.

\begin{defn}[core]\label{defn=core}
Assume Setting~$\ref{setting=GLN}$. Fix a tuple $(\Ff,\Kf,\Mf,\Rel,\ppi)$ associated with $(\Gg,\Lg,\Ng)$. Let $\nug\in\QQQ(\Ng)^{\Gg}$. An invariant \emph{homomorphism} $\hf\in \HHH^1(\Mf)^{\Ff}$ is called a \emph{core} of $\nug$ \textup{(}with respect to the tuple $(\Ff,\Kf,\Mf,\Rel,\ppi)$\textup{)} if
\[
\ppi^{\ast}\nug-\hf\in i^{\ast}_{\Mf,\Kf}\QQQ(\Kf)^{\Ff}.
\]
\end{defn}

\begin{rem}
In general, a core of $\nug \in \QQQ(\Ng)^{\Gg}$ may not exist. Even if it exists, the uniqueness of cores \emph{may fail} (this is why we call it \emph{a} core, not the core). Also, this notion of cores \emph{does} depend on the choices of a tuple $(\Ff,\Kf,\Mf,\Rel,\ppi)$ for $(\Gg,\Lg,\Ng)$.
\end{rem}

The following lemma indicates the role of condition (i) in Definition~\ref{defn=FKMcondition}.

\begin{lem}\label{lem=(i)}
Assume Setting~$\ref{setting=GLN}$. Let $(\Ff,\Kf,\Mf,\Rel,\ppi)$ be a tuple associated with $(\Gg,\Lg,\Ng)$. Assume that $(\Ff,\Kf,\Mf,\Rel,\ppi)$ satisfies  condition \textup{(i)} in Definition~$\ref{defn=FKMcondition}$. Then for every $\nug\in  \QQQ(\Ng)^{\Gg}$, the following hold true.
\begin{enumerate}[label=\textup{(}$\arabic*$\textup{)}]
 \item  There exists a core of $\nug$ \textup{(}with respect to the tuple $(\Ff,\Kf,\Mf,\Rel,\ppi)$\textup{)}.
 \item Let $\hf_1,\hf_2\in \HHH^1(\Mf)^{\Ff}$ be two cores of $\nug$. Then, $\hf_1-\hf_2\in i^{\ast}_{\Mf,\Kf}\HHH^1(\Kf)^{\Ff}$. In particular, for every $\wf\in \Mf\cap [\Ff,\Kf]$, $\hf_1(\wf)=\hf_2(\wf)$ holds.
\end{enumerate}
\end{lem}

Lemma~\ref{lem=(i)}~(2) states that if $(\Ff,\Kf,\Mf,\Rel,\ppi)$ satisfies  condition (i), then the value of an element in $\Mf\cap [\Ff,\Kf]$ does \emph{not} depend on the choices of a core of $\nug$. Under the assumption $\Mf\geqslant [\Ff,\Kf]$, we will present in  Theorem~\ref{thm=valueofcore} a \emph{limit formula} for the value $\hf(\wf)$ of a core $\hf$ of $\nug\in \QQQ(\Ng)^{\Gg}$ at $\wf\in [\Ff,\Kf]$.

\begin{proof}[Proof of Lemma~$\ref{lem=(i)}$]
Item~(1) holds since the first half of condition~(i) states that $\Wcal(\Ff,\Kf,\Mf)=0$. For (2), note that $\hf_1-\hf_2\in \HHH^1(\Mf)^{\Ff}\cap i_{\Mf,\Kf}^{\ast}\QQQ(\Kf)^{\Ff}$. Now the latter half of condition (i) ends the proof (observe also that for every $\tilde{\hf}\in \HHH^1(\Kf)^{\Ff}$, $\tilde{\hf}$ vanishes on $[\Ff,\Kf]$).
\end{proof}

The following corollary to Lemma~\ref{lem=(i)} is important in our theory of core extractors.

\begin{cor}\label{cor=torsion}
Assume Setting~$\ref{setting=GLN}$. Fix a tuple $(\Ff,\Kf,\Mf,\Rel,\ppi)$ associated with $(\Gg,\Lg,\Ng)$. Assume that the tuple $(\Ff,\Kf,\Mf,\Rel,\ppi)$ satisfies condition \textup{(i)} in Definition~$\ref{defn=FKMcondition}$. Let $\nug\in \QQQ(\Ng)^{\Gg}$. Assume that $\wf\in \Mf$ represents a torsion element in the quotient group $\Mf/(\Mf\cap [\Ff,\Kf])$. Then the value $\hf(\wf)$ of a core $\hf$ of $\nug$ is independent of the choices of $\hf$.
\end{cor}

\begin{proof}
First by Lemma~\ref{lem=(i)}~(1), there exists a core $\hf_1$ of $\nug$. Take an arbitrary core $\hf$ of $\nug$. Secondly, by assumption  we can take $m\in \NN$ such that $\wf^m\in \Mf \cap [\Ff,\Kf]$. Then by Lemma~\ref{lem=(i)}~(2),
\[
\hf(\wf)=\frac{1}{m}\hf(\wf^m)=\frac{1}{m}\hf_1(\wf^m)=\hf_1(\wf).\qedhere
\]
\end{proof}

\begin{defn}[core-surviving element]\label{defn=surviving}
Assume Setting~$\ref{setting=GLN}$. Fix a tuple $(\Ff,\Kf,\Mf,\Rel,\ppi)$ associated with $(\Gg,\Lg,\Ng)$. Assume that the tuple $(\Ff,\Kf,\Mf,\Rel,\ppi)$ satisfies condition (i) in Definition~\ref{defn=FKMcondition}. Let $\nug\in \QQQ(\Ng)^{\Gg}$. We say that $\wf\in \Mf$ is a \emph{$\nug$-core-surviving element} \textup{(}with respect to the tuple $(\Ff,\Kf,\Mf,\Rel,\ppi)$\textup{)} if the following two conditions are satisfied:
\begin{enumerate}[label=\textup{(}$\arabic*$\textup{)}]
  \item the element $\wf$ represents a torsion element in the quotient group $\Mf/(\Mf\cap [\Ff,\Kf])$; and
  \item $\hf(\wf)\ne 0$ for some (equivalently, for every) core $\hf$ of $\nug$.
\end{enumerate}
\end{defn}

Here, the equivalence mentioned in Definition~\ref{defn=surviving}~(2) follows from by Corollary~\ref{cor=torsion}.
The following lemma describes the role of condition (ii) in Definition~\ref{defn=FKMcondition}.

\begin{lem}\label{lem=(ii)}
Assume Setting~$\ref{setting=GLN}$. Fix a tuple $(\Ff,\Kf,\Mf,\Rel,\ppi)$ associated with $(\Gg,\Lg,\Ng)$. Assume that the tuple $(\Ff,\Kf,\Mf,\Rel,\ppi)$ satisfies conditions \textup{(i)} and \textup{(ii)} in Definition~$\ref{defn=FKMcondition}$. Let $\nug\in \QQQ(\Ng)^{\Gg}$. Then, the following are all equivalent.
\begin{enumerate}[label=\textup{(\arabic*)}]
 \item There does not exist any $\nug$-core-surviving element in $\Mf\cap \Rel$. \item For every core $\hf$ of $\nug$, $\hf|_{\Mf\cap \Rel}\equiv 0$.
 \item There exists a core $\hf$ of $\nug$ such that $\hf|_{\Mf\cap \Rel}\equiv 0$.
\end{enumerate}
\end{lem}

\begin{proof}
By (ii), $\wf\in \Mf\cap \Rel$ is a $\nug$-core-surviving element if and only if $\hf(\wf)\ne 0$ for some core $\hf$ of $\nug$; by Corollary~\ref{cor=torsion}, it is also equivalent to $\hf(\wf)\ne 0$ for every core $\hf$ of $\nug$. Now the equivalences of (1), (2) and (3) immediately follow.
\end{proof}

\begin{lem}\label{lem=CoreE}
Assume Setting~$\ref{setting=GLN}$. Fix a tuple $(\Ff,\Kf,\Mf,\Rel,\ppi)$ associated with $(\Gg,\Lg,\Ng)$. Assume that the tuple $(\Ff,\Kf,\Mf,\Rel,\ppi)$ satisfies condition \textup{(i)} in Definition~$\ref{defn=FKMcondition}$. Then, the following hold.
\begin{enumerate}[label=\textup{(}$\arabic*$\textup{)}]
  \item If there exists a $\nug$-core-surviving element $\wf$ in $\Mf$, then $\nug$ represents a non-zero element in the space $\QQQ(\Ng)^{\Gg}/i^{\ast}\QQQ(\Lg)^{\Gg}$.
  \item If there exists a $\nug$-core-surviving element $\wf$ in $\Mf\cap \Rel$, then $\nug$ represents a non-zero element in the space $\Wcal(\Gg,\Lg,\Ng)$.
\end{enumerate}
\end{lem}

\begin{proof}
We prove the contrapositives of these assertions. For (2), assume that $\nug$ represents the zero element in $\Wcal(\Gg,\Lg,\Ng)$. Then we can take $\kg\in \HHH^1(\Ng)^{\Gg}$ and $\psg\in \QQQ(\Lg)^{\Gg}$ such that $\nug=\kg+i_{\Ng,\Lg}^{\ast}\psg$. Then we have $\pi^{\ast}\nug=\pi^{\ast}\kg+ i_{\Mf,\Kf}^{\ast}(\pi^{\ast}\psg)$. Hence $\pi^{\ast}\kg$ is a core of $\nug$. Since $\pi^{\ast}\kg|_{\Mf\cap \Rel}\equiv 0$, we complete the proof of (2). We can prove (1) by an  argument similar to one above.
\end{proof}

Under the descending condition (D) (recall Definition~\ref{defn=descending}), we have the following equivalence.

\begin{prop}[characterization of non-vanishing in $\Wcal(\Gg,\Lg,\Ng)$ in terms of core-surviving elements]\label{prop=CoreE}
Assume Setting~$\ref{setting=GLN}$. Fix a tuple $(\Ff,\Kf,\Mf,\Rel,\ppi)$ associated with $(\Gg,\Lg,\Ng)$. Assume that the tuple $(\Ff,\Kf,\Mf,\Rel,\ppi)$ satisfies conditions \textup{(i)} and \textup{(ii)} in Definition~$\ref{defn=FKMcondition}$ and the descending condition \textup{(D)}. Let $\nug\in \QQQ(\Ng)^{\Gg}$. Then, $\nug$ represents a non-zero element in $\Wcal(\Gg,\Lg,\Ng)$ if and only if there exists a $\nug$-core-surviving element $\wf$ in $\Mf\cap \Rel$.
\end{prop}

\begin{proof}
The `if' part is proved in Lemma~\ref{lem=CoreE}. In what follows, we verify the `only if' part by showing the contrapositive. Assume that there does not exist any $\nug$-core-surviving element $\wf$ in $\Mf\cap \Rel$. Take a core $\hf$ of $\nug$ and $\phf\in \QQQ(\Kf)^{\Ff}$ such that $\pi^{\ast}\nug=\hf+i_{\Mf,\Kf}^{\ast}\phf$. Then by Lemma~\ref{lem=(ii)}, $\hf|_{\Mf\cap \Rel}\equiv 0$. Hence, $\hf$ descends to a invariant homomorphism $\kg\in \HHH^1(\Ng)^{\Gg}$ (namely, $\pi^{\ast}\kg=\hf$). Then we have
\[
\pi^{\ast}(\nug-\kg)=i_{\Mf,\Kf}^{\ast}\phf\in i_{\Mf,\Kf}^{\ast}\QQQ(\Kf)^{\Ff}.
\]
By applying the descending condition (D) to $\nug-\kg$, we have $\nug-\kg\in i_{\Ng,\Lg}^{\ast}\QQQ(\Lg)^{\Gg}$. Therefore, $\nug\in \HHH^1(\Ng)^{\Gg}+i_{\Ng,\Lg}^{\ast}\QQQ(\Lg)^{\Gg}$. This completes our proof of the contrapositive.
\end{proof}

In the `abelian case,' we have the equivalence in Lemma~\ref{lem=CoreE}~(1); see Theorem~\ref{thm=abeliancore} for details.

\subsection{Core extractor}

By Corollary~\ref{cor=torsion} and Lemma~\ref{lem=CoreE}, we can define the following $\RR$-linear map, which we call the \emph{core extractor}.

\begin{defn}[core extractor]\label{defn=CoreE}
Assume Setting~\ref{setting=GLN}. Fix a tuple $(\Ff,\Kf,\Mf,\Rel,\ppi)$ associated with $(\Gg,\Lg,\Ng)$. Assume that the tuple $(\Ff,\Kf,\Mf,\Rel,\ppi)$ satisfies conditions (i) and (ii) in Definition~\ref{defn=FKMcondition}.
\begin{enumerate}[label=\textup{(}$\arabic*$\textup{)}]
 \item Define a well-defined $\RR$-linear map
\[
\CEt\colon \QQQ(\Ng)^{\Gg}\to i_{\Mf\cap \Rel,\Mf}^{\ast}\HHH^1(\Mf)^{\Ff}
\]
in the following manner: for every $\nug \in \QQQ(\Ng)^{\Gg}$, take a core $\hf$ of $\nug$ and set $\CEt(\nug)=\hf|_{\Mf\cap \Rel}$.
 \item The map $\CEt\colon \QQQ(\Ng)^{\Gg}\to i_{\Mf\cap \Rel,\Mf}^{\ast}\HHH^1(\Mf)^{\Ff}$ induces an $\RR$-linear map
\[
\CE\colon \Wcal(\Gg,\Lg,\Ng)\to i_{\Mf\cap \Rel,\Mf}^{\ast}\HHH^1(\Mf)^{\Ff};\quad  [\nug]\mapsto \hf|_{\Mf\cap \Rel},
\]
where  $\hf$ is a core of $\nug$. We call this map $\CE$ the \emph{core extractor} with respect to the tuple $(\Ff,\Kf,\Mf,\Rel,\ppi)$.
\end{enumerate}
\end{defn}

Indeed, by Corollary~\ref{cor=torsion}, $\hf|_{\Mf\cap \Rel}$ is independent of the choices of a core $\hf$ of $\nug\in \QQQ(\Ng)^{\Gg}$. By Lemma~\ref{lem=CoreE}~(2), we have
\[
\Ker(\CEt)\supseteq \HHH^1(\Ng)^{\Gg}+i_{\Ng,\Lg}^{\ast}\QQQ(\Lg)^{\Gg}
\]
so that $\CEt$ induces the core extractor $\CE$.

The following theorem explains the importance of the descending condition (D) in the theory of core extractors.

\begin{thm}[injectivity of the core extractor]\label{thm=InjCE}
Assume Setting~$\ref{setting=GLN}$. Fix a tuple $(\Ff,\Kf,\Mf,\Rel,\ppi)$ associated with $(\Gg,\Lg,\Ng)$. Assume that the tuple $(\Ff,\Kf,\Mf,\Rel,\ppi)$ satisfies conditions \textup{(i)} and \textup{(ii)} in Definition~$\ref{defn=FKMcondition}$. Assume furthermore that the tuple $(\Ff,\Kf,\Mf,\Rel,\ppi)$ satisfies the descending condition \textup{(D)}. Then the core extractor $\CE\colon \Wcal(\Gg,\Lg,\Ng)\to i_{\Mf\cap \Rel,\Mf}^{\ast}\HHH^1(\Mf)^{\Ff}$ is \emph{injective}.
\end{thm}

\begin{proof}
Under (D), Proposition~\ref{prop=CoreE} states that $\Ker(\CEt)= \HHH^1(\Ng)^{\Gg}+i_{\Ng,\Lg}^{\ast}\QQQ(\Lg)^{\Gg}$.
\end{proof}

With the aid of Lemma~\ref{lem=dual}, the core extractor indicates a way of taking `good' tuples $(\wf_1,\ldots,\wf_{\ell})$ and $(\nug_1,\ldots,\nug_{\ell})$ to construct the map $\PPs$ in Theorem~\ref{mthm=main}.

\begin{thm}[outcome of the theory of core extractors]\label{thm=Goodtuple}
Assume Setting~$\ref{setting=GLN}$. Assume that $\Wcal(\Gg,\Lg,\Ng)\ne 0$ and let $\ell\in \NN$ satisfy $\ell\leq \Rdim \Wcal(\Gg,\Lg,\Ng)$. Fix a tuple $(\Ff,\Kf,\Mf,\Rel,\ppi)$ associated with $(\Gg,\Lg,\Ng)$. Assume that the tuple $(\Ff,\Kf,\Mf,\Rel,\ppi)$ satisfies conditions \textup{(i)} and \textup{(ii)} in Definition~$\ref{defn=FKMcondition}$ and the descending condition \textup{(D)}. Then there exist $\wf_1,\ldots \wf_{\ell}\in \Mf\cap \Rel\cap [\Ff,\Kf]$ and $\nug_1,\ldots,\nug_{\ell}\in \QQQ(\Ng)^{\Gg}$ such that for every $i,j\in \{1,\ldots,\ell\}$,
\[
\CEt_{\nug_{j}}(\wf_i)=\delta_{i,j}.
\]
Here, for $\nug\in \QQQ(\Ng)^{\Gg}$ we write $\CEt_{\nug}$ for $\CEt(\nug)$.
\end{thm}

\begin{proof}
By Theorem~\ref{thm=InjCE}, the image $\CE(\Wcal(\Gg,\Lg,\Ng))$ of $\Wcal(\Gg,\Lg,\Ng)$ under the core extractor $\CE$ is a real vector space of dimension at least $\ell$. By applying Lemma~\ref{lem=dual}, we obtain $\wf'_1,\ldots \wf'_{\ell}\in \Mf\cap \Rel$ and $\nug'_1,\ldots,\nug'_{\ell}\in \QQQ(\Ng)^{\Gg}$ such that for every $i,j\in \{1,\ldots,\ell\}$, $\CEt_{\nug'_{j}}(\wf'_i)=\delta_{i,j}$ holds.
By condition (ii), there exist $m_1,\ldots,m_{\ell}\in \NN$ such that for every $i\in \{1,\ldots,\ell\}$, $w_i'^{m_i}\in \Mf\cap \Rel\cap [\Ff,\Kf]$ holds. Let $m$ be the least common multiple of $m_1,\ldots,m_{\ell}$. Finally, for every $i,j\in \{1,\ldots,\ell\}$, set $w_i=w_i'^{m}$ and $\nug_j=(1/m)\nug'_j$. Then these $(w_1,\ldots,w_{\ell})$ and $(\nug_1,\ldots,\nug_{\ell})$ satisfy the desired equalities.
\end{proof}

\subsection{Explicit elements in Theorem~\ref{thm=Goodtuple}}\label{subsec=explicit}

Under a certain condition, we can take an \emph{explicit} tuple  $(\wf_1,\ldots,\wf_{\ell})$ in Theorem~\ref{thm=Goodtuple}.

\begin{thm}[explicit elements in Theorem~\ref{thm=Goodtuple}]\label{thm=explicitkernel}
Assume Setting~$\ref{setting=GLN}$. Let $\ell\in \NN$. Assume that $\Rdim \Wcal(\Gg,\Lg,\Ng)\geq \ell$. Let $(\Ff,\Kf,\Mf,\Rel,\ppi)$ be a tuple associated with $(\Gg,\Lg,\Ng)$. Assume that $(\Ff,\Kf,\Mf,\Rel,\ppi)$ satisfies conditions \textup{(i)} and \textup{(ii)} in Definition~$\ref{defn=FKMcondition}$ and the descending condition \textup{(D)}. Assume moreover that there exist $\ell$ elements $r_1,\ldots,r_{\ell}$ in $\Mf\cap \Rel$ such that $r_1,\ldots ,r_{\ell}$ normally generate $\Mf\cap \Rel$ in $\Ff$. Then, $\Rdim \Wcal(\Gg,\Lg,\Ng)= \ell$, and there exist $\nug_1,\ldots,\nug_{\ell}\in \QQQ(\Ng)^{\Gg}$ such that for every $i,j\in \{1,\ldots,\ell\}$,
\begin{equation}\label{eq=deltar}
\CEt_{\nug_{j}}(r_i)=\delta_{i,j}
\end{equation}
holds.
\end{thm}

\begin{proof}
By Theorem~\ref{thm=Goodtuple}, we obtain $\wf_1,\ldots,\wf_{\ell}\in \Mf\cap \Rel$ and $\nug'_1,\ldots,\nug'_{\ell}\in \QQQ(\Ng)^{\Gg}$ such that for every $i,j\in \{1,\ldots,\ell\}$, $\CEt_{\nug'_{j}}(\wf_i)=\delta_{i,j}$ holds. Let $i\in \{1,\ldots,\ell\}$. Since $\wf_i\in \Mf \cap \Rel$, we can fix an expression of $\wf_i$ as the product of some conjugates in $F$ of $r_1^{\pm},\ldots,r_{\ell}^{\pm}$. For every $j\in \{1,\ldots ,\ell\}$, set $m_{i,j}\in \ZZ$ as the difference between the number of conjugates of $r_j$ and that of conjugates of $r_j^{-1}$ in the fixed expression of $\wf_i$. Then since $\CEt_{\nug'_{1}},\ldots,\CEt_{\nug'_{\ell}}\in \HHH^1(\Mf \cap \Rel)^{\Ff}$, we have the following matrix equality:
\[
[m_{i,j}]_{i,j} [\CEt_{\nug'_{j}}(r_i)]_{i,j}=I_{\ell},
\]
where $I_{\ell}$ denotes the identity matrix of degree $\ell$. Hence, $[\CEt_{\nug'_{j}}(r_i)]_{i,j}=([m_{i,j}]_{i,j})^{-1}$ so that  $[\CEt_{\nug'_{j}}(r_i)]_{i,j}[m_{i,j}]_{i,j} =I_{\ell}$. For every $j\in \{1,\ldots,\ell\}$, set
\[
\nug_{j}=\sum_{s\in\{1,\ldots,\ell\}}m_{j,s}\nug'_{s}.
\]
Then, we have \eqref{eq=deltar}. This argument will also draw a contradiction if we assume that $\Rdim \Wcal(\Gg,\Lg,\Ng) > \ell$.
\end{proof}

\begin{cor}\label{cor=explicitkernel}
Assume Setting~$\ref{setting=GLN}$. Let $\ell\in \NN$ and $\Rdim \Wcal(\Gg,\Lg,\Ng)\geq \ell$. Assume that $\QG=\Gg/\Ng$ is boundedly $3$-acyclic and that $\Lg/\Ng$ is boundedly $2$-acyclic. Assume furthermore that $\Gg$ admits a presentation by $\ell$ relators: more precisely, assume that there exist a free group $F$, a surjective group homomorphism $\ppi\colon F\twoheadrightarrow \Gg$ and $r_1,\ldots,r_{\ell}\in F$ such that $R=\Ker(\ppi)$ is normally generated by $r_1,\ldots,r_{\ell}$ in $F$. Set $\Kf=\ppi^{-1}(\Lg)$ and $\Mf=\ppi^{-1}(\Ng)$. Assume that $R/(R\cap [F,\Kf])$ is a torsion group. Then $\Rdim \Wcal(\Gg,\Lg,\Ng)= \ell$, and there exist $\nug_1,\ldots,\nug_{\ell}\in \QQQ(\Ng)^{\Gg}$ such that for every $i,j\in \{1,\ldots,\ell\}$, $\eqref{eq=deltar}$ holds.
\end{cor}

\begin{proof}
By Corollary~\ref{cor=free}, the tuple $(F,\Kf,\Mf,R,\ppi)$ satisfies condition (i) in Definition~\ref{defn=FKMcondition}. Since $R/(R\cap [F,\Kf])$ is a torsion group, this tuple satisfies condition (ii) in Definition~\ref{defn=FKMcondition}. Since $R\leqslant \Mf$ by construction, condition (iii) in Proposition~\ref{prop=D} is fulfilled for this tuple, and  hence (D) holds. Therefore, Theorem~\ref{thm=explicitkernel} applies to this situation.
\end{proof}

\begin{exa}[examples in surface groups]\label{exa=surface}
Let $\genus\in \NN_{\geq 2}$. Then, the surface group $\pi_1(\Sigma_{\genus})$ admits the following presentation $(F\,|\,R)$ (recall Definition~\ref{defn=pres} for our formulation of group representations): let $F=F(\tilde{a}_1,\ldots,\tilde{a}_{\genus},\tilde{b}_1,\ldots,\tilde{b}_{\genus})$ be the free group with free basis $(\tilde{a}_1,\ldots,\tilde{a}_{\genus},\tilde{b}_1,\ldots,\tilde{b}_{\genus})$ and $R$ the normal closure in $F$ of $r=[\tilde{a}_1,\tilde{b}_1]\cdots [\tilde{a}_{\genus},\tilde{b}_{\genus}] (\in [F,F])$. Let $\Ng_{\genus}$ be a normal subgroup of $\pi_1(\Sigma_{\genus})$ with $\Ng_{\genus}\geqslant [\pi_1(\Sigma_{\genus}),\pi_1(\Sigma_{\genus})]$ and $\WW(\pi_1(\Sigma_{\genus}),\Ng_{\genus})\ne 0$: we can take for instance $\Ng_{\genus}=\gamma_2(\pi_1(\Sigma_{\genus}))$ by Theorem~\ref{thm=WW}~(2). Then, by Corollary~\ref{cor=explicitkernel} (with $\Lg=\Gg$) there exists $\nug\in \QQQ(\Ng_{\genus})^{\pi_1(\Sigma_{\genus})}$ such that $\CEt_{\nug}(r)=1$. In Example~\ref{exa=surfacecup}~(2), we will exhibit an example of such $\Ng_{\genus}$ other than $\gamma_2(\pi_1(\Sigma_{\genus}))$.

As we mentioned at the beginning of this section, in \cite[Section~6]{MMM} a concrete representative of a generator $\nug'$ of $\WW(\pi_1(\Sigma_{\genus}),\gamma_2(\pi_1(\Sigma_{\genus})))$ was provided; that generator satisfies $\CEt_{\nug'}(r)=1$. The construction of $\nug'$ stems from an elaborate study of surface group actions on the circle (see \cite[Section~3]{MMM} for details). Our emphasis here is that the theory of core extractors may ensure the existence of certain `interesting' invariant quasimorphisms $\nug\in \QQQ(\Ng)^{\Gg}$ in Setting~\ref{setting=GLN}  \emph{merely by algebraic conditions on $(\Gg,\Lg,\Ng)$, a group presentation of $\Gg$ and dimension computation of $\Wcal(\Gg,\Lg,\Ng)$}.
\end{exa}

\section{Core extractor in the abelian case}\label{sec=CEabel}

As we discuss in Section~\ref{sec=CE}, our theory of core extractors works well for a triple $(\Gg,\Lg,\Ng)$ in Setting~\ref{setting=GLN} if there exists a tuple $(\Ff,\Kf,\Mf,\Rel,\ppi)$ associated with $(\Gg,\Lg,\Ng)$ that satisfies conditions (i) and (ii) in Definition~\ref{defn=FKMcondition} and the descending condition (D) (recall Definition~\ref{defn=descending}). Proposition~\ref{prop=D} provides a criterion (iii) for (D). In this section, we treat the case where $\Ng\geqslant [\Gg,\Lg]$, and discuss (D) and core extractors in this setting. We call this case the `\emph{abelian case}' because it exactly corresponds to the case where $\QG=\Gg/\Ng$ is abelian if $\Lg=\Gg$. (For a general $\Lg$, the terminology of `central case' might be more appropriate. However, we will study for instance `nilpotent case' in our forthcoming work, and the terminology of `abelian case' fits better in this context.) We will prove the following theorem, which state that the descending condition (D) is automatically fulfilled in the `abelian case.'

\begin{thm}[condition (D) in the abelian case]\label{thm=abelD}
Assume Setting~$\ref{setting=GLN}$. Fix a tuple $(\Ff,\Kf,\Mf,\Rel,\ppi)$ associated with $(\Gg,\Lg,\Ng)$. Assume furthermore that
\begin{equation}\label{eq=abelian}
\Mf\geqslant [\Ff,\Kf].
\end{equation}
Then, the tuple $(\Ff,\Kf,\Mf,\Rel,\ppi)$ satisfies the descending condition \textup{(D)}.
\end{thm}

We note that under \eqref{eq=abelian}, $\Mf\cap [\Ff,\Kf]=[\Ff,\Kf]$ holds.

We will prove Theorem~\ref{thm=abelD} in Subsection~\ref{subsec=keycor}. There, we  deduce Corollary~\ref{cor=InjCE} from Theorem~\ref{thm=abelD}. Corollary~\ref{cor=InjCE} is the \emph{upshot} of the theory that we have been introducing and developing in Sections~\ref{sec=CE} and \ref{sec=CEabel}. This corollary plays a key role in the proof of Theorem~\ref{mthm=main}; as we mentioned in Subsection~\ref{subsec=org}, we employ it in the construction of the map $\PPs\colon (\ZZ^{\ell},\|\cdot\|_1)\to ([\Gg,\Ng],d_{\scl_{\Gg,\Ng}})$ (see Theorem~\ref{thm=PPs} for the concrete construction of $\PPs$).

\begin{rem}\label{rem=normal}
Let $\Gg$ be a group and $\Ng$ a normal subgroup of $\Gg$. If a subgroup $\Lg\leqslant \Gg$ satisfies that $\Lg\geqslant \Ng$ and $\Ng\geqslant [\Gg,\Lg]$, then $\Lg$ is normal in $\Gg$.
\end{rem}

\subsection{The map $\alpha_{\bff}$}

To show Theorem~\ref{thm=abelD}, we study the value $\hf(\wf)$ where $\wf\in [\Ff,\Kf]$ and $\hf$ is a core of $\nug\in \QQQ(\Ng)^{\Gg}$ in the abelian case (namely, under \eqref{eq=abelian}). We employ the following definitions.

\begin{defn}[{$[\Ff,\Kf]$-expression}]\label{defn=tuple}
Let $\Ff$ be a group, and let $\Kf$ and $\Mf$ be two normal subgroups of $\Ff$ with $\Kf\geqslant \Mf$. Assume that $\Mf\geqslant [\Ff,\Kf]$. For $\wf\in [\Ff,\Kf](\leqslant \Mf)$, we say that $\bff=(\ff_1,\ldots ,\ff_t;\ff'_1,\ldots,\ff'_t)$ is an \emph{$[\Ff,\Kf]$-expression} for $\wf$ if the following conditions are satisfied:
\begin{enumerate}[label=\textup{(}$\arabic*$\textup{)}]
 \item $t\in \NN$;
 \item for every $i\in \{1,\ldots,t\}$, $\ff_i\in \Ff$ and $\ff'_i\in \Kf$; and
 \item $\wf=[\ff_1,\ff'_1]\cdots [\ff_t,\ff'_t]$.
\end{enumerate}
\end{defn}

For an $[\Ff,\Kf]$-expression $\bff$ for $\wf\in [\Ff,\Kf]$, we define a map $\alpha_{\bff}\colon \ZZ\to [\Ff,\Kf]$ in the following manner.

\begin{defn}[the map $\alpha_{\bff}$]\label{defn=Pspower}
Let $\Ff$ be a group, and let $\Kf$ and $\Mf$ be two normal subgroups of $\Ff$ with $\Kf\geqslant \Mf$. Assume that $\Mf\geqslant [\Ff,\Kf]$. Let $\wf\in [\Ff,\Kf]$ and $\bff=(\ff_1,\ldots ,\ff_t;\ff'_1,\ldots,\ff'_t)$ be an $[\Ff,\Kf]$-expression for $\wf$. Then, we define a map $\alpha_{\bff}\colon \ZZ\to [\Ff,\Kf]$ by
\[
\ZZ\ni m\mapsto \alpha_{\bff}(m)=[\ff_1,\ff'_1{}^m]\cdots [\ff_t,\ff'_t{}^m]\in [\Ff,\Kf].
\]
\end{defn}

The following proposition states key properties of this map $\alpha_{\bff}$.

\begin{prop}\label{prop=alpha}
Let $\Ff$ be a group, and $\Kf$ and $\Mf$ two normal subgroups of $\Ff$ with $\Kf\geqslant \Mf$. Assume that $\Mf\geqslant [\Ff,\Kf]$. Let $\wf\in [\Ff,\Kf]$ and $\bff=(\ff_1,\ldots ,\ff_t;\ff'_1,\ldots,\ff'_t)$ be an $[\Ff,\Kf]$-expression  for $\wf$. Then, the following hold true for the map $\alpha_{\bff}$.
\begin{enumerate}[label=\textup{(}$\arabic*$\textup{)}]
  \item The map $\alpha_{\bff}\colon \ZZ\to ([\Ff,\Kf],d_{\cl_{\Ff,\Mf}})$ is a  pre-coarse homomorphism. More precisely, for every $m,n\in \ZZ$,
\begin{equation}\label{eq=2t-1}
d_{\cl_{\Ff,\Mf}}(\alpha_{\bff}(m+n),\alpha_{\bff}(m)\alpha_{\bff}(n))\leq 2t-1.
\end{equation}
  \item Let $\mathrm{proj}_{[\Ff,\Mf]}^{\Mf}\colon \Mf\twoheadrightarrow M/[\Ff,\Mf]$ be the natural group quotient map. Then, $\mathrm{proj}_{[\Ff,\Mf]}^{\Mf}\circ \alpha_{\bff}\colon \ZZ\to M/[\Ff,\Mf]$ is a genuine group homomorphism.
\end{enumerate}
\end{prop}

Recall from Subsection~\ref{subsec=scl-distance} that we set $d_{\cl_{\Ff,\Mf}}(\wf_1,\wf_2)=\infty$ if $\wf_1^{-1}\wf_2\not\in [\Ff,\Mf]$.

\begin{proof}[Proof of Proposition~$\ref{prop=alpha}$]
We use the symbol `${\underset{C}{\eqsim}}$' to mean `$\overset{(\Ff,\Mf)}{\underset{C}{\eqsim}}$;' recall Definition~\ref{defn=choron}. First we prove (1). Let $i\in \{1,\ldots,t\}$. Then by Lemma~\ref{lem=commutatoreq}~(1),
\[
[\ff_i,\ff'_i{}^{m+n}]=[\ff_i,\ff'_i{}^m \ff'_i{}^n ]=[\ff_i,\ff'_i{}^m][\ff_i,\ff'_i{}^n][[\ff_i,\ff'_i{}^n]^{-1},\ff'_i{}^m]\mathrel{\underset{1}{\eqsim}}[\ff_i,\ff'_i{}^m][\ff_i,\ff'_i{}^n].
\]
Here, recall Lemma~\ref{lem=Ginvmetric} and that $[\Ff,\Kf]\leqslant \Mf$. Hence, we have
\begin{equation}\label{eq=m+n}
\alpha_{\bff}(m+n)\mathrel{\underset{t}{\eqsim}}[\ff_1,\ff'_1{}^m][\ff_1,\ff'_1{}^n]\cdots [\ff_t,\ff'_t{}^m][\ff_t,\ff'_t{}^n].
\end{equation}
By Lemma~\ref{lem=choron}~(3), for every $\wf_1,\wf_2\in \Mf$, $\wf_1\wf_2\mathrel{\underset{1}{\eqsim}}\wf_2\wf_1$ holds. This implies that
\begin{equation}\label{eq=mton}
[\ff_1,\ff'_1{}^m][\ff_1,\ff'_1{}^n]\cdots [\ff_t,\ff'_t{}^m][\ff_t,\ff'_t{}^n]\mathrel{\underset{t-1}{\eqsim}}\alpha_{\bff}(m)\alpha_{\bff}(n).
\end{equation}
By \eqref{eq=m+n} and \eqref{eq=mton}, we conclude (1). Now (2) immediately follows from (1).
\end{proof}

The following theorem is the limit formula for the value $\hf(\wf)$ in the abelian case.

\begin{thm}[limit formula for the value of a core]\label{thm=valueofcore}
Assume Setting~$\ref{setting=GLN}$. Fix a tuple $(\Ff,\Kf,\Mf,\Rel,\ppi)$ associated with $(\Gg,\Lg,\Ng)$. Assume that the tuple $(\Ff,\Kf,\Mf,\Rel,\ppi)$ satisfies condition~\textup{(i)} in Definition~$\ref{defn=FKMcondition}$. Assume that $\Mf\geqslant [\Ff,\Kf]$. Let $\nug\in \QQQ(\Ng)^{\Gg}$ and $\hf$ be a core of $\nug$ \textup{(}with respect to $(\Ff,\Kf,\Mf,\Rel,\ppi)$\textup{)}. Then, for every $\wf\in [\Ff,\Kf]$ and for every $[\Ff,\Kf]$-expression $\bff=(\ff_1,\ldots,\ff_t;\ff'_1,\ldots,\ff'_t)$ for $\wf$, we have
\[
\hf(\wf)=\lim_{m\to\infty}\frac{\ppi^{\ast}\nug(\alpha_{\bff}(m))}{m}.
\]
\end{thm}

\begin{proof}
Take $\phf\in \QQQ(\Kf)^{\Ff}$ such that $\pi^{\ast}\nug=\hf+i^{\ast}\phf$. Let $m\in \ZZ$. Then, we have
\[
\pi^{\ast}\nug(\alpha_{\bff}(m))=\hf(\alpha_{\bff}(m))+\phf(\alpha_{\bff}(m)).
\]
By Proposition~\ref{prop=alpha}~(2), we have $\hf(\alpha_{\bff}(m))=m\hf(\alpha_{\bff}(1))=m\hf(\wf)$. Since $\cl_{\Ff,\Kf}(\alpha_{\bff}(m))\leq t$ by construction, we have $|\phf(\alpha_{\bff}(m))|\leq (2t-1)\DD(\phf)$. Therefore, we obtain the desired limit formula.
\end{proof}

In particular, if $\wf$ is a $\nug$-core-surviving element in the setting of Theorem~\ref{thm=valueofcore}, then for every $[\Ff,\Kf]$-expression $\bff$ for $\wf$, we have
\[
\lim_{m\to\infty}\left|\nug\bigr(\ppi(\alpha_{\bff}(m))\bigl)\right|=\infty;
\]
a phenomenon of this type was called an \emph{overflowing} in \cite[Subsection~8.2]{MMM}. We summarize the argument above as the following corollary.

\begin{cor}[characterization of the $\nug$-surviving property in terms of  overflowing]\label{cor=abelianlimit}
Assume Setting~$\ref{setting=GLN}$. Fix a tuple $(\Ff,\Kf,\Mf,\Rel,\ppi)$ associated with $(\Gg,\Lg,\Ng)$. Assume that the tuple $(\Ff,\Kf,\Mf,\Rel,\ppi)$ satisfies condition~\textup{(i)} in Definition~$\ref{defn=FKMcondition}$. Assume that $\Mf\geqslant [\Ff,\Kf]$. Let $\nug\in \QQQ(\Ng)^{\Gg}$. Then for every $\wf\in [\Ff,\Kf]$, the following three conditions are all equivalent.
\begin{enumerate}[label=\textup{(}$\arabic*$\textup{)}]
 \item The element $\wf$ is a $\nug$-core-surviving element.
 \item For every $[\Ff,\Kf]$-expression $\bff$ for $\wf$, $\lim\limits_{m\to\infty}\left|\nug\bigr(\ppi(\alpha_{\bff}(m))\bigl)\right|=\infty$.
 \item There exists an $[\Ff,\Kf]$-expression $\bff$ for $\wf$ such that $\lim\limits_{m\to\infty}\left|\nug\bigr(\ppi(\alpha_{\bff}(m))\bigl)\right|=\infty$ holds.
\end{enumerate}
\end{cor}

\subsection{Characterizations of the extendability  in the abelian case}\label{subsec=orerei}

Here, we present characterization of the extendability of invariant quasimorphisms in the abelian case. First, we define the following quantity.

\begin{defn}\label{defn=D_t}
Assume Setting~\ref{setting=GLN}. Assume moreover that
\begin{equation}\label{eq=abelianG}
\Ng\geqslant [\Gg,\Lg].
\end{equation}
Let $\nug\in \QQQ(\Ng)^{\Gg}$. Then for every $t\in \NN$, we define
\[
\DD^t_{\Gg,\Lg}(\nug)=\sup\left\{\left|\nug([\gG_1,\gG'_1]\cdots [\gG_t,\gG'_t])\right|\,|\, \gG_1,\ldots ,\gG_t\in \Gg,\ \gG'_1,\ldots ,\gG'_t\in \Lg\right\} \in \RR_{\geq 0}\cup \{\infty\}.
\]
\end{defn}

We note that every element of the form $[\gG_1,\gG'_1]\cdots [\gG_t,\gG'_t]$ above lies in $\Ng$ by \eqref{eq=abelianG}.

\begin{thm}[characterizations of the extendability in the abelian case]\label{thm=orerei}
Assume Setting~$\ref{setting=GLN}$. Assume that $\Mf\geqslant [\Ff,\Kf]$. Then for every $\nug\in \QQQ(\Ng)^{\Gg}$, the following are all equivalent.
\begin{enumerate}[label=\textup{(}$\arabic*$\textup{)}]
  \item The $\Gg$-invariant homogeneous quasimorphism $\nug$ is extendable to $\Lg$, namely, $\nug\in i_{\Ng,\Lg}^{\ast}\QQQ(\Lg)^{\Gg}$ holds.
  \item For every $t\in \NN$, $\DD^t_{\Gg,\Lg}(\nug)<\infty$.
  \item There exists $t\in \NN$ such that $\DD^t_{\Gg,\Lg}(\nug)<\infty$.
  \item The equality $\DD^1_{\Gg,\Lg}(\nug)=\DD(\nug)$ holds.
  \item The inequality $\DD^1_{\Gg,\Lg}(\nug)<\infty$ holds.
\end{enumerate}
\end{thm}

\begin{proof}
First we prove that (5) implies (1); this is the main part of the proof. Set
\[
\mathcal{S}=\{(\tilde{\Ng},\tilde{\nug})\,|\, \Ng \leqslant \tilde{\Ng}\leqslant \Lg,\ \tilde{\nug}\in \QQQ(\tilde{\Ng}),\ \tilde{\nug}|_{\Ng}\equiv\nug\}.
\]
We first claim that for every $(\tilde{\Ng},\tilde{\nug})\in \mathcal{S}$, $\tilde{\nug}\in \QQQ(\tilde{\Ng})^{\Gg}$ holds (recall also Remark~\ref{rem=normal}). Indeed, if not, then there must exist $\gG\in \Gg$ and $\xg\in \tilde{\Ng}$ such that $\tilde{\nug}(\gG\xg\gG^{-1})-\tilde{\nug}(\xg)\ne 0$. Set $\kappa=\tilde{\nug}(\gG\xg\gG^{-1})-\tilde{\nug}(\xg) (\ne 0)$. Then we have
\begin{align*}
\DD^1_{\Gg,\Lg}(\nug)&\geq \sup\left\{\left|\tilde{\nug}([\gG,\xg^m])\right|
\,\middle|\,m\in \NN\right\}\\
&= \sup\left\{\left|\tilde{\nug}((\gG\xg\gG^{-1})^m \xg^{-m})\right|
\,\middle|\,m\in \NN\right\}\\
&\geq \sup \left\{  m|\kappa|-\DD(\tilde{\nug})  \,|\,m\in \NN\right\}=\infty,
\end{align*}
a contradiction. The points here are that $[\gG,\xg^m]\in [\Gg,\Lg]\leqslant \Ng$ for every $m\in\NN$ by $\Ng\geqslant [\Gg,\Lg]$; hence, $\tilde{\nug}([\gG,\xg^m])=\nug([\gG,\xg^m])$. We define a partial order $\leq$ on $\mathcal{S}$ in the following natural way: for $(\tilde{\Ng}_1,\tilde{\nug}_1),(\tilde{\Ng}_2,\tilde{\nug}_2)\in \mathcal{S}$, $(\tilde{\Ng}_1,\tilde{\nug}_1)\leq (\tilde{\Ng}_2,\tilde{\nug}_2)$ if $\tilde{\Ng}_1\leqslant \tilde{\Ng}_2$ and $\tilde{\nug}_2|_{\tilde{\Ng}_1}\equiv \tilde{\nug}_1$. Then, it is straightforward to see that this $(\mathcal{S},\leq)$ satisfies the condition of applying Zorn's lemma. By Zorn's lemma, there exists a maximal element $(\tilde{\Ng}^{\sharp},\tilde{\nug}^{\sharp})\in \mathcal{S}$. In what follows, we verify that $\tilde{\Ng}^{\sharp}=\Lg$ and that $\tilde{\nug}^{\sharp}\in \QQQ(\Lg)^{\Gg}$.

Suppose that $\tilde{\Ng}^{\sharp}$ is a proper subgroup of $\Lg$. Then, there must exist $\gG'\in \Lg\setminus \tilde{\Ng}^{\sharp}$. Let $\Lg_{\gG'}$ be the group generated by $\tilde{\Ng}^{\sharp}$ and $\gG'$. Then the group quotient $\Lg_{\gG'}/\tilde{\Ng}^{\sharp}$ is a non-trivial cyclic group. In particular, the short exact sequence
\[
1\longrightarrow \tilde{\Ng}^{\sharp}\longrightarrow \Lg_{\gG'} \longrightarrow \Lg_{\gG'}/\tilde{\Ng}^{\sharp}\longrightarrow 1
\]
virtually splits. Indeed, it splits if $\Lg_{\gG'}/\tilde{\Ng}^{\sharp}\cong \ZZ$; if $\Lg_{\gG'}/\tilde{\Ng}^{\sharp}$ is finite, then observe that the trivial subgroup is of finite index in $\Lg_{\gG'}/\tilde{\Ng}^{\sharp}$. By the claim above, we have $\tilde{\nug}^{\sharp}\in \QQQ(\tilde{\Ng}^{\sharp})^{\Gg}\subseteq \QQQ(\tilde{\Ng}^{\sharp})^{\Lg_{\gG'}}$. By Proposition~\ref{prop=vsplit}, we now conclude that $\tilde{\nug}^{\sharp}\in i_{\tilde{\Ng}^{\sharp},\Lg_{\gG'}}^{\ast}\QQQ(\Lg_{\gG'})$. Take $\psg_{\gG'}\in \QQQ(\Lg_{\gG'})$ such that $\psg_{\gG'}|_{\tilde{\Ng}^{\sharp}}\equiv \tilde{\nug}^{\sharp}$. Then, $(\Lg_{\gG'},\psg_{\gG'})$ is an element in $\mathcal{S}$ that is strictly greater than $(\tilde{\Ng}^{\sharp},\tilde{\nug}^{\sharp})$ with respect to the order $\leq$; this is a contradiction. Therefore, we have verified that $\tilde{\Ng}^{\sharp}=\Lg$ and $\tilde{\nug}^{\sharp}\in \QQQ(\Lg)^{\Gg}$, thus completing the proof of `(5)$\Rightarrow$(1).'

Now we will prove remaining implications; clearly, (1) implies (2), (2) implies (3), (3) implies (5), and (4) implies (5). Finally, we will verify that (5) implies (4) by showing the contrapositive. By Corollary~\ref{cor=BavardD}, we have $\DD^1_{\Gg,\Lg}(\nug)\geq \DD(\nug)$. Assume now that $\DD^1_{\Gg,\Lg}(\nug)>\DD(\nug)$. Then there exist $\gG\in \Gg$ and $\gG'\in \Lg$ such that $|\nug([\gG,\gG'])|>\DD(\nug)$. By Lemma~\ref{lem=commutatoreq}~(2), for every $m\in \NN$ there exist $\gG_1,\ldots ,\gG_{m-1}\in \Gg$ such that
\[
[\gG,\gG'^m]=[\gG,\gG'](\gG_1[\gG,\gG']\gG_1^{-1})\cdots(\gG_{m-1}[\gG,\gG']\gG_{m-1}^{-1})
\]
holds. This, together with $|\nug([\gG,\gG'])|>\DD(\nug)$, implies that
\[
|\nug([\gG,\gG'{}^m])|\geq |m|\bigl(|\nug([\gG,\gG'])|-\DD(\nug)\bigr)
\]
for every $m\in \ZZ$. Hence, (5) implies (4). This completes the proof.
\end{proof}

\subsection{Proof of Theorem~\ref{thm=abelD}, and the key corollary}\label{subsec=keycor}
Theorem~\ref{thm=orerei} is a powerful tool for the study of the abelian case. For instance, we are now ready to prove Theorem~\ref{thm=abelD}.

\begin{proof}[Proof of Theorem~$\ref{thm=abelD}$]
Let $\nug\in \QQQ(\Ng)^{\Gg}$ and consider $\pi^{\ast}\nug\in \QQQ(\Mf)^{\Ff}$. By Theorem~\ref{thm=orerei}, the fact that $\pi^{\ast}\nug \in i_{\Kf,\Ff}^{\ast}\QQQ(\Kf)^{\Ff}$ implies that $\DD^1_{\Mf,\Kf}(\pi^{\ast}\nug)<\infty$. Note that  we have $\DD^1_{\Ff,\Kf}(\pi^{\ast}\nug)=\DD^1_{\Gg,\Lg}(\nug)$ because $\ppi(\Ff)=\Gg$ and $\ppi(\Kf)=\Lg$. Hence, $\DD^1_{\Gg,\Lg}(\nug)<\infty$. Again by Theorem~\ref{thm=orerei}, we conclude that $\nug \in i_{\Ng,\Lg}^{\ast}\QQQ(\Lg)^{\Gg}$. This completes our proof.
\end{proof}

Theorem~\ref{thm=abelD}, together with Theorems~\ref{thm=InjCE} and \ref{thm=Goodtuple}, immediately provides the following corollary, which is one of the keys to the proofs of Theorem~\ref{mthm=main} and Theorem~\ref{mthm=dim}.

\begin{cor}[outcome of the theory of core extractors: the abelian case]\label{cor=InjCE}
Assume Setting~$\ref{setting=GLN}$. Fix a tuple $(\Ff,\Kf,\Mf,\Rel,\ppi)$ associated with $(\Gg,\Lg,\Ng)$. Assume that the tuple $(\Ff,\Kf,\Mf,\Rel,\ppi)$ satisfies conditions \textup{(i)} and \textup{(ii)} in Definition~$\ref{defn=FKMcondition}$. Assume furthermore that $\Mf\geqslant  [\Ff,\Kf]$. Then the following hold true.
\begin{enumerate}[label=\textup{(}$\arabic*$\textup{)}]
 \item The core extractor $\CE\colon \Wcal(\Gg,\Lg,\Ng)\to i_{\Mf\cap \Rel,\Mf}^{\ast}\HHH^1(\Mf)^{\Ff}$ is injective.
 \item Assume that $\Wcal(\Gg,\Lg,\Ng)\ne 0$, and let $\ell \in \NN$ satisfy  that $\ell\leq \Rdim \Wcal(\Gg,\Lg,\Ng)$.  Then there exist $\wf_1,\ldots \wf_{\ell}\in [\Ff,\Kf]\cap \Rel$ and $\nug_1,\ldots,\nug_{\ell}\in \QQQ(\Ng)^{\Gg}$ such that $\CEt_{\nug_{j}}(\wf_i)=\delta_{i,j}$ for every $i,j\in \{1,\ldots,\ell\}$.
\end{enumerate}
\end{cor}

\subsection{More results on core extractors in the abelian case}
We start this subsection by stating characterizations of the extendability in the abelian case in terms of core-surviving elements under the assumption that $\QG=\Gg/\Ng$ is boundedly $3$-acyclic.

\begin{thm}[characterizations of the extendability in the abelian case in terms of core-surviving elements]\label{thm=abeliancore}
Assume Setting~$\ref{setting=GLN}$.  Assume that $\Ng\geqslant  [\Gg,\Lg]$. Assume furthermore that $\QG=\Gg/\Ng$ is boundedly $3$-acyclic. Then for every $\nug\in \QQQ(\Ng)^{\Gg}$, conditions \textup{(1)--(5)} in Theorem~$\ref{thm=orerei}$ are also equivalent to the following two conditions.
\begin{enumerate}
  \item[\textup{(6)}] For every tuple $(\Ff,\Kf,\Mf,\Rel,\ppi)$ associated with $(\Gg,\Lg,\Ng)$ satisfying condition \textup{(i)} in Definition~$\ref{defn=FKMcondition}$ such that $\Mf\geqslant  [\Ff,\Kf]$ holds, no element $\wf$ in $[\Ff,\Kf]$ is $\nug$-core-surviving.
  \item[\textup{(7)}] There exists a tuple $(\Ff,\Kf,\Mf,\Rel,\ppi)$ associated with $(\Gg,\Lg,\Ng)$ satisfying condition \textup{(i)} in Definition~$\ref{defn=FKMcondition}$ such that $\Mf\geqslant  [\Ff,\Kf]$ holds and no element $\wf$ in $[\Ff,\Kf]$ is $\nug$-core-surviving.
\end{enumerate}
\end{thm}

For the proof of Theorem~\ref{thm=abeliancore}, we employ the following lemma.

\begin{lem}\label{lem=exlift}
Assume Setting~$\ref{setting=GLN}$.  Assume that $\Ng\geqslant  [\Gg,\Lg]$. Assume that $\QG=\Gg/\Ng$ is boundedly $3$-acyclic. Then, there exists a tuple $(\Ff,\Kf,\Mf,\Rel,\ppi)$ associated with $(\Gg,\Lg,\Ng)$ satisfying condition \textup{(i)} in Definition~$\ref{defn=FKMcondition}$ such that $\Mf\geqslant  [\Ff,\Kf]$ holds.
\end{lem}

\begin{proof}
Take an arbitrary free group $\Ff$ such that there exists a surjective group homomorphism $\ppi\colon \Ff\twoheadrightarrow \Gg$, and let $\Rel=\Ker(\ppi)$. Set
\[
(\Ff,\Kf,\Mf,\Rel,\ppi)=(\Ff,\pi^{-1}(\Lg),\pi^{-1}(\Ng),\Rel,\ppi);
\]
this tuple is associated with $(\Gg,\Lg,\Ng)$.
Then, since $\QG$ is boundedly $3$-acyclic, Corollary~\ref{cor=free} shows that the tuple $(\Ff,\Kf,\Mf,\Rel,\ppi)$ satisfies condition (i) in Definition~\ref{defn=FKMcondition}. Indeed,  $\Lg/\Ng$ is abelian by $\Ng\geqslant  [\Gg,\Lg]$ and hence is boundedly $2$-acyclic. Moreover, $\Ng\geqslant  [\Gg,\Lg]$ implies that $\Mf\geqslant  [\Ff,\Kf]$.
\end{proof}

\begin{proof}[Proof of Theorem~$\ref{thm=abeliancore}$]
Lemma~\ref{lem=CoreE}~(1) shows that (1) in Theorem~\ref{thm=orerei} implies (6).  By Lemma~\ref{lem=exlift}, (6) immediately implies (7).

It only remains to show that  (7) implies (4) in Theorem~\ref{thm=orerei}. We prove this implication by proving the contrapositive. Suppose that there exist $\gG\in \Gg$ and $\gG'\in \Lg$ such that $|\nu([\gG,\gG'])|>\DD(\nug)$. Take an arbitrary tuple $(\Ff,\Kf,\Mf,\Rel,\ppi)$ associated with $(\Gg,\Lg,\Ng)$ satisfying (i) such that $\Mf\geqslant  [\Ff,\Kf]$ holds. Take $\ff\in \Ff$ and $\ff'\in \Kf$ such that $\ppi(\ff)=\gG$ and $\ppi(\ff')=\gG'$. Set $\wf=[\ff,\ff']\in [\Ff,\Kf]$. Then, $\bff=(\ff;\ff')$ is an $[\Ff,\Kf]$-expression for $\wf$. Observe that $\DD(\ppi^{\ast}\nug)=\DD(\nug)<|\ppi^{\ast}\nu([\ff,\ff'])|$. Hence, by an argument similar to the proof of `(5)$\Rightarrow$(4)' in Theorem~\ref{thm=orerei}, we conclude $\lim\limits_{m\to\infty}\left|\ppi^{\ast}\nug\bigl(\alpha_{\bff}(m)\bigr)\right|=\infty$. Therefore, by Corollary~\ref{cor=abelianlimit} the element $\wf$ is $\nug$-core-surviving. This completes our proof.
\end{proof}

The following lemma will be employed in Section~\ref{sec=Psi}.

\begin{lem}\label{lem=eqdefect}
Assume Setting~$\ref{setting=GLN}$.  Fix a tuple $(\Ff,\Kf,\Mf,\Rel,\ppi)$ associated with $(\Gg,\Lg,\Ng)$. Assume that $\Mf\geqslant  [\Ff,\Kf]$. Let $\nug\in \QQQ(\Ng)^{\Gg}$. Assume that there exists a core $\hf$ of $\nug$. Then, for every $\phf\in \QQQ(\Kf)^{\Ff}$ satisfying $\ppi^{\ast}\nug=\hf+i_{\Mf,\Kf}^{\ast}\phf$, we have $\DD(\phf)=\DD(\nug)$.
\end{lem}

\begin{proof}
Set $\muf=\ppi^{\ast}\nug-\hf$. Then, $\muf=i_{\Mf,\Kf}^{\ast}\phf\in i_{\Mf,\Kf}^{\ast}\QQQ(\Kf)^{\Ff}$. Hence by Theorem~\ref{thm=orerei}, we have $\DD^1_{\Ff,\Kf}(\muf)=\DD(\muf)$. By Corollary~\ref{cor=BavardD}, we also have $\DD^1_{\Ff,\Kf}(\muf)=\DD(\phf)$. Therefore,  $\DD(\phf)=\DD(\muf)=\DD(\muf+\hf)=\DD(\pi^{\ast}\nug)=\DD(\nug)$, as desired.
\end{proof}

Let $(\Gg,\Lg,\Ng)$ be a triple in Setting~\ref{setting=GLN}. Assume that $\Ng\geqslant  [\Gg,\Lg]$ and that $\QG=\Gg/\Ng$ is boundedly $3$-acyclic. Then by Lemma~\ref{lem=exlift}, there always exists a tuple $(\Ff,\Kf,\Mf,\Rel,\ppi)$ associated with $(\Gg,\Lg,\Ng)$ that satisfies condition (i) in Definition~\ref{defn=FKMcondition} and $\Mf\geqslant  [\Ff,\Kf]$. Hence, to apply Corollary~\ref{cor=InjCE}, the main issue is to find such a tuple $(\Ff,\Kf,\Mf,\Rel,\ppi)$ that furthermore meets condition (ii) in Definition~\ref{defn=FKMcondition}. The following lemma supplies such examples, and this explains the role of the assumptions in Theorem~\ref{mthm=main}.

\begin{lem}\label{lem=maintogeneral}
Assume Setting~$\ref{setting=GLN}$. Assume that there exists $q\in \NN_{\geq 2}$ such that either of the following two conditions is fulfilled:
\begin{enumerate}
  \item[\textup{(a$_q$)}]  $\Ng=\gamma_{q}(\Gg)$ and $\Lg=\gamma_{q-1}(\Gg)$; or
  \item[\textup{(b$_q$)}]  $\Ng\geqslant \gamma_{q}(\Gg)$, $\Lg=\gamma_{q-1}(\Gg)\Ng$, and $\Gg$ admits a group presentation $(\tilde{\Ff}\,|\,\tilde{\Rel})$ such that $\tilde{\Rel}\leqslant \gamma_{q}(\tilde{\Ff})$.
\end{enumerate}
Then, there exists a tuple $(\Ff,\Kf,\Mf,\Rel,\ppi)$ associated with $(\Gg,\Lg,\Ng)$ satisfying conditions \textup{(i)} and \textup{(ii)} in Definition~$\ref{defn=FKMcondition}$ such that  $\Mf\geqslant [\Ff,\Kf]$ holds.
\end{lem}

\begin{proof}
In case (a$_q$), fix a free group $\Ff$ admitting a surjective group homomorphism $\ppi\colon\Ff\twoheadrightarrow \Gg$. Then, we set
\[
(\Ff,\Kf,\Mf,\Rel,\ppi)=(\Ff,\gamma_{q-1}(\Ff),\gamma_q(\Ff),\Ker(\ppi),\ppi).
\]
In case (b$_q$), let $\tilde{\ppi}\colon \tilde{\Ff}\twoheadrightarrow \tilde{\Ff}/\tilde{\Rel}\cong \Gg$ be the natural group quotient map. Then, we set
\[
(\Ff,\Kf,\Mf,\Rel,\ppi)=(\tilde{\Ff}, \gamma_{q-1}(\tilde{\Ff})\ppi^{-1}(\Ng), \ppi^{-1}(\Ng),\tilde{\Rel},\tilde{\ppi}).
\]
In both cases, with the aid of Proposition~\ref{prop=onetwo}, we can verify that these tuples $(\Ff,\Kf,\Mf,\Rel,\ppi)$ satisfy all the indicated  conditions.
\end{proof}

\begin{exa}\label{exa=MGL}
Assume Setting~$\ref{setting=GLN}$. Assume that $\QG=\Gg/\Ng$ is boundedly 3-acyclic and that $\Ng=[\Gg,\Lg]$. Then, there exists a tuple $(\Ff,\Kf,\Mf,\Rel,\ppi)$ associated with $(\Gg,\Lg,\Ng)$ satisfying condition (i) in Definition~\ref{defn=FKMcondition} such that $\Mf=[\Ff,\Kf]$ holds. Indeed, fix a free group $\Ff$ admitting a surjective group homomorphism $\ppi\colon\Ff\twoheadrightarrow \Gg$ and set
\[
(\Ff,\Kf,\Mf,\Rel,\ppi)=(\Ff,\ppi^{-1}(\Lg),[\Ff,\ppi^{-1}(\Lg)],\Ker(\ppi),\ppi).
\]
To see that this tuple satisfies (i) in Definition~\ref{defn=FKMcondition}, first we observe that $\Lg/\Ng$ and  $\Kf/\Mf$ are abelian and hence boundedly acyclic. By Theorem~\ref{thm=bdd_acyc}~(2), $\Ff/\Kf\cong \Gg/\Lg$ is boundedly $3$-acyclic. Again by Theorem~\ref{thm=bdd_acyc}~(2), $\Ff/\Mf$ is boundedly $3$-acyclic. Therefore, by Proposition~\ref{prop=onetwo}, the tuple $(\Ff,\Kf,\Mf,\Rel,\ppi)$ satisfies (i). Since $\Mf=[\Ff,\Kf]$, this tuple moreover satisfies condition (ii) in Definition~\ref{defn=FKMcondition}.
\end{exa}

\section{Construction of $\PPs$}\label{sec=Psi}
In this section, we take Step~2 in the outlined proof of Theorem~\ref{mthm=main}. In fact, we will construct the map $\PPs$ in a more general situation. The key tools here are Corollary~\ref{cor=InjCE} and the map $\alpha_{\bff}\colon \ZZ\to [\Ff,\Kf]$ defined in Definition~\ref{defn=Pspower}. We use the following settings.

\begin{setting}\label{setting=main1}
Let $(\Gg,\Lg,\Ng)$ be a triple in  Setting~$\ref{setting=GLN}$, \emph{i.e.,} let $\Gg$ be a group, and  let $\Lg$ and $\Ng$ be two normal subgroups of $\Gg$ with $\Lg\geqslant \Ng$. Fix a tuple $(\Ff,\Kf,\Mf,\Rel,\ppi)$ associated with $(\Gg,\Lg,\Ng)$ \textup{(}recall Definition~$\ref{defn=FKM}$\textup{)}. Assume that $(\Ff,\Kf,\Mf,\Rel,\ppi)$ satisfies conditions \textup{(i)} and \textup{(ii)} in Definition~$\ref{defn=FKMcondition}$. Assume that $\Mf\geqslant [\Ff,\Kf]$.
\end{setting}

\begin{setting}\label{setting=main2}
Under Setting~$\ref{setting=main1}$, assume furthermore that $\Wcal(\Gg,\Lg,\Ng)\ne 0$. Let $\ell\in \NN$ satisfy that $\ell\leq \Rdim \Wcal(\Gg,\Lg,\Ng)$. Fix arbitrary $\wf_1,\ldots ,\wf_{\ell}\in [\Ff,\Kf] \cap \Rel$ and $\nug_1,\ldots,\nug_{\ell}\in \QQQ(\Ng)^{\Gg}$ such that for every $i,j\in \{1,\ldots,\ell\}$, $\CEt_{\nug_{j}}(\wf_i)=\delta_{i,j}$ holds; Corollary~$\ref{cor=InjCE}$ ensures the existence of such $\wf_1,\ldots ,\wf_{\ell}$ and $\nug_1,\ldots,\nug_{\ell}$. For every $i\in \{1,\ldots,\ell\}$, fix an arbitrary $[\Ff,\Kf]$-expression $\bffi=(\ff_1^{(i)},\ldots ,\ff_{t_i}^{(i)};\ff'_1{}^{(i)},\ldots ,\ff'_{t_i}{}^{(i)})$ for $\wf_i$ \textup{(}recall Definition~$\ref{defn=tuple}$\textup{)}. Set $\ttil=\sum\limits_{i\in \{1,\ldots,\ell\}}t_i$.
\end{setting}

We remark that in Setting~\ref{setting=main2}, we only assume that $\Wcal(\Gg,\Lg,\Ng)\ne 0$ in addition to Setting~\ref{setting=main1}. The rest of Setting~\ref{setting=main2} is only to fix our notation for Theorem~\ref{thm=PPs}.

\subsection{The construction  of $\PPs$}\label{subsec=PPs}
The following theorem corresponds to Step~2 in the outlined proof of Theorem~\ref{mthm=main}. Recall our definition of QI-type estimates from below/above from Definition~\ref{defn=QI}.

\begin{thm}[construction of $\PPs$]\label{thm=PPs}
Assume Settings~$\ref{setting=GLN}$, $\ref{setting=main1}$ and $\ref{setting=main2}$. Let $\alpha_{\bffone},\ldots ,\alpha_{\bffell}\colon \ZZ\to [\Ff,\Kf]$ be the maps defined in Definition~$\ref{defn=Pspower}$. For every $i\in \{1,\ldots,\ell\}$, set $\beta_{\bffi}=\ppi\circ \alpha_{\bffi}$. Set a map $\PPs$ as
\begin{equation}\label{eq=defnPPs}
\PPs(\vm)=\beta_{\bffone}(m_1)\cdots \beta_{\bffell}(m_{\ell})
\end{equation}
for every $\vm=(m_1,\ldots,m_{\ell})\in \ZZ^{\ell}$. Then, the following hold true.
\begin{enumerate}[label=\textup{(}$\arabic*$\textup{)}]
  \item For every $i\in \{1,\ldots,\ell\}$ and for every $m_i\in \ZZ$, we have $\beta_{\bffi}(m_i)\in [\Gg,\Ng]$. In particular, we regard $\PPs$ as the map from $\ZZ^{\ell}$ to $[\Gg,\Ng]$.
  \item The map $\PPs\colon \ZZ^{\ell}\to ([\Gg,\Ng],d_{\cl_{\Gg,\Ng}})$ is a pre-coarse homomorphism.
  \item The map $\PPs$ is $d_{\cl_{\Gg,\Lg}}$-bounded.
  \item We have the following QI-type estimate from above on  $\ZZ^{\ell}$:
\begin{equation}\label{eq=PPsabove}
d_{\cl_{\Gg,\Ng}}(\PPs(\vm),\PPs(\vn))\leq \left\{2\left(\max_{i\in \{1,\ldots ,\ell\}}t_i\right)-1\right\}\cdot \|\vm-\vn\|_1 +2\ttil -1.
\end{equation}
  \item We have the following QI-type estimate from below on $\ZZ^{\ell}$:
\begin{equation}\label{eq=PPsbelow}
d_{\scl_{\Gg,\Ng}}(\PPs(\vm),\PPs(\vn))\geq \frac{1}{2\ell \left(\max\limits_{j\in \{1,\ldots,\ell\}}\DD(\nug_j)\right)}\cdot \|\vm-\vn\|_1-2 \ttil +\frac{1}{2}.
\end{equation}
\end{enumerate}
\end{thm}

We emphasize the following point: the construction of $\PPs$ in Theorem~\ref{thm=PPs} is built on Setting~\ref{setting=main2}, and the existences of $\wf_1,\ldots ,\wf_{\ell}$ and $\nug_1,\ldots,\nug_{\ell}$ in Setting~\ref{setting=main2} are ensured by Corollary~\ref{cor=InjCE}. Thus, Corollary~\ref{cor=InjCE} is  the key to the construction above of $\PPs$.

We note that by \eqref{eq=hikaku}, the inequality \eqref{eq=PPsabove} in particular implies that
\[
d_{\scl_{\Gg,\Ng}}(\PPs(\vm),\PPs(\vn))\leq \left\{2\left(\max_{i\in \{1,\ldots ,\ell\}}t_i\right)-1\right\}\cdot \|\vm-\vn\|_1 +2\ttil -1.
\]

By setting $\gG_{s}^{(i)}=\pi(\ff_{s}^{(i)})$ and $\gG_{s}^{(i)}{}'=\pi(\ff_{s}^{(i)}{}')$ for every $i\in \{1,\ldots,\ell\}$ and for every $s\in \{1,\ldots ,t_i\}$, we rewrite the formula \eqref{eq=defnPPs} for every $\vm=(m_1,\ldots ,m_{\ell})\in\ZZ^{\ell}$ as
\begin{equation}\label{eq=defnPPs2}
\PPs(\vm)=[\gG_{1}^{(1)},(\gG'_{1}{}^{(1)})^{m_1}]\cdots [\gG_{t_1}^{(1)},(\gG'_{t_1}{}^{(1)})^{m_1}]\cdots [\gG_{1}^{(\ell)},(\gG'_{1}{}^{(\ell)})^{m_{\ell}}]\cdots [\gG_{t_{\ell}}^{(\ell)},(\gG'_{t_{\ell}}{}^{(\ell)})^{m_{\ell}}].
\end{equation}

We employ the following lemma for the proof of Theorem~\ref{thm=PPs}.
\begin{lem}\label{lem=nugalpha}
Assume Settings~$\ref{setting=GLN}$, $\ref{setting=main1}$ and $\ref{setting=main2}$. Let $\beta_{\bffone},\ldots ,\beta_{\bffell}\colon \ZZ\to [\Gg,\Ng]$ and $\PPs\colon \ZZ^{\ell}\to [\Gg,\Ng]$ be the maps defined in Theorem~$\ref{thm=PPs}$. Then for every $j\in \{1,\ldots ,\ell\}$, the following hold true.
\begin{enumerate}[label=\textup{(}$\arabic*$\textup{)}]
 \item For every $i\in \{1,\ldots,\ell\}$ and for every $m_i\in \ZZ$, $|\nug_j(\beta_{\bffi}(m_i))-m_i\delta_{i,j}|\leq (2t_i-1)\DD(\nug_j)$.
 \item For every $\vm=(m_1,\ldots ,m_{\ell})\in \ZZ^{\ell}$, $\left|\nug_j(\PPs(\vm))-m_j\right|\leq (2 \ttil-1) \DD(\nug_j)$.
\end{enumerate}
\end{lem}

For the following proof of Lemma~\ref{lem=nugalpha}, recall  Lemma~\ref{lem=2D-1}.

\begin{proof}[Proof of Lemma~$\ref{lem=nugalpha}$]
Take a core $\hf_j$ of $\nug_j$. Then, we can take $\phf_j\in \QQQ(\Kf)^{\Ff}$ satisfying $\ppi^{\ast}\nug_j=\hf_j+i_{\Mf,\Kf}^{\ast}\phf_j$. By definition, we have $i^{\ast}_{\Mf\cap \Rel,\Mf}\hf_j=\CEt_{\nug_j}$. Let $i\in \{1,\ldots ,\ell\}$. Then, since $\cl_{\Ff,\Kf}(\alpha_{\bffi}(m_i))\leq t_i$, we obtain
\begin{align*}
|\nug_j(\beta_{\bffi}(m_i))-\hf_j(\alpha_{\bffi}(m_i))|&=|\ppi^{\ast}\nug_j(\alpha_{\bffi}(m_i))-\hf_j(\alpha_{\bffi}(m_i))|\\
&=|\phf_j(\alpha_{\bffi}(m_i))|\\
& \leq (2t_i-1)\DD(\phf_j).
\end{align*}
By Proposition~\ref{prop=alpha}~(2), we have
\[
\hf_j(\alpha_{\bffi}(m_i))=\hf_j(\alpha_{\bffi}(1)^{m_i})=\hf_j(\wf_i^{m_i})=m_i\delta_{i,j}.
\]
Therefore, Lemma~\ref{lem=eqdefect} ends the proof of (1). Since
\[
\left|\nug_j(\PPs(\vm))-\sum_{i\in \{1,\ldots ,\ell\}}\nug_j(\beta_{\bffi}(m_i))\right|\leq (\ell-1)\DD(\nug_j),
\]
(2) follows from (1).
\end{proof}

\subsection{Proof of Theorem~\ref{thm=PPs}}\label{subsec=proofPPs}
\begin{proof}[Proof of Theorem~$\ref{thm=PPs}$]
Let $\vm=(m_1,\ldots ,m_{\ell})$ and $\vn=(n_1,\ldots ,n_{\ell})$ be elements in $\ZZ^{\ell}$. By Proposition~\ref{prop=alpha}~(2), we have for every $i\in \{1,\ldots,\ell\}$,
\[
\beta_{\bffi}(m_i)=\pi(\alpha_{\bffi}(m_i))\in \pi\left( \alpha_{\bff}(1)^{m_i} [\Ff,\Mf]\right) =\ppi(\alpha_{\bff}(1))^{m_i}[\Gg,\Ng]=[\Gg,\Ng].
\]
Here, note that $\ppi(\wf_i)=e_{\Gg}$ since $\wf_i\in \Rel$. This proves (1).
Secondly, we will show (2). By Proposition~\ref{prop=alpha}~(1) we have for every $i\in \{1,\ldots ,\ell\}$, $\alpha_{\bffi}(m_i+n_i)\mathrel{\overset{(\Ff,\Mf)}{\underset{2t_i-1}{\eqsim}}}\alpha_{\bffi}(m_i)\alpha_{\bffi}(n_i)$; this implies that
\begin{equation}\label{eq=t_i}
\beta_{\bffi}(m_i+n_i)\mathrel{\overset{(\Gg,\Ng)}{\underset{2t_i-1}{\eqsim}}}\beta_{\bffi}(m_i)\beta_{\bffi}(n_i).
\end{equation}
By (1), for every $i\in \{1,\ldots ,\ell\}$, the element $\beta_{\bffi}(m_i)$ in particular lies in $\Ng$. Hence, by Lemma~\ref{lem=choron}~(3), we have
\begin{equation}\label{eq=t}
\beta_{\bffone}(m_1)\beta_{\bffone}(n_1)\cdots \beta_{\bffell}(m_{\ell})\beta_{\bffell}(n_{\ell})\mathrel{\overset{(\Gg,\Ng)}{\underset{\ell-1}{\eqsim}}} \PPs(\vm)\PPs(\vn)
\end{equation}
By \eqref{eq=t_i} and \eqref{eq=t}, we have
\begin{equation}\label{eq=2sumt-1}
d_{\cl_{\Gg,\Ng}}(\PPs(\vm+\vn),\PPs(\vm)\PPs(\vn))\leq 2\ttil-1;
\end{equation}
this \eqref{eq=2sumt-1} proves (2).
By construction (recall \eqref{eq=defnPPs2}), we have
\begin{equation}\label{eq=tt}
\sup_{\vm\in \ZZ^{\ell}}\cl_{\Gg,\Lg}(\PPs(\vm))\leq \ttil;
\end{equation}
this \eqref{eq=tt} proves (3).

Next, we prove (4).  Let $i\in \{1,\ldots ,\ell\}$. Then by \eqref{eq=t_i}, we in particular have for every $m\in \NN$,
\[
\beta_{\bffi}(-m)\mathrel{\overset{(\Gg,\Ng)}{\underset{2t_i-1}{\eqsim}}}\beta_{\bffi}(m)^{-1}\qquad \mathrm{and}\qquad
\beta_{\bffi}(m)\mathrel{\overset{(\Gg,\Ng)}{\underset{(m-1)(2t_i-1)}{\eqsim}}}\beta_{\bffi}(1)^m=e_{\Gg}.
\]
In particular, we have $\beta_{\bffi}(m_i)\mathrel{\overset{(\Gg,\Ng)}{\underset{2t_i-1}{\eqsim}}} \beta_{\bffi}(|m_i|)^{\mathrm{sign}(m_i)}$, where $\mathrm{sign}(m_i)$ is defined to be $1$ if $m_i\geq 0$ and $-1$ if $m_i<0$. Hence we conclude that
\[
\beta_{\bffi}(m_i)\mathrel{\overset{(\Gg,\Ng)}{\underset{|m_i|(2t_i-1)}{\eqsim}}} e_{\Gg}.
\]
Therefore, we obtain that
\[
\cl_{\Gg,\Ng}(\PPs(\vm))\leq \sum_{i\in \{1,\ldots ,\ell\}}|m_i|(2t_i-1)\leq \left\{2\left(\max_{i\in \{1,\ldots ,\ell\}}t_i\right)-1\right\}\cdot \|\vm\|_1.
\]
By \eqref{eq=2sumt-1}, we have
\begin{align*}
d_{\cl_{\Gg,\Ng}}(\PPs(\vm),\PPs(\vn))&\leq  \cl_{\Gg,\Ng}\left(\PPs(\vm-\vn)\right)+2\ttil-1\\
&\leq \left\{2\left(\max_{i\in \{1,\ldots ,\ell\}}t_i\right)-1\right\}\cdot \|\vm-\vn\|_1 +2\ttil -1,
\end{align*}
thus obtaining \eqref{eq=PPsabove}.

Finally, we prove (5). Given $\vm$ and $\vn$, we can take $j_{\vm,\vn}\in \{1,\ldots ,\ell\}$ such that $|m_{j_{\vm,\vn}}-n_{j_{\vm,\vn}}|\geq \ell^{-1}\cdot \|\vm-\vn\|_1$. By Lemma~\ref{lem=nugalpha}~(2), we have
\begin{align*}
\left|\nug_{j_{\vm,\vn}}(\PPs(\vm)^{-1}\PPs(\vn))\right|&\geq \left|\nug_{j_{\vm,\vn}}(\PPs(\vn))-\nug_{j_{\vm,\vn}}(\PPs(\vm))\right|-\DD(\nug_{j_{\vm,\vn}})\\
&\geq  |n_{j_{\vm,\vn}}-m_{j_{\vm,\vn}}|-(4 \ttil -1)\cdot \DD(\nug_{j_{\vm,\vn}})\\
&\geq \frac{1}{\ell}\cdot \|\vm-\vn\|_1-(4 \ttil -1)\cdot \DD(\nug_{j_{\vm,\vn}}).
\end{align*}
Therefore, Theorem~\ref{thm=Bavard} implies \eqref{eq=PPsbelow}.
\end{proof}

We summarize a part of the arguments in this section for a future use as follows.

\begin{thm}\label{thm=nilphe}
Assume Setting~$\ref{setting=GLN}$. Assume that $\Ng=[\Gg,\Lg]$ and that $\QG=\Gg/\Ng$ is boundedly $3$-acyclic. Assume that $\Wcal(\Gg,\Lg,\Ng)$ is finite dimensional. Let $\ell=\Rdim \Wcal(\Gg,\Lg,\Ng)$. Then, there exist quasimorphisms $\nug_1,\ldots,\nug_{\ell}\in \QQQ(\Ng)^{\Gg}$ and maps $\beta_1,\ldots,\beta_{\ell}\colon \ZZ\to [\Gg,\Ng]$ such that the following conditions are all fulfilled.
\begin{enumerate}[label=\textup{(\arabic*)}]
  \item The equivalence classes $[\nug_1],\ldots,[\nug_{\ell}]$ form a basis of $\Wcal(\Gg,\Lg,\Ng)$.
  \item The maps $\beta_1,\ldots,\beta_{\ell}$ are all $d_{\cl_{\Gg,\Lg}}$-bounded.
  \item There exists $D\in \RR_{\geq 0}$ such that for every $i,j\in \{1,\ldots,\ell\}$ and for every $m\in \ZZ$,
\[
\left|\nug_j(\beta_i(m))-m \delta_{i,j}\right| \leq D.
\]
\end{enumerate}
\end{thm}

Here, for the case of $\ell=0$ we regard Theorem~\ref{thm=nilphe} as a trivial statement.

\begin{proof}[Proof of Theorem~$\ref{thm=nilphe}$]
Apply Corollary~\ref{cor=InjCE} and Lemma~\ref{lem=nugalpha} (and Theorem~\ref{thm=PPs}) to Example~\ref{exa=MGL}.
\end{proof}

\section{Proofs of Theorem~\ref{mthm=main} and Theorem~\ref{mthm=dim}}\label{sec=proofmain}
In this section, we prove Theorem~\ref{mthm=main} and Theorem~\ref{mthm=dim}. For the proof of Theorem~\ref{mthm=main}, we in fact prove Theorem~\ref{thm=maingeneral} appearing in Subsection~\ref{subsec=generalform}, which is a general form of Theorem~\ref{mthm=main}. To state Theorem~\ref{thm=maingeneral}, we will use the following setting, which is stronger than Setting~\ref{setting=main2}. More precisely, we assume $\ell= \Rdim \Wcal(\Gg,\Lg,\Ng)$, rather than $\ell\leq  \Rdim \Wcal(\Gg,\Lg,\Ng)$ in Setting~\ref{setting=main3}.

\begin{setting}\label{setting=main3}
Under Setting~$\ref{setting=main1}$, assume furthermore that $\Wcal(\Gg,\Lg,\Ng)$ is non-zero finite dimensional, and set $\ell= \Rdim \Wcal(\Gg,\Lg,\Ng)$. Fix arbitrary $\wf_1,\ldots ,\wf_{\ell}\in [\Ff,\Kf] \cap \Rel$ and $\nug_1,\ldots,\nug_{\ell}\in \QQQ(\Ng)^{\Gg}$ such that for every $i,j\in \{1,\ldots,\ell\}$, $\CEt_{\nug_{j}}(\wf_i)=\delta_{i,j}$ holds. For every $i\in \{1,\ldots,\ell\}$, fix an arbitrary $[\Ff,\Kf]$-expression 
\[\bffi=(\ff_1^{(i)},\ldots ,\ff_{t_i}^{(i)};\ff'_1{}^{(i)},\ldots ,\ff'_{t_i}{}^{(i)})\]
 for $\wf_i$. Set $\ttil=\sum\limits_{i\in \{1,\ldots,\ell\}}t_i$.
\end{setting}

\subsection{Proof of Theorem~\ref{thm=itte}}\label{subsec=itte}
The following theorem corresponds to Step~3 in the outlined proof of Theorem~\ref{mthm=main}.

\begin{thm}\label{thm=itte}
Assume Settings~$\ref{setting=GLN}$, $\ref{setting=main1}$ and $\ref{setting=main3}$. Let $\PPh^{\RR}=\PPh^{\RR}_{(\nug_1,\ldots,\nug_{\ell})}\colon ([\Gg,\Ng],d_{\scl_{\Gg,\Ng}})\to (\RR^{\ell},\|\cdot\|_1)$ be the coarse homomorphism associated with $(\nug_1,\ldots ,\nug_{\ell})$ constructed in Theorem~$\ref{thm=evmap}$.
Take an arbitrary map $\rho\colon (\RR^{\ell},\|\cdot\|_1)\to (\ZZ^{\ell},\|\cdot\|_1)$ such that $\sup\limits_{\vu\in \RR^{\ell}}\|\vu-\rho(\vu)\|_1<\infty$ holds. Set $\PPh\colon ([\Gg,\Ng],d_{\scl_{\Gg,\Ng}})\to (\ZZ^{\ell},\|\cdot\|_1)$ as $\PPh=\rho\circ \PPh^{\RR}$. Let   $\PPs\colon (\ZZ^{\ell},\|\cdot\|_1)\to ([\Gg,\Ng],d_{\scl_{\Gg,\Ng}})$ be the coarse homomorphism constructed in Theorem~$\ref{thm=PPs}$. Then, the following hold true.
\begin{enumerate}[label=\textup{(}$\arabic*$\textup{)}]
  \item Let $A\subseteq [\Gg,\Ng]$ be a $d_{\scl_{\Gg,\Lg}}$-bounded set, and set $D_A'=\sup\{\scl_{\Gg,\Lg}(\yg)\,|\,\yg\in A\}$. Then, there exists $D\in \RR_{\geq 0}$, whose  dependence on $A$ only comes from  $D_A'$, such that for every $\yg\in A$, $ d_{\scl_{\Gg,\Ng}}(\yg,(\PPs\circ \PPh)(\yg))\leq D$ holds.
  \item We have $\PPh\circ \PPs\approx \mathrm{id}_{(\ZZ^{\ell},\|\cdot\|_1)}$.
\end{enumerate}
\end{thm}

\begin{proof}
Set $\kappa=\sup\limits_{\vu\in \RR^{\ell}}\|\vu-\rho(\vu)\|_1$. First we will prove (2). By Lemma~\ref{lem=nugalpha}~(2), we have
\[
\sup_{\vm\in \ZZ^{\ell}}\|\vm-(\PPh^{\RR}\circ \PPs)(\vm)\|_1 \leq (2 \ttil -1) \cdot \sum_{j\in \{1,\ldots ,\ell\}} \DD(\nug_j).
\]
By the definition of $\kappa$,  this implies that
\begin{equation}\label{eq=error2}
 \sup_{\vm\in \ZZ^{\ell}}\|\vm- (\PPh\circ \PPs)(\vm)  \|_1\leq \kappa+(2 \ttil -1) \cdot \sum_{i\in \{1,\ldots ,\ell\}} \DD(\nug_i);
\end{equation}
this \eqref{eq=error2} yields (2).

In what follows, we prove (1). Let $\mathscr{C}_{1,\mathrm{ctd}}$ and $\mathscr{C}_{2,\mathrm{ctd}}$ be the constants associated with $(\nug_1,\ldots,\nug_{\ell})$ appearing in Theorem~$\ref{thm=comparisonDD}$. Let $\nug\in \QQQ(\Ng)^{\Gg}$. Then, by Theorem~\ref{thm=comparisonDD}, there exist $\kg\in \HHH^1(\Ng)^{\Gg}$, $\psg\in \QQQ(\Lg)^{\Gg}$ and $(a_1,\ldots ,a_{\ell})\in \RR^{\ell}$ such that $\nug=\kg+i^{\ast}\psg+\sum\limits_{j\in \{1,\ldots,\ell\}}a_j\nug_j$ and
\[
\DD(\nug)\geq \mathscr{C}_{1,\mathrm{ctd}}{}^{-1}\left(\DD(\psg)+\mathscr{C}_{2,\mathrm{ctd}}{}^{-1}\cdot\sum_{j\in \{1,\ldots,\ell\}}|a_j|\right).
\]
By recalling Remark~\ref{rem=a_j}, we note that the inequality above in particular implies that
\begin{equation}\label{eq=a_j10}
\DD(\psg)\leq \mathscr{C}_{1,\mathrm{ctd}}\cdot \DD(\nug)\quad \textrm{and} \quad \sum_{j\in \{1,\ldots,\ell\}}|a_j| \leq \mathscr{C}_{1,\mathrm{ctd}}\mathscr{C}_{2,\mathrm{ctd}}\cdot \DD(\nug).
\end{equation}
Set $\PPh^{\RR}(\yg)=(u_1,\ldots ,u_{\ell})$ and $\rho((u_1,\ldots ,u_{\ell}))=(m_1,\ldots ,m_{\ell})$. Then, we have for every $j\in \{1,\ldots ,\ell\}$, $|m_j-u_j|\leq \kappa$. Theorem~\ref{thm=Bavard}, together with the definition of $D_A'$ and \eqref{eq=a_j10}, implies that
\begin{equation}\label{eq=yg}
|\psg(\yg)|\leq 2D_A'\mathscr{C}_{1,\mathrm{ctd}}\cdot \DD(\nug).
\end{equation}
By \eqref{eq=tt} (Theorem~\ref{thm=PPs}~(3)) and \eqref{eq=a_j10}, we also have
\begin{equation}\label{eq=yg2}
|\psg((\PPs\circ \PPh)(\yg))|\leq (2\ttil -1)\mathscr{C}_{1,\mathrm{ctd}}\cdot \DD(\nug).
\end{equation}
By Lemma~\ref{lem=nugalpha}~(2), for every $j\in \{1,\ldots,\ell\}$, $|\nug_j((\PPs\circ \PPh)(\yg))-m_j|\leq (2\ttil -1)\DD(\nug_j)$; hence,
\begin{equation}\label{eq=u_i}
|\nug_j((\PPs\circ \PPh)(\yg))-u_j|\leq \kappa+(2\ttil -1)\DD(\nug_j).
\end{equation}
By combining \eqref{eq=yg}, \eqref{eq=yg2}, \eqref{eq=u_i} and \eqref{eq=a_j10}, we obtain that
\begin{align*}
&|\nug(\yg^{-1}(\PPs\circ \PPh)(\yg)) |\\
\leq{} &\DD(\nug)+\left|\nug((\PPs\circ \PPh)(\yg))-\nug(\yg)\right|\\
\leq{} &\DD(\nug)+|\psg(\yg)|+\left|\psg((\PPs\circ \PPh)(\yg))\right|+\sum_{j\in \{1,\ldots,\ell\}} |a_j| \left|\nug_j((\PPs\circ \PPh)(\yg))-\nug_j(\yg)\right|\\
\leq{} &\DD(\nug)+(2D_A'+2\ttil-1)\mathscr{C}_{1,\mathrm{ctd}}\cdot \DD(\nug)+\sum_{j\in \{1,\ldots,\ell\}}|a_j|\bigl(\kappa +(2\ttil -1)\DD(\nug_j)\bigr)\\
\leq{} &\DD(\nug)+(2D_A'+2\ttil-1)\mathscr{C}_{1,\mathrm{ctd}}\cdot \DD(\nug)+ \left(\kappa +(2\ttil -1)\max_{i\in \{1,\ldots ,\ell\}}\DD(\nug_i)\right)\cdot \sum_{j\in \{1,\ldots,\ell\}}|a_j|\\
\leq{} &\DD(\nug)+\left\{\left(\kappa +(2\ttil-1)\max_{j\in \{1,\ldots ,\ell\}}\DD(\nug_j)\right)\mathscr{C}_{2,\mathrm{ctd}}+2D_A'+2\ttil-1\right\}\mathscr{C}_{1,\mathrm{ctd}}\cdot \DD(\nug).
\end{align*}
Therefore, by Theorem~\ref{thm=Bavard}, we conclude that
\begin{equation}\label{eq=error1}
 \sup_{\yg\in A} d_{\scl_{\Gg,\Ng}}(\yg,(\PPs\circ \PPh)(\yg))  
 \leq \frac12 \left\{ \left(\kappa +(2\ttil -1)\max\limits_{j\in \{1,\ldots ,\ell\}}\DD(\nug_j)\right)\mathscr{C}_{2,\mathrm{ctd}}+2D_A'+2\ttil -1 \right\}\cdot \mathscr{C}_{1,\mathrm{ctd}}+\frac{1}{2};
\end{equation}
this \eqref{eq=error1} yields (1). This completes the proof.
\end{proof}

\subsection{Proof of Theorem~\ref{mthm=main}}\label{subsec=generalform}

We formulate Theorem~\ref{thm=maingeneral}, a general form of Theorem~\ref{mthm=main}.

\begin{thm}[general form of Theorem~\ref{mthm=main}: coarse kernel in the abelian case]\label{thm=maingeneral}
Assume Settings~$\ref{setting=GLN}$, $\ref{setting=main1}$ and $\ref{setting=main3}$. Then, the following hold.
\begin{enumerate}[label=\textup{(}$\arabic*$\textup{)}]
\item There exist maps $\PPh\colon [\Gg,\Ng]\to \ZZ^{\ell}$ and $\PPs\colon  \ZZ^{\ell}\to [\Gg,\Ng]$ that satisfy the following properties.
\begin{enumerate}
  \item[\textup{(1-1)}] Two maps $\PPh\colon ([\Gg,\Ng],d_{\scl_{\Gg,\Ng}})\to (\ZZ^{\ell},\|\cdot\|_1)$ and $\PPs\colon  (\ZZ^{\ell},\|\cdot\|_1)\to( [\Gg,\Ng],d_{\cl_{\Gg,\Ng}})$ are both pre-coarse homomorphisms.
  \item[\textup{(1-2)}] The map $\PPs$ is $d_{\cl_{\Gg,\Lg}}$-bounded.
  \item[\textup{(1-3)}] Let $A\subseteq [\Gg,\Ng]$ be a $d_{\scl_{\Gg,\Lg}}$-bounded set and set $D_A=\diam_{d_{\scl_{\Gg,\Lg}}}(A)$. Then, there exist $C_1,C_2\in \RR_{>0}$ and $D_1,D_2\in \RR_{\geq 0}$ such that for every $\yg_1,\yg_2\in A$,
\[
C_1 \cdot d_{\scl_{\Gg,\Ng}}(\yg_1,\yg_2)-D_1\leq \|\PPh(\yg_1)-\PPh(\yg_2)\|_1\leq C_2 \cdot d_{\scl_{\Gg,\Ng}}(\yg_1,\yg_2)+D_2.
\]
Here, $C_1$, $C_2$ and $D_2$ can be taken to be independent of $A$; the dependence of $D_1$ on $A$ only comes from $D_A$.
 \item[\textup{(1-4)}] There exist $C'_1,C'_2\in \RR_{>0}$ and $D'_1,D'_2\in \RR_{\geq 0}$ such that for every $\vm,\vn\in \ZZ^{\ell}$,
\begin{align*}
C'_1 \cdot \|\vm-\vn\|_1-D'_1&\leq d_{\scl_{\Gg,\Ng}}(\PPs(\vm),\PPs(\vn))\\
&\leq  d_{\cl_{\Gg,\Ng}}(\PPs(\vm),\PPs(\vn))\leq C'_2 \cdot \|\vm-\vn\|_1+D'_2.
\end{align*}
  \item[\textup{(1-5)}]Let $A\subseteq [\Gg,\Ng]$ be a $d_{\scl_{\Gg,\Lg}}$-bounded set and set $D'_A=\sup\{\scl_{\Gg,\Lg}(\yg)\,|\,\yg\in A\}$. Then there exist $D_3\in \RR_{\geq 0}$, whose dependence on $A$ only comes from $D'_A$, such that
\[
\sup\limits_{\yg\in A}d_{\scl_{\Gg,\Ng}}(\yg,(\PPs\circ \PPh)(\yg))\leq D_3.
\]
\item[\textup{(1-6)}] We have $\PPh\circ \PPs\approx \mathrm{id}_{(\ZZ^{\ell},\|\cdot\|_1)}$.
\end{enumerate}
\item Furthermore, in \textup{(1)} we may take $\PPh$ and $\PPs$ in such a way that $\PPh\circ \PPs=\mathrm{id}_{\ZZ^{\ell}}$ holds.
\item In \textup{(1)}, the coarse subspace represented by $\PPs(\ZZ^{\ell})$ is the coarse kernel of $\iota_{(\Gg,\Lg,\Ng)}$. The map $\PPs$ gives a coarse group isomorphism $(\ZZ^{\ell},\|\cdot\|_1)\cong (\PPs(\ZZ^{\ell}),d_{\scl_{\Gg,\Ng}})$ by a quasi-isometry.
\end{enumerate}
\end{thm}

\begin{proof}
Item~(1) follows from Theorems~\ref{thm=evmap}, \ref{thm=PPs} and \ref{thm=itte}. For item~(2), set
\[
\tilde{D}=\left\lfloor 2(2 \ttil-1) \cdot \sum_{i\in \{1,\ldots ,\ell\}} \DD(\nug_i)\right\rfloor +1.
\]
Then, we can go along a line similar to the proof of Lemma~\ref{lem=sigmatau}. Finally, we will prove (3). Indeed, we will show that (1-1)--(1-6) imply (3). Item~(1-2) implies that $\PPs(\ZZ^{\ell})$ is $d_{\scl_{\Gg,\Lg}}$-bounded. Furthermore, let $A\subseteq [\Gg,\Ng]$ be a $d_{\scl_{\Gg,\Lg}}$-bounded set. Then, by (1-5), $A$ is coarsely contained in $(\PPs\circ \PPh)(A)$. Since $(\PPs\circ \PPh)(A)\subseteq \PPs(\ZZ^{\ell})$, it follows that $A$ is coarsely contained in $\PPs(\ZZ^{\ell})$. Therefore, the coarse subspace represented by  $\PPs(\ZZ^{\ell})$ is the coarse kernel of $\iota_{(\Gg,\Lg,\Ng)}$ (recall Example~\ref{exa=CKiota}). By (1-1), (1-3), (1-4), (1-5) and (1-6), $\PPs$ and $\PPh|_{\PPs(\ZZ^{\ell})}$ are coarse group isomorphisms between $(\ZZ^{\ell},\|\cdot\|_1)$ and $(\PPs(\ZZ^{\ell}),d_{\scl_{\Gg,\Ng}})$; hence (3) holds.
\end{proof}

\begin{proof}[Proof of Theorem~\textup{\ref{mthm=main}}]
By Lemma~\ref{lem=maintogeneral}, Theorem~\ref{mthm=main} for  $\ell\in \NN$ is a special case of Theorem~\ref{thm=maingeneral}. If $\ell=0$, then by Theorem~\ref{thm=biLip} the coarse kernel of $\iota_{(\Gg,\Lg,\Ng)}$ is trivial, and  Theorem~\ref{mthm=main} also holds.
\end{proof}

\subsection{Proof of Theorem~\ref{mthm=dim}}\label{subsec=dim}

The following theorem is a general form of Theorem~\ref{mthm=dim}.

\begin{thm}[general form of Theorem~\ref{mthm=dim}: coarse group theoretic characterization of $\Rdim\Wcal(\Gg,\Lg,\Ng)$ in the abelian case]\label{thm=dimgeneral}
Assume Settings~$\ref{setting=GLN}$ and $\ref{setting=main1}$. Then,
\begin{align*}
\Rdim \Wcal(\Gg,\Lg,\Ng)  ={} &\sup\left\{\ell\in \ZZ_{\geq 0}\,\middle|\,\begin{gathered}\exists \mathrm{\ coarsely\ proper\ }d_{\scl_{\Gg,\Lg}}\textrm{-}\mathrm{bounded\ coarse\ homomorphism}\\ (\ZZ^{\ell},\|\cdot\|_1)\to ([\Gg,\Ng],d_{\scl_{\Gg,\Ng}})
\end{gathered}\right\} \\
={} &\inf\left\{\ell\in \ZZ_{\geq 0}\,\middle|\,\begin{gathered}\forall A\subseteq [\Gg,\Ng]\ d_{\scl_{\Gg,\Lg}}\textrm{-}\mathrm{bounded;}\\
\exists \mathrm{\ coarsely\ proper\ coarse\ homomorphism}\ (A,d_{\scl_{\Gg,\Ng}})\to (\ZZ^{\ell},\|\cdot\|_1)
\end{gathered}\right\}
\end{align*}
and
\begin{equation*}
\Rdim \Wcal(\Gg,\Lg,\Ng)
\leq \sup\left\{\ell\in \ZZ_{\geq 0}\,\middle|\, \begin{gathered}\exists \mathrm{\ coarsely\ proper\ }d_{\cl_{\Gg,\Lg}}\textrm{-}\mathrm{bounded\ coarse\ homomorphism}\\ (\ZZ^{\ell},\|\cdot\|_1)\to ([\Gg,\Ng],d_{\cl_{\Gg,\Ng}})
\end{gathered}\right\}.
\end{equation*}
\end{thm}

\begin{proof}[Proofs of Theorem~$\ref{thm=dimgeneral}$ and Theorem~\textup{\ref{mthm=dim}}]
First we prove Theorem~\ref{thm=dimgeneral}. The proof goes along a line similar to that of Proposition~\ref{prop=metricdim}, which employs Theorem~\ref{thm=asdimZ} and Proposition~\ref{prop=asdimCE}. More precisely, if $\Wcal(\Gg,\Lg,\Ng)$ is infinite dimensional, then the assertions of Theorem~\ref{thm=dimgeneral} follow from Theorem~\ref{thm=PPs}. If $\Wcal(\Gg,\Lg,\Ng)$ is finite dimensional, then Theorem~\ref{thm=maingeneral} implies the conclusions. This ends our proof of Theorem~\ref{thm=dimgeneral}.

Now, Theorem~\ref{mthm=dim} immediately follows from Theorem~\ref{thm=dimgeneral} and Lemma~\ref{lem=maintogeneral}.
\end{proof}

We present the following immediate corollary to Theorem~\ref{thm=dimgeneral}.

\begin{cor}\label{cor=notbilip}
Assume Settings~$\ref{setting=GLN}$ and $\ref{setting=main1}$. Assume that $\Wcal(\Gg,\Lg,\Ng)\ne 0$. Then, $\scl_{\Gg,\Lg}$ and $\scl_{\Gg,\Ng}$ are \emph{not} bi-Lipschitzly equivalent on $[\Gg,\Ng]$.
\end{cor}

\section{Examples of $(\Gg,\Lg,\Ng)$ for the main theorems}\label{sec=examplemain}

In this section, we collect examples of triples $(\Gg,\Lg,\Ng)$ to which Theorem~\ref{mthm=main} and Theorem~\ref{mthm=dim} apply. Before presenting such examples in Subsections~\ref{subsec=WW} and \ref{subsec=Wcal}, we discuss in Subsection~\ref{subsec=boundt} a variant $\PPs^{\flat}$ of the map $\PPs$ in the statement of Theorem~\ref{thm=maingeneral}, which is a modification of the map $\PPs$ constructed in Theorem~\ref{thm=PPs}. This map $\PPs^{\flat}$ enables us to obtain some examples with completely explicit coarse kernels in Subsection~\ref{subsec=explicitkernel}, including Proposition~\ref{prop=explicitexample}. In Subsection~\ref{subsec=application_cod}, we present an application of Theorem~\ref{thm=comparisonDD}, which supplies a subset on which $\scl_{\Gg,\Lg}$ and $\scl_{\Gg,\Ng}$ are bi-Lipschitzly equivalent.

\subsection{A variant $\PPs^{\flat}$ of the map $\PPs$}\label{subsec=boundt}
The map $\PPs$ in the statement of Theorem~\ref{thm=maingeneral} (and Theorem~\ref{mthm=main}) is constructed in Theorem~\ref{thm=PPs}. In this subsection, we discuss a slight modification of this construction; this provides a variant $\PPs^{\flat}$ of the map $\PPs$, which still serves in the role of $\PPs$ in Theorem~\ref{thm=maingeneral} (and Theorem~\ref{mthm=main}).

\begin{prop}\label{prop=boundt}
Assume Settings~$\ref{setting=GLN}$, $\ref{setting=main1}$ and $\ref{setting=main2}$. For every $i\in \{1,\ldots,\ell\}$, take $\wf_i^{\flat}\in [\Ff,\Kf]$ such that $\wf_i^{\flat}$ represents the same equivalence class as $\wf_i$ in $[\Ff,\Kf]/[\Ff,\Mf]$; take an $[\Ff,\Kf]$-expression $\bffi^{\flat}=(\ff_1^{(i) \flat},\ldots,\ff_{s_i}^{(i) \flat}; \ff'_1{}^{(i) \flat},\ldots,\ff'_{s_i}{}^{(i) \flat})$ for $\wf_i^{\flat}$. Define a map $\PPs^{\flat}\colon \ZZ^{\ell}\to [\Gg,\Ng]$  by
\[
\PPs^{\flat}(\vm)=\beta_{\bffone^{\flat}}(m_1)\cdots\beta_{\bffell^{\flat}}(m_{\ell})
\]
for every $\vm=(m_1,\ldots,m_{\ell})\in \ZZ^{\ell}$, where $\beta_{\bffi^{\flat}}=\ppi\circ \alpha_{\bffi^{\flat}}$   and $\alpha_{\bffi^{\flat}}$ is the map defined in Definition~$\ref{defn=Pspower}$ for every $i\in \{1,\ldots,\ell\}$.
Then, $\PPs^{\flat}(\ZZ^{\ell})\subseteq [\Gg,\Ng]$. In particular, we can regard $\PPs^{\flat}$ as a map from $\ZZ^{\ell}$ to $[\Gg,\Ng]$. Moreover, for the map $\PPh\colon [\Gg,\Ng]\to \ZZ^{\ell}$ constructed in Theorem~$\ref{thm=itte}$ for $(\nug_1,\ldots,\nug_{\ell})$, the pair $(\PPh,\PPs^{\flat})$ satisfies the assertions \textup{(1)}--\textup{(7)} in Theorem~{\textup{\ref{mthm=main}}} with $\PPs$ being replaced with $\PPs^{\flat}$.
\end{prop}

\begin{proof}
Let $\vm=(m_1,\ldots,m_{\ell})\in \ZZ^{\ell}$. Then, for every $i\in \{1,\ldots,\ell\}$, we have by Proposition~\ref{prop=alpha}~(2)
\[
\alpha_{\bffi^{\flat}}(m_i)\in \alpha_{\bffi^{\flat}}(1)^{m_i}[\Ff,\Mf] =(\wf_i^{\flat})^{m_i}[\Ff,\Mf]=\wf_i^{m_i}[\Ff,\Mf].
\]
Since $\ppi(\wf_i)=e_{\Gg}$, this implies that $\beta_{\bffi^{\flat}}(m_i)\in [\Gg,\Ng]$. Hence, we have $\PPs^{\flat}(\ZZ^{\ell})\subseteq [\Gg,\Ng]$.

Now we proceed to the proof of the latter assertion. The key here is the following observation: for every $i,j\in \{1,\ldots,\ell\}$ and for every core $\hf_j$ of $\nug_j$,
\begin{equation}\label{eq=deltasharp}
\hf_j(\wf_i^{\flat})=\delta_{i,j}.
\end{equation}
Indeed, since $\hf_j\in \HHH^1(\Mf)^{\Ff}$, the value $\hf_j(\vf)$ for  $\vf\in \Mf$ only depends on $\vf[\Ff,\Mf]\in \Mf/[\Ff,\Mf]$. Hence, $\hf_j(\wf_i^{\flat})=\hf_j(\wf_i)=\CEt_{\nug_j}(\wf_i)=\delta_{i,j}$, as desired. By \eqref{eq=deltasharp}, our proof goes along arguments similar to ones of Theorems~\ref{thm=PPs} and \ref{thm=itte} except one of Theorem~\ref{thm=PPs}~(2): there, for every $i\in \{1,\ldots,\ell\}$ and for every $m\in \NN$ we argued
\[
\beta_{\bffi}(m)\mathrel{\overset{(\Gg,\Ng)}{\underset{(m-1)(2t_i-1)}{\eqsim}}}\beta_{\bffi}(1)^m=e_{\Gg}.
\]
We can modify this part by arguing that
\[
\beta_{\bffi^{\flat}}(m)\mathrel{\overset{(\Gg,\Ng)}{\underset{(m-1)(2s_i-1)}{\eqsim}}}\beta_{\bffi^{\flat}}(1)^m \mathrel{\overset{(\Gg,\Ng)}{\underset{m\cdot\cl_{\Gg,\Ng}(\wf_i^{\flat})}{\eqsim}}} e_{\Gg}.
\]
Then, we may obtain the corresponding statements of (1)--(7) in Theorem~\ref{mthm=main} to the pair $(\PPh,\PPs^{\flat})$.
\end{proof}

\begin{lem}\label{lem=sharp}
Let $n\in \NN_{\geq 2}$ and $F_n$ be the free group of rank $n$. Then for every $\vf\in \gamma_2(F_n)$, there exists $\vf^{\flat}\in \gamma_2(F_n)$ such that the following two conditions are satisfied:
\begin{enumerate}[label=\textup{(\arabic*)}]
  \item the element $\vf^{\flat}$ represents the same equivalence class as $\vf$ in $\gamma_2(F_n)/\gamma_3(F_n)$; and
  \item  $\cl_{F_n}(\vf^{\flat})\leq n-1$.
\end{enumerate}
\end{lem}

\begin{proof}
Fix a free basis $\{\tilde{a}_1,\ldots ,\tilde{a}_n\}$ of $F_n$. Then by Lemma~\ref{lem=commutatoreq}, there exists $\overline{\vf}$ of the form
\[
\overline{\vf}=[\tilde{a}_{1},\tilde{a}_{2}^{m_{(1,2)}}]\cdots [\tilde{a}_{1},\tilde{a}_{n}^{m_{(1,n)}}][\tilde{a}_{2},\tilde{a}_{3}^{m_{(2,3)}}]\cdots [\tilde{a}_{2},\tilde{a}_{n}^{m_{(2,n)}}]\cdots [\tilde{a}_{n-1},\tilde{a}_{n}^{m_{(n-1,n)}}],
\]
where $m_{(i,j)}\in \ZZ$ for every $(i,j)\in \ZZ^2$ with $1\leq i<j\leq n$, such that $\overline{\vf}$ represents the same equivalence class as $\vf$ in $\gamma_2(F_n)/\gamma_3(F_n)$. Again by Lemma~\ref{lem=commutatoreq}, for such $\overline{\vf}$ we can take
\[
\vf^{\flat}=[\tilde{a}_{1},\tilde{a}_{2}^{m_{(1,2)}}\cdots \tilde{a}_{n}^{m_{(1,n)}}][\tilde{a}_{2},\tilde{a}_{3}^{m_{(2,3)}}\cdots \tilde{a}_{n}^{m_{(2,n)}}]\cdots [\tilde{a}_{n-1},\tilde{a}_{n}^{m_{(n-1,n)}}]
\]
such that $\vf^{\flat}$ represents the same equivalence class as $\overline{\vf}$ in $\gamma_2(F_n)/\gamma_3(F_n)$. This $\vf^{\flat}$ works.
\end{proof}

In the setting of Theorem~\ref{thm=main1precise}, we moreover assume that  $\Gg$ is finitely generated and that $\Ng=\gamma_2(\Gg)$. Set  $n\in \NN$  as the number of elements of a fixed generating set of $\Gg$; construct the map $\PPs^{\flat}$ in Proposition~\ref{prop=boundt} for carefully chosen $\bffone^{\flat},\ldots,\bffell^{\flat}$ from Lemma~\ref{lem=sharp}. Then, we can take this map $\PPs^{\flat}$ as the map $\PPs$ in the statement of Theorem~\ref{thm=main1precise}; this map in addition satisfies that
\[
\sup\limits_{\vm\in \ZZ^{\ell}}\cl_{\Gg}(\PPs^{\flat}(\vm))\leq \ell(n-1).
\]
In contrast, even in this case, in Settings~\ref{setting=GLN}, \ref{setting=main1} and \ref{setting=main3} we cannot hope for any bounds, depending only on $n$,  on $t_1,\ldots,t_{\ell}$.

\subsection{Examples with explicit coarse kernels}\label{subsec=explicitkernel}
In Theorem~\ref{thm=maingeneral}, the image $\PPs(\ZZ^{\ell})$ might not be fully explicit in general; when we take $\wf_1,\ldots ,\wf_{\ell}$ and $\nug_1,\ldots,\nug_{\ell}$ by applying Corollary~\ref{cor=InjCE}, the existence of them is ensured by the theory of core extractors, but they are not constructive in general. Nevertheless, under certain conditions, we may take explicit $\wf_1,\ldots ,\wf_{\ell}$; recall Corollary~\ref{cor=explicitkernel}. Together with Proposition~\ref{prop=boundt}, we can construct the variant $\PPs^{\flat}$ of $\PPs$ in an explicit manner as well.

\begin{prop}\label{prop=explicitck}
Assume Setting~$\ref{setting=GLN}$. Let $q\in \NN_{\geq 2}$. Let $\ell\in \NN$. Assume that $\Rdim \WW(\Gg,\gamma_{q-1}(\Gg),\gamma_q(\Gg))= \ell$.  Assume that $\Gg$ admits a presentation $(F\,|\,R)$ such that $R$ is normally generated in $F$ by $\ell$ elements. Let $\ppi\colon F\twoheadrightarrow F/R\cong \Gg$ be the natural group quotient map. Let $r_1,\ldots,r_{\ell}\in F$ be normal generators of $R$ in $F$. Assume that $R\leqslant \gamma_q(F)$. Let $r_1^{\flat},\ldots,r_{\ell}^{\flat}$ be elements in $\gamma_q(F)$ such that for every $i\in \{1,\ldots,\ell\}$, $r_i^{\flat}\gamma_{q+1}(F)=r_i\gamma_{q+1}(F)$ holds. Fix $[F,\gamma_{q-1}(F)]$-expressions $\bffone^{\flat}=(f_1^{(1) \flat},\ldots,f_{s_1}^{(1) \flat};f'_1{}^{(1) \flat},\ldots,f'_{s_1}{}^{(1) \flat}),$ $\ldots,$ $\bffell^{\flat}=(f_1^{(\ell) \flat},\ldots,f_{s_{\ell}}^{(\ell) \flat};f'_{1}{}^{(\ell) \flat},\ldots,f'_{s_{\ell}}{}^{(\ell) \flat})$ of $r_1^{\flat},\ldots,r_{\ell}^{\flat}$, respectively. Then, the coarse subspace represented by the set
\[
A^{\flat}=\{\beta_{\bffone^{\flat}}(m_1)\cdots \beta_{\bffell^{\flat}}(m_{\ell})\,|\,(m_1,\ldots,m_{\ell})\in \ZZ^{\ell}\}\subseteq \gamma_{q+1}(\Gg)
\]
is the coarse kernel of the map $\iota_{\Gg,\gamma_q(\Gg)}$, and the map
\[
\PPs^{\flat}\colon (\ZZ^{\ell},\|\cdot\|_1)\to (A^{\flat},d_{\scl_{\Gg,\gamma_q(\Gg)}});\quad (m_1,\ldots,m_{\ell})\mapsto \beta_{\bffone^{\flat}}(m_1)\cdots \beta_{\bffell^{\flat}}(m_{\ell})
\]
gives a coarse group isomorphism. Here, for every $i\in\{1,\ldots,\ell\}$, we set $\beta_{\bffi^{\flat}}=\pi\circ \alpha_{\bffi^{\flat}}$.
\end{prop}

\begin{proof}
Combine Corollary~\ref{cor=explicitkernel} and Proposition~\ref{prop=boundt}. Here, note that $\ppi^{-1}(\gamma_q(\Gg))=\gamma_q(\Ff)$ since $\Rel\leqslant \gamma_q(\Ff)$.
\end{proof}

\begin{exa}[examples with explicit coarse kernels]\label{exa=explicitker}
Here we exhibit examples to which Proposition~\ref{prop=explicitck} applies with $q=2$; these examples are based on Theorem~\ref{thm=WW} (recall terminology in Theorem~\ref{thm=WW} from discussions above Theorem~\ref{thm=WW}). Recall also our formulation of group presentations from Definition~\ref{defn=pres}. For a set $S$, let $F(S)$ denote the free group of free basis $S$. For a free group $F$ and $\mathcal{R}\subseteq F$, let $\llangle \mathcal{R}\rrangle=\llangle \mathcal{R}\rrangle_F$ denote the normal closure of $\mathcal{R}$ in $F$. 
\begin{enumerate}[label=\textup{(\arabic*)}]
  \item Let $\genus\in \NN_{\geq 2}$, and consider the surface group $\pi_1(\Sigma_{\genus})$ of genus $\genus$. Then, $\pi_1(\Sigma_{\genus})$ admits a group presentation
\[
\pi_1(\Sigma_{\genus})=(F\,|\,R)=(F(\tilde{a}_1,\ldots,\tilde{a}_{\genus},\tilde{b}_1,\ldots,\tilde{b}_{\genus})\,|\, \llangle r\rrangle),
\]
where $r=[\tilde{a}_1,\tilde{b}_1]\cdots [\tilde{a}_{\genus},\tilde{b}_{\genus}]$. Let $\ppi\colon F\twoheadrightarrow F/R\cong \pi_1(\Sigma_{\genus})$ be the natural group quotient map, and let $a_1,\ldots,b_{\genus}$ be the images of $\tilde{a}_1,\ldots,\tilde{b}_{\genus}$ by $\ppi$, respectively.
Since $r\in \gamma_2(F)$ and $\Rdim \WW(\Gg,\gamma_2(\Gg))=1$ (Theorem~\ref{thm=WW}~(2)), the coarse subspace represented by the set
\[
A=\{[a_1,b_1^m]\cdots [a_{\genus},b_{\genus}^m]\,|\, m\in \ZZ\}
\]
is the coarse kernel of $\iota_{\Gg,\gamma_2(\Gg)}$ for $\Gg=\pi_1(\Sigma_{\genus})$.
\item Let $F_n$ be a free group of rank $n\in \NN_{\geq 2}$. Recall the Andreadakis--Johnson filtration \eqref{eq=Andreadakis}. Let $\chi\in \Aut(F_n)$ be an atoroidal automorphism. Consider the semi-direct product $F_n\rtimes_{\chi}\ZZ$. Then, this group admits a group presentation
\[
F_n\rtimes_{\chi}\ZZ=(F\,|\,R)=(F(\tilde{a}_1,\ldots,\tilde{a}_{n},\tilde{c})\,|\, \llangle r_1,\ldots,r_{n}\rrangle);
\]
for $i\in \{1,\ldots,n\}$, set $r_i=\tilde{c}\tilde{a}_i\tilde{c}^{-1}\chi(\tilde{a}_i)^{-1}$. Here, $\chi$ can be seen as an automorphism on $F(\tilde{a}_1,\ldots,\tilde{a}_{n})\cong F(a_1,\ldots,a_n)=F_n$; for the natural group quotient map $\ppi\colon F\twoheadrightarrow F/R\cong F_n\rtimes_{\chi}\ZZ$, set $a_1,\ldots,a_{n},c$ as the images of $\tilde{a}_1,\ldots,\tilde{a}_{n},\tilde{c}$ under $\ppi$, respectively.  (The element $c$ corresponds to $1\in \ZZ$.)

Now, we assume that $\chi$ lies in $\mathcal{A}_n(2)$. Since $\chi\in \mathcal{A}_n(2)\leqslant \mathcal{A}_n(1)=\IAA_n$, for every $i\in \{1,\ldots, n\}$ we in particular have $\tilde{a}_i\gamma_2(F)=\chi(\tilde{a}_i)\gamma_2(F)$. Hence, we have $R\leqslant \gamma_2(F)$. Together with Theorem~\ref{thm=WW}~(4), we can apply Proposition~\ref{prop=explicitck} to this setting.

Since $\chi\in \mathcal{A}_n(2)$, for every $i\in \{1,\ldots, n\}$ we in fact have $\tilde{a}_i\gamma_3(F)=\chi(\tilde{a}_i)\gamma_3(F)$. Hence, in the setting of Proposition~\ref{prop=explicitck}, we can take $r_1^{\flat},\ldots,r_n^{\flat}$ as
\[
r_i^{\flat}=[\tilde{c},\tilde{a}_i]  \quad\textrm{for every $i\in \{1,\ldots,n\}$}.
\]
Therefore, the coarse subspace represented by the set
\begin{align*}
A^{\flat}&=\left\{[c, a_1^{m_1}]\cdots [c,a_n^{m_{n}}]\,\middle|\, (m_1,\ldots,m_{n})\in \ZZ^{n}\right\}\\
&=\left\{\chi(a_1)^{m_1}a_1^{-m_1}\cdots \chi(a_n)^{m_{n}}a_n^{-m_n}\,\middle|\, (m_1,\ldots,m_{n})\in \ZZ^{n}\right\}
\end{align*}
is the coarse kernel of $\iota_{\Gg,\gamma_2(\Gg)}$ for $\Gg=F_n\rtimes_{\chi}\ZZ$.
\item Here, we discuss what happens  in the setting of (2) if we only assume that the atoroidal automorphism $\chi\in \Aut(F_n)$ belongs to $\mathcal{A}_n(1)=\IAA_n$. In this case, for every $i\in \{1,\ldots,n\}$ we have 
\[
\tilde{a}_i\chi(\tilde{a}_i)^{-1}\in \gamma_2(F(\tilde{a}_1,\ldots,\tilde{a}_{n})).
\]
Let $i\in \{1,\ldots,n\}$. We fix an $[F(\tilde{a}_1,\ldots,\tilde{a}_{n}),F(\tilde{a}_1,\ldots,\tilde{a}_{n})]$-expression 
\[
(\ff_1^{(i)},\ldots ,\ff_{t_i}^{(i)};\ff'_1{}^{(i)},\ldots ,\ff'_{t_i}{}^{(i)})
\]
for $\tilde{a}_i\chi(\tilde{a}_i)^{-1}$. Then, we have an $[F,F]$-expression 
\[
\bffi=(\tilde{c},\ff_1^{(i)},\ldots ,\ff_{t_i}^{(i)};\tilde{a}_i,\ff'_1{}^{(i)},\ldots ,\ff'_{t_i}{}^{(i)})
\]
for $r_i$. We note that the map $\beta_{\bff_i}\colon \ZZ\to \gamma_3(G)$ equals the map
\[
m\mapsto \chi(a_i)^{m}a_i^{-m}[\pi(\ff_1^{(i)}),\pi(\ff'_1{}^{(i)})^m]\cdots [\pi(\ff_{t_i}^{(i)}),\pi(\ff'_{t_i}{}^{(i)})^m].
\]
By Proposition~\ref{prop=explicitck}, the coarse subspace represented by the set
\[
A=\left\{\beta_{\bffone}(m_1)\cdots \beta_{\bff_n}(m_{n})\,\middle|\, (m_1,\ldots,m_{n})\in \ZZ^{n}\right\}
\]
is the coarse kernel of $\iota_{\Gg,\gamma_2(\Gg)}$ for $\Gg=F_n\rtimes_{\chi}\ZZ$.
  \item We stick to the setting of (1). Similar to the Andreadakis--Johnson filtration of $\Aut(F_n)$, the mapping class group $\Mod(\Sigma_{\genus})\cong \mathrm{Out}_+(\pi_1(\Sigma_{\genus}))$ admits the following filtration. For $q\in \NN$, let $\Mod_{(q)}(\Sigma_{\genus})$ be the kernel of the natural group homomorphism $\mathrm{Out}(\pi_1(\Sigma_{\genus}))\to \mathrm{Out}(\pi_1(\Sigma_{\genus})/\gamma_{q+1}(\pi_1(\Sigma_{\genus})))$; this yields the filtration
\begin{equation}\label{eq=Johnson}
\Mod(\Sigma_{\genus}) \geqslant \Mod_{(1)}(\Sigma_{\genus})\geqslant \Mod_{(2)}(\Sigma_{\genus})\geqslant \cdots\geqslant \Mod_{(q)}(\Sigma_{\genus})\geqslant  \cdots.
\end{equation}
We note that $\Mod_{(1)}(\Sigma_{\genus})$ equals the Torelli group $\mathcal{I}(\Sigma_{\genus})$. Also, it follows from the work \cite{Johnson} of Johnson that $\Mod_{(2)}(\Sigma_{\genus})$ equals the Johnson kernel $\mathcal{K}(\Sigma_{\genus})$, provided that $\genus\geq 3$.

Now, we assume that the element $[\chi]$ in $\Mod(\Sigma_{\genus})$ represented by $\chi\in \Aut_+(\pi_1(\Sigma_{\genus}))$ is a pseudo-Anosov element. Moreover, we assume that this element $[\chi]$ belongs to the group $\Mod_{(2)}(\Sigma_{\genus})$ appearing in the filtration \eqref{eq=Johnson}. In particular, the latter assumption implies that $[\chi]\in \mathcal{I}(\Sigma_{\genus})$. Consider the semi-direct product $\pi_1(\Sigma_{\genus})\rtimes_{\chi}\ZZ$. 

Observe that  the following diagram commutes.
\[
\begin{tikzpicture}[auto]
\node (a) at (0, 2) {$\Aut(\pi_1(\Sigma_{\genus}))$}; 
\node (x) at (4, 2) {$\Aut\Big(\pi_1(\Sigma_{\genus})/\gamma_3(\pi_1(\Sigma_{\genus}))\Big)$};
\node (b) at (0, 0) {$\Out(\pi_1(\Sigma_{\genus}))$};   
\node (y) at (4, 0) {$\Out\Big(\pi_1(\Sigma_{\genus})/\gamma_3(\pi_1(\Sigma_{\genus}))\Big)$};
\node (z) at (2, 1) {$\circlearrowleft$};
\draw[->] (a) to node {$\rho$} (x);
\draw[->] (x) to (y);
\draw[->] (a) to (b);
\draw[->] (b) to (y);
\end{tikzpicture}
\]
Since $[\chi]$ lies in $\Mod_{(2)}(\Sigma_{\genus})$, we have $\rho(\chi)\in \Inn (\pi_1(\Sigma_{\genus})/\gamma_{3}(\pi_1(\Sigma_{\genus})))$. Take $\sigma\in \pi_1(\Sigma_{\genus})/\gamma_{3}(\pi_1(\Sigma_{\genus}))$ such that $\rho(\chi)$ equals the inner conjugation of $\sigma$. Take an arbitrary element $s\in \pi_1(\Sigma_{\genus})$ that is sent to $\sigma$ by the group quotient map $\pi_1(\Sigma_{\genus})\twoheadrightarrow\pi_1(\Sigma_{\genus})/\gamma_{3}(\pi_1(\Sigma_{\genus}))$. Then, for every $l\in \pi_1(\Sigma_{\genus})$, we have
\[
\chi(l)\gamma_3(\pi_1(\Sigma_{\genus}))=sls^{-1}\gamma_3(\pi_1(\Sigma_{\genus})).
\]
Define $\chi'\in \Aut_+(\pi_{1}(\Sigma_g))$ as the composition of the inner conjugation of $s^{-1}$ and the automorphism $\chi$. Then, $[\chi']=[\chi]$ in $\Mod(\Sigma_g)$, and for every $l\in \pi_1(\Sigma_{\genus})$, we have
\begin{equation}\label{eq=inngamma_3}
\chi'(l)\gamma_3(\pi_1(\Sigma_{\genus}))=l\gamma_3(\pi_1(\Sigma_{\genus})).
\end{equation}

Since $[\chi']=[\chi]$ in $\Mod(\Sigma_g)\cong \mathrm{Out}_+(\pi_1(\Sigma_{\genus}))$, the group $\pi_1(\Sigma_{\genus})\rtimes_{\chi'}\ZZ$ is isomorphic to $\pi_1(\Sigma_{\genus})\rtimes_{\chi}\ZZ$. Hence, we set $G=\pi_1(\Sigma_{\genus})\rtimes_{\chi'}\ZZ$, and hereafter we treat $G$ instead of $\pi_1(\Sigma_{\genus})\rtimes_{\chi}\ZZ$. Let $c$ be the element in $G$ corresponding to $1\in \ZZ$. Consider the free group $F=F(\tilde{a}_1,\ldots,\tilde{a}_{\genus},\tilde{b}_1,\ldots,\tilde{b}_{\genus},\tilde{c})$. Let $\pi\colon F\twoheadrightarrow G$ be the group quotient map sending $\tilde{a}_1,\ldots,\tilde{b}_{\genus},\tilde{c}$ to $a_1,\ldots,b_{\genus},c$, respectively. Then, by \eqref{eq=inngamma_3}, we can take 
\[
\widetilde{\chi'}_{\tilde{a}_1},\ldots,\widetilde{\chi'}_{\tilde{b}_{\genus}}\in F(\tilde{a}_1,\ldots,\tilde{a}_{\genus},\tilde{b}_1,\ldots,\tilde{b}_{\genus})
\]
in such a way that for each $i\in \{1,\ldots,\genus\}$,
\begin{equation}\label{eq=surfacegamma_3}
\widetilde{\chi'}_{\tilde{a}_i} \gamma_3(F)=\tilde{a}_i\gamma_3(F)\quad \textrm{and}\quad \widetilde{\chi'}_{\tilde{b}_i} \gamma_3(F)=\tilde{b}_i\gamma_3(F)
\end{equation}
as well as
\[
\pi (\widetilde{\chi'}_{\tilde{a}_i})=\chi'(a_i)\quad \textrm{and}\quad \pi (\widetilde{\chi'}_{\tilde{b}_i})=\chi'(b_i).
\]
Then, this group $G=\pi_1(\Sigma_{\genus})\rtimes_{\chi'}\ZZ$ admits a group presentation
\[
G=(F\,|\,R)=(F(\tilde{a}_1,\ldots,\tilde{a}_{\genus},\tilde{b}_1,\ldots,\tilde{b}_{\genus},\tilde{c})\,|\, \llangle r_1,r_2,\ldots,r_{2\genus},r_{2\genus+1}\rrangle),
\]
where $r_i=\tilde{c}\tilde{a}_i\tilde{c}^{-1}\widetilde{\chi'}_{\tilde{a}_i}^{-1}$ and $r_{\genus+i}=\tilde{c}\tilde{b}_i\tilde{c}^{-1}\widetilde{\chi'}_{\tilde{b}_i}^{-1}$ for $i\in \{1,\ldots,\genus\}$, and $r_{2\genus +1}=[\tilde{a}_1,\tilde{b}_1]\cdots [\tilde{a}_{\genus},\tilde{b}_{\genus}]$.

By Theorem~\ref{thm=WW}~(3), we can apply Proposition~\ref{prop=explicitck} to this setting. More precisely, by \eqref{eq=surfacegamma_3} we can take $r_1^{\flat},\ldots,r_{2\genus}^{\flat},r_{2\genus +1}^{\flat}$ as $r_{2\genus +1}^{\flat}=r_{2\genus +1}$ and 
\[
r_i^{\flat}=[\tilde{c},\tilde{a}_i]\quad \textrm{and}\quad r_{\genus+i}^{\flat}=[\tilde{c},\tilde{b}_i]
\]
for every $i\in \{1,\ldots,\genus\}$. Therefore, by Proposition~\ref{prop=explicitck} we conclude that the coarse subspace represented by the set
\[
A^{\flat}=\left\{\chi'(a_1)^{m_1}a_1^{-m_1}\cdots \chi'(b_{\genus})^{m_{2\genus}}b_{\genus}^{-m_{2\genus}}[a_1,b_1^{m_{2\genus+1}}]\cdots [a_{\genus},b_{\genus}^{m_{2\genus+1}}]\,\middle|\, (m_1,\ldots,m_{2\genus+1})\in\ZZ^{2\genus+1}\right\}
\]
is the coarse kernel of $\iota_{\Gg,\gamma_2(\Gg)}$ for $\Gg=\pi_1(\Sigma_{\genus})\rtimes_{\chi'}\ZZ$.
\end{enumerate}
\end{exa}

\begin{proof}[Proof of Proposition~$\ref{prop=explicitexample}$]
Immediate from Example~\ref{exa=explicitker}~(1) and (2).
\end{proof}

\begin{rem}\label{rem=appli}
In Example~\ref{exa=explicitker}~(2), we can replace the system of relators $(r_1,\ldots,r_n)$ with another system $(r'_1,\ldots,r'_n)$, where for every $i\in \{1,\ldots,n\}$, we set $r'_i=\chi(\tilde{a}_i)^{-1}\tilde{c}\tilde{a}_i\tilde{c}^{-1}$. In this case, for every $i\in \{1,\ldots,n\}$ we can take $r'_i{}^{\flat}$  as $r'_i{}^{\flat}=[\tilde{a}_i^{-1},\tilde{c}]$. From these $r'_1{}^{\flat},\ldots,r'_n{}^{\flat}$, we obtain another representative of the coarse kernel of  $\iota_{\Gg,\gamma_2(\Gg)}$ for $\Gg=F_n\rtimes_{\chi}\ZZ$:
\[
A'{}^{\flat}=\left\{a_1^{-1}\chi^{m_1}(a_1)\cdots a_n^{-1}\chi^{m_n}(a_n)\,\middle|\, (m_1,\ldots,m_{n})\in \ZZ^{n}\right\}.
\]
Similarly, in Example~\ref{exa=explicitker}~(4) we obtain another representative $A'{}^{\flat}$ of the coarse kernel of  $\iota_{\Gg,\gamma_2(\Gg)}$ for $\Gg=\ppi_1(\Sigma_{\genus})\rtimes_{\chi'}\ZZ$ as follows:
\[
A'{}^{\flat}=\left\{a_1^{-1}\chi'{}^{m_1}(a_1)\cdots b_{\genus}^{-1}\chi'{}^{m_{2\genus}}(b_{\genus})[a_1,b_1^{m_{2\genus+1}}]\cdots [a_{\genus},b_{\genus}^{m_{2\genus+1}}]\,\middle|\, (m_1,\ldots,m_{2\genus+1})\in \ZZ^{2\genus+1}\right\}.
\]
\end{rem}

\subsection{More examples of $(\Gg,\Ng)$ to which Theorem~\ref{mthm=main} applies with $q=2$}\label{subsec=WW}

In this subsection, we provide a family of examples of pairs $(\Gg,\Ng)$ such that $\Ng\geqslant [\Gg,\Gg]$ and $\WW(\Gg,\Ng)$ is non-zero finite dimensional, other than in Example~\ref{exa=explicitker}; for such $(\Gg,\Ng)$, we can apply Theorem~\ref{mthm=main} and Theorem~\ref{mthm=dim} for the triple $(\Gg,\Gg,\Ng)$ with $q=2$. In particular, we provide such $(\Gg,\Ng)$ in which $\Ng$ is a strictly larger group than $[\Gg,\Gg]$. We start from the following theorem. In this subsection, for the natural group quotient map $p\colon \Gg\twoheadrightarrow \QG$, we use the symbols $\HHH^1(p)$ and $\HHH^2(p)$ for the induced maps $\HHH^1(\QG) \to \HHH^1(\Gg)$ and $\HHH^2(\QG) \to \HHH^2(\Gg)$ by $p$, respectively. (We have been using the symbol $p^{\ast}$ for $\HHH^2(p)$ in the present paper. However, in this subsection  $\HHH^1(p)$ also appears so that we need to make distinctions between $\HHH^2(p)$ and $\HHH^1(p)$.)

\begin{thm} \label{thm=cup}
Let $\Gg$ be a group and $\Ng$ a normal subgroup of $\Gg$. Assume that $\QG=\Gg/\Ng$ is a finitely generated abelian group. Assume that the comparison map  $c_{\Gg}^2\colon \HHH^2_b(\Gg)\to \HHH^2(\Gg)$ of degree $2$ is surjective. Let $\mathcal{H}$ be the image of $\HHH^1(p)$, where $p\colon \Gg\twoheadrightarrow \QG$ is the natural group quotient map. Then,
\[ \Rdim \WW(\Gg,\Ng) = \Rdim \Im (\smile \colon \mathcal{H} \otimes \mathcal{H} \to \HHH^2(\Gg)),\]where $\smile$ denotes the cup product on the cohomology ring $\HHH^{\ast}(\Gg)$.
\end{thm}

Theorem~\ref{thm=Mineyev} yields the following corollary to Theorem~\ref{thm=cup}.

\begin{cor}\label{cor=cup}
Let $\Gg$ be a non-elementary Gromov-hyperbolic group and $\Ng$ a normal subgroup of $\Gg$ such that $\QG=\Gg/\Ng$ is abelian. Let $\mathcal{H}$ be the image of $\HHH^1(p)$, where $p\colon \Gg\twoheadrightarrow \QG$ is the natural group quotient map. Then, $\Rdim \WW(\Gg,\Ng) = \Rdim \Im (\smile \colon \mathcal{H} \otimes \mathcal{H} \to \HHH^2(\Gg))$ holds.
\end{cor}

We employ the following lemma for the proof of Theorem~\ref{thm=cup}.

\begin{lem} \label{lem=cup}
Let $\Ng$ be a normal subgroup of a finitely generated group $\Gg$ such that $\QG=\Gg/\Ng$ is abelian, and let $\mathcal{H}$ be the image of $\HHH^1(p)$, where $p\colon \Gg\twoheadrightarrow \QG$ is the natural group quotient map. Then the image of the cup product
\[ \smile \colon \mathcal{H} \otimes \mathcal{H} \to \HHH^2(\Gg) \]
coincides with the image of $\HHH^2(p)$.
\end{lem}
\begin{proof}
Let $\mathcal{H}'$ be the image of $\smile \colon \mathcal{H} \otimes \mathcal{H} \to \HHH^2(\Gg)$. Let $a,b \in \HHH^1(\QG)$. Since $\HHH^1(p) (a) \smile \HHH^1(p)(b) = \HHH^2(p)(a \smile b)$, we have  $\mathcal{H}' \subseteq \Im(\HHH^2(p))$. Conversely, since $\QG$ is a finitely generated abelian group, every element $u \in \HHH^2(\QG)$ may be expressed as a sum of cup products of elements in $\HHH^1(\QG)$:
\[
 u = a_1 \smile b_1 + \cdots + a_{t} \smile b_{t},
\]
where $t\in \NN$ and $a_1,\ldots,a_t, b_1,\ldots,b_t \in \HHH^1(\QG)$. This implies  that $\mathcal{H}' \supseteq \Im(\HHH^2(p))$.
\end{proof}

\begin{proof}[Proof of Theorem~$\ref{thm=cup}$]
Since $\QG$ is abelian and in particular boundedly $3$-acyclic, we conclude by Theorem~\ref{thm=KKMMMmain2} and the surjectivity of $c_{\Gg}^2$ that
\[
\WW(\Gg,\Ng) \cong \Im (\HHH^2(p)).
\]
Now Lemma~\ref{lem=cup} ends the proof.
\end{proof}

\begin{exa}\label{exa=surfacecup}
In this example, we apply Corollary~\ref{cor=cup} to the case where $\Gg$ is the surface group $\pi_1(\Sigma_{\genus})$ of genus $\genus\in \NN_{\geq 2}$:
\begin{equation}\label{eq=surfacegroup}
G=\pi_1(\Sigma_{\genus})= \langle a_1,  \cdots,a_{\genus}, b_1, \cdots, b_{\genus} \, | \, [a_1, b_1] \cdots [a_{\genus}, b_{\genus}]=e_G \rangle.
\end{equation}
Let $\Ng$ be a normal subgroup of $\Gg$ satisfying  $\Ng\geqslant [\Gg,\Gg]$. Set $\QG=\Gg/\Ng$, and let $p\colon \Gg\twoheadrightarrow \QG$ be the natural group quotient map.
By Theorem~\ref{thm=cup}, there exist $a,b \in \Im (\HHH^1(p))$ such that $a \smile b \ne 0$ if and only if $\WW(\Gg,\Ng) \ne 0$.  In what follows, we provide two contrasting examples. Here, let $a_1^\dagger, \cdots, a_\genus^\dagger, b_1^\dagger, \cdots, b_\genus^\dagger \in \HHH^1(\Gg)$ be the dual basis of $[a_1], \cdots, [a_\genus], [b_1], \cdots, [b_\genus]$; for $i\in \{1,\ldots ,\genus\}$, $[a_i]$ and $[b_j]$ respectively denote the elements of $\HHH_1(\Gg)=\HHH_1(\Gg;\RR)$, the first real homology of $\Gg$, determined by $a_i$ and $b_i$. For elements $\lambda_1,\ldots,\lambda_m$ ($m\in\NN$), the symbol $\ZZ[\lambda_1,\ldots,\lambda_m]$ denotes the free abelian group with basis $\lambda_1,\ldots,\lambda_m$.
\begin{enumerate}[label=\textup{(\arabic*)}]
\item Let $\Ng$ be the subgroup of $\Gg = \pi_1(\Sigma_\genus)$ containing $[\Gg, \Gg]$ such that $\QG= \ZZ[p(a_1),\ldots, p(a_{\genus})]$.
We note that for every $i,j\in \{1,\ldots,\genus\}$, $a_i^\dagger \smile a_j^\dagger = 0$ holds. Hence, by Corollary~\ref{cor=cup} we have $\WW(\Gg, \Ng) = 0$.

\item Now, from (1) we change $\Ng$ to be the subgroup of $\Gg$ containing $[\Gg, \Gg]$ such that $\QG = \ZZ[p(a_1), p(b_1)]$.
We note that $a_1^\dagger \smile b_1^\dagger$ generates $\HHH^2(\Gg) \cong \RR$. Hence, by Corollary~\ref{cor=cup}, in this case $\Rdim\WW(\Gg, \Ng) = 1$.
\end{enumerate}
In particular, for the pair $(\Gg,\Ng)$ as in (2), the group triple $(\Gg,\Lg,\Ng)=(\Gg,\Gg,\Ng)$  satisfies condition (b$_2$) in Theorem~\ref{mthm=main} with $\ell=1$.
\end{exa}

As is stated as Corollary~\ref{cor=cup}, Theorem~\ref{thm=cup} works for the case where $\Gg$ is non-elementary Gromov-hyperbolic. If $\Gg$ is a one-relator group,  then we have a criterion of this surjectivity in terms of the simplicial volume of $\Gg$, introduced by  \cite{HL}. We use the following setting.

\begin{setting}\label{setting=SV}
Let $F$ be a finitely generated free group and let $r$ be a non-trivial element in  $\gamma_2(F)$. Set $\Gg_r=F / \llangle r \rrangle$.
\end{setting}
In Setting~\ref{setting=SV}, we have $\HHH_2(\Gg_r; \ZZ) \cong \ZZ$. Motivated by the celebrated concept of the simplicial volume of closed connected oriented manifolds by Gromov \cite{Gr}, Heuer and L\"oh \cite{HL} defined the \emph{simplicial volume of $\Gg_r$} to be the $\ell^1$-seminorm of the generator of $\HHH_2(\Gg_r; \ZZ)$. We collect in Example~\ref{exa=SV} of several examples of one-relator groups with positive simplicial volume by \cite{HL}. Proposition~\ref{prop=SV} explains the importance of the simplical volume in relation to $\WW(\Gg,\Ng)$.

\begin{prop}\label{prop=SV}
Assume Setting~$\ref{setting=SV}$. Then the following are equivalent.
\begin{enumerate}[label=\textup{(\arabic*)}]
\item The simplicial volume of $\Gg_r$ is non-zero.
\item The comparison map $c_{\Gg_r}^2 \colon \HHH^2_b(\Gg_r) \to \HHH^2(\Gg_r)$ is surjective.
\end{enumerate}
\end{prop}
\begin{proof}
This follows from \cite[Lemma~6.1]{Frigerio} (see also \cite[Subsection~2.1]{HL}).
\end{proof}

\begin{rem}\label{rem=hyp}
There are some criteria on $r\in \gamma_2(F)$ for $\Gg_r$ being non-elementary Gromov-hyperbolic. Typical ones are  small cancellation conditions (see for instance \cite{Strebel}). Another one comes from  Newman's spelling theorem \cite{Newman}: if $r$ is a proper power and if the rank of $F$ is at least $2$, then $\Gg_r$ is non-elementary Gromov-hyperbolic. We note that if $\Gg_r$ is a non-elementary Gromov-hyperbolic group, then the simplicial volume of $\Gg_r$ is non-zero. Indeed, this follows from Theorem~\ref{thm=Mineyev} and Proposition~\ref{prop=SV}.
\end{rem}

The following theorem extends a previous result \cite[Theorem~4.18]{KKMMM}; we will prove a more general theorem, Theorem~\ref{thm=one_relator}, in Subsection~\ref{subsec=Wcal_nilp}.

\begin{thm}\label{thm=SV}
Assume Setting~$\ref{setting=SV}$. Assume that $r\in \gamma_2(F)\setminus \gamma_3(F)$. Assume that the simplicial volume of $\Gg_r$ is non-zero. Then $\Rdim \WW(\Gg_r,\gamma_2(\Gg_r)) = 1$.
\end{thm}






\begin{exa}\label{exa=SV}
Here we collect some examples appearing in \cite{HL} of groups $\Gg_r$ in Setting~\ref{setting=SV} with non-zero simplicial volume, possibly different from ones appearing in Remark~\ref{rem=hyp}.
\begin{enumerate}[label=\textup{(\arabic*)}]
  \item Let $F=F(a,b)$. Set $r_3=[a,b][a,b^{-3}]$ and $r_4=[a,b][a,b^{-4}]$ in $F$. Then, by \cite[Example~G]{HL} the simplicial volumes of $\Gg_{r_3}$ and $\Gg_{r_4}$ equal $1$ and $\frac{4}{3}$, respectively.
  \item Let $S_1$ and $S_2$ be two disjoint finite sets. Take two non-trivial elements $r_1\in \gamma_2(F(S_1))$ and $r_2\in \gamma_2(F(S_2))$. Then, with regarding $r=r_1r_2$ as an element in $F=F(S_1\sqcup S_2)$, \cite[Theorem~A~1]{HL} showed that the simplicial volume of $\Gg_r$ equals $4(\scl_{F}(r)-1/2)$. Together with \cite[Theorem~2.93]{Calegari}, the simplicial volume of this $\Gg_r$ can be shown to be non-zero.
\end{enumerate}
\end{exa}

\subsection{One-relator groups $\Gg$ with $\Wcal(\Gg,\gamma_{q-1}(\Gg),\gamma_q(\Gg))\ne 0$}\label{subsec=Wcal_nilp}

In Subsections~\ref{subsec=explicitkernel} and \ref{subsec=WW}, we have presented examples of triples $(\Gg,\Lg,\Ng)$ to which Theorem~\ref{mthm=main} and Theorem~\ref{mthm=dim} apply. However, in every example there, the exponent $q$ equals $2$; in particular, we have $\Lg=\Gg$ in these examples. For each fixed $q\in \NN_{>2}$, the following theorem supplies examples   of groups $\Gg$  such that $\Wcal(\Gg,\gamma_{q-1}(\Gg),\gamma_q(\Gg))\ne 0$. Theorem~\ref{thm=one_relator} in particular implies Theorem~\ref{thm=SV}.

\begin{thm}[one-relator group with non-vanishing $\Wcal(\Gg,\gamma_{q-1}(\Gg),\gamma_q(\Gg))$] \label{thm=one_relator}
Let $F$ be a finitely generated free group. Let $q\in \NN_{\geq 2}$. Let $r \in \gamma_q(F) \setminus \gamma_{q+1}(F)$. Set $\Gg=F / \llangle r \rrangle$. Assume that the simplicial volume of $\Gg$ is non-zero.
Then, we have
\[
\WW(\Gg, \gamma_{q-1}(\Gg))=0 \quad \textrm{and} \quad \Rdim \WW(\Gg, \gamma_q(\Gg)) = 1.
\]
In particular, we have $\Rdim\Wcal(\Gg, \gamma_{q-1}(\Gg), \gamma_{q}(\Gg)) = 1$.
\end{thm}

In what follows, we prove Theorem~\ref{thm=one_relator}. We will employ the following lemmata. The proof of Lemma~\ref{lem=invariant_homomorphism} is straightforward, and we omit the proof.
\begin{setting}\label{setting=nilp}
Let $F$ be a finitely generated free group. Let $q\in \NN_{\geq 2}$. Let $r \in \gamma_q(F) $. Set $\Gg=F/\llangle r\rrangle$.
\end{setting}

\begin{lem} \label{lem=invariant_homomorphism}
Let $\Gg$ be a group and $\Ng$ a normal subgroup of $\Gg$. For a group homomorphism $\kg \colon \Ng \to \RR$, $\kg$ is $\Gg$-invariant if and only if $\Ker(\kg)\geqslant [\Gg,\Ng]$.
\end{lem}



\begin{lem} \label{lem=invariant_homomorphism2}
Let $F$ be a finitely generated free group and $q\in \NN_{\geq 2}$. Assume that $r \in \gamma_{q}(F) \setminus \gamma_{q+1}(F)$. Then there exists $\hf \in \HHH^1(\gamma_q(F))^{F}$ such that $\hf (r) \ne 0$.
\end{lem}

\begin{proof}
It is well known that $\gamma_q(F) / \gamma_{q+1}(F)$ is a free abelian group of finite rank; see for instance \cite[Corollary~5.12~(iv)]{MKSbook}. Since $\RR$ is an injective abelian group, there exists a homomorphism
\[
\bar{\hf} \colon \gamma_q(F) / \gamma_{q+1}(F) \to \RR
\]
such that $\bar{\hf}(r) \ne 0$. Define $\hf$ to be the composition of the sequence $\gamma_q(F) \to \gamma_q(F) / \gamma_{q+1}(F) \xrightarrow{\bar{\hf}} \RR$. Then $\hf(r) \ne 0$; it also follows from Lemma~\ref{lem=invariant_homomorphism} that $\hf$ is $F$-invariant.
\end{proof}

\begin{lem}\label{lem=invhom}
Assume Setting~$\ref{setting=nilp}$.  Then, for every $s\in \NN$ with $s\leq q$, we have
\[
\HHH^1(\gamma_s(\Gg))^{\Gg} \cong \{ \hf \in \HHH^1(\gamma_s(F))^F \, | \, \hf(r) = 0\}.
\]
\end{lem}
\begin{proof}
Set $\mathcal{H}$ as the space of the right hand side. Let $\ppi$ be the natural group quotient map from $F$ onto $\Gg$, and let $\ppi^{\ast}$ denote the map $\HHH^1(\gamma_s(\Gg))^{\Gg} \to \HHH^1(\gamma_s(F))^F$ induced by the projection $\ppi$. Then $\ppi^{\ast}$ is injective since $\ppi$ is surjective. It suffices to show that the image of $\ppi^{\ast}$ coincides with $\mathcal{H}$. It is clear that the image of $\ppi^{\ast}$ is contained in $\mathcal{H}$. Conversely, let $\hf\in \mathcal{H}$. Then, $\Ker(\hf)$ is a normal subgroup of $F$ containing $r$ (note that $r\in \gamma_q(F)\leqslant \gamma_s(F)$). Hence $\hf$ induces a group homomorphism $\kg \colon \gamma_s(G) \to \RR$. Then $\kg$ is $G$-invariant since $\hf$ is $F$-invariant. Hence, we have $\ppi^{\ast} \kg = \hf$. Hence, we have $\ppi^{\ast}(\HHH^1(\gamma_s(\Gg))^{\Gg})\supseteq \mathcal{H}$. We now conclude that $\ppi^{\ast}$ gives an isomorphism $\HHH^1(\gamma_s(\Gg))^{\Gg}\cong \mathcal{H}$, as desired.
\end{proof}

\begin{lem}\label{lem=H^1dim}
Assume Setting~$\ref{setting=nilp}$. Assume moreover that $r\not \in \gamma_{q+1}(F)$. Then, we have
\begin{align*}
&\Rdim \HHH^1(\gamma_{q-1}(\Gg))^{\Gg} = \Rdim \HHH^1(\gamma_{q-1}(F))^F\\
 \textrm{and}\quad &
\Rdim \HHH^1(\gamma_q(\Gg))^{\Gg} = \Rdim \HHH^1(\gamma_q(F))^F - 1.
\end{align*}
\end{lem}

\begin{proof}
The first assertion holds by Lemma~\ref{lem=invhom}; indeed, note that for every $\hf \in \HHH^1(\gamma_{q-1}(F))^F$, $\hf(r)=0$ holds since $r\in \gamma_q(F)$. The second assertion follows from Lemmata~\ref{lem=invariant_homomorphism2} and \ref{lem=invhom}.
\end{proof}

\begin{proof}[Proof of Theorem~$\ref{thm=one_relator}$]
Let $s\in \{q-1,q\}$. Set $\QG^{(s)}=\Gg/ \gamma_{s}(\Gg)$; this  is a nilpotent group and in particular boundedly $3$-acyclic. By Proposition~\ref{prop=SV},  $c_{\Gg}^2\colon \HHH^2_b(\Gg)\to \HHH^2(\Gg)$ is surjective. Hence, Theorem~\ref{thm=KKMMMmain2} implies the isomorphism:
\[
\WW(\Gg, \gamma_{s}(\Gg)) \cong \Im \left(p^{\ast}_s \colon \HHH^2(\QG^{(s)}) \to \HHH^2(\Gg)\right).
\]

Set $\Lambda^{(s)}=F/ \gamma_{s}(F)$, and consider the following two exact sequences:
\begin{equation} \label{eq=one_relator_1}
0 \to \HHH^1(\Lambda^{(s)}) \to \HHH^1(F) \to \HHH^1(\gamma_s(F))^F \to \HHH^2(\Lambda^{(s)})\to  \HHH^2(F),
\end{equation}
and
\begin{equation} \label{eq=one_relator_2}
 0 \to \HHH^1(\QG^{(s)}) \to \HHH^1(\Gg) \to \HHH^1(\gamma_s(\Gg))^{\Gg} \to \HHH^2(\QG^{(s)}) \xrightarrow{p^{\ast}_s} \HHH^2(\Gg).
\end{equation}
In \eqref{eq=one_relator_1}, since $r\in \gamma_q(F)\leqslant \gamma_2(F)$, the abelianization of $F$ and $\Lambda^{(s)}$ coincide. Hence, the map $\HHH^1(\Lambda^{(s)}) \to \HHH^1(F)$ is an isomorphism. Since $\HHH^2(F) = 0$, we have $\Rdim \HHH^2(\Lambda^{(s)}) = \Rdim \HHH^1(\gamma_q(F))^F$. In a manner similar to this argument, in \eqref{eq=one_relator_2}, we have the injectivity of the map $\HHH^1(\gamma_s(\Gg))^{\Gg} \to \HHH^2(\QG^{(s)})$.

Since $r\in \gamma_q(F)\leqslant \gamma_s(F)$, we have $\Lambda^{(s)}\cong \QG^{(s)}$. Together with Lemma~\ref{lem=H^1dim}, we obtain that
\begin{align*}
\Rdim {\rm Im}(p^{\ast}_s) & =  \Rdim \HHH^2(\QG^{(s)}) - \Rdim \HHH^1(\gamma_s(\Gg))^{\Gg} \\
&=  \Rdim \HHH^2(\Lambda^{(s)}) - \Rdim \HHH^1(\gamma_s(\Gg))^{\Gg} \\
&= \Rdim \HHH^1(\gamma_s(F))^F -\Rdim \HHH^1(\gamma_s(\Gg))^{\Gg} \\
&=   \begin{cases}
\Rdim \HHH^1(\gamma_{q-1}(F))^F - \Rdim \HHH^1(\gamma_{q-1}(F))^F, & \textrm{if }s=q-1,\\
\Rdim \HHH^1(\gamma_q(F))^F - \Bigl(\Rdim \HHH^1(\gamma_q(F))^F - 1\Bigr), & \textrm{if }s=q,
\end{cases} \\
& =  \begin{cases}
0, & \textrm{if }s=q-1,\\
1, & \textrm{if }s=q.
\end{cases}
\end{align*}
Therefore, we conclude that $\WW(\Gg, \gamma_{q-1}(\Gg)) = 0$ and $\Rdim \WW(\Gg, \gamma_q(\Gg)) = 1$.

Note that $\WW(\Gg, \gamma_{q-1}(\Gg)) = 0$ means that
\[
\QQQ(\gamma_{q-1}(\Gg))^{\Gg} = \HHH^1(\gamma_{q-1}(\Gg))^{\Gg} + i_{\gamma_{q-1}(\Gg),\Gg}^{\ast} \QQQ(\Gg).
\]
It is then straightforward to deduce that $\Wcal(\Gg, \gamma_{q-1}(\Gg), \gamma_q(\Gg)) = \WW(\Gg, \gamma_q(\Gg))$. In particular, we obtain that $\Rdim \Wcal(\Gg, \gamma_{q-1}(\Gg), \gamma_q(\Gg))=1$. This completes our proof.
\end{proof}

\subsection{Examples of $(\Gg,\Lg,\Ng)$ to which Theorem~\ref{mthm=main} applies with $q>2$}\label{subsec=Wcal}

Theorem~\ref{thm=one_relator}, together with Theorem~\ref{mthm=main}, yields  the following result.

\begin{thm}[coarse kernel for one-relator groups]\label{thm=coarsekernelonerel}
Let $F$ be a finitely generated free group. Let $q\in \NN_{\geq 2}$. Let $r\in \gamma_q(F)\setminus \gamma_{q+1}(F)$. Set $\Gg=F/\llangle r\rrangle$. Assume that the simplicial volume of $\Gg$ is non-zero. Let $r^{\flat}\in \gamma_q(F)$ be an element representing the same equivalence class as $r$ in $\gamma_q(F)/\gamma_{q+1}(F)$. Take an $[F,\gamma_{q-1}(F)]$-expression $\bff^{\flat}=(\ff_1^{\flat},\ldots,\ff_s^{\flat};\ff'_1{}^{\flat},\ldots,\ff'_s{}^{\flat})$ of $r^{\flat}$. For every $i\in \{1,\ldots,s\}$, set $\gG_i^{\flat}=\ppi(\ff_i^{\flat})$ and $\gG'_i{}^{\flat}=\ppi(\ff'_i{}^{\flat})$, where $\ppi\colon F\twoheadrightarrow \Gg$ is the natural group quotient map. Set $\PPs^{\flat}\colon \ZZ\to \gamma_{q+1}(\Gg)$ by
\[
\ZZ\ni m\mapsto [\gG_1^{\flat},(\gG'_1{}^{\flat})^m]\cdots [\gG_s^{\flat},(\gG'_s{}^{\flat})^m]\in \gamma_{q+1}(\Gg).
\]
Then, the coarse subspace represented by $A^{\flat}=\PPs^{\flat}(\ZZ)$ is the coarse kernel of
\[
\iota_{\Gg,\gamma_q(G)}\colon (\gamma_{q+1}(\Gg),d_{\scl_{\Gg,\gamma_q(\Gg)}})\to (\gamma_{q+1}(\Gg),d_{\scl_{\Gg}}),
\]
and $\PPs^{\flat}$, viewed as a map from $(\ZZ,|\cdot|)$ to $(A^{\flat},d_{\scl_{\Gg,\gamma_q(\Gg)}})$, is a coarse isomorphism that is also a quasi-isometry.
\end{thm}

\begin{proof}
By Theorem~\ref{thm=one_relator}, we have $\Rdim \Wcal(\Gg,\gamma_{q-1}(\Gg),\gamma_q(\Gg))=1$. Then, by Theorem~\ref{mthm=main} and Proposition~\ref{prop=explicitck}, the coarse subspace represented by  $A^{\flat}$ is the coarse kernel of the map
\[
\iota_{(\Gg,\gamma_{q-1}(G),\gamma_q(\Gg))}\colon (\gamma_{q+1}(\Gg),d_{\scl_{\Gg,\gamma_q(\Gg)}})\to (\gamma_{q+1}(\Gg),d_{\scl_{\Gg},\gamma_{q-1}(\Gg)})
\]
and that $\PPs^{\flat}\colon (\ZZ,|\cdot|)\to (A^{\flat},d_{\scl_{\Gg,\gamma_q(\Gg)}})$ is a quasi-isometry and a a coarse isomorphism.

Finally, we move from discussions on $\iota_{(\Gg,\gamma_{q-1}(G),\gamma_q(\Gg))}$ to those on $\iota_{\Gg,\gamma_q(\Gg)}$. By Theorem~\ref{thm=one_relator}, $\WW(\Gg,\gamma_{q-1}(\Gg))=0$. By Theorem~\ref{thm=biLip}, $\scl_{\Gg}$ and $\scl_{\Gg,\gamma_{q-1}(\Gg)}$ are bi-Lipschitzly equivalent on $\gamma_q(\Gg)(\geqslant \gamma_{q+1}(\Gg))$ (in fact, \cite[Theorem~2.1~(3)]{KKMMM} implies that they coincide on $\gamma_q(\Gg)$). Therefore, the coarse subspace represented by $A^{\flat}$ is the coarse kernel of $\iota_{\Gg,\gamma_q(\Gg)}$ as well.
\end{proof}

We note that in the proof of Theorem~\ref{thm=coarsekernelonerel}, we ascend in the lower central series $(\gamma_i(\Gg))_{i\in \NN}$, from $\gamma_q(\Gg)$ to $\gamma_1(\Gg)=\Gg$, \emph{by two steps}: the first step is from $\gamma_q(\Gg)$ to $\gamma_{q-1}(\Gg)$, which falls in the abelian case and we may apply Theorem~\ref{mthm=main}; the second step is from $\gamma_{q-1}(\Gg)$ to $\gamma_{1}(\Gg)=\Gg$.

In this aspect, it may be also important to provide examples for each fixed $q\in \NN_{\geq 3}$ of groups $\Gg$ such that both  $\Wcal(\Gg,\gamma_{q-1}(\Gg),\gamma_{q}(\Gg))$ and $\WW(\Gg,\gamma_{q-1}(\Gg))$ are non-zero spaces. This is possible by Theorem~\ref{thm=one_relator} and the following proposition. Here, recall that a subgroup $H$ of a group $\Gg$ is called a \emph{retract} of $\Gg$ if there exists a group homomorphism $\alpha\colon \Gg\to \Gg$ such that $\alpha(\Gg)=H$ and $\alpha|_H=\mathrm{id}_H$.

\begin{prop}\label{prop=retract}
Let $q\in \NN_{\geq 3}$. Let $H_1$ and $H_2$ be groups such that
\[
\Wcal(H_1,\gamma_{q-1}(H_1),\gamma_q(H_1))\ne 0\quad \textrm{and} \quad\WW(H_2,\gamma_{q-1}(H_2))\ne 0.
\]
Let $\Gg$ be a group such that it admits two isomorphic copies of $H_1$ and $H_2$ that are both retracts of $\Gg$. Then,
\[
\Wcal(\Gg,\gamma_{q-1}(\Gg),\gamma_q(\Gg))\ne 0\quad \textrm{and}\quad  \WW(\Gg,\gamma_{q-1}(\Gg))\ne 0.
\]
\end{prop}

\begin{proof}
Note that  the correspondence $H  \mapsto \Wcal( H, \gamma_{q-1}(H), \gamma_q(H))$ and $H\mapsto \WW(H , \gamma_{q-1}(H))$ (where, $H$ is a group) are both contravariant functors.  By the assumptions on $H_1$ and $H_2$, we conclude that $\Wcal(H_1, \gamma_{q-1}(H_1), \gamma_q(H_1)) \ne 0$ is a retract of $\Wcal(\Gg, \gamma_{q-1}(\Gg),\gamma_{q}(\Gg))$ and that $\WW(H_2, \gamma_{q-1}(H_2)) \ne 0$ is a retract of $\WW(\Gg, \gamma_{q-1}(\Gg)))$, thus obtaining the conclusion.
\end{proof}

\begin{cor}
Let $F^{(1)}$ and $F^{(2)}$ be  finitely generated free groups. Let $q\in \NN_{\geq 3}$. Let $r_1\in \gamma_{q}(F^{(1)})\setminus \gamma_{q+1}(F^{(1)})$ and $r_2\in \gamma_{q-1}(F^{(2)})\setminus \gamma_{q}(F^{(2)})$ . Set $H_1=F^{(1)}/\llangle r_1\rrangle_{F^{(1)}}$ and $H_2=F^{(2)}/\llangle r_2\rrangle_{F^{(2)}}$. Assume that the simplicial volumes of $H_1$ and $H_2$ are both non-zero. Set $\Gg=H_1\star H_2$, the free product of $H_1$ and $H_2$. Then, the spaces $\Wcal(\Gg,\gamma_{q-1}(\Gg),\gamma_q(\Gg))$ and  $\WW(\Gg,\gamma_{q-1}(\Gg))$ are both non-zero finite dimensional.
\end{cor}

\begin{proof}
By applying Proposition~\ref{prop=retract} and Theorem~\ref{thm=one_relator}, we obtain that  $\Wcal(\Gg,\gamma_{q-1}(\Gg),\gamma_q(\Gg))$ and $\WW(\Gg,\gamma_{q-1}(\Gg))$ are both non-zero spaces. Recall also Corollary~\ref{cor=nilp}.
\end{proof}

\subsection{Subsets on which $\scl_{\Gg,\Lg}$ and $\scl_{\Gg,\Ng}$ are bi-Lipschitzly equivalent}\label{subsec=application_cod}
Corollary~\ref{cor=notbilip} can supply many triples $(\Gg,\Lg,\Ng)$ for which $\scl_{\Gg,\Lg}$ and $\scl_{\Gg,\Ng}$ are not bi-Lipschitzly equivalent on $[\Gg,\Ng]$.  Nevertheless, the following theorem states that if we know the set of zeros for a set of representative of a basis of $\Wcal(\Gg,\Lg,\Ng)$, then we may find a subset $Z\subseteq [\Gg,\Ng]$ on which $\scl_{\Gg,\Lg}$ and $\scl_{\Gg,\Ng}$ are bi-Lipschitzly equivalent. It is worth noting that the sets of zeros of certain \emph{quasi}morphisms have such a property of interest. The following theorem is an application of Theorem~\ref{thm=comparisonDD}.

\begin{thm}\label{thm=application_cod}
Assume Setting~$\ref{setting=GLN}$. Assume that $\Wcal(\Gg,\Lg,\Ng)$ is non-zero finite dimensional, and set $\ell=\Wcal(\Gg,\Lg,\Ng)$. Take an arbitrary basis of $\Wcal(\Gg,\Lg,\Ng)$. Take an arbitrary  set $\{\nu_1,\ldots ,\nu_{\ell}\}\subseteq  \QQQ(\Ng)^{\Gg}$ of representatives of this basis. Let $Z_{\nug_1,\ldots ,\nug_{\ell}}=[\Gg,\Ng]\cap \bigcap\limits_{j\in \{1,\ldots ,\ell\}}\nug_j^{-1}(\{0\})$. Then, we have for every $\zg\in Z_{\nu_1,\ldots ,\nu_{\ell}}$,
\[
\scl_{\Gg,\Ng}(\zg)\leq \mathscr{C}_{1,\mathrm{ctd}}\cdot \scl_{\Gg,\Lg}(\zg),
\]
where $\mathscr{C}_{1,\mathrm{ctd}}$ is the constant associated with $\nu_1,\ldots ,\nu_{\ell}$ appearing in Theorem~$\ref{thm=comparisonDD}$.
\end{thm}

\begin{proof}
Take the constants $\mathscr{C}_{1,\mathrm{ctd}}, \mathscr{C}_{2,\mathrm{ctd}}$  associated with $\nu_1,\ldots ,\nu_{\ell}$ as in Theorem~\ref{thm=comparisonDD}. Let $\zg\in Z_{\nug_1,\ldots,\nug_{\ell}}$. Let $\varepsilon\in \RR_{>0}$. Then, by Theorem~\ref{thm=Bavard} (for the pair $(\Gg,\Ng)$), we have $\nug_{\varepsilon}\in \QQQ(\Ng)^{\Gg}$ such that
\begin{equation}\label{eq=z1}
|\nug_{\varepsilon}(\zg)|\geq 2(1-\varepsilon)\scl_{\Gg,\Ng}(\zg)\cdot \DD(\nug_{\varepsilon}).
\end{equation}
By applying Theorem~\ref{thm=comparisonDD} to $\nug_{\varepsilon}$, we have $\kg\in \HHH^1(\Ng)^{\Gg}$, $\psg\in \QQQ(\Lg)^{\Gg}$ and $a_1,\ldots,a_{\ell}\in \RR$ such that $\nug_{\varepsilon}=\kg+i^{\ast}\psg+\sum\limits_{j\in\{1,\ldots,\ell\}}a_j\nug_j$ and
\[
\DD(\nug_{\varepsilon})\geq \mathscr{C}_{1,\mathrm{ctd}}^{-1}\left(\DD(\psg)+\mathscr{C}_{2,\mathrm{ctd}}^{-1}\cdot \sum_{j\in \{1,\ldots,\ell\}}|a_j|\right).
\]
In particular, we have $\DD(\psg)\leq \mathscr{C}_{1,\mathrm{ctd}}\cdot \DD(\nug_{\varepsilon})$. Since $\zg\in Z_{\nug_1,\ldots,\nug_{\ell}}$, we have $\psg(\zg)=\nug_{\varepsilon}(\zg)$. Therefore, by \eqref{eq=z1}, Theorem~\ref{thm=Bavard}  (for the pair $(\Gg,\Lg)$) yields that
\[
\mathscr{C}_{1,\mathrm{ctd}}\cdot \scl_{\Gg,\Lg}(\zg)\geq (1-\varepsilon)\scl_{\Gg,\Ng}(\zg).
\]
By letting $\varepsilon \searrow 0$, we have the conclusion.
\end{proof}

\begin{exa}[example for surface groups]
Let $\genus\in \NN_{\geq 2}$. We consider the surface group $\pi_1(\Sigma_{\genus})$ (\eqref{eq=surfacegroup}). As is mentioned in Theorem~\ref{thm=WW}~(2), $\Rdim \WW(\pi_1(\Sigma_{\genus}),\gamma_2(\pi_1(\Sigma_{\genus})))=1$. Define
\[
\overline{G}=\langle c,d\,|\, [c,d]^2=e_{\overline{G}}\rangle.
\]
Let $I_{\genus}=\{(i,j)\in \ZZ^2\,|\,1\leq i<j\leq \genus\}$. For every $(i,j)\in I_{\genus}$, we can define the surjective group homomorphism $p_{(i,j)}\colon \pi_1(\Sigma_{\genus})\twoheadrightarrow \overline{G}$ that sends $a_i$ and $a_j$ to $c$, $b_i$ and $b_j$ to $d$, and $a_s$, $b_s$ for every $s\in \{1,\ldots,\genus\}\setminus \{i,j\}$ to $e_{\overline{G}}$. In \cite[Section~5]{MMM}, an element $\overline{\nug}\in \QQQ(\gamma_2(\overline{G}))^{\overline{G}}$ with the following property is constructed: for every $(i,j)\in I_{\genus}$, $[p_{(i,j)}^{\ast}\overline{\nug}]$ generates the space $\WW(\gamma_2(\pi_1(\Sigma_{\genus})))^{\pi_1(\Sigma_{\genus})}$. Therefore, Theorem~\ref{thm=application_cod} implies that for every $(i,j)\in I_{\genus}$ there exists $C_{(i,j)}\in \RR_{\geq 1}$ such that on
\[
Z_{(i,j)}=\gamma_3(\pi_1(\Sigma_{\genus}))\cap (p_{(i,j)}^{\ast}\overline{\nug})^{-1}(\{0\}),
\]
$\scl_{\pi_1(\Sigma_{\genus}),\gamma_2(\pi_1(\Sigma_{\genus}))}\leq C_{(i,j)}\cdot \scl_{\pi_1(\Sigma_{\genus})}$ holds. Set $Z=\bigcup\limits_{(i,j)\in I_{\genus}}Z_{(i,j)}$ and $C=\max\limits_{(i,j)\in I_{\genus}}C_{(i,j)}$. Then for every $\zg\in Z$, we have $\scl_{\pi_1(\Sigma_{\genus}),\gamma_2(\pi_1(\Sigma_{\genus}))}(\zg)\leq C\cdot \scl_{G_{\genus}}(\zg)$. In particular, $\scl_{\pi_1(\Sigma_{\genus})}$ and $\scl_{\pi_1(\Sigma_{\genus}),\gamma_2(\pi_1(\Sigma_{\genus}))}$ are bi-Lipschitzly equivalent on $Z$.  We note that
\[
Z \supseteq \gamma_3(\pi_1(\Sigma_{\genus}))\cap \left(\bigcup_{(i,j)\in I_{\genus}}\Ker(p_{(i,j)})\right).
\]
We exhibit concrete examples of elements in this set $Z$ when $\genus=4$. Every element of the form
\[
\lambda \lambda'_1[a_1,b_1]^{t}\lambda'_2[a_2^{s_1},b_2^{s_2}]^{s_3}\lambda'_3[a_3,b_3]^{t}\lambda'_4[a_4^{s_4},b_4^{s_5}]^{s_6} \lambda'_5\lambda^{-1},
\]
where $\lambda\in \pi_1(\Sigma_4)$, $\lambda'_1,\lambda'_2,\lambda'_3,\lambda'_4,\lambda'_5\in \gamma_3(\pi_1(\Sigma_4))\cap \langle a_2,b_2,a_4,b_4\rangle$ and $s_1,s_2,s_3,s_4,s_5,s_6,t\in \ZZ$  with $s_1s_2s_3=s_4s_5s_6=t$, belongs to $\gamma_3(\pi_1(\Sigma_{\genus}))\cap\Ker(p_{(1,3)})\subseteq Z$. Every  product of such elements also lies in the set $Z$.
\end{exa}

\section{Applications of the coarse kernels}\label{sec=remarks}
In this section, we present applications of the coarse kernels obtained in Theorem~\ref{mthm=main}. In particular, we prove Proposition~\ref{prop=extexample} and Theorem~\ref{thm=crushingq2}.

\subsection{Coarse kernel and extendability}\label{subsec=coarsekerext}
The following theorem connects the extendability problem up to invariant homomorphisms and behaviors of invariant quasimorphisms on the coarse kernel.

\begin{thm}[coarse kernel and extendability in the abelian case]\label{thm=extkernel}
Assume Settings~$\ref{setting=GLN}$, $\ref{setting=main1}$ and $\ref{setting=main3}$. Let $\bA$ be the coarse kernel of $\iota_{(\Gg,\Lg,\Ng)}$ \textup{(}as a coarse subspace\textup{)}. Let $\nug\in \QQQ(\Ng)^{\Gg}$. For every $i\in \{1,\ldots,\ell\}$, set $\vec{e}_i$ as the unit vector in $\ZZ^{\ell}$ supported on the $i$-th entry. Then, the following are all equivalent.
\begin{enumerate}[label=\textup{(\arabic*)}]
  \item The invariant quasimorphism $\nug$ represents the zero element in $\Wcal(\Gg,\Lg,\Ng)$.
  \item For every representative $A\subseteq [\Gg,\Ng]$ of $\bA$, $\nug$ is bounded on $A$.
  \item There exists a representative $A\subseteq [\Gg,\Ng]$ of $\bA$ such that $\nug$ is bounded on $A$.
  \item For every representative $A\subseteq [\Gg,\Ng]$ of $\bA$, for every coarse group isomorphism $\PPs\colon \ZZ^{\ell}\to A$ and for every $i\in \{1,\ldots,\ell\}$, $\nug$ is bounded on the set $\PPs(\ZZ\vec{e}_i)$.
  \item There exist a representative $A\subseteq [\Gg,\Ng]$ of $\bA$ and a coarse group isomorphism $\PPs\colon \ZZ^{\ell}\to A$ such that for every $i\in \{1,\ldots,\ell\}$, $\nug$ is bounded on the set $\PPs(\ZZ\vec{e}_i)$.
\end{enumerate}
\end{thm}

\begin{proof}
Recall from Theorem~\ref{thm=maingeneral} that $\bA$ is isomorphic to $(\ZZ^{\ell},\|\cdot\|_1)$ as a coarse group and that the map $\PPs$ defined in Theorem~\ref{thm=PPs} is a coarse homomorphism between $\ZZ^{\ell}$ and a certain representative $A$ of $\bA$. Hence, (2) implies  (4), and (4) implies (5).
Since every representative $A$ of $\bA$ is $d_{\scl_{\Gg,\Lg}}$-bounded and $A\subseteq [\Gg,\Ng]$, by Theorem~\ref{thm=Bavard} (for the pair $(\Gg,\Lg)$) (1) implies (2). Since two representatives $A,A'$ of $\bA$ are asymptotic ($A\asymp A'$), by Theorem~\ref{thm=Bavard} (for the pair $(\Gg,\Ng)$) (3) implies (2). Every element $\vm=(m_1,\ldots,m_{\ell})\in \ZZ^{\ell}$ is written as $\vm=m_1\vec{e}_1+\cdots +m_{\ell}\vec{e}_{\ell}$. Together with Theorem~\ref{thm=Bavard}, we conclude that (5) implies (3).
Now it suffices to show that (4) implies (1) in order to close up our proof. Take a basis of $\Wcal(\Gg,\Lg,\Ng)$, and take a set $\{\nug_1,\ldots,\nug_{\ell}\}$ of  representatives of the basis. For this $(\nug_1,\ldots,\nug_{\ell})$, construct a map $\PPs$ as in Theorem~\ref{thm=PPs}. Write $\nug=\kg+i_{\Ng,\Lg}^{\ast}\psg+\sum\limits_{j\in \{1,\ldots,\ell\}}a_j\nug_j$, where $\kg\in \HHH^1(\Ng)^{\Gg}$, $\psg\in \QQQ(\Lg)^{\Gg}$ and $(a_1,\ldots,a_{\ell})\in \RR^{\ell}$. Then by the limit formula in Theorem~\ref{thm=valueofcore}, the construction of $\PPs$ implies that for every $i\in \{1,\ldots,\ell\}$,
\[
a_i=\lim_{m\to \infty}\frac{\nug(\PPs(m\vec{e}_i))}{m}.
\]
Under (4), the limit above equals zero, and we  obtain (1). This completes our proof.
\end{proof}

\begin{proof}[Proof of Proposition~$\ref{prop=extexample}$]
Recall Lemma~\ref{lem=maintogeneral}. Then, by Proposition~\ref{prop=explicitexample} the equivalence between (1) and (5) of Theorem~\ref{thm=extkernel} ends the proof.
\end{proof}

\subsection{Induced $\RR$-linear maps and the coarse duality formula}\label{subsec=induced}
Let $\Gg,H$ be two groups, and let $\varphi\colon \Gg\to H$ be a group homomorphism. For each $p\in \NN$, $\varphi$ sends every $(\Gg,\gamma_p(\Gg))$-commutator to a  $(H,\gamma_p(H))$-commutator. Therefore, for each $p\in \NN$, $\varphi$ induces a map $\varphi_{p}\colon (\gamma_{p+1}(\Gg),d_{\scl_{\Gg,\gamma_p(\Gg)}})\to (\gamma_{p+1}(H),d_{\scl_{H,\gamma_p(H)}})$. This map $\varphi_p$ is a coarse homomorphism because it does not increase the corresponding $\scl$-almost-metrics. For every $q\in \NN_{\geq 2}$ we present the construction of  induced $\RR$-linear maps $T^{\varphi_{q-1,q}}\colon \Wcal(H,\gamma_{q-1}(H),\gamma_q(H))\to \Wcal(\Gg,\gamma_{q-1}(\Gg),\gamma_q(\Gg))$ and $S_{\varphi_{q-1,q}}$. Here, to define $S_{\varphi_{q-1,q}}$, we assume that $\Wcal(\Gg,\gamma_{q-1}(\Gg),\gamma_q(\Gg))$ and $\Wcal(H,\gamma_{q-1}(H),\gamma_q(H))$ are both finite dimensional; it is satisfied if $\Gg$ and $H$ are finitely generated (Corollary~\ref{cor=nilp}). We briefly discussed the case of $q=2$ in Subsection~\ref{subsec=intro_main}.

In the setting above, let $\ell=\Rdim\Wcal(\Gg,\gamma_{q-1}(\Gg),\gamma_q(\Gg))$ and $\ell'=\Rdim\Wcal(H,\gamma_{q-1}(H),\gamma_q(H))$. Then by Theorem~\ref{mthm=main}, the coarse kernels of $\iota_{(\Gg,\gamma_{q-1}(\Gg),\gamma_q(\Gg))}$ and $\iota_{(H,\gamma_{q-1}(H),\gamma_q(H))}$ are isomorphic to $(\ZZ^{\ell},\|\cdot\|_1)$ and $(\ZZ^{\ell'},\|\cdot\|_1)$ as coarse groups, respectively. Hence, they are isomorphic to $(\RR^{\ell},\|\cdot\|_1)$ and $(\RR^{\ell'},\|\cdot\|_1)$ as coarse groups, respectively.

\begin{prop}\label{prop=induced}
Let $\Gg$ and $H$ be groups, and let $\varphi \colon \Gg\to H$ be a group homomorphism. Let $q\in \NN_{\geq 2}$.
\begin{enumerate}[label=\textup{(\arabic*)}]
  \item The following map is well-defined:
\[
T \colon \Wcal(H,\gamma_{q-1}(H),\gamma_q(H))\to \Wcal(\Gg,\gamma_{q-1}(\Gg),\gamma_q(\Gg));\quad [\nug]\mapsto [\varphi^{\ast}\nug ].
\]
  \item Assume furthermore that $\Wcal(\Gg,\gamma_{q-1}(\Gg),\gamma_q(\Gg))$ and $\Wcal(H,\gamma_{q-1}(H),\gamma_q(H))$ are finite dimensional. Set $\mathbf{A}_{\Gg,q-1,q}$ and $\mathbf{A}_{H,q-1,q}$ as the coarse kernels \textup{(}as coarse subspaces\textup{)} of $\iota_{\Gg,\gamma_{q-1}(\Gg),\gamma_q(\Gg)}$ and $\iota_{H,\gamma_{q-1}(H),\gamma_q(H)}$. Then, the  coarse homomorphism $\varphi_{q}\colon (\gamma_{q+1}(\Gg),d_{\scl_{\Gg,\gamma_q(\Gg)}})\to (\gamma_{q+1}(H),d_{\scl_{H,\gamma_q(H)}})$ induces a coarse homomorphism $\mathbf{S}\colon \mathbf{A}_{\Gg,q-1,q}\to \mathbf{A}_{H,q-1,q}$.
\end{enumerate}
\end{prop}

The  existences of $\mathbf{A}_{\Gg,q-1,q}$ and $\mathbf{A}_{H,q-1,q}$ are ensured by Theorem~\ref{mthm=main}.

\begin{proof}[Proof of Proposition~$\ref{prop=induced}$]
Item (1) follows because $\varphi^{\ast}$ sends elements in $\HHH^1(\gamma_{q}(H))^H+i_{\gamma_{q}(H),\gamma_{q-1}(H)}^{\ast}\QQQ(\gamma_{q-1}(H))^H$ to those in $\HHH^1(\gamma_{q}(\Gg))^{\Gg}+i_{\gamma_{q}(\Gg),\gamma_{q-1}(\Gg)}^{\ast}\QQQ(\gamma_{q-1}(\Gg))^{\Gg}$. Item (2) follows from the universality of the coarse kernel $\mathbf{A}_{H,q-1,q}$. More precisely, fix representatives $A_{\Gg}\subseteq \gamma_{q+1}(\Gg)$ and $A_{H}\subseteq \gamma_{q+1}(H)$ of $\mathbf{A}_{\Gg,q-1,q}$ and $\mathbf{A}_{H,q-1,q}$, respectively. Then, $A_{\Gg}$ is $d_{\scl_{\Gg,\gamma_{q-1}(\Gg)}}$-bounded. Hence, $\varphi_{q}(A_{\Gg})=\varphi_{q-1}(A_{\Gg})$ is $d_{\scl_{H,\gamma_{q-1}(H)}}$-bounded. Therefore, $\varphi_q(A_{\Gg})$ is coarsely contained in $A_H$. Thus, $\varphi_q$ induces a  coarse homomorphism $\mathbf{S}\colon \mathbf{A}_{\Gg,q-1,q}\to \mathbf{A}_{H,q-1,q}$.
\end{proof}

\begin{defn}[induced $\RR$-linear maps by $\varphi$]\label{defn=induced}
Let $\Gg$ and $H$ be groups, and let $\varphi \colon \Gg\to H$ be a group homomorphism. Let $q\in \NN_{\geq 2}$.
\begin{enumerate}[label=(\arabic*)]
  \item We define the $\RR$-linear map $T^{\varphi_{q-1,q}}$$\colon \Wcal(H,\gamma_{q-1}(H),\gamma_q(H))\to \Wcal(\Gg,\gamma_{q-1}(\Gg),\gamma_q(\Gg))$ by the map defined in Proposition~\ref{prop=induced}~(1).
  \item Assume furthermore that $\Wcal(\Gg,\gamma_{q-1}(\Gg),\gamma_q(\Gg))$ and $\Wcal(H,\gamma_{q-1}(H),\gamma_q(H))$ are both finite dimensional. Set $\ell=\Rdim \Wcal(\Gg,\gamma_{q-1}(\Gg),\gamma_q(\Gg))$ and $\ell'=\Rdim \Wcal(H,\gamma_{q-1}(H),\gamma_q(H))$. Set $\mathbf{A}_{\Gg,q-1,q}$ and $\mathbf{A}_{H,q-1,q}$ as the coarse kernels  of $\iota_{\Gg,\gamma_{q-1}(\Gg),\gamma_q(\Gg)}$ and $\iota_{H,\gamma_{q-1}(H),\gamma_q(H)}$. Fix coarse isomorphisms $\mathbf{A}_{\Gg,q-1,q}\cong (\RR^{\ell},\|\cdot \|_1)$ and $\mathbf{A}_{H,q-1,q}\cong (\RR^{\ell'},\|\cdot \|_1)$. Then, the coarse homomorphism $\mathbf{S}\colon \mathbf{A}_{\Gg,q-1,q}\to \mathbf{A}_{H,q-1,q}$ constructed in Proposition~\ref{prop=induced}~(2) can be seen as a coarse homomorphism $\mathbf{S}\colon (\RR^{\ell},\|\cdot \|_1)\to (\RR^{\ell'},\|\cdot \|_1)$. Define $S_{\varphi_{q-1,q}}$ as the unique  representative of $\mathbf{S}\colon (\RR^{\ell},\|\cdot \|_1)\to (\RR^{\ell'},\|\cdot \|_1)$ being an $\RR$-linear map.
\end{enumerate}
\end{defn}
Here, in Definition~\ref{defn=induced}~(2), the existence and the uniqueness of $S_{\varphi_{q-1,q}}$ are both ensured by Lemma~\ref{lem=crushing}.

Strictly speaking, the $\RR$-linear map $S=S_{\varphi_{q-1,q}}$ in this formulation depends on the choices of coarse isomorphisms $\mathbf{A}_{\Gg,q-1,q}\cong (\RR^{\ell},\|\cdot \|_1)$ and $\mathbf{A}_{H,q-1,q}\cong (\RR^{\ell'},\|\cdot \|_1)$. If we choose such coarse isomorphisms in a certain manner, then we may take $S$ more concretely, as in the following remark.

\begin{rem}\label{rem=S}
In the setting of Definition~\ref{defn=induced}~(2), let $i_{\ell}\colon (\ZZ^{\ell},\|\cdot\|_1)\to (\RR^{\ell},\|\cdot\|_1)$ and $i_{\ell'}\colon (\ZZ^{\ell'},\|\cdot\|_1)\to (\RR^{\ell'},\|\cdot\|_1)$ be the inclusion maps. Let $\rho_{\ell}\colon  (\RR^{\ell},\|\cdot\|_1)\to (\ZZ^{\ell},\|\cdot\|_1)$ and $\rho_{\ell'}\colon (\RR^{\ell'},\|\cdot\|_1) \to (\ZZ^{\ell'},\|\cdot\|_1)$ be coarse inverses to $i_{\ell}$ and $i_{\ell'}$, respectively.
 By applying Theorem~\ref{thm=itte} to the triple $(\Gg,\gamma_{q-1}(\Gg),\gamma_q(\Gg))$, we obtain $\PPh^{\RR}$ and $\PPs$. Set $\PPh^{\RR}_{\Gg,q-1,q}=\PPh^{\RR}$ and $\PPs^{\RR}_{\Gg,q-1,q}=\PPs\circ \rho_{\ell}$. In a similar manner, we construct maps $\PPh^{\RR}_{H,q-1,q}$ and $\PPs^{\RR}_{H,q-1,q}$ for the triple $(H,\gamma_{q-1}(H),\gamma_q(H))$.
Then, by the proof of Proposition~\ref{prop=induced}~(2), the map
\[
S_0=\PPh^{\RR}_{H,q-1,q}\circ \varphi_q\circ \PPs^{\RR}_{\Gg,q-1,q}\colon (\RR^{\ell},\|\cdot\|_1)\to (\RR^{\ell'},\|\cdot\|_1)
\]
is a representative of the coarse homomorphism $\mathbf{S}\colon (\RR^{\ell},\|\cdot \|_1)\to (\RR^{\ell'},\|\cdot \|_1)$. Hence, the $\RR$-linear map $S_{\varphi,q-1,q}$ associated with $\PPh^{\RR}_{\Gg,q-1,q}\colon  (\RR^{\ell},\|\cdot \|_1)\stackrel{\cong}{\to}\mathbf{A}_{\Gg,q-1,q}$ and $\PPh^{\RR}_{H,q-1,q}\colon  (\RR^{\ell'},\|\cdot \|_1)\stackrel{\cong}{\to}\mathbf{A}_{H,q-1,q}$ is the map obtained by Lemma~\ref{lem=crushing} to $S_0$. See the diagram below, which commutes up to closeness.
\begin{equation}\label{eq=diagramch}
 \xymatrix@C=40pt{
(\RR^{\ell}, \| \cdot \|_1) \ar[d]_{S_{\varphi,q-1,q}}\ar[r]^-{\PPs^{\RR}_{\Gg,q-1,q}} & (\gamma_{q+1}(\Gg), d_{\scl_{\Gg, \gamma_q(\Gg)}}) \ar[r]^-{\iota_{\Gg,q-1,q}} \ar[d]^{\varphi_{q}} & (\gamma_{q+1}(\Gg), d_{\scl_{\Gg,\gamma_{q-1}(\Gg)}}) \ar[d]^{\varphi_{q-1}} \\
(\RR^{\ell'}, \| \cdot \|_1) \ar[r]^-{\PPs^{\RR}_{H,q-1,q}} & (\gamma_{q+1}(H), d_{\scl_{H, \gamma_q(H)}}) \ar[r]^-{\iota_{H,q-1,q}} & (\gamma_{q+1}(H), d_{\scl_{H,\gamma_{q-1}(H)}}).
}
\end{equation}
\end{rem}

The two $\RR$-linear maps $S_{\varphi_{q-1,q}}$ and $T^{\varphi_{q-1,q}}$ may be seen as induced maps to the setting of \emph{coarse subspaces of the groups} and \emph{function spaces on the groups}, respectively. As we mentioned in Subsection~\ref{subsec=intro_main} for $q=2$, we obtain the following \emph{coarse duality formula} between them.

\begin{thm}[coarse duality formula]\label{thm=duality}
We stick to the setting of Definition~$\ref{defn=induced}$~\textup{(2)}. Fix  coarse inverses $\rho_{\ell}$ and  $\rho_{\ell'}$ to the inclusion maps $i_{\ell}\colon (\ZZ^{\ell},\|\cdot\|_1)\to (\RR^{\ell},\|\cdot\|_1)$ and $i_{\ell'}\colon (\ZZ^{\ell'},\|\cdot\|_1)\to (\RR^{\ell'},\|\cdot\|_1)$, respectively. Fix tuples $(\nug_1,\ldots,\nug_{\ell})$ and $(\wf_1,\ldots,\wf_{\ell})$ as in Setting~$\ref{setting=main2}$ for $(\Gg,\gamma_{q-1}(\Gg),\gamma_q(\Gg))$; fix tuples $(\nug'_1,\ldots,\nug'_{\ell})$ and $(\wf'_1,\ldots,\wf'_{\ell})$ as in Setting~$\ref{setting=main2}$ for $(H,\gamma_{q-1}(H),\gamma_q(H))$. Let $\PPs^{\RR}_{\Gg,q-1,q}=\PPs_{\Gg}\circ \rho_{\ell}$, where $\PPs_{\Gg}$ is the map constructed in Theorem~$\ref{thm=itte}$ for the tuples $(\nug_1,\ldots,\nug_{\ell})$ and $(\wf_1,\ldots,\wf_{\ell})$; let $\PPs^{\RR}_{H,q-1,q}=\PPs_H\circ \rho_{\ell}$, where $\PPs_H$ is constructed in Theorem~$\ref{thm=itte}$ for the tuples $(\nug'_1,\ldots,\nug'_{\ell})$ and $(\wf'_1,\ldots,\wf'_{\ell})$.  Let $S_{\varphi_{q-1},q}\colon \RR^{\ell}\to \RR^{\ell'}$ be the $\RR$-linear map defined in Definition~$\ref{defn=induced}$~\textup{(2)} with respect to the identifications of coarse kernels to $\RR^{\ell}$ and $\RR^{\ell'}$ by $\PPs^{\RR}_{\Gg,q-1,q}$ and $\PPs^{\RR}_{H,q-1,q}$, respectively. Let $\tilde{T}^{\varphi_{q-1,q}}\colon \RR^{\ell'}\to \RR^{\ell}$ be the $\RR$-linear map defined in Definition~$\ref{defn=induced}$~\textup{(1)} with respect to the identifications of $\Wcal(\Gg,\gamma_{q-1}(\Gg),\gamma_q(\Gg))$ and $\Wcal(H,\gamma_{q-1}(H),\gamma_q(H))$ to $\RR^{\ell}$ and $\RR^{\ell'}$ by the bases $([\nug_1],\ldots,[\nug_{\ell}])$ and $([\nug'_1],\ldots,[\nug'_{\ell'}])$, respectively.
Define two maps $\langle \cdot | \cdot \rangle_{\Gg}\colon \RR^{\ell}\times \RR^{\ell}\to \RR$ and $\langle \cdot | \cdot \rangle_{H}\colon \RR^{\ell'}\times \RR^{\ell'}\to \RR$ in the following manner: for every $\va\in \RR^{\ell}$, $\va'\in \RR^{\ell'}$ and for every $\vb=(b_1,\ldots,b_{\ell})\in \RR^{\ell}$, $\vb'=(b'_1,\ldots,b'_{\ell})\in \RR^{\ell'}$, set
\[
\langle  \vb| \va\rangle_{\Gg} =(b_1\nug_1+\cdots +b_{\ell}\nug_{\ell})\Bigl(\PPs^{\RR}_{\Gg,q-1,q}(\va)\Bigr)\ \ \textrm{and}\ \
\langle  \vb'| \va'\rangle_{H} =(b'_1\nug'_1+\cdots +b'_{\ell}\nug'_{\ell'})\Bigl(\PPs^{\RR}_{H,q-1,q}(\va')\Bigr).
\]
Then, the following hold  true.
\begin{enumerate}[label=\textup{(\arabic*)}]
\item For every $\vb'\in \RR^{\ell'}$, we have the following closeness relation:
\begin{equation}\label{eq=duality}
\langle  \vb'  | S_{\varphi_{q-1,q}}\ \cdot\ \rangle_{H} \approx \langle  \tilde{T}^{\varphi_{q-1,q}} \vb'  | \ \cdot\ \rangle_{\Gg}.
\end{equation}
\item Let $M_{S_{\varphi}}$ and $M_{T^{\varphi}}$ be the representation matrices of $S_{\varphi_{q-1,q}}\colon \RR^{\ell}\to \RR^{\ell'}$ and $\tilde{T}^{\varphi_{q-1,q}}\colon \RR^{\ell'}\to \RR^{\ell}$, respectively. Then,
\[
M_{T^{\varphi}}={}^t M_{S_{\varphi}}
\]
holds. Here, for a matrix $M$, the symbol ${}^tM$ means the transpose of $M$.
\end{enumerate}
\end{thm}
Here, \eqref{eq=duality} means that
\[
\sup_{\va\in \RR^{\ell}}\left|\langle  \vb' | S_{\varphi_{q-1,q}}\va\rangle_H -\langle  \tilde{T}^{\varphi_{q-1,q}} \vb' | \va \rangle_{\Gg} \right| <\infty.
\]

We note that $\langle \cdot | \cdot \rangle_{\Gg}$ or $\langle \cdot | \cdot \rangle_{H}$ is \emph{not} $\RR$-bilinear in general. Nevertheless, for every $\va\in \RR^{\ell}$ and for every $\va'\in \RR^{\ell'}$, the two maps $\langle \cdot | \va \rangle_{\Gg}\colon \RR^{\ell}\to \RR$ and $\langle \cdot | \va' \rangle_{H}\colon \RR^{\ell'}\to \RR$ are $\RR$-linear.

\begin{proof}[Proof of Theorem~$\ref{thm=duality}$]
Fix $\vb'=(b'_1,\ldots,b'_{\ell'})\in \RR^{\ell'}$. Set $\nug'=b'_1\nug'_1+\cdots +b'_{\ell'}\nug'_{\ell'}$. We use the symbols and diagram \eqref{eq=diagramch} appearing in Remark~\ref{rem=S} for the setting associate with the tuples $(\nug_1,\ldots,\nug_{\ell})$, $(\wf_1,\ldots,\wf_{\ell})$, $(\nug'_1,\ldots,\nug'_{\ell'})$ and $(\wf'_1,\ldots,\wf'_{\ell'})$. Let $\va\in \RR^{\ell}$. Then by the definition of $\tilde{T}^{\varphi_{q-1,q}}$, we have
\begin{align*}
\langle \tilde{T}^{\varphi_{q-1,q}} \vb'| \va\rangle_{\Gg} &=\bigl(\varphi_q^{\ast}(b'_1\nug'_1+\cdots +b'_{\ell'}\nug'_{\ell'})\bigr)\Bigl(\PPs^{\RR}_{\Gg,q-1,q}(\va)\Bigr)\\
&=\bigl(\varphi_q^{\ast}\nug'\bigr)\Bigl(\PPs^{\RR}_{\Gg,q-1,q}(\va)\Bigr)= \nug'\Bigl(\varphi(\PPs^{\RR}_{\Gg,q-1,q}(\va))\Bigr).
\end{align*}
By Remark~\ref{rem=S}, $S_{\varphi_{q-1,q}}\colon (\RR^{\ell},\|\cdot\|_1)\to (\RR^{\ell'},\|\cdot\|_1)$ is close to $S_0=\PPh^{\RR}_{H,q-1,q}\circ \varphi_q\circ \PPs^{\RR}_{\Gg,q-1,q}$. By Theorem~\ref{thm=Bavard} and Theorem~\ref{mthm=main}, the map $(\RR^{\ell'},\|\cdot\|_1)\to (\RR,|\cdot|)$; $\va'\mapsto \langle \vb' | \va'\rangle_H$ is a coarse homomorphism. Hence,
\begin{align*}
\langle \vb' | S_{\varphi_{q-1,q}} \va \rangle_H & \mathrel{\sim_{\vb'}}
\langle \vb' | S_0 \va \rangle_H \\
& =
\nug'\Bigl((\PPs^{\RR}_{H,q-1,q}\circ  \PPh^{\RR}_{H,q-1,q}) (\varphi_{q-1}(\PPs^{\RR}_{\Gg,q-1,q}(\va)))\Bigr) \\
& \mathrel{\sim_{\vb'}}
\nug'\Bigl( \varphi(\PPs^{\RR}_{\Gg,q-1,q}(\va))\Bigr),
\end{align*}
where $\sim_{\vb'}$ means that the difference between two real numbers uniformly bounded in such a way that the bound is independent of  $\va$ (the bound may depend on $\vb'$). Here, recall Theorem~\ref{thm=itte}~(1). Therefore,  \eqref{eq=duality} holds.

Finally, we will deduce (2)  from (1). We can define two forms $\langle \cdot | \cdot\rangle_{\Gg}^{\sharp}$ and $\langle \cdot | \cdot\rangle_{H}^{\sharp}$ as follows:
\begin{align*}
\langle \cdot | \cdot\rangle_{\Gg}^{\sharp}&\colon \RR^{\ell}\times \RR^{\ell}\to \RR;\quad (\va,\vb)\mapsto \langle \vb | \va \rangle_{\Gg}^{\sharp}=\lim_{u\to \infty}\frac{1}{u} \langle \vb | u\va \rangle_{\Gg};\\
\langle \cdot | \cdot\rangle_{H}^{\sharp}&\colon \RR^{\ell'}\times \RR^{\ell'}\to \RR;\quad (\va',\vb')\mapsto \langle \vb' | \va' \rangle_{H}^{\sharp}=\lim_{u\to \infty}\frac{1}{u} \langle \vb' | u\va' \rangle_{H}.
\end{align*}
Then, we can show that $\langle \cdot | \cdot\rangle_{\Gg}^{\sharp}$ and $\langle \cdot | \cdot\rangle_{H}^{\sharp}$ are both genuine $\RR$-bilinear forms. Therefore, the coarse duality formula \eqref{eq=duality} for $\langle \cdot | \cdot\rangle_{\Gg}$ and $\langle \cdot | \cdot\rangle_{H}$ implies the following genuine duality formula for $\langle \cdot | \cdot\rangle_{\Gg}^{\sharp}$ and $\langle \cdot | \cdot\rangle_{H}^{\sharp}$: for every $\vb'\in \RR^{\ell'}$ and for every $\va\in \RR^{\ell}$,
\begin{equation}\label{eq=trueduality}
\langle  \vb'  | S_{\varphi_{q-1,q}}\va \rangle_{H}^{\sharp} = \langle  \tilde{T}^{\varphi_{q-1,q}} \vb'  | \va \rangle_{\Gg}^{\sharp}.
\end{equation}
By our choices, we can moreover show that $\langle \cdot | \cdot\rangle_{\Gg}^{\sharp}$ and $\langle \cdot | \cdot\rangle_{H}^{\sharp}$ are the standard forms; here recall the limit formula (Theorem~\ref{thm=valueofcore}). Now, \eqref{eq=trueduality} implies (2). This completes our proof.
\end{proof}

\subsection{Crushing theorem}\label{subsec=crushing}

The following theorem generalizes Theorem~\ref{thm=crushingq2}.

\begin{thm}[crushing theorem] \label{thm=crushing}
Let $q\in \NN_{\geq 2}$. Let $\Gg$ and $H$ be groups such that $\Wcal(\Gg, \gamma_{q-1}(\Gg),\gamma_q(\Gg))$ and $\Wcal(H, \gamma_{q-1}(H), \gamma_q(H))$ are both finite dimensional. Let  $\varphi \colon G \to H$ be a group homomorphism, and let $T^{\varphi_{q-1,q}}\colon \Wcal(H, \gamma_{q-1}(H),\gamma_q(H))\to \Wcal(\Gg, \gamma_{q-1}(\Gg),\gamma_q(\Gg))$ be the induced map by $\varphi$. Assume that $T^{\varphi_{q-1,q}}$ is not surjective. Set $t^{\varphi}\in \NN$ as $t^{\varphi}=\Rdim \mathrm{Coker}(T^{\varphi_{q-1,q}})$, where $\mathrm{Coker}(T^{\varphi_{q-1,q}})$ means the cokernel of $T^{\varphi_{q-1,q}}$.
Then, there exists $X\subseteq \gamma_{q+1}(\Gg)$ that satisfies the following three conditions:
\begin{enumerate}[label=\textup{(\arabic*)}]
  \item the set $(X,d_{\scl_{\Gg, \gamma_{q}(\Gg)}})$ is isomorphic to $(\ZZ^{t^{\varphi}},\|\cdot\|_1)$ as a coarse group. In particular, $X$ is \emph{not} $d_{\scl_{\Gg,\gamma_{q}(\Gg)}}$-bounded;
  \item the set $X$ is $d_{\scl_{\Gg,\gamma_{q-1}(\Gg)}}$-bounded; and
  \item $\varphi(X)$ is $d_{\scl_{H,\gamma_{q}(H)}}$-bounded.
\end{enumerate}
\end{thm}

\begin{proof}
Set $\ell=\Rdim \Wcal(\Gg, \gamma_{q-1}(\Gg),\gamma_q(\Gg))$ and $\ell'=\Rdim \Wcal(H, \gamma_{q-1}(H),\gamma_q(H))$. Use the setting of Theorem~\ref{thm=duality}, and take $\RR$-linear maps $S_{\varphi_{q-1,q}}\colon \RR^{\ell}\to \RR^{\ell'}$ and $\tilde{T}^{\varphi_{q-1,q}}\colon \RR^{\ell'}\to \RR^{\ell}$. By the coarse duality formula (Theorem~\ref{thm=duality}), we have
\[
\Rdim \mathrm{Ker}(S_{\varphi_{q-1,q}})=\Rdim \mathrm{Coker}(\tilde{T}^{\varphi_{q-1,q}})=t^{\varphi}.
\]
Define $X$ by $X=\PPs^{\RR}_{\Gg,q-1,q}(\Ker (S_{\varphi_{q-1,q}}))$. In what follows, we will prove (1)--(3). Since $\PPs^{\RR}_{\Gg,q-1,q}(\RR^{\ell})$ is $d_{\scl_{\Gg,\gamma_{q-1}(\Gg)}}$-bounded, (2) holds. Since $\PPs^{\RR}_{H,q-1,q}\circ S_{\varphi}$ is close to $\varphi_q\circ \PPs^{\RR}_{\Gg,q-1,q}$, (3) holds. Since  $(X,d_{\scl_{\Gg, \gamma_q(\Gg)}})\cong(\Ker(S_{\varphi_{q-1,q}}),\|\cdot\|_1)$ as a coarse group, $(X,d_{\scl_{\Gg, \gamma_q(\Gg)}})$ is isomorphic to $(\ZZ^{t^{\varphi}},\|\cdot\|_1)$ as a coarse subgroup. Hence, (1) holds.
\end{proof}

\begin{proof}[Proof of Theorem~$\ref{thm=crushingq2}$]
This is now immediate from Theorem~\ref{thm=crushing} for the case of $q=2$ (recall also Theorem~\ref{thm=asdimZ}). Indeed, under the assumption of Theorem~\ref{thm=crushingq2}, we have $t^{\varphi}\geq \ell-\ell' >0$.
\end{proof}

As we mentioned in Subsection~\ref{subsec=intro_main}, we will extend the study of the $\RR$-linear maps $T^{\varphi_{q-1,q}}$ an $S_{\varphi_{q-1,q}}$ induced by $\varphi$ in a much broader framework; this study will show up as a forthcoming work.

\section*{Acknowledgment}
The authors thank Arielle Leitner and Federico Vigolo for helpful comments. The first-named author and the fifth-named author are partially supported by JSPS KAKENHI Grant Number JP21K13790 and JP21K03241, respectively.
The second-named author is supported by JST-Mirai Program Grant Number JPMJMI22G1.
The third-named author is supported by JSPS KAKENHI Grant Number JP23KJ1938 and JP23K12971.
The fourth-named author is partially supported by JSPS KAKENHI Grant Number JP19K14536 and JP23K12975.

\bibliography{reference}

\begin{thebibliography}{10}

\bibitem{Andreadakis}
S.~Andreadakis.
\newblock On the automorphisms of free groups and free nilpotent groups.
\newblock {\em Proc. London Math. Soc. (3)}, 15:239--268, 1965.

\bibitem{Bavard}
C.~Bavard.
\newblock Longueur stable des commutateurs.
\newblock {\em Enseign. Math. (2)}, 37(1-2):109--150, 1991.

\bibitem{BD}
G.~Bell and A.~Dranishnikov.
\newblock Asymptotic dimensions.
\newblock {\em Topology Appl.}, 155(11):1265--1296, 2008.

\bibitem{BBF16}
M.~Bestvina, K.~Bromberg, and K.~Fujiwara.
\newblock Stable commutator length on mapping class groups.
\newblock {\em Ann. Inst. Fourier (Grenoble)}, 66(3):871--898, 2016.

\bibitem{BF1992}
M.~Bestvina and M.~Feighn.
\newblock A combination theorem for negatively curved groups.
\newblock {\em J. Differential Geom.}, 35(1):85--101, 1992.

\bibitem{BF}
M.~Bestvina and K.~Fujiwara.
\newblock Bounded cohomology of subgroups of mapping class groups.
\newblock {\em Geom. Topol.}, 6:69--89, 2002.

\bibitem{BK13}
M.~Brandenbursky and J.~K\k{e}dra.
\newblock On the autonomous metric on the group of area-preserving
  diffeomorphisms of the 2-disc.
\newblock {\em Algebr. Geom. Topol.}, 13(2):795--816, 2013.

\bibitem{BucherMonod}
M.~Bucher and N.~Monod.
\newblock The bounded cohomology of {$\rm SL_2$} over local fields and
  {$S$}-integers.
\newblock {\em Int. Math. Res. Not. IMRN}, (6):1601--1611, 2019.

\bibitem{BurgerMozes}
M.~Burger and S.~Mozes.
\newblock Finitely presented simple groups and products of trees.
\newblock {\em C. R. Acad. Sci. Paris S\'{e}r. I Math.}, 324(7):747--752, 1997.

\bibitem{Calegari}
D.~Calegari.
\newblock {\em scl}, volume~20 of {\em MSJ Memoirs}.
\newblock Mathematical Society of Japan, Tokyo, 2009.

\bibitem{Calegari09}
D.~Calegari.
\newblock Stable commutator length is rational in free groups.
\newblock {\em J. Amer. Math. Soc.}, 22(4):941--961, 2009.

\bibitem{CaFu10}
D.~Calegari and K.~Fujiwara.
\newblock Stable commutator length in word-hyperbolic groups.
\newblock {\em Groups Geom. Dyn.}, 4(1):59--90, 2010.

\bibitem{CZ08}
D.~Calegari and D.~Zhuang.
\newblock Large scale geometry of commutator subgroups.
\newblock {\em Algebr. Geom. Topol.}, 8(4):2131--2146, 2008.

\bibitem{Chen20}
L.~Chen.
\newblock Scl in graphs of groups.
\newblock {\em Invent. Math.}, 221(2):329--396, 2020.

\bibitem{EpsteinFujiwara}
D.~B.~A. Epstein and K.~Fujiwara.
\newblock The second bounded cohomology of word-hyperbolic groups.
\newblock {\em Topology}, 36(6):1275--1289, 1997.

\bibitem{FLM2}
F.~Fournier-Facio, C.~L\"{o}h, and M.~Moraschini.
\newblock Bounded cohomology and binate groups.
\newblock {\em J. Aust. Math. Soc.}, 115(2):204--239, 2023.

\bibitem{FLM1}
F.~Fournier-Facio, C.~Loh, and M.~Moraschini.
\newblock Bounded cohomology of finitely presented groups: vanishing,
  non-vanishing, and computability.
\newblock {\em Ann. Sc. Norm. Super. Pisa Cl. Sci. (5)}, 25(2):1169--1202,
  2024.

\bibitem{FW}
F.~Fournier-Facio and R.~D. Wade.
\newblock {\rm {A}ut}-invariant quasimorphisms on groups.
\newblock {\em Trans. Amer. Math. Soc.}, 376(10):7307--7327, 2023.

\bibitem{Frigerio}
R.~Frigerio.
\newblock {\em Bounded cohomology of discrete groups}, volume 227 of {\em
  Mathematical Surveys and Monographs}.
\newblock American Mathematical Society, Providence, RI, 2017.

\bibitem{GenevoisHorbez}
A.~Genevois and C.~Horbez.
\newblock Acylindrical hyperbolicity of automorphism groups of infinitely ended
  groups.
\newblock {\em J. Topol.}, 14(3):963--991, 2021.

\bibitem{Gr}
M.~Gromov.
\newblock Volume and bounded cohomology.
\newblock {\em Inst. Hautes \'{E}tudes Sci. Publ. Math.}, (56):5--99 (1983),
  1982.

\bibitem{Gromov1993}
M.~Gromov.
\newblock Asymptotic invariants of infinite groups.
\newblock In {\em Geometric group theory, {V}ol. 2 ({S}ussex, 1991)}, volume
  182 of {\em London Math. Soc. Lecture Note Ser.}, pages 1--295. Cambridge
  Univ. Press, Cambridge, 1993.

\bibitem{HL}
N.~Heuer and C.~L\"{o}h.
\newblock Simplicial volume of one-relator groups and stable commutator length.
\newblock {\em Algebr. Geom. Topol.}, 22(4):1615--1661, 2022.

\bibitem{Ivanov12}
N.~V. Ivanov.
\newblock Leray theorems in bounded cohomology theory.
\newblock {\em \emph{preprint, arXiv:2012.08038v1}}, 2020.

\bibitem{Johnson}
D.~Johnson.
\newblock The structure of the {T}orelli group. {II}. {A} characterization of
  the group generated by twists on bounding curves.
\newblock {\em Topology}, 24(2):113--126, 1985.

\bibitem{KK}
M.~Kawasaki and M.~Kimura.
\newblock {$\hat G$}-invariant quasimorphisms and symplectic geometry of
  surfaces.
\newblock {\em Israel J. Math.}, 247(2):845--871, 2022.

\bibitem{KKMMM}
M.~Kawasaki, M.~Kimura, S.~Maruyama, T.~Matsushita, and M.~Mimura.
\newblock The space of non-extendable quasimorphisms.
\newblock {\em \emph{to appear in} Algebr. Geom. Topol.,
  \emph{arXiv:2107.08571v5}}, 2021.

\bibitem{KKMMMgmB}
M.~Kawasaki, M.~Kimura, S.~Maruyama, T.~Matsushita, and M.~Mimura.
\newblock Invariant quasimorphisms and generalized mixed bavard duality.
\newblock {\em \emph{preprint, arXiv:2406.04319v1}}, 2024.

\bibitem{KKMMMsurvey}
M.~Kawasaki, M.~Kimura, S.~Maruyama, T.~Matsushita, and M.~Mimura.
\newblock Survey on invariant quasimorphisms and stable mixed commutator
  length.
\newblock {\em Topology Proc.}, 64:129--174, 2024.

\bibitem{KKMM1}
M.~Kawasaki, M.~Kimura, T.~Matsushita, and M.~Mimura.
\newblock Bavard's duality theorem for mixed commutator length.
\newblock {\em Enseign. Math.}, 68(3-4):441--481, 2022.

\bibitem{KKMM2}
M.~Kawasaki, M.~Kimura, T.~Matsushita, and M.~Mimura.
\newblock Commuting symplectomorphisms on a surface and the flux homomorphism.
\newblock {\em Geom. Funct. Anal.}, 33(5):1322--1353, 2023.

\bibitem{KM}
M.~Kawasaki and S.~Maruyama.
\newblock On boundedness of characteristic class via quasi-morphism.
\newblock {\em \emph{to appear in} J. Topol. Anal., \emph{arXiv:2012.10612v2}},
  2020.

\bibitem{KO19}
M.~Kawasaki and R.~Orita.
\newblock Disjoint superheavy subsets and fragmentation norms.
\newblock {\em J. Topol. Anal.}, 13(2):443--468, 2021.

\bibitem{Kotschick}
D.~Kotschick.
\newblock Quasi-homomorphisms and stable lengths in mapping class groups.
\newblock {\em Proc. Amer. Math. Soc.}, 132(11):3167--3175, 2004.

\bibitem{LV}
A.~Leitner and F.~Vigolo.
\newblock {\em An invitation to coarse groups}, volume 2339 of {\em Lecture
  Notes in Mathematics}.
\newblock Springer, Cham, [2023] \copyright 2023.

\bibitem{MKSbook}
W.~Magnus, A.~Karrass, and D.~Solitar.
\newblock {\em Combinatorial group theory}.
\newblock Dover Publications, Inc., Mineola, NY, second edition, 2004.
\newblock Presentations of groups in terms of generators and relations.

\bibitem{Mannbook}
A.~Mann.
\newblock {\em How groups grow}, volume 395 of {\em London Mathematical Society
  Lecture Note Series}.
\newblock Cambridge University Press, Cambridge, 2012.

\bibitem{MR}
K.~Mann and C.~Rosendal.
\newblock Large-scale geometry of homeomorphism groups.
\newblock {\em Ergodic Theory Dynam. Systems}, 38(7):2748--2779, 2018.

\bibitem{MMM}
S.~Maruyama, T.~Matsushita, and M.~Mimura.
\newblock Invariant quasimorphisms for groups acting on the circle and
  non-equivalence of {SCL}.
\newblock {\em \emph{to appear in} Israel J. Math.\emph{, arXiv:2203.09221v3}},
  2022.

\bibitem{MatsumotoMorita}
S.~Matsumoto and S.~Morita.
\newblock Bounded cohomology of certain groups of homeomorphisms.
\newblock {\em Proc. Amer. Math. Soc.}, 94(3):539--544, 1985.

\bibitem{Mineyev}
I.~Mineyev.
\newblock Straightening and bounded cohomology of hyperbolic groups.
\newblock {\em Geom. Funct. Anal.}, 11(4):807--839, 2001.

\bibitem{Monodbook}
N.~Monod.
\newblock {\em Continuous bounded cohomology of locally compact groups}, volume
  1758 of {\em Lecture Notes in Mathematics}.
\newblock Springer-Verlag, Berlin, 2001.

\bibitem{Mon04}
N.~Monod.
\newblock Stabilization for {${\rm SL}_n$} in bounded cohomology.
\newblock In {\em Discrete geometric analysis}, volume 347 of {\em Contemp.
  Math.}, pages 191--202. Amer. Math. Soc., Providence, RI, 2004.

\bibitem{Monod2021}
N.~Monod.
\newblock Lamplighters and the bounded cohomology of {T}hompson's group.
\newblock {\em Geom. Funct. Anal.}, 32(3):662--675, 2022.

\bibitem{MN21}
N.~Monod and S.~Nariman.
\newblock Bounded and unbounded cohomology of homeomorphism and diffeomorphism
  groups.
\newblock {\em Invent. Math.}, 232(3):1439--1475, 2023.

\bibitem{MS04}
N.~Monod and Y.~Shalom.
\newblock Cocycle superrigidity and bounded cohomology for negatively curved
  spaces.
\newblock {\em J. Differential Geom.}, 67(3):395--455, 2004.

\bibitem{MR21}
M.~Moraschini and G.~Raptis.
\newblock Amenability and acyclicity in bounded cohomology.
\newblock {\em Rev. Mat. Iberoam.}, 39(6):2371--2404, 2023.

\bibitem{Newman}
B.~B. Newman.
\newblock Some results on one-relator groups.
\newblock {\em Bull. Amer. Math. Soc.}, 74:568--571, 1968.

\bibitem{Osin15}
D.~Osin.
\newblock On acylindrical hyperbolicity of groups with positive first
  {$\ell^2$}-{B}etti number.
\newblock {\em Bull. Lond. Math. Soc.}, 47(5):725--730, 2015.

\bibitem{OsinAH}
D.~Osin.
\newblock Acylindrically hyperbolic groups.
\newblock {\em Trans. Amer. Math. Soc.}, 368(2):851--888, 2016.

\bibitem{MR1402300}
J.-P. Otal.
\newblock Le th\'{e}or\`eme d'hyperbolisation pour les vari\'{e}t\'{e}s
  fibr\'{e}es de dimension 3.
\newblock {\em Ast\'{e}risque}, (235):x+159, 1996.

\bibitem{Py06}
P.~Py.
\newblock Quasi-morphismes et invariant de {C}alabi.
\newblock {\em Ann. Sci. \'{E}cole Norm. Sup. (4)}, 39(1):177--195, 2006.

\bibitem{Roe}
J.~Roe.
\newblock {\em Lectures on coarse geometry}, volume~31 of {\em University
  Lecture Series}.
\newblock American Mathematical Society, Providence, RI, 2003.

\bibitem{Strebel}
R.~Strebel.
\newblock Appendix. {S}mall cancellation groups.
\newblock In {\em Sur les groupes hyperboliques d'apr\`es {M}ikhael {G}romov
  ({B}ern, 1988)}, volume~83 of {\em Progr. Math.}, pages 227--273.
  Birkh\"{a}user Boston, Boston, MA, 1990.

\bibitem{ThompsonHemmila}
J.~Thompson and D.~Hemmila.
\newblock Collapsing {M}aps and {Q}uasi-{I}sometries.
\newblock {\em \emph{preprint, arXiv:2202.05915v1}}, 2022.

\bibitem{Thu1986}
W.~P. Thurston.
\newblock Hyperbolic structures on 3-manifolds, {II}: Surface groups and
  3-manifolds which fiber over the circle.
\newblock {\em \emph{arXiv:math/9801045}}, 1986.

\end{thebibliography}
\bibliographystyle{abbrv}
\end{document}